\documentclass[reqno,10pt]{amsart}
\usepackage[margin=1in]{geometry}
\usepackage{amsthm,amsmath,amssymb,dsfont}
\usepackage{mathrsfs,amsfonts,functan,extarrows}

\usepackage{hyperref}
\hypersetup{hidelinks}

 \usepackage{mathtools}

\usepackage[utf8]{inputenc}  
\usepackage[T1]{fontenc}
\usepackage{lmodern}         
\input{glyphtounicode}

\usepackage{xcolor}
\usepackage{cite}
\usepackage{indentfirst, latexsym, amssymb, enumerate,amsmath,graphicx}
\usepackage{float}
\usepackage{relsize}
\usepackage{marginnote}
\usepackage{stmaryrd}
\usepackage{esint}
\usepackage{graphicx}
\usepackage{bm}
\usepackage{caption}
\usepackage{subfigure}
\usepackage{cite}
\usepackage{graphicx}
\usepackage{subfigure}
\usepackage{float}
\usepackage{paralist}
\usepackage{indentfirst}
\usepackage{cite}
\allowdisplaybreaks 
\theoremstyle{plain}
\newtheorem{theorem}{Theorem}[section]
\newtheorem{lemma}[theorem]{Lemma}
\newtheorem{coro}[theorem]{Corollary}
\newtheorem{prop}{Proposition}[section]
\newtheorem{remark}{Remark}[section]
\newtheorem{defn}{Definition}[section]

\numberwithin{equation}{section}

\definecolor{red}{rgb}{1.00,0.00,0.00}
\definecolor{blue}{rgb}{0.00,0.00,0.63}
\definecolor{black}{rgb}{0.00,0.00,0.00}
\definecolor{purple}{rgb}{0.00,1.00,0.00}
\definecolor{pink}{rgb}{0.95,0.01,0.08}
\newcommand{\blue}{\color{blue}}

\theoremstyle{remark}

\numberwithin{equation}{section}

\DeclareMathOperator{\dive}{div}

\usepackage{appendix}
\usepackage{xcolor}

\def\blue#1{\textcolor{blue}{#1}}

\begin{document}


\title[Partially dissipative  hyperbolic systems]{Sharp decay characterization for
partially dissipative hyperbolic systems of balance laws}

\author{Ling-Yun Shou}
\author{Jiang Xu}
\author{Ping Zhang}

\date{}


\begin{abstract}\setlength{\baselineskip}{1.1\baselineskip}
The partially dissipative systems that characterize many physical phenomena were first pointed out by Godunov (1961), then investigated by Friedrichs-Lax (1971) who introduced the convex entropy, and later by Shizuta-Kawashima (1984,1985) who initiated a simple sufficient criterion ensuring the global existence of smooth solutions and their large-time asymptotics. There has been remarkable progress in the past several decades, through various different attempts. However, the decay character theory for partially dissipative hyperbolic systems remains largely open, as the Fourier transform of Green's function is generally not explicit in multi-dimensions. In this paper, we provide a positive answer to the open question by means of the general $L^p$ energy method.
Precisely, a new {\emph{effective quantity}} $\Psi(t,x)$ motivated by the compressible Euler system with damping is introduced, which enables us to capture 
 leading diffusion profiles of the large-time behavior in the spirit of the Chapman-Enskog
expansion. 
Consequently, we prove that the solutions approach the constant equilibrium state in the $\dot{\!B}^{\sigma}_{p,1}$-norm at the rate $t^{-(\sigma-\sigma_1)/2}$ as $t\rightarrow\infty$, and the corresponding norm of dissipative components decays at the enhanced rate $t^{-(\sigma-\sigma_1+1)/2}$, where the boundedness assumption in the  $\dot{B}^{\sigma_1}_{p,\infty} (-d/p\leq \sigma_1<d/p-1$)-norm of the low frequencies of conservative components is not only sufficient, but also necessary to achieve those upper bounds of decay estimates. Furthermore, both upper and lower bounds for time-decay estimates are obtained if and only if the low-frequency part of $\Psi_0(x)$ (the initial effective quantity)  is bounded in a non-trivial subset of $\dot{B}^{\sigma_1}_{p,\infty}$.



\end{abstract}


 \date{}

\maketitle


\tableofcontents

\section{Introduction}

\subsection{Problem formulation}
In this paper, we consider general $n$-component dissipative hyperbolic systems of balance laws, which take the form
\begin{equation}
\partial_{t} U + \sum_{i=1}^d \partial_{x_{i}} F^{i}(U)=G(U), \label{GEQSYM}
\end{equation}
where $U=U(t,x)$ is the unknown vector-valued function of $(t,x)\in \mathbb{R}_+\times\mathbb{R}^d$ ($d\geq1$) taking values in an open convex set $\mathcal{O}_{U}\subset\mathbb{R}^n$ ($n\geq2$), and $F^i$ and $G$ are given vector-valued smooth functions on $\mathcal{O}_{U}$. System \eqref{GEQSYM} is supplemented
with the initial data
\begin{equation}\label{data}
U(0,x)=U_{0}(x), \quad x\in \mathbb{R}^d.
\end{equation}

Note that in the absence of the source term $G(U)$, the system \eqref{GEQSYM} is reduced to the hyperbolic conservation laws in continuum physics, we refer to \cite{dafermos1,kato1,majda1,serre1,Whitham}. Under rather general conditions, for example, whenever  \eqref{GEQSYM} is Friedrichs' symmetrizable, \eqref{GEQSYM}-\eqref{data} admits a unique classical solution $U\in\mathcal{C}([0,T];H^{s}(\mathbb{R}^d))\cap
\mathcal{C}^{1}([0,T];H^{s-1}(\mathbb{R}^d))$ with $s>d/2+1$ on some time interval $[0,T]$. At the same time, it is well known that even for arbitrarily small and smooth initial data \eqref{data}, there is no global
continuation for these smooth solutions, which can develop singularities, shocks, or blow up in finite time. On the other hand, the system \eqref{GEQSYM} with $G(U)$ typically
governs non-equilibrium processes in physics for media with
hyperbolic response, as, for example, in gas dynamics (see \cite{wangyang1}). The system also
arises in the numerical simulation of conservation laws by relaxation
schemes (see \cite{ADN,JinXin1} and references cited therein). In those
applications, $G(U)$ has, or can be transformed by a
linear transformation into, the form
\begin{equation}\label{GU}
\begin{aligned}
   G(U)=\begin{pmatrix}
  0 \\
   g(U)
   \end{pmatrix},
\end{aligned}
\end{equation}
where $0\in \mathbb{R}^{n_{1}}$ and $g(U)\in \mathbb{R}^{n_{2}}$ with $n_{1}+n_{2}=n\,(n_{1}\neq 0)$. As observed, the source term  $G(U)$ does not appear in all components of \eqref{GEQSYM}. In fact, it does not matter for the local well-posedness; however, it may induce a dissipative effect, which completes with the hyperbolicity. A perfect example is
the isentropic Euler system with damping.
It was shown that
the system admits  global-in-time  classical solutions, at least for some restricted classes of initial data (see, for example, \cite{s1}).

\vspace{1mm}

In the following, we recall analytical results in the literature on the Cauchy problem to \eqref{GEQSYM}-\eqref{data}, which can be divided into two main subsets, the two directions being related but complementary.
We refer to a broad overview of hyperbolic conservation law theory by Dafermos \cite{dafermos1} or the recent survey by Danchin \cite{danchinnote2}.
\begin{itemize}
\item \textbf{Sobolev framework.} Let us point out that the coupling condition that is now referred to as the [SK] condition, was discovered for the first time by Kawashima and Shizuta \cite{SK} such that the dissipation present in the second block of \eqref{GU} has effect also to the first block of equation, which
plays a key role for the global existence of smooth solutions.
Up to now, it is well known that if the problem \eqref{GEQSYM}-\eqref{data} is endowed with a convex entropy and all time the [SK] condition, then the global existence of classical solutions can be achieved for small initial data (close to equilibrium). A brief history is stated as follows. Chen, Levermore and Liu \cite{CLL} formulated a notion of the mathematical entropy for \eqref{GEQSYM}, which was a natural extension of the classical entropy due to Godunov \cite{Godunov}, Friedrichs and Lax \cite{FL} for conservation laws.
However, their dissipative entropy was not strong enough to
develop the global existence theory for (\ref{GEQSYM}). Under
some technical requirements on
the entropy dissipation, Hanouzet and Natalini \cite{HN} in one dimensional space and Yong \cite{Yong041} in the multidimensional space, have proved the global existence of classical solutions to \eqref{GEQSYM}-\eqref{data} for initial data close to equilibrium in $H^{s}(\mathbb{R}^d)$ with $s>d/2+1$. Moreover, Ruggeri and Serre \cite{RS1} showed that constant equilibrium states are time-asymptotically
$L^2$-stable. By using the Duhamel principle and a
detailed analysis of Green's kernel estimates for the linearized
problem, Bianchini, Hanouzet and Natalini \cite{BHN}
have proved that smooth solutions approached the constant equilibrium state
$\bar{U}$ in the $L^{p}$-norm at the rate
$O(t^{-\frac{d}{2}\left(1-\frac{1}{p}\right)})$, as $t\rightarrow\infty$, for
$p\in[\min\{d,2\},\infty]$,
in $H^{s}(\mathbb{R}^d)\cap L^1(\mathbb{R}^d)$. Furthermore, in the spirit of Chapman-Enskog expansion, they also justified the validity of the approximation to a parabolic system  with a faster time convergence rate. Subsequently, Kawashima and Yong \cite{KY1} gave a perfect definition
of entropy for \eqref{GEQSYM} and
removed
the technical requirements on the entropy dissipation assumed in
\cite{BHN,HN,RS1,Yong041}, see Definition below. Kawashima and Yong \cite{KY2} employed the time-weighted energy
method and established the asymptotic decay estimate as in \cite{BHN}. Beauchard and Zuazua \cite{BZ} framed those results of \cite{Kawashimadoctoral,SK} in the spirit of Villani's hypocoercivity theory \cite{Villani} and showed the equivalence of the [SK] condition and the Kalman rank condition in control theory. A recent work \cite{CSZ}  developed  the ``physical space version" of the hyperbolic
hypocoercivity approach  introduced by
\cite{BZ} and obtained
 time-decay estimates for partially dissipative hyperbolic systems.

 \vspace{2mm}

\item \textbf{Critical regularity framework.} It is so interesting to investigate the limit case that $s=d/2+1$, where the classical local-in-time existence theory for symmetric hyperbolic systems fails. Iftimie \cite{I} first established the local existence for symmetric conservation laws pertaining to data in the Besov space $B^{d/2+1}_{2,1}(\mathbb{R}^d)$, which is a subalgebra embedded in $C^{1}(\mathbb{R}^d)$. In addition, a lower bound of the maximal
time of existence was also obtained. Chae \cite{C1} subsequently
assumed the structural condition
$C^{-1}I_{n}\leq\tilde{A}^{0}(V)\leq CI_{n} (\forall \,V\in
\mathbb{R}^{d})$ and established a similar local existence. In \cite{XK1}, the second author and Kawashima removed the crucial assumption, although it is satisfied by many concrete examples, where the iteration strategy and Friedrichs' regularization method were mainly employed.
The blowup criterion in terms of the $L^1$-time integrability of Lipschitz bound was also obtained. By exploring the theory of Besov spaces, furthermore, the global existence result in the spatially critical Besov space $B^{d/2+1}_{2,1}$ was established under that assumption of the entropy notion as in \cite{KY1} and the [SK] condition. A natural question then arises: how can the large-time behavior of solutions
be described in spatially critical Besov spaces? The second author and Kawashima \cite{XK2,XK3} gave a new decay framework for linearized partially dissipative systems in $L^2(\mathbb{R}^{d})\cap\dot{B}^{-s}_{2,\infty}\,(\mathbb{R}^{d})(s>0)$, which
improves the integral framework $L^2(\mathbb{R}^{d})\cap L^{p}(\mathbb{R}^{d})(1\leq p<2)$ in \cite{UKS}, since $L^{1}(\mathbb{R}^{d})\hookrightarrow\dot{B}^{0}_{1,\infty}(\mathbb{R}^{d})\hookrightarrow\dot{B}^{-d/2}_{2,\infty}(\mathbb{R}^{d})$. Furthermore, the time-decay rates of solutions to the nonlinear problem
\eqref{GEQSYM}-\eqref{data} are also obtained in \cite{XK2,XK3} when the
initial perturbation is suitably small in
$B^{d/2+1}_{2,1}(\mathbb{R}^{d})\cap \dot{B}^{-s}_{2,\infty}(\mathbb{R}^{d})$ with $0<s\leq d/2$. The smallness of the $\dot{B}^{-s}_{2,\infty}$-norm was subsequently removed by Yu \cite{Yu1} in Sobolev spaces with higher regularity. Very recently, Crin-Barat and Danchin \cite{c2,c3} constructed global solutions in hybrid Besov space $\dot{B}^{d/p,d/2+1}_{p,2}(\mathbb{R}^{d})$, that is with regularity $B^{d/p}_{p,1}(\mathbb{R}^{d})$  in low frequencies and $B^{d/2+1}_{2,1}(\mathbb{R}^{d})$ in high frequencies. To the best of our knowledge, their results are in sharp contrast with
the pure hyperbolic case (even linear), where well-posedness in $L^p$ spaces for $p\neq 2$ fails (see \cite{brenner1}). Moreover, the time-decay estimates of solutions can be deduced if initial data additionally belong to the Besov space $\dot{B}^{\sigma_1}_{2,\infty}(\mathbb{R}^{d})$ for some $\sigma_1\in [-d/2,d/2-1]$.
\end{itemize}

It is natural to investigate whether those decay rates obtained in the literature are optimal for generally hyperbolic systems, specifically relying on the lower bounds of time-decay estimates. More importantly, the theory of decay character for the
Cauchy problem \eqref{GEQSYM}-\eqref{data} remains open so far, and there are few results regarding the sharp regularity assumption of the initial data for the large-time behavior of solutions.



\subsection{Presentation of main results}
Let us now state the contents of this paper in more details. To make the paper as self-contained as possible, in Section \ref{section2}, we
recall the entropy and symmetrization on the hyperbolic systems of balance laws as well  as the [SK] stability condition. Section \ref{section3} is devoted to the decay characterization theory for linear systems. In Section \ref{decay-d}, we develop the decay characterization
for a large class of
diffusive systems  with Fourier multiplier
\begin{equation}\label{lineardis0}
\left\{
\begin{aligned}
    &\partial_t U=\mathcal{L}U,\\
    &U|_{t=0}=U_0.
\end{aligned}
\right.
\end{equation}
Here $U: \mathbb{R}_+\times\mathbb{R}^d\rightarrow \mathbb{R}^d$ is the unknown vector-valued function, and $\mathcal{L}$ is a Fourier multiplier with the symbol
\begin{align}
&S(\xi)\coloneqq P(\xi)^{-1}D(\xi) P(\xi),\quad \text{a.e.}\quad \xi\in\mathbb{R}^{d},\label{symbol0}
\end{align}
where $D(\xi)$ and $P(\xi)$ are, respectively, diagonal and orthogonal matrices of size $n$, with $D(\xi)_{ij}=-c_{i}|\xi|^{2\alpha}\delta_{ij}$ and $0<c_*\leq c_{i}\leq c^*$ for all $i=1,\dots,n$ and $\alpha>0$. 
Clearly, the constant heat equation is a basic example.   
Bjorland-Schonbek \cite{bjorland1} and Niche-Schonbek \cite{niche1} proved that the solution has a two-sided time decay estimate
$(1+t)^{-s/2\alpha}\lesssim \|e^{t\mathcal{L}}U_0\|_{L^{2}}\lesssim (1+t)^{-s/2\alpha}$, if the initial datum $U_0$ satisfies
\begin{align}\label{ch1}
    &0<\lim_{r\rightarrow 0+} r^{-2s}\int_{|\xi|\leq r} |\widehat{U}_{0}(\xi)|^2 d\xi <\infty,\quad s>0,
\end{align}
which is closely linked with the decay character; however, it
is somehow too stringent as such a limit might not exist. Brandolese \cite{brandolese1}
improved a slight modification of  \eqref{ch1} and proved two-sided bounds of the solution to \eqref{lineardis0}-\eqref{symbol0}:
 \begin{equation}\label{ch2}
    \left\{
        \begin{aligned}
       &\liminf_{r\rightarrow 0+} r^{-2s}\int_{|\xi|\leq r} |\widehat{U}_{0}(\xi)|^2 d\xi>0,\\
&\limsup_{r\rightarrow 0+} r^{-2s}\int_{|\xi|\leq r} |\widehat{U}_{0}(\xi)|^2 d\xi <\infty
\end{aligned}
\right.
\Longleftrightarrow U_{0}\in \dot{\mathcal{B}}^{-s}_{2,\infty}\Longleftrightarrow (1+t)^{-\frac{s}{2\alpha}}\lesssim \|e^{t\mathcal{L}}U_0\|_{L^{2}}\lesssim (1+t)^{-\frac{s}{2\alpha}},
\end{equation}
 where
the subset of the Besov space $\dot{B}^{\sigma}_{p,\infty}(\sigma \in \mathbb{R})$ is defined by
\begin{equation}\label{Bsubset}
\begin{aligned}
\dot{\mathcal{B}}^{\sigma}_{p,\infty}\coloneqq  \Big\{ f\in \dot{B}^{\sigma}_{p,\infty} ~|~& \exists~ \text{two constants}~c, M>0~\text{and a sequence of integers $\{j_{k}\}_{k=1,2,\dots}$} \\
&\text{s.t.~$\lim_{k\rightarrow\infty}j_{k}= -\infty$,~$|j_{k}-j_{k+1}|\leq M$~and  $2^{\sigma j_{k}}\|\dot{\Delta}_{j_{k}}f\|_{L^p}\geq c$}\Big\}.
\end{aligned}
\end{equation}
As a continuous step, we are interested in the general $L^p$ extension in the present paper, since Plancherel's theorem  no longer holds. For any $1\leq p,r \leq \infty$, we prove
(see Proposition \ref{propgeneral2}):
\begin{equation}\label{chp}
        \begin{aligned}
        U_{0}\in \dot{\mathcal{B}}^{\sigma_1}_{p,\infty}\Longleftrightarrow (1+t)^{-\frac{\sigma-\sigma_1}{2\alpha}}\lesssim \|e^{t\mathcal{L}}U_0\|_{\dot{B}^{\sigma}_{p,r}}\lesssim (1+t)^{-\frac{\sigma-\sigma_1}{2\alpha}},\quad \sigma>\sigma_1.
        \end{aligned}
\end{equation}

In general, the above parabolic theory cannot be directly applied to the following partially dissipative hyperbolic systems:\vspace{-2mm}
\begin{equation}\label{hp}
\left\{
\begin{aligned}
&\partial_{t}V+\sum_{i=1}^{d}A^{i}\partial_{x_{i}}V+LV=0,\\
&V|_{t=0}=V_{0},
\end{aligned}
\right.
\end{equation}
where $V=V(t,x)$ with $(t,x)\in\mathbb{R}_+\times\mathbb{R}^d$ ($d\geq1$) is the unknown $N$-vector-valued function taking values in an open convex set $\mathcal{O}_{V}\subset \mathbb{R}^{N}$, and $A^{i}$ and $L$ are constant symmetric matrices that satisfy \eqref{block1}, \eqref{dissipationU211}, \eqref{blockL}, and the [SK] condition. To show the decay estimates, Bianchini, Hanouzet and Natalini \cite{BHN}
performed a very detailed analysis of the asymptotic behavior of Green's kernel to \eqref{hp} in one dimension, and an analogous result for multidimensional systems can be presented by assuming the [SK] condition. In fact,  the major difficulty lies in the fact that the form of Green's function to \eqref{hp} is not explicit in the multidimensional case.
The decay characterization has been a challenging problem so far. In Section \ref{section:linear}, we aim to address the open question.
Our key idea lies in deducing the leading diffusion profile
induced by the partial relaxation structure, and therefore deriving the two-sided bounds of decay
for \eqref{hp} via the energy method, without relying on any
explicit expression of the Green's function that is typically required in the existing literature.
To the best of our knowledge, the new finding is of independent interest,
with potential application beyond the present setting.
Precisely, inspired by the spectral analysis of the damped Euler system
\eqref{Eulerlinear} (see Appendix~\ref{appendixB}), we introduce the following initial
effective quantity \vspace{-1mm}
\begin{align}
\Psi_0\coloneqq   V_{1,0}-\sum_{i=1}^{d} A^{i}_{1,2}D^{-1}\partial_{x_{i}} V_{2,0},\label{Psi0}
\end{align}
describing the interaction of conservative and dissipative components, which originates from the quantity
$\tilde{\Psi}_0=a_0-\dive m_0$ in \eqref{Eulerlinear}. Furthermore, by the effective quantity $\Psi=V_{1}-\sum_{i=1}^{d} A^{i}_{1,2}D^{-1} \partial_{x_{i}} V_{2}$ associated with \eqref{Psi0}, we can employ the $L^p$ energy estimates and capture the corresponding leading diffusion part to obtain two-sided bounds for decay estimates \textit{if and only if} $\Psi^\ell_0\in\dot{\mathcal{B}}^{\sigma_1}_{p,\infty}$.

In Section \ref{global existence}, we focus on the  global well-posedness result for the Cauchy problem \eqref{GEQSYM}-\eqref{data}. As pointed out by Crin-Barat and Danchin \cite{c3}, it is possible to consider data in the hybrid Besov space $\dot{B}^{d/p,d/2+1}_{p,2}$, since the low frequencies'
eigenvalues are real and can thus be handled in the $L^p$-type framework, whereas the high frequencies'
eigenvalues are complex-valued, and hence are in the space related to $L^2$. We shall improve the low-frequency regularities of initial data
from the viewpoint
of spectral analysis and the blowup criterion with respect to the $L^1$-time integrability of Lipschitz bound (see \cite{XK1}).
Precisely, we assume that
\begin{align*}
(\mathbf{P}V_0)^{\ell}\in \dot{B}^{\frac{d}{p}-1}_{p,1} \quad\text{and}\quad (\{\mathbf{I} - \mathbf{P}\}V_0)^\ell\in \dot{B}^{\frac{d}{p}}_{p,1}
\end{align*}
for some $2\leq p\leq 4$, where $\mathbf{P}$ and $\mathbf{\{I-P}\}$ represent the orthogonal projectors of the conservative and dissipative parts, respectively. The low-frequency assumption may be more suitable, as it coincides with the quantitative estimates of spectral localization in Proposition \ref{LemmaspectrallocalHp}. Then, it is essentially a matter of establishing global-in-time \emph{a priori} estimates, which 
enables us to construct global solutions.

In Section \ref{section:decay},
our main task is to justify the sharp characterization of large-time asymptotics of solutions to the nonlinear problem \eqref{GEQSYM}-\eqref{data} (see Theorem \ref{theorem0}). We introduce the following low-frequency assumption,
according to the different regularity indices of the conservative component
$\mathbf{P}V_0$ and the dissipative component
$\{\mathbf{I}-\mathbf{P}\}V_0$:
\begin{align*}
(\mathbf{P}V_0)^{\ell}\in \dot{B}^{\sigma_1}_{p,\infty}\quad\text{and}\quad (\{\mathbf{I} - \mathbf{P}\}V_0)^{\ell}\in \dot{B}^{\sigma_1+1}_{p,\infty}\quad\text{for}\quad -\frac{d}{p}\leq \sigma_1<\frac{d}{p}-1.
\end{align*}
To the best of our knowledge, the $L^p$-type low-frequency regularity condition was first imposed to investigate the large-time behavior of global solutions so far, revealing the new parabolic profiles that enable us to achieve the upper bounds for decay rates.
Precisely, it is shown that the boundedness assumption that $(\mathbf{P}V_0)^{\ell}\in \dot{B}^{\sigma_1}_{p,\infty}(-d/p\leq \sigma_1<d/p-1)$
is
not only sufficient, but also necessary to achieve upper bounds of decay estimates. Furthermore,
both two-sided bounds for time-decay estimates  hold if and
only if $\Psi^\ell_0\in\dot{\mathcal{B}}^{\sigma_1}_{p,\infty}$. The key error estimate for $\delta V \coloneqq   V-V_L$ between solutions to the nonlinear problem \eqref{m1} and the linear problem \eqref{hp}  originates from the Wiegner's idea for viscous incompressible flows in \cite{wiegner1}, which can be
adapted to suit the dissipative hyperbolic systems: namely, we compute faster time-decay rates of $\mathbf{P}\delta V$ and $\{\mathbf{I} - \mathbf{P}\}\delta V$  compared with that of $\mathbf{P}V_{L}$ and $\{\mathbf{I} - \mathbf{P}\}V_{L}$, respectively, to the linear problem \eqref{hp} in $L^p$-type Besov spaces (see Proposition \ref{properror1}).
With faster rates of $\delta V$, the sharp decay characterization (Theorems \ref{thm1}-\ref{thm3}) is equivalent to establishing the large-time asymptotic relation as $t\rightarrow\infty$:
\begin{equation*}
\setlength\abovedisplayskip{2pt} {\textit{Decay of the linear problem \eqref{hp}}} \xLongleftrightarrow{\raisebox{0.3ex}{}}   {\textit{Decay of the nonlinear problem \eqref{m1}.}}
\setlength\belowdisplayskip{2pt} \end{equation*}
The approach developed here is closely linked with Wiegner's argument \cite{wiegner1} and the inverse Wiegner's argument \cite{skalak1}, which can be regarded as a new attempt
to study partially dissipative hyperbolic systems. In addition, as a side effect,
the smallness of low frequencies of initial data can be removed
in the Fourier semigroup framework, in contrast to
those previous works, for instance, \cite{BHN,XK2}.


In Appendix \ref{appendixA}, we briefly recall the Littlewood–Paley theory, including Besov spaces and Chemin–Lerner spaces. Moreover, we develop several non-standard product laws, composition estimates, and commutator estimates to handle nonlinear terms in hybrid $L^p$–$L^2$ Besov spaces. Note that the damped compressible Euler system \eqref{Eulerlinear} for perfect gas flows belongs to the class of partially dissipative hyperbolic systems, for which the physical energy can be viewed as a convex entropy and the [SK] stability condition is satisfied (see \cite{XK1}). Consequently, in Appendix \ref{appendixB}, we take \eqref{Eulerlinear} as a toy model to present the underlying idea of sharp decay characterization for partially dissipative hyperbolic systems, where the initial effective quantity $\Psi_0=a_0-\dive m_0$ is first identified. For the reader’s convenience, we provide a detailed analysis; this appendix can also be read independently of the main results.

Last but not least, we would like to point out again that our main results were obtained by assuming the dissipative entropy in \cite{KY1,SK} and, all the time, the [SK] condition, where the dissipation matrix $L$ is symmetric in \eqref{hp}. For the non-symmetric dissipation, the situation becomes more complicated. Several years ago,
Ueda, Duan and Kawashima \cite{UDK} formulated a new stability condition and analyzed the weaker dissipative structure for hyperbolic systems (see also \cite{CLSZ}).
The suitable modification of our approach is likely to be effective for the non-symmetric case, which is the next consideration.

\vspace{2mm}

\noindent
\textbf{Notations.} Throughout the paper, $C>0$ denotes a generic constant that may vary from line to line.
We write $A\lesssim B$ (resp. $A\gtrsim B$) if $A\le CB$ (resp. $A\ge CB$) for some $C>0$,
and $A\sim B$ if both $A\lesssim B$ and $A\gtrsim B$ hold.
For $1\le p\le\infty$, the space $L^p(0,T;X)$ (or $L^p_T(X)$) consists of measurable functions
$f:[0,T]\to X$ such that $t\mapsto\|f(t)\|_X$ belongs to $L^p(0,T)$, endowed with the norm
$\|f\|_{L^p_T(X)}:=\|\|f(\cdot)\|_X\|_{L^p(0,T)}$.
We denote by $\mathcal{F}f=\widehat f$ and $\mathcal{F}^{-1}f=\breve f$ the Fourier transform
of $f$ and its inverse, respectively, and define
$\Lambda^\sigma f:=\mathcal{F}^{-1}\big(|\xi|^\sigma \widehat f\big)$ for $\sigma\in\mathbb{R}$, which corresponds to a fractional derivative of order $\sigma$.
Finally, we set $\langle t\rangle=(1+t^2)^{1/2}$.

 \vspace{2mm}

\section{Entropy and [SK] stability condition}\label{section2}
As shown by \cite{HN, KY1,SK,Yong041,XK1}, a rigorous framework for a general class
of partially dissipative hyperbolic systems has been well
developed based on the dissipative entropy and the [SK] stability condition. In this section, we present it for convenience of readers. 
To begin with, we set
\begin{equation}\label{Mset}
\begin{aligned}
&\mathcal{M}:=\Big\{\psi\in\mathbb{R}^{n}~:~\langle\psi,G(U)\rangle=0\quad~\text{for any}~U\in \mathcal{O}_{U}\Big\},
\end{aligned}
\end{equation}
where $\mathcal{M}$ is a subset of $\mathbb{R}^{n}$ with ${\rm{dim}}~\mathcal{M}=n_{1}$. It follows from the definition of $\mathcal{M}$ that $G(U)\in \mathcal{M}^{\perp}$ (the orthogonal complement of $\mathcal{M}$) for any $U\in \mathcal{O}_{U}$. Corresponding to the orthogonal decomposition $\mathbb{R}^{n}=\mathcal{M}\oplus\mathcal{M}^{\perp}$,
we write $U\in \mathbb{R}^{n}$
as
\begin{eqnarray*}
U=
\begin{pmatrix}
      U_{1} \\
      U_{2} \\
   \end{pmatrix}
\end{eqnarray*}
such that $U\in\mathcal{M}$ holds if and only if $U_{2}=0$. Moreover,
we denote the set of equilibria for the balance laws \eqref{GEQSYM}:
\begin{align}
&\mathcal{E}=\{U\in \mathcal{O}_{U}~:~G(U)=0\}.
\end{align}

To symmetrize the system \eqref{GEQSYM}, we recall the  notion of dissipative entropy introduced in \cite{KY1,SK}.


\begin{defn}\label{defnentropy}
    Let $\eta=\eta(U)$ be a smooth function defined in a convex open set $\mathcal{O}_{U}\subset \mathbb{R}^{n}$. Then $\eta=\eta(U)$ is called an entropy for \eqref{GEQSYM} if the following statements hold{\rm:}
    \begin{itemize}
    \item [$(\bullet)$]$\eta=\eta(U)$ is strictly convex in $\mathcal{O}_{U}$ in the sense that the Hessian $D^2_{U}\eta(U)$ is positive definite for $U\in \mathcal{O}_{U} ${\rm;}
    \item [$(\bullet)$]$D_{U}F^{i}(U)\big( D^2_{U}\eta(U)\big)^{-1}$ is  symmetric for $i=1,2...,d$ and $U\in \mathcal{O}_{U} $;
    \item [$(\bullet)$] $U\in \mathcal{E}$ if and only if $\big( D_{U}\eta(U) \big)^{\top}\in \mathcal{M}${\rm;}
    \item [$(\bullet)$] For $U\in \mathcal{E}$, the matrix $D_{U}G(U)\big( D^2_{U}\eta(U)\big)^{-1}$ is symmetric and non-positive definite, and its null space coincides with $\mathcal{M}$.
    \end{itemize}
\end{defn}

Here and below, $D_{U}$ stands for the (row) gradient operator with respect to $U$ and the
superscript $\top$ denotes the transpose. Let $\eta=\eta(U)$ be an entropy in Definition \ref{defnentropy}. Set
\begin{equation}\label{w}
   W=W(U)\coloneqq  (D_{U}\eta(U))^{\top}.
\end{equation}
It was shown in \cite{KY1} that the mapping $W=W(U)$ is  a diffeomorphism from $\mathcal{O}_{U}$ onto its range $\mathcal{O}_{W}$. Let $U=U(W)$ be the inverse mapping, which is also a diffeomorphism from $\mathcal{O}_{W}$ onto its range $\mathcal{O}_{U}$. Then \eqref{GEQSYM} can be rewritten as
\begin{equation}
A^{0}(W)\partial_{t}W+\sum_{i=1}^{d}A^{i}(W)\partial_{x_{i}}W=H(W)
\label{entropeq}
\end{equation}
with
\begin{equation}\nonumber
\begin{aligned}
&A^{0}(W)=D_{W}U(W),\\
&A^{i}(W)=D_{W}F^{i}(U(W))=D_{U}F^{i}(U(W))D_{W}U(W),\\
&H(W)=G(U(W)).
\end{aligned}
\end{equation}
Moreover, we define
\begin{equation}\label{tildeL}
\begin{aligned}
&L(W)=-D_{W}H(W)=-(D_{U}G)(U(W))D_{W}U(W).
\end{aligned}
\end{equation}
We see that $H(W)\in \mathcal{M}^{\perp}$ for any $W\in \mathcal{O}_{W}$. By virtue of \eqref{w}, we have $D_{W}U(W)=\big(D^2_{U}\eta(U(W))\big)^{-1}$.
It follows from \eqref{entropeq} that the system \eqref{GEQSYM} is symmetric dissipative (see \cite{KY1}).





Let $\bar{U}\in \mathcal{E}$ be an equilibrium state of $U$ such that $G(\bar{U})=0$. Then there exists an equilibrium state $\bar{W}$ of $W$. From the definition before, we know that $H(\bar{W})=G(\bar{U})=0$, and the source term $H(W)$ in \eqref{entropeq} can be expressed by
\begin{equation}\label{tildeH}
\begin{aligned}
H(W)=-L(\bar{W})(W-\bar{W})+\mathbf{r}(W),
\end{aligned}
\end{equation}
where $\mathbf{r}(W)$ is given by
\begin{align}
\mathbf{r}(W)\coloneqq   H(W)-H(\bar{W})-D_{W}H(\bar{W})(W-\bar{W}),\label{rtildeU}
\end{align}
which satisfies $\mathbf{r}(W)\in \mathcal{M}^{\perp}$ for all $W\in \mathcal{O}_{W}$.


Unfortunately, even for very simple cases, the dissipative entropy alone is not enough to prevent the appearance of singularities in finite time. Therefore,
an additional assumption is needed, that is, the [SK] condition, which was initiated by Shizuta and Kawashima \cite{SK}. The condition guarantees
the necessary coupling between conservative and dissipative components in order to have the enough
dissipation in each state variable.
Let $\bar{W}\in \mathcal{M}$ be the constant state and consider the linearized symmetric form of \eqref{entropeq} at $W=\bar{W}$:
\begin{align}
A^{0}(\bar{W})\partial_t W+\sum_{i=1}^d A^{i}(\bar{W}) \partial_{x_i} W+L(\bar{W})W=0,\label{linear*}
\end{align}
where $A^{i}(\bar{W})$ ($i=0,1,2,\dots,d$) and $L(\bar{W})$ are constant symmetric matrices given by \eqref{entropeq} and \eqref{tildeL}. Then the stability condition for
\eqref{linear*} is formulated as follows.

\begin{defn}\label{defnSK}
The symmetric form \eqref{linear*} satisfies the [SK] stability condition if the following property holds true:
Let $\phi\in \mathbb{R}^n$ satisfy $\phi\in \mathcal{M}$ {\rm(}that is, $L(\bar{W})\phi=0${\rm)} and $\lambda A^{0}(\bar{W})\phi +A(\omega)\phi=0$ for some $(\lambda,\omega)\in \mathbb{R}\times \mathbb{S}^{d-1}$, then $\phi=0$.  Here $A(\omega)$ with $\omega |\xi|=\xi$ is given by
$A(\omega)=\sum\limits_{i=1}^{d} A^{i}(\bar{W}) \omega_{i}$.
\end{defn}









According to \cite{SK}, the [SK] condition
implies the equivalent characterization that the eigenvalues of the Fourier transform of \eqref{linear*} satisfy
\begin{align}
{\rm Re} \,\lambda(i\xi)\leq -\frac{c|\xi|^2}{1+|\xi|^2} \label{eigh:SK}
\end{align}
for $\xi\in\mathbb{R}^{d}$ and $c>0$, which reflects the intrinsic nature of partial dissipation: diffusion dominates at low frequencies due to the indirect dissipation mechanism, whereas high frequencies are directly damped. In other words, the [SK] condition precisely characterizes
the dissipation transfer via the hyperbolic coupling.

If \eqref{GEQSYM} admits an entropy in the sense of Definition \ref{defnentropy}, then  one can  also transform \eqref{linear*} to a symmetric dissipative system of the normal form proposed by Yong and Kawashima (see \cite{KY1}). Furthermore, we have after a suitable transform
\begin{equation}\label{normal:linear}
    \begin{aligned}
&\partial_{t}V+\sum_{i=1}^{d}A^{i}\partial_{x_{i}}V+LV=0
\end{aligned}
\end{equation}
with
\begin{equation}\nonumber
\begin{aligned}
    A^{i}=\begin{pmatrix} A^{i}_{1,1} & A^{i}_{1,2} \\ A^{i}_{2,1} & A^{i}_{2,2} \end{pmatrix}
\end{aligned} \quad\text{and}\quad L=\begin{pmatrix} 0 & 0 \\ 0 & D
    \end{pmatrix},
\end{equation}
where  $A^i$ and $L$ are constant symmetric matrices of size $n\times n$, and $D$ is the symmetric positive definite matrix of size $n_2 \times n_2$. 


In addition, the [SK] condition may be linked to the diffusion structure: If $A^i_{1,1}=0$, then
\begin{equation}\label{eq:SK-elliptic}
-\sum_{i,l=1}^{d}A^{i}_{1,2}D^{-1}A^{l}_{2,1}\partial_{x_{i}x_{l}}^{2}
\ \text{is strongly elliptic}
\quad \Longleftrightarrow \quad
\text{the {\rm[SK]} condition holds}.
\end{equation}
The property \eqref{eq:SK-elliptic} is widely used in the study of relaxation limit problems; see \cite{peng1,Yong041,c2}.
In the present paper, we shall take advantage of the implication of the
{\rm[SK]} condition, which induces the underlying diffusion profile on the large-time behavior of solutions.

\section{Decay character theory of linear systems} \label{section3}
In this section, the main purpose is to develop the characterization of optimal time-decay estimates
for partially dissipative hyperbolic
systems, since we consider the difference between the solutions of the nonlinear equation \eqref{GEQSYM} and the corresponding linearized system at equilibrium $\bar{U}$. For that purpose, we assume that the system \eqref{GEQSYM} admits an entropy. Then there is a function $V=V(U)$ depending smoothly on $U$ and satisfying $V(\bar{U})=0$ such that the system \eqref{entropeq} can be rewritten as the  normal symmetric dissipative system
\begin{equation}\label{m1}
    \begin{aligned}
&\partial_{t}V+\sum_{i=1}^{d}A^{i}\partial_{x_{i}}V+LV=R^{0}+\sum_{i=1}^{d}\partial_{x_{i}}R^{i}
\end{aligned}
\end{equation}
with initial data
\begin{align}
V(0,x)=V_0(x),\quad x\in\mathbb{R}^d.\label{m1d}
\end{align}
Here, $A^{i}$ ($i=1,\dots,d$) are symmetric constant matrices, $L$ is the constant symmetric matrix given by
\begin{equation}\label{block1}
\begin{aligned}
L=\begin{pmatrix}
    0 & 0 \\ 0 & D
    \end{pmatrix},
\end{aligned}
\end{equation}
where $D$ is a $n_2\times n_2$ positive definite matrix satisfying
\begin{align}
&\langle D V_2, V_2\rangle\geq \kappa |V_2|^2
\label{dissipationU211}
\end{align}
for some constant $\kappa>0$. In addition, the nonlinear terms $R^i$ ($i=0,\dots,d$) are given by
\begin{equation}\label{R0Ri}
\left\{
    \begin{aligned}
    &R^{0}\coloneqq  \mathbb{A}^0\Big(G(U)-G(\bar{U})-D_{U}G(\bar{U})(U-\bar{U})\Big),\\
    &R^{i}\coloneqq  \mathbb{A}^i \Big( F^i(U)-F^i(\bar{U})-D_{U}F^i(\bar{U})(U-\bar{U})\Big),\quad i=1,\dots,d,
\end{aligned}
\right.
\end{equation}
where $\mathbb{A}^i$ ($i=0,\dots,d$) are some constant matrices.  


To highlight the property of partial dissipativity, we need to reformulate \eqref{m1} into the conservative-dissipative form. For that end,
corresponding to the orthogonal decomposition $\mathbb{R}^{n}=\mathcal{M}\oplus\mathcal{M}^{\top}$, we write
\begin{equation}\label{block}
\begin{aligned}
V=\begin{pmatrix}
V_{1}\\
V_{2}
    \end{pmatrix},\quad\quad V_0=\begin{pmatrix}
V_{1,0}\\
V_{2,0}
    \end{pmatrix}\quad\text{and}\quad
    A^{i}= A^{i}(\bar{U})=\begin{pmatrix} A^{i}_{1,1} & A^{i}_{1,2} \\ A^{i}_{2,1} & A^{i}_{2,2} \end{pmatrix}.
\end{aligned}
\end{equation}
 Then the problem \eqref{m1}-\eqref{m1d} reads
\begin{equation}\label{hp121}
\left\{
\begin{aligned}
& \partial_{t}V_{1}+\sum_{i=1}^{d}\Big(A_{1,1}^{i}\partial_{x_{i}}V_{1}+A_{1,2}^{i}\partial_{x_{i}}V_{2}\Big)=\sum_{i=1}^{d}\partial_{x_{i}}(R^{i})^1,\\
&\partial_{t}V_{2}+\sum_{i=1}^{d}\Big(A_{2,1}^{i}\partial_{x_{i}}V_{1}+A_{2,2}^{i}\partial_{x_{i}}V_{2}\Big)+DV_{2}=(R^{0})^2+\sum_{i=1}^{d}\partial_{x_{i}}(R^{i})^2,\\
&(V_{1},V_2)(0,x)=(V_{1,0},V_{2,0}).
\end{aligned}
\right.
\end{equation}
Let $\mathbf{P}$ and $\{\mathbf{I} - \mathbf{P}\}$ denote the orthogonal projections onto $\mathcal{M}$ and $\mathcal{M}^{\top}$, respectively. That is,
\begin{equation}\label{3}
\begin{aligned}
\mathbf{P}V=\begin{pmatrix}
V_{1}\\
0
    \end{pmatrix}\quad\text{and}\quad
    \{\mathbf{I} - \mathbf{P}\}V=\begin{pmatrix}
0\\
V_{2}
    \end{pmatrix}.
\end{aligned}
\end{equation}
It follows from
\eqref{Mset} that $\mathbf{P}R^0=0$. This property will be used in the analysis of large-time asymptotics. Moreover, it will also be shown that, with the faster order of convergence, the conservative
component $\mathbf{P}V$ and
dissipative component $\{\mathbf{I} - \mathbf{P}\}V$ are asymptotically approximated
by profiles associated with  the diffusive systems in
the spirit of Chapman-Enskog expansion. As a first step,
we try to give the $L^p$ decay characterization for a class of diffusive systems with Fourier multiplier.

\subsection{Diffusive systems with Fourier multiplier}
\label{decay-d}
Consider a large class of diffusive systems  with Fourier multiplier
\begin{equation}\label{lineardis}
\left\{
\begin{aligned}
    &\partial_t U=\mathcal{L}U,\\
    &U|_{t=0}=U_0.
\end{aligned}
\right.
\end{equation}
Here $U=(U^1,U^2,\dots, U^{n})$ $(n\geq1)$ is the unknown at $(t,x)\in \mathbb{R}_{+}\times\mathbb{R}^{d}$ ($d\geq1$), and $\mathcal{L}$ is a  Fourier multiplier defined via its symbol
\begin{align}
&S(\xi)\coloneqq   P(\xi)^{-1}D(\xi) P(\xi),\quad \text{a.e. } \xi \in \mathbb{R}^d\setminus\{0\},\label{symbol}
\end{align}
where $D(\xi)$ is a diagonal matrix of size $n$ with $D(\xi)_{ij}=-c_{i}|\xi|^{2\alpha}\delta_{ij}$ for some $\alpha>0$ and $0<c_*\leq c_{i}\leq c^*$ ($i=1,\dots,n$), and
$P(\xi)$ is an orthogonal matrix of size $n$ whose entries $P(\xi)_{i,j}$ are homogeneous functions for a.e. $\xi \in \mathbb{R}^d\setminus\{0\}$. Basic examples include the heat equation ($\mathcal{L}=\Delta$ with $P(\xi)=\mathbf{I}_{n}$ and $D(\xi)=-|\xi|^{2} \mathbf{I}_{n}$) and the fractional diffusion equation ($\mathcal{L}=-(-\Delta)^{\alpha}$ with $P(\xi)=\mathbf{I}_{n}$ and $D(\xi)=-|\xi|^{2\alpha} \mathbf{I}_{n}$). For any $U_0\in \mathcal{S}'_h(\mathbb{R}^d)$, due to the similarity invariance $e^{tS(\xi)}=P(\xi)^{-1}e^{tD(\xi)}P(\xi)$, the Fourier transform of the solution $U=e^{t\mathcal{L}}U_0$ for \eqref{lineardis} can be represented by
\begin{equation}\nonumber
\begin{aligned}
\mathcal{F}(U)(t,\xi)=P(\xi)^{-1}e^{tD(\xi)} P(\xi) \mathcal{F}(U_0)(\xi),\quad \text{a.e.}~~ \xi \in \mathbb{R}^d\setminus\{0\}~~\text{and}~~\forall t>0.
\end{aligned}
\end{equation}

\vspace{2mm}
The theory of decay characters was first introduced by
Bjorland and Schonbek in \cite{bjorland1}, and further
developed by Niche and Schonbek in \cite{niche1}, with the motivation of obtaining sharp upper and lower decay bounds for the $L^2$-norm of solutions to a large class of linear diffusive systems with Fourier multiplier. This theory, in fact, has been successfully applied to the heat equation, the Navier–Stokes equations \cite{bjorland1}, the quasi-geostrophic equations \cite{niche1} and so on. Brandolese \cite{brandolese1}
slightly modified the original notion of decay character, which allows, in particular, for obtaining the two-sided decay bounds of $L^2$ estimates
\eqref{ch2} if and only if the initial data belong to the suitable subset of Besov spaces. Recently, Brandolese et al. \cite{BSXZ} have developed the theory of decay characters for compressible Navier-Stokes equations.
Here, we are interested in establishing $L^{p}$ estimates  from above and below for solutions to the diffusive system \eqref{lineardis}.
We need the following  estimates for the semigroup $e^{t\mathcal{L} }$ of \eqref{lineardis} restricted to functions with compact support away from the origin in Fourier variables.

\begin{lemma}\label{lemma22}
Let $1\leq p\leq \infty$, and $f$
be a tempered distribution
such that $\operatorname{Supp}\mathcal{F}(f)$ is in the annulus $\lambda \mathcal{C}$, where $\mathcal{C}\coloneqq\{\xi\in\mathbb{R}^{d}~:~0<a_{1}<|\xi|<a_{2}<\infty\}$ and $ \lambda>0$. Then, for all $t>0$ it holds that
\begin{align}
&c_{0}e^{-r_{0}\lambda^{2\alpha}t}\|f\|_{L^p}\leq \|e^{t\mathcal{L} }f\|_{L^p}\leq C_{0}e^{-R_{0}\lambda^{2\alpha}t}\|f\|_{L^p},\label{lowupper}
\end{align}
where $c_{0}, C_{0}$, $r_{0}$, $R_{0}>0$ are positive constants independent of $\lambda$ and $t$.
\end{lemma}

\begin{proof}
The upper bound in \eqref{lowupper} is now classical (see Theorem 2.34 in \cite{bahouri1} or Lemma 3.3 in \cite{brandolese1}), so it suffices to show the lower bound in \eqref{lowupper}. Without loss of generality, we focus on the case that $\lambda=1$.

Let $\psi(\xi)$ be a smooth cutoff function identically equal to one on $ \mathcal{C}$ with compact support in $\{\xi\in\mathbb{R}^{d}:~\frac{1}{2}a_{1}<|\xi|<2a_{2}\}$. As $\mathcal{F}f=\psi(\xi)\mathcal{F}f$, we have
\begin{equation}\nonumber
\begin{aligned}
&f=e^{-t\mathcal{L} }e^{t\mathcal{L}}f=\mathcal{F}^{-1}\Big(\psi(\xi)e^{-tS(\xi)} \mathcal{F}(e^{t\mathcal{L}}f)(\xi)\Big)=G_{t,\psi}\ast (e^{t\mathcal{L}} f)
\end{aligned}
\end{equation}
with  $G_{t,\psi}\coloneqq  \mathcal{F}^{-1}(\psi(\xi) e^{-tS(\xi)})$. Remark that $G_{t,\psi}$ is a well-defined tempered distribution (in fact smooth with rapid decay) due to the compact support of $\psi(\xi)$. With the aid of Young's inequality for convolution, it only remains to deduce that 
\begin{equation}
\begin{aligned}
&\|G_{t,\psi}\|_{L^1}\leq Ce^{Rt} \label{GL1}
\end{aligned}
\end{equation}
for uniform positive constants $C$ and $R$ independent of $t$. Indeed, it follows from  the inverse Fourier transform and integration by parts that
\begin{equation}\nonumber
\begin{aligned}
G_{t,\psi}&=\int_{\mathbb{R}^{d}}e^{i\xi\cdot x}\psi(\xi) P(\xi)^{-1} e^{-t D(\xi)} P(\xi) d\xi\\
&=(1+|x|^2)^{-d}\int_{\mathbb{R}^{d}}\big( ({\rm Id}-\Delta_{\xi})^{d} e^{i\xi\cdot x}\big) \psi(\xi) P(\xi)^{-1} e^{-t D(\xi)} P(\xi) d\xi\\
&=(1+|x|^2)^{-d}\int_{\mathbb{R}^{d}} e^{i\xi\cdot x} \Big( ({\rm Id}-\Delta_{\xi})^{d} \big( \psi(\xi) P(\xi)^{-1} e^{-t D(\xi)} P(\xi)\big) \Big) d\xi.
\end{aligned}
\end{equation}
By Leibniz’s formula and Fa\'a di Bruno’s formula, through a direct computation, we obtain
\begin{equation}\nonumber
\begin{aligned}
\Big|({\rm Id}-\Delta_{\xi})^{d} \big( \psi(\xi) P(\xi)^{-1} e^{-t D(\xi)} P(\xi)\big)\Big|\leq C \langle t\rangle^{2d} e^{c^*(2a_{2})^{2\alpha} t}\leq Ce^{2c^*(2a_{2})^{2\alpha} t}.
\end{aligned}
\end{equation}
Hence, the assertion of \eqref{GL1} follows, and  the proof of Lemma \ref{lemma22} is finished.
\end{proof}

In what follows,  we present the sharp decay characterization for \eqref{lineardis} in the $L^{p}$ framework
in terms of Besov regularity. 
In particular, it allows us to recover the classical decay character theorem as in \cite{bjorland1} and \cite{niche1} by choosing $p=r=2$.

\begin{prop} \label{propgeneral2}
Let $1\leq p, r\leq \infty$, $\sigma_{1}\in\mathbb{R}$, and $\sigma>\sigma_1$. Let $U=e^{t\mathcal{L}}U_0$ be the solution of \eqref{lineardis} with  $U_0\in \dot{B}^{\sigma}_{p,r}$.
Then
\begin{itemize}
\item  {\rm (}Upper bound{\rm )} For all $t>0$, $U$
satisfies
 \begin{align}
        & \|U(t)\|_{\dot{B}^{\sigma}_{p,r}}\leq C (1+t)^{-\frac{1}{2\alpha}(\sigma-\sigma_{1})},\label{Ul1}
    \end{align}
if and only if $U_{0}\in \dot{B}^{\sigma_{1}}_{p,\infty}${\rm;}

 \item  {\rm (}Upper and lower bounds{\rm )} For all $t>0$, $U$ satisfies 
  \begin{align}
        &c (1+t)^{-\frac{1}{2\alpha}(\sigma-\sigma_{1})}\leq \|U(t)\|_{\dot{B}^{\sigma}_{p,r}}\leq C (1+t)^{-\frac{1}{2\alpha}(\sigma-\sigma_{1})},\quad  t> 0,\label{Ul2}
    \end{align}
if and only if $U_{0}\in \dot{\mathcal{B}}^{\sigma_{1}}_{p,\infty}$.
\end{itemize}
Here $c,C>0$ are two constants independent of $t$, and the subset $\dot{\mathcal{B}}^{\sigma_{1}}_{p,\infty}$ 
is  defined by \eqref{Bsubset}. 
\end{prop}

\begin{proof}

Under the assumption $U_{0}\in \dot{B}^{\sigma_{1}}_{p,\infty}$, it follows from
Lemma \ref{lemma22} that
\begin{equation*}
\begin{aligned}
\|U(t)\|_{\dot{B}^{\sigma}_{p,r}}&\lesssim \|U(t)\|_{\dot{B}^{\sigma}_{p,1}}\lesssim \sum_{j\in \mathbb{Z}}  e^{- R_{0} 2^{2\alpha j} t}  2^{j\sigma}\|\dot{\Delta}_{j}U_{0}\|_{L^p}\lesssim t^{-\frac{1}{2\alpha}(\sigma-\sigma_{1})} \|U_{0}\|_{\dot{B}^{\sigma_{1}}_{p,\infty}}
\end{aligned}
\end{equation*}
for $\sigma>\sigma_1$
and $t>0$, where we used the fact
\begin{align}
\sup_{t>0}\sum_{j\in\mathbb{Z}}  t^{\frac{\sigma'}{2\alpha'}} 2^{j\sigma'} e^{-C' 2^{2j\alpha'}t}<\infty\quad \text{for all}~\sigma',\alpha',C'>0.\label{opoo}
\end{align}
On the other hand, since $U_0\in \dot{B}^{\sigma}_{p,r}$, one easily deduces that
\begin{align*}
\|U(t)\|_{\dot{B}^{\sigma}_{p,r}}&\lesssim \Big\| \{  e^{- R_{0} 2^{2\alpha j} t}  2^{j\sigma}\|\dot{\Delta}_{j}U_{0}\|_{L^p}\}_{j\in\mathbb{Z}} \Big\|_{l^r} \lesssim \|U_0\|_{\dot{B}^{\sigma}_{p,r}}.
\end{align*}
Thus, the upper bound \eqref{Ul1} follows.


Conversely, assume that $U$ satisfies \eqref{Ul1}.  For any vector-valued function $G=(G^1,G^2,\cdots G^n)\in \mathcal{S}'_h(\mathbb{R}^d)$, one has
\begin{align}
&\quad\Gamma(\frac{\sigma-\sigma_{1}}{2\alpha}+1)\widehat{G}\nonumber\\
&=\int_{0}^{\infty}t^{\frac{1}{2\alpha}(\sigma-\sigma_{1})}{\rm diag} \Big((2c_1)^{1+\frac{1}{2\alpha}(\sigma-\sigma_{1})} e^{-2c_1 t}\widehat{G}^1,\dots,(2c_{n})^{1+\frac{1}{2\alpha}(\sigma-\sigma_{1})}e^{-2c_n t}\widehat{G}^n\Big)\,dt \nonumber\\
&=\int_{0}^{\infty}\tau^{\frac{1}{2\alpha}(\sigma-\sigma_{1})}|\xi|^{\sigma-\sigma_{1}+2\alpha}\nonumber\\
&\quad\quad \times{\rm diag} \Big((2c_1)^{1+\frac{1}{2\alpha}(\sigma-\sigma_{1})} e^{-2c_1|\xi|^{2\alpha} \tau}\widehat{G}^1,\dots,(2c_{n})^{1+\frac{1}{2\alpha}(\sigma-\sigma_{1})}e^{-2c_n |\xi|^{2\alpha}\tau}\widehat{G}^n\Big)  \,d\tau~(\tau=|\xi|^{-2\alpha}t) \nonumber\\
&=\int_{0}^{\infty} \tau^{\frac{1}{2\alpha}(\sigma-\sigma_{1})} (-D(\xi))^{1-\frac{\sigma_{1}}{2\alpha}} e^{2\tau D(\xi)}|\xi|^{\sigma}\widehat{G}\,d\tau \nonumber
\end{align}
with $\Gamma(s)=\int_{0}^{\infty} t^{s-1}e^{-t}dt$. Taking $G=\mathcal{F}^{-1}(P(\xi) \widehat{\Delta_{j}U_0})$, for all $j\in\mathbb{Z}$, we get
\begin{align}
\Delta_{j}U_0&=\frac{1}{\Gamma(\frac{\sigma-\sigma_{1}}{2\alpha}+1)} \mathcal{F}^{-1}\bigg(\int_{0}^{\infty}t^{\frac{\sigma-\sigma_{1}}{2\alpha}}P(\xi)^{-1} (-D(\xi))^{1-\frac{\sigma_1}{2\alpha}} e^{2t\mathcal{L}} P(\xi)|\xi|^{\sigma}\widehat{\Delta_{j}U_0} \,dt \bigg)\nonumber\\
&=\frac{1}{\Gamma(\frac{\sigma-\sigma_{1}}{2\alpha}+1)} \int_{0}^{\infty}t^{\frac{\sigma-\sigma_{1}}{2\alpha}}\mathcal{F}^{-1}\bigg(P(\xi)^{-1} (-D(\xi))^{1-\frac{\sigma_1}{2\alpha}} P(\xi)\mathcal{F}(\Delta_{j}\Lambda^{\sigma} e^{2t\mathcal{L}} U_0) \bigg)\,dt.\nonumber
\end{align}
Hence, the classical Fourier multiplier theorem (see Lemma 2.2 in \cite{bahouri1}) implies that
\begin{equation}\nonumber
\begin{aligned}
\|\dot{\Delta}_{j}U_0\|_{L^p}&\lesssim \int_{0}^{\infty}t^{\frac{\sigma-\sigma_{1}}{2\alpha}} 2^{(2\alpha-\sigma_1)j} \| \Delta_{j}\Lambda^{\sigma} e^{2t\mathcal{L}}U_0\|_{L^p} dt\\
&\lesssim \int_{0}^{\infty}t^{\frac{\sigma-\sigma_{1}}{2\alpha}} 2^{(2\alpha-\sigma_1)j} e^{-c2^{2\alpha j}t} \|  \Delta_{j} \Lambda^{\sigma} e^{t\mathcal{L}}U_0\|_{L^p} dt,\quad j\in\mathbb{Z}
\end{aligned}
\end{equation}
Note that the decay assumption \eqref{Ul1} leads to
\begin{equation}\nonumber
\begin{aligned}
\| \Delta_{j}\Lambda^{\sigma} e^{t\mathcal{L}}U_0\|_{L^p}&\lesssim \|e^{t\mathcal{L}}U_0\|_{\dot{B}^{\sigma}_{p,r}}\lesssim t^{-\frac{1}{2\alpha}(\sigma-\sigma_1)},\quad j\in\mathbb{Z},\quad  t>0.
\end{aligned}
\end{equation}
Hence, we arrive at
\begin{equation}\nonumber
\begin{aligned}
\|U_0\|_{\dot{B}^{\sigma_{1}}_{p,\infty}}&= \sup_{j\in\mathbb{Z}} 2^{\sigma_{1}j} \|\dot{\Delta}_{j}U_0\|_{L^p}\lesssim \sup_{j\in\mathbb{Z}}\int_{0}^{\infty} 2^{2\alpha j} e^{-c2^{2\alpha j}t}dt\lesssim 1,
\end{aligned}
\end{equation}
which yields $U_{0}\in\dot{B}^{\sigma_{1}}_{p,\infty}$.

Next, we turn to prove the property (2). Assume $U_{0}\in \dot{\mathcal{B}}^{\sigma_{1}}_{p,\infty}$. The upper bound in \eqref{Ul2} follows from \eqref{Ul1} directly. In order to derive the lower bound, it follows from the definition $\dot{\mathcal{B}}^{\sigma_{1}}_{p,\infty}$ as in \eqref{Bsubset} that there exist two constants $c,M>0$ and a sequence $\{j_{k}\}_{k=1,2,\dots}$ such that
 \begin{align}
j_{k}\rightarrow-\infty~~\text{as}~~k\rightarrow \infty,\quad |j_{k}-j_{k+1}|\leq M\quad\text{and}\quad  2^{\sigma_{1}j_{k}}\|\dot{\Delta}_{j_{k}}U_{0}\|_{L^p}\geq c,\quad k=1,2,\dots\label{A}
 \end{align}
 It follows from Lemma \ref{lemma22} that
\begin{equation}\label{maomaomao}
\begin{aligned}
 \|U(t)\|_{\dot{B}^{\sigma }_{p,r}}&\geq 2^{\sigma j_{k}} \|e^{t\mathcal{L}}\dot{\Delta}_{j_k} U_{0}\|_{L^{p}}\gtrsim 2^{\sigma j_{k}}   e^{- r_{0} 2^{2j_{k}\alpha}t} \|\dot{\Delta}_{j_k}U_{0}\|_{L^p}.
 \end{aligned}
 \end{equation}
 For all $t>t_{L}$ with some $t_L>0$, since $j_{k}$ tends to $-\infty$ as $k\rightarrow\infty$, we are able to find a maximal integer $j_{k_{0}}$ satisfying $j_{k_{0}}\leq -\frac{1}{2\alpha}\log_{2} (t_{L}+ t)$. Then we have $j_{k_{0}}> -M-\frac{1}{2\alpha}\log_{2} (t_{L}+ t)$; otherwise, from \eqref{A} another integer $j_{k_{0}-1}$ fulfills $j_{k_{0}-1}\leq j_{k_{0}}+M\leq -\frac{1}{2\alpha}\log_{2} (t_{L}+ t)$ which contradicts the maximality of $j_{k_{0}}$. Therefore, it follows from \eqref{A}, \eqref{maomaomao} and the fact $2^{j_{k_{0}}}\sim \langle t\rangle ^{-\frac{1}{2\alpha}}$  that
 \begin{equation}\nonumber
\begin{aligned}
 \|U(t)\|_{\dot{B}^{\sigma }_{p,r}}& \gtrsim 2^{(\sigma-\sigma_{1}) j_{k_{0}}}\gtrsim  \langle t\rangle ^{-\frac{1}{2\alpha}(\sigma-\sigma_{1})},\quad t>t_L.
 \end{aligned}
 \end{equation}
We now deal with the case $0<t\leq t_L$. We first claim $\|e^{t_L\mathcal{L}}U_0\|_{\dot{B}^{\sigma}_{p,r}}\neq 0$. In fact, using the spectral lower bound in Lemma \ref{lemma22} and $U_0\in \dot{\mathcal{B}}^{\sigma_1}_{p,\infty}$
\begin{align*}
\|e^{t_L\mathcal{L}}U_0\|_{\dot{B}^{\sigma}_{p,r}}\geq c_0 e^{-r_0 2^{2\alpha j_{k} } t_L} 2^{j_k \sigma}\|\dot{\Delta}_{j_k}U_0\|_{L^p} \geq
 c c_0 e^{-r_0 2^{2\alpha j_{k} } t_L}  2^{j_k(\sigma-\sigma_1)} >0
\end{align*}
for some constant $c>0$ and fixed $j_{k}$ given in the definition of $\dot{\mathcal{B}}^{\sigma_1}_{p,\infty}$. Then, we consider the  problem $\partial_s \mathbf{U}(s,x)=\mathcal{L}\mathbf{U}(s,x)$ in $(s,x)\in [t,t_L]\times\mathbb{R}^d$ with the initial datum at $s=t$: $\mathbf{U}(t,x)=U(t,x)=e^{t\mathcal{L}}U_0(x)$.
The upper bound in Lemma \ref{lemma22} in particular implies that $\|\mathbf{U}(t_L)\|_{\dot{B}^{\sigma}_{p,r}}\lesssim \|\mathbf{U}(t)\|_{\dot{B}^{\sigma}_{p,r}}$, or equivalently, 
$\|U(t)\|_{\dot{B}^{\sigma}_{p,r}}\gtrsim\|e^{t_L\mathcal{L}}U_0\|_{\dot{B}^{\sigma}_{p,r}}\gtrsim \langle t\rangle^{-\frac{1}{2\alpha}(\sigma-\sigma_1)} $, where we used $\mathbf{U}(t_L)=e^{(t_L-t)\mathcal{L}}U(t)=e^{t_L\mathcal{L}}U_0$.  
Therefore, the lower bound in \eqref{Ul2} follows.

 Conversely, we aim to infer $U_{0}\in \dot{B}^{\sigma_1}_{p,\infty}$ if we assume that $U$ satisfies the two-sided bounds \eqref{Ul2} for $t>t_{L}$ with $t_L>0$ (in fact, a weaker assumption in the second statement). The upper bound in \eqref{Ul2} implies that $U_{0}\in \dot{B}^{\sigma_{1}}_{p,\infty}$. It suffices to construct a sequence $\{j_{k}\}_{k=1,2,\dots}$ satisfying \eqref{A}. For that end, it follows from Lemma \ref{lemma22},  \eqref{Ul2} and $r<\infty$ that
\begin{equation}\label{mm11}
\begin{aligned}
\sum_{ j\in \mathbb{Z} }e^{- R_{0}  2^{2\alpha j}t}  2^{\sigma j} t^{\frac{1}{2\alpha}(\sigma-\sigma_{1})}\|\dot{\Delta}_{j}U_{0}\|_{L^p} \gtrsim t^{\frac{1}{2\alpha}(\sigma-\sigma_{1})} \|U(t)\|_{\dot{B}^{\sigma}_{p,1}}\gtrsim t^{\frac{1}{2\alpha}(\sigma-\sigma_{1})} \|U(t)\|_{\dot{B}^{\sigma}_{p,r}}>\eta_{*}>0
\end{aligned}
\end{equation}
for all $t>t_{L}$ and some global-in-time constant $\eta_{*}>0$. In particular, \eqref{mm11} holds true with $t=t_{L}+k$ for all $k=1,2,\dots$, and then we define
$$
j_{1,k}\coloneqq  -\bigg[\frac{1}{2\alpha}\log_{2}(t_{L}+k)\bigg]-1.
$$
Making use of \eqref{mm11} and the fact $2^{-2\alpha (j_{1,k}+1)}\leq t_{L}+k\leq 2^{-2\alpha j_{1,k}}$, we get
 \begin{equation}\nonumber
\begin{aligned}
\sum_{ j\in\mathbb{Z} }e^{- R_{0} 2^{2\alpha(j-j_{1,k}-1)}}  2^{(\sigma-\sigma_{1}) (j-j_{1,k})} 2^{\sigma_{1} j}\|\dot{\Delta}_{j}U_{0}\|_{L^p} > \eta_{*}>0.
\end{aligned}
\end{equation}
Shifting the index $j-j_{1,k}$ to $j'$, we deduce that
 \begin{equation}\label{226}
\begin{aligned}
&\sum_{  j' \in\mathbb{Z} } e^{- R_{0} 2^{2\alpha(j'-1)} }  2^{(\sigma-\sigma_{1})j'} \Big( 2^{(j'+j_{1,k})\sigma_{1}}\|\dot{\Delta}_{j'+j_{1,k}}U_{0}\|_{L^p}\Big) >\eta_{*}>0.
\end{aligned}
\end{equation}
Due to $e^{-  R_{0} 2^{2\alpha(j-j_{1,k}-1)}} 2^{j'(\sigma-\sigma_{1})} \in l^1(\mathbb{Z})$ for $\sigma>\sigma_{1}$,  there exists a sufficiently large integer $J>0$ such that
\begin{align}
&\sum_{|j'|>J} e^{- R_{0} 2^{2\alpha(j-j_{1,k}-1)}} 2^{(\sigma-\sigma_{1})j'}<\frac{\eta_{*}}{2\|U_{0}\|_{\dot{B}^{\sigma_{1}}_{p,\infty}}+1}.\nonumber
\end{align}
Consequently, we have
\begin{align}
&\sum_{|j'|>J} e^{- R_{0} 2^{2\alpha(j'-1)}} 2^{(\sigma-\sigma_{1})j'} \Big( 2^{(j'+j_{1,k})\sigma_{1}}\|\dot{\Delta}_{j'+j_{1,k}}U_{0}\|_{L^p}\Big) <\frac{\eta_{*}}{2}. \label{2261}
\end{align}
It follows from \eqref{226} and \eqref{2261} that
 \begin{equation}\label{267}
\begin{aligned}
&\sum_{|j'|\leq J} e^{- R_{0} 2^{2\alpha(j'-1)}}  2^{(\sigma-\sigma_{1})j'} \Big( 2^{(j'+j_{1,k})\sigma_{1}}\|\dot{\Delta}_{j'+j_{1,k}}U_{0}\|_{L^p}\Big) >\frac{\eta_{*}}{2}>0.
\end{aligned}
\end{equation}
For every given $j_{1,k}$, let $j_{2,k}\in [-J,J]$ be the integer such that
 \begin{equation}\nonumber
\begin{aligned}
&2^{(j_{2,k}+j_{1,k})\sigma_{1}}\|\dot{\Delta}_{j_{2,k}+j_{1,k}}U_{0}\|_{L^p}=\max_{|j'|\leq J}2^{(j'+j_{1,k})\sigma_{1}}\|\dot{\Delta}_{j'+j_{1,k}}U_{0}\|_{L^p}.
\end{aligned}
\end{equation}
If we define
\begin{equation}\nonumber
\begin{aligned}
&j_{k}\coloneqq  j_{1,k}+j_{2,k},\quad\quad k=1,2,\dots,
\end{aligned}
\end{equation}
then it follows from \eqref{267} and the definitions of $j_{1,k}$, $j_{2,k}$ that $j_{k}\rightarrow-\infty$ as $k\rightarrow\infty$,
\begin{equation}\nonumber
\begin{aligned}
&|j_{k}-j_{k+1}|\leq 2J+\frac{1}{2\alpha}\log_{2}(1 +\frac{1}{t_{L}})+1\quad \text{and}\quad 2^{\sigma_{1}j_{k}}\|\dot{\Delta}_{j_{k}}U_{0}\|_{L^p}\geq \tilde{c}
\end{aligned}
\end{equation}
with some constant $\tilde{c}>0$. This implies that $U_{0}\in \dot{\mathcal{B}}^{\sigma_{1}}_{p,\infty}.$ The proof of Proposition \ref{propgeneral2} is now complete.
\end{proof}

\subsection{Partially dissipative hyperbolic systems}\label{section:linear}
In this section, we are concerned with the Cauchy problem of the following conservative-dissipative form
\begin{equation}\label{normal:linear0}
\left\{
    \begin{aligned}
&\partial_{t}V+\sum_{i=1}^{d}A^{i}\partial_{x_{i}}V+LV=0,\\
&V(0,x)=V_0(x).
\end{aligned}
\right.
\end{equation}
Here $V=(V_1,V_2)^{\top}(t,x)$ with $V_1\in \mathbb{R}^{n_1}$ and $V_2\in \mathbb{R}^{n_2}$ is the unknown on $(t,x)\in \mathbb{R}_+\times\mathbb{R}^d$ ($d\geq1$) taking values in an open convex set $\mathcal{O}_{V}\subset\mathbb{R}^n$ ($n=n_1+n_2\geq2$). The constant symmetric matrices $A^i$ ($i=1,2,\dots,d$) and $L$ take the form
\begin{equation*}
\begin{aligned}
    A^{i}= \begin{pmatrix} 0  & A^{i}_{1,2} \\ A^{i}_{2,1} & A^{i}_{2,2} \end{pmatrix}\quad \text{and}\quad L=\begin{pmatrix}
    0 & 0 \\ 0 & D
    \end{pmatrix}.
\end{aligned}
\end{equation*}
Here, $D$ is an $n_2\times n_2$ positive definite matrix. Suppose that the symmetric form \eqref{normal:linear0} satisfies the [SK] condition (see Definition \ref{defnSK}).

 To emphasize the partially dissipative structure, we rewrite \eqref{normal:linear0} as
\begin{equation}\label{hplinear}
\left\{
\begin{aligned}
& \partial_{t}V_{1}+\sum_{i=1}^{d}A_{1,2}^{i}\partial_{x_{i}}V_{2}=0,\\
&\partial_{t}V_{2}+\sum_{i=1}^{d}\Big(A_{2,1}^{i}\partial_{x_{i}}V_{1}+A_{2,2}^{i}\partial_{x_{i}}V_{2}\Big)+DV_{2}=0,\\
&(V_1,V_2)|_{t=0}=(V_{1,0},V_{2,0}).
\end{aligned}
\right.
\end{equation}

In the multidimensional case, the explicit form of Green's function associated with \eqref{hplinear} cannot,
in general, be explicitly  expressed. As a consequence, we propose a general $L^p$ hypercoercivity method and give the upper and lower bounds of time-decay estimates (Proposition \ref{proplinear}). Our key ingredient lies in the introduction of a new {\emph{effective quantity}}
\begin{align}
&\Psi\coloneqq   V_1-\sum_{i=1}^{d}A^{i}_{1,2}D^{-1} \partial_{x_{i}} V_2\label{effective1}
\end{align}
associated with the initial data
\begin{align}
&\Psi_0\coloneqq  V_{1,0}-\sum_{i=1}^{d}A^{i}_{1,2}D^{-1} \partial_{x_{i}} V_{2,0},\label{W0}
\end{align}
 which describes the coupling of conservative and dissipative parts. The additional regularity of the dissipative component $V_2$
compensates the spatial derivative in the definition of $\Psi$, which leads to weaker low-frequency assumptions on $V_{2,0}$.

Using \eqref{effective1}, we can capture the precise dependence on the initial data for the leading heat diffusion part and then establish sharp time asymptotics. Precisely, let us consider the linear parabolic problem
\begin{equation}\label{diffusion}
\left\{
\begin{aligned}
&\partial_{t}v=\mathcal{A}v,\\
&v(0,x)=\Psi_0(x),
\end{aligned}
\right.
\end{equation}
where the operator $\mathcal{A}$ is defined by
\begin{align}
\mathcal{A}=\sum_{i,l=1}^{d}A_{1,2}^{i}D^{-1}A_{2,1}^{l} \partial^2_{x_{i} x_{l}}.\label{mathcalA}
\end{align}
Under the [SK] condition, $-\mathcal A$ is strongly elliptic (cf. \eqref{eq:SK-elliptic}). Observe that the parabolic system \eqref{diffusion} can be derived by the Chapman-Enskog expansion (see for example \cite{BHN,Yong041}), which
is a good approximation of the conservative part of the solution to the first-order hyperbolic systems \eqref{hplinear}. We aim to show the asymptotic equivalence of the solution $V$ to \eqref{hplinear} and the new profile
\begin{equation}\label{Vp}
\begin{aligned}
 V^{*}=
    \begin{pmatrix}
   V_1^*  \\  V_2^*
    \end{pmatrix}\quad\text{with}\quad V_1^*\coloneqq   e^{t\mathcal{A}}\Psi_0\quad\text{and}\quad V_2^*\coloneqq   -\sum_{i=1}^{d}D^{-1} A^i_{2,1} \partial_{x_i} e^{t\mathcal{A}}\Psi_0,
\end{aligned}
\end{equation}
where $V_1^*$ is the solution to the linear parabolic equation \eqref{diffusion}. The structure allows us to develop the decay character theorem from linear diffusive systems with Fourier multiplier (see Proposition \ref{propgeneral2}) to the partially dissipative hyperbolic systems \eqref{hplinear}. Indeed, the main idea comes from the spectral analysis for damped compressible Euler equations (see Appendix \ref{appendixB}).

First, we establish the following spectral localization
estimates for the conservative component $V_1$ and the dissipative component $V_2$, respectively.

\begin{prop}\label{LemmaspectrallocalHp}
Let $1<p<\infty$, $ \lambda>0$, and define the annulus  $\mathcal{C}\coloneqq   \{\xi\in\mathbb{R}^{d}~:~0<a_{1}<|\xi|<a_{2}<\infty\}$. 
Let $V^{\lambda}=(V^{\lambda}_{1},V^{\lambda}_{2})^{\top}$ with $V^{\lambda}_{1}\in \mathbb{R}^{n_{1}}$ and $V^{\lambda}_{2}\in \mathbb{R}^{n_{2}}$ be the global solution to  \eqref{hplinear} supplemented with the initial data 
$V^{\lambda}_{0}=(V^{\lambda}_{1,0},V^{\lambda}_{2,0})^{\top}$. Assume that the Fourier transforms of $V^{\lambda}$ and $V^{\lambda}_{0}$ are both supported in $\lambda \mathcal{C}$. Then there exists a small constant $\lambda_{0}$ such that for all $\lambda\in (0,\lambda_{0})$, 
\begin{itemize}

\item [$(\rm i)$]  The whole solution $V^{\lambda}$ satisfies
\begin{align}
    \|V^{\lambda}(t)\|_{L^p}\leq C_{1} e^{-R_{1}\lambda^2t} \|V^{\lambda}_{0}\|_{L^p}\label{lowlocaluppertotal}
\end{align}
and more accurate estimates in terms of $(V^{\lambda}_1,V^{\lambda}_2)${\rm:}
\begin{equation}\label{lowlocalupper}
\left\{
\begin{aligned}
\|V_{1}^{\lambda}(t)\|_{L^p}&\leq C_{1} e^{-R_{1}\lambda^2t}\Big( \|V^{\lambda}_{1,0}\|_{L^p}+\lambda \|V^{\lambda}_{2,0}\|_{L^p}\Big)+C_{1} e^{-\kappa_{1} t}\Big(\lambda^2 \|V^{\lambda}_{1,0}\|_{L^p}+\lambda \|V^{\lambda}_{2,0}\|_{L^p}\Big),\\
\|V^{\lambda}_{2}(t)\|_{L^p}&\leq C_{1} \lambda e^{-R_{1}\lambda^2 t}\Big(\|V^{\lambda}_{1,0}\|_{L^p}+\lambda \|V^{\lambda}_{2,0}\|_{L^p}\Big)+C_{1} e^{-\kappa_{1} t}\Big(\lambda \|V^{\lambda}_{1,0}\|_{L^p}+\|V^{\lambda}_{2,0}\|_{L^p}\Big).
\end{aligned}
\right.
\end{equation}

\item [$(\rm ii)$] Furthermore, we provide the
asymptotic profiles for the solution $(V^{\lambda}_{1},V^{\lambda}_{2})^{\top}$. Precisely,
\begin{equation}\label{asy:error}
\begin{aligned}
\|(V_{1}^{\lambda}-V_{1,*}^\lambda)(t)\|_{L^p}&\leq C_{1}\lambda  e^{-R_{1}\lambda^2t} \Big(\|V^{\lambda}_{1,0}\|_{L^p}+\lambda \|V^{\lambda}_{2,0}\|_{L^p}\Big)+C_{1} e^{-\kappa_{1} t}\Big(\lambda^2 \|V^{\lambda}_{1,0}\|_{L^p}+\lambda \|V^{\lambda}_{2,0}\|_{L^p}\Big)
\end{aligned}
\end{equation}
and
\begin{equation}\label{asy:error2}
\begin{aligned}
\| (V_2^{\lambda}- V_{2,*}^{\lambda})(t)\|_{L^p}
&\leq C_1 \lambda^2 e^{-R_1\lambda^2t} \Big(\|V^\lambda_{1,0}\|_{L^p}+\lambda\|V^\lambda_{2,0}\|_{L^p}\Big)+ C_1 e^{-\kappa_1 t} \Big(\|V^\lambda_{1,0}\|_{L^p}+\lambda\|V^\lambda_{2,0}\|_{L^p}\Big).
\end{aligned}
\end{equation}
\end{itemize}
Here $C_{1}, R_{1}$ and $\kappa_{1}$ are generic constants independent of $\lambda$ and $t$. The profiles $V_{1,*}^\lambda$ and $V_{2,*}^\lambda$ are defined by
\begin{equation}\label{profilelambda}
\begin{aligned}
V_{1,*}^\lambda\coloneqq e^{t\mathcal{A}}\Psi_0^{\lambda}\quad \text{and}\quad  V_{2,*}^\lambda\coloneqq -\sum_{i=1}^{d} D^{-1} A_{2,1}^i \partial_{x_i} e^{t\mathcal{A}}\Psi^{\lambda}_0,
\end{aligned}
\end{equation}
where the operator $\mathcal{A}$ is defined in \eqref{mathcalA}, and $\Psi^{\lambda}_0$ is given by
 \begin{align}
\Psi^{\lambda}_0\coloneqq   V^{\lambda}_{1,0}-\sum_{i=1}^{d} A^{i}_{1,2}D^{-1} \partial_{x_{i}} V^{\lambda}_{2,0}.\label{Psi0lambda}
\end{align}


\end{prop}

\begin{remark}\normalfont
In the spectral analysis
of \eqref{hplinear}, a challenging obstacle lies in the absence of the explicit formula of Green's function. To the best of our knowledge, there is little existing literature on the lower bound of decay estimates of solutions so far. In order to overcome the difficulty, as in Appendix \ref{appendixB}, we observe that the quantity  $\Psi_0^{\lambda}$ plays a key role in
optimal decay estimates, which implies that  $V_{1}^{\lambda}-V_{1,*}^\lambda$ and $V_2^\lambda- V_{2,*}^\lambda$ enjoy better decay properties compared to those of $V_{1}^{\lambda}$ and $V_{2}^{\lambda}$, respectively (see \eqref{lowlocalupper}-\eqref{asy:error2}). 
It should be emphasized that this is a completely new viewpoint of independent interest, leading to the sharp decay characterization for partially dissipative hyperbolic systems in the multidimensional setting.



\end{remark}

\begin{proof}
The proof of Proposition \ref{LemmaspectrallocalHp} can be divided into several steps for clarity.

\begin{itemize}

\item  {\emph{Step 1: Decoupling 
}}

\end{itemize}

We introduce the new {\emph{effective variable}}
\begin{align}
&\Psi^\lambda\coloneqq  V_1^\lambda-\sum_{i=1}^{d}A^{i}_{1,2}D^{-1} \partial_{x_{i}} V_2^\lambda \label{Psilambda}
\end{align}
associated with the initial data \eqref{Psi0lambda}. Moreover, in the spirit of \cite{c3}, we define
\begin{equation}
\begin{aligned}
&Z^\lambda\coloneqq    V_{2}^\lambda+\sum_{i=1}^{d}D^{-1}\Big(A_{2,1}^{i}\partial_{x_{i}}V_{1}^\lambda+A_{2,2}^{i}\partial_{x_{i}}V_{2}^\lambda\Big), \label{Z}
\end{aligned}
\end{equation}
and the initial data are correspondingly  prescribed by
\begin{equation}
\begin{aligned}
&Z_0^\lambda\coloneqq    V_{2,0}^\lambda+\sum_{i=1}^{d}D^{-1}\Big(A_{2,1}^{i}\partial_{x_{i}}V_{1,0}^\lambda+A_{2,2}^{i}\partial_{x_{i}}V_{2,0}^\lambda\Big) .\label{Z0}
\end{aligned}
\end{equation}

Furthermore, the first two equations in \eqref{hplinear} can be
 decoupled in terms of $(\Psi^\lambda,Z^\lambda)$ such that they are able to be written as a parabolic model and a damped model, up to some higher-order terms.

\begin{lemma}\label{lemma32}
Let $\Psi^\lambda$ and $Z^\lambda$ be defined by \eqref{Psilambda}-\eqref{Z}. Then, for all $0<\lambda\leq \lambda_0\ll1$, it holds that
\begin{equation}\label{3.29}
\left\{
\begin{aligned}
&\partial_{t}\Psi^\lambda-\mathcal{A}\Psi^\lambda=L_1(\Psi^\lambda,Z^\lambda),\\
&\partial_t Z^\lambda+ D Z^\lambda=L_2(\Psi^\lambda,Z^\lambda),
\end{aligned}
\right.
\end{equation}
where  $L_i(\Psi^\lambda,Z^\lambda)$ {\rm(}$i=1,2${\rm)} are given by
\begin{equation*}
\begin{aligned}
L_1(\Psi^\lambda,Z^\lambda)&\coloneqq    \bigg(\sum_{i,l=1}^{d} A^i_{1,2}D^{-1}A^l_{2,2} \partial^2_{x_{i} x_{l}}+\sum_{i=1}^{d}\mathcal{A}A_{1,2}^i D^{-1}\partial_{x_i} \bigg)     \mathcal{B}^{-1} \Big(Z^\lambda-\sum_{i=1}^d D^{-1} A_{2,1}^i  \partial_{x_{i}}\Psi^\lambda\Big) , \\
L_2(\Psi^\lambda,Z^\lambda)&\coloneqq     -\sum_{i=1}^{d} D^{-1}\bigg(A_{2,1}^{i}\sum_{l=1}^d A_{1,2}^l\mathcal{B}^{-1} \Big(\partial^2_{x_i x_l} Z^\lambda-\sum_{m=1}^d D^{-1} A_{2,1}^m  \partial_{x_i x_l x_m}^3 \Psi^\lambda\Big)+A_{2,2}^{i}D \partial_{x_i} Z^\lambda \bigg).
\end{aligned}
\end{equation*}
For any distribution $f^\lambda$ fulfilling  $\operatorname{Supp}\mathcal{F}(f^\lambda)\subset \lambda \mathcal{C}$ with $0<\lambda\leq \lambda_0\ll1$, the Fourier multiplier
\begin{align*}
\mathcal{B}^{-1}\coloneqq    \mathcal{F}^{-1}\bigg( \bigg({\rm Id}+\sum_{i=1}^d D^{-1} A^i_{2,2} \mathrm{i} \xi_i-\sum_{i,l=1}^d D^{-1} A_{2,1}^i  A_{1,2}^l D^{-1}  \xi_i \xi_l \bigg)^{-1} \mathcal{F}\Big(\cdot\Big)\bigg)
\end{align*}
satisfies $\|\mathcal{B}^{-1}(f^\lambda)\|_{L^p}\leq 2 \|f^\lambda\|_{L^p}$   $(1\leq p\leq \infty)$.
\end{lemma}

\begin{proof}
It follows from the second equation in $\eqref{hplinear}$ that
\begin{align}
V_{2}^\lambda=-D^{-1} \partial_{t} V_{2}^\lambda-\sum_{i=1}^{d}D^{-1}\Big(A_{2,1}^{i}\partial_{x_{i}}V_{1}^\lambda+A_{2,2}^{i}\partial_{x_{i}}V_{2}^\lambda\Big).\label{V200}
\end{align}
Substituting \eqref{V200} into the first equation in $\eqref{hplinear}$ yields
\begin{equation}\label{formalcouple}
\begin{aligned}
&\partial_{t}\Psi^\lambda-\mathcal{A}\Psi^\lambda=\bigg(\sum_{i,l=1}^{d} A^i_{1,2}D^{-1}A^l_{2,2} \partial^2_{x_{i} x_{l}}+\sum_{i=1}^{d}\mathcal{A}A_{1,2}^iD^{-1}\partial_{x_i} \bigg)V_2^\lambda.
\end{aligned}
\end{equation}
On the other hand, differentiating $Z$ in time and using the fact $\partial_{t}V_{2}=-DZ$ due to $\eqref{hplinear}$, we have
\begin{equation}\label{Zformal}
\begin{aligned}
\partial_{t}Z^\lambda+D Z^\lambda&=-\sum_{i=1}^{d}\partial_{x_{i}} D^{-1}\Big(A_{2,1}^{i}\sum_{l=1}^d A_{1,2}^l\partial_{x_l}V_{2}^\lambda+A_{2,2}^{i}DZ^\lambda\Big).
\end{aligned}
\end{equation}
According to the definitions \eqref{effective1} of $\Psi$ and \eqref{Z} of $Z$, one knows that
\begin{align}
Z^\lambda-\sum_{i=1}^d D^{-1} A_{2,1}^i \Psi^\lambda=\mathcal{B} V_2^\lambda.
\end{align}
where the operator $\mathcal{B}$ is given by
$$
\mathcal{B}:={\rm Id}+\widetilde{\mathcal{B}}\quad\text{with} \quad  \widetilde{\mathcal{B}}:=\sum_{i=1}^d D^{-1} A^i_{2,2}\partial_{x_i}+\sum_{i,l=1}^d D^{-1} A_{2,1}^i  A_{1,2}^l D^{-1}\partial_{x_{i}x_{l}}^2.
$$
Let $E_\lambda:=\{f^\lambda\in\mathcal S' : \operatorname{Supp}\widehat f\subset\lambda\mathcal C\}$.
By Bernstein's inequality,  $\widetilde{\mathcal{B}} :E_\lambda\to E_\lambda$ and for $f^\lambda\in E_\lambda$, $
\|\widetilde{\mathcal{B}} f^\lambda\|_{L^p}\lesssim (\lambda+\lambda^2)\,\|f^\lambda\|_{L^p}$. If $0<\lambda\leq \lambda_0\ll1$, then  $\|\widetilde{\mathcal{B}}\|_{E_\lambda\to L^p}\le\tfrac12$. Consequently, $\mathcal B|_{E_\lambda}:E_\lambda\to E_\lambda$ is a bounded bijection (isomorphism) and satisfies $\|\mathcal B|_{E_\lambda}\|_{L^p\to L^p}\geq \frac{1}{2}$. This gives $\|\mathcal B^{-1}f^\lambda\|_{L^p}\le 2\|f^\lambda\|_{L^p}$ for all $f^\lambda\in E_\lambda$, $1\leq p\leq\infty$ and $0<\lambda\le\lambda_0\ll1$.


Therefore, $(V_1^\lambda,V_2^\lambda)$ can be represented in terms of $(\Psi^\lambda,Z^\lambda)$ as
\begin{align}
V_2^\lambda&=\mathcal{B}^{-1}\Big(Z^\lambda-\sum_{i=1}^d D^{-1} A_{2,1}^i  \partial_{x_{i}}\Psi^\lambda\Big),\label{V2}\\
V_1^\lambda&=\Psi^\lambda+\sum_{i=1}^d A^{i}_{1,2}D^{-1} \partial_{x_i}\mathcal{B}^{-1}\Big(Z^\lambda-\sum_{i=1}^d D^{-1} A_{2,1}^i  \partial_{x_{i}} \Psi^\lambda\Big).\label{V1}
\end{align}
Substituting \eqref{V2} and \eqref{V1} into \eqref{formalcouple} and \eqref{Zformal}, respectively, we obtain \eqref{3.29}.
\end{proof}

Next, we establish the localized Lyapunov inequality for the new variable $(\Psi^\lambda,Z^\lambda)$ in low frequencies.

\begin{lemma}
Let $1<p<\infty$, and let $\Psi^\lambda$ and $Z^\lambda$ be defined in \eqref{Psilambda} and \eqref{Z}, respectively. Then, for any $\lambda\in(0, \lambda_0]$ and $\varepsilon>0$, there exists a suitably small constant $\lambda_0$ such that we have
\begin{equation}\label{3.38}
\begin{aligned}
&\frac{d}{dt}\Big(\|\Psi^\lambda\|_{\varepsilon,L^p}+\lambda\|Z^\lambda\|_{\varepsilon,L^p}\Big)+R_1 \lambda^2\Big( \|\Psi^\lambda\|_{\varepsilon,L^p}+\lambda\|Z^\lambda\|_{\varepsilon,L^p}\Big)\leq C\varepsilon.
\end{aligned}
\end{equation} where $\|\cdot\|_{\varepsilon,L^p}\coloneqq   (\|\cdot\|_{L^p}^p+\varepsilon^p)^{1/p}$, and $C, R_1>0$ are some generic constants.
\end{lemma}

\begin{proof}
Multiplying each component of the equations for $\Psi^\lambda$ in $\eqref{3.29}$ by $|(\Psi^\lambda)^i|^{p-2}(\Psi^\lambda)^i$ ($i=1,\dots,n_1$) and integrating over $\mathbb{R}^d$, we arrive at
\begin{equation}\label{fffffffs}
    \begin{aligned}
    &\frac{1}{p}\frac{d}{dt}\|(\Psi^\lambda)^i\|_{L^p}^p-\int_{\mathbb{R}^d} \mathcal{A}(\Psi^\lambda)^i |(\Psi^\lambda)^i|^{p-2}(\Psi^\lambda)^i \,dx=\int_{\mathbb{R}^d}|(\Psi^\lambda)^i|^{p-2}(\Psi^\lambda)^i L_1(\Psi^\lambda,Z^\lambda)^i \,dx.
    \end{aligned}
\end{equation}
Note that $-\mathcal{A}$ is strongly elliptic due to the [SK] condition, so Lemma \ref{lemmaA2} implies that there exists a constant $c^*_p$ such that the second term in the left-hand side of \eqref{fffffffs} can be greater than $c_p^* \lambda^2 \|(\Psi^\lambda)^i\|_{L^p}^p$. Consequently, performing the summation of \eqref{fffffffs} over $i=1,\dots,n_1$ leads to
\begin{equation}\label{340}
    \begin{aligned}
    &\frac{1}{p}\frac{d}{dt}\|\Psi^\lambda\|_{L^p}^p+c^*_p\lambda^2 \|\Psi^\lambda\|_{L^p}^p\leq \|L_1(\Psi^\lambda,Z^\lambda)\|_{L^p}\|\Psi^\lambda\|_{L^p}^{p-1}.
    \end{aligned}
\end{equation}
On the other hand, by multiplying the equations for $Z^\lambda$ in $\eqref{3.29}$ by $|Z^\lambda|^{p-2}Z^\lambda$ and using
$$
\langle D Z^\lambda,|Z^\lambda|^{p-2}Z^\lambda \rangle=\langle D(|Z^\lambda|^{\frac{p-2}{2}}Z^\lambda),|Z^\lambda|^{\frac{p-2}{2}}Z^\lambda \rangle \geq \kappa |Z^\lambda|^{p}
$$
due to \eqref{dissipationU211}, we obtain
\begin{equation}\label{341}
    \begin{aligned}
    &\frac{1}{p}\frac{d}{dt}\|Z^\lambda\|_{L^p}^p+\kappa\|Z^\lambda\|_{L^p}^{p}\leq \|L_2(\Psi^\lambda,Z^\lambda)\|_{L^p}\|Z^\lambda\|_{L^p}^{p-1}.
    \end{aligned}
\end{equation}
It is convenient to take $\lambda\leq \lambda_0^*$ with $\lambda_0^*$ suitably small such that $\|\mathcal{B}^{-1}f^{\lambda}\|_{L^p}\leq  2\|f^\lambda\|_{L^p}$ for distribution $f^\lambda$ satisfying  $\operatorname{Supp}\mathcal{F}(f^\lambda)\subset \lambda \mathcal{C}$. Then, for all $\lambda\leq \min\{1,\lambda_0^*\}$, Bernstein's inequality  \ref{lemma21} implies that the high-order remainder terms $L_i(\Psi^\lambda,Z^\lambda)$ ($i=1,2$) can be bounded by
\begin{align}
\|L_1(\Psi^\lambda,Z^\lambda)\|_{L^p}&\leq C\lambda^2 \|Z^\lambda\|_{L^p}+C\lambda^3 \|\Psi^\lambda\|_{L^p},\label{342}\\
\|L_2(\Psi^\lambda,Z^\lambda)\|_{L^p}&\leq C\lambda \|Z^\lambda\|_{L^p}+C\lambda^3 \|\Psi^\lambda\|_{L^p}\label{343}
\end{align}
for some uniform constant $C>0$. Dividing both sides of \eqref{340} and \eqref{341} by $\|\Psi^\lambda\|_{\varepsilon,L^p}^{p-1}$ and $\|Z^\lambda\|_{\varepsilon,L^p}^{p-1}$, respectively, combining \eqref{342} and \eqref{343} and then choosing $\lambda_0\leq \min\{1,\lambda_0^*, \frac{c_p^*}{4C}, \frac{\kappa}{4C}\}$ and $R_1\coloneqq   \min\{\frac{c_p^*}{2},\frac{\kappa}{2}\}$, we get \eqref{3.38} immediately.
\end{proof}

\begin{itemize}

\item {\emph{Step 2: The $L^p$ upper bounds of $V_1^\lambda$ and $V_2^\lambda$}}

\end{itemize}

In this step, we shall prove the estimates \eqref{lowlocaluppertotal} and \eqref{lowlocalupper}. For that end, taking advantage of Gr\"onwall's inequality for \eqref{3.38} and letting $\varepsilon\rightarrow 0$, we arrive at
\begin{equation}\label{totalPsiZ}
\begin{aligned}
\|\Psi^\lambda\|_{L^p}+\lambda \|Z^\lambda\|_{L^p}\lesssim e^{-R_1 \lambda^2t}\Big(\|\Psi^\lambda_0\|_{L^p}+\lambda \|Z_0^\lambda\|_{L^p}\Big).
\end{aligned}
\end{equation}
The above estimate exhibits the asymptotic behavior of $(\Psi^\lambda,Z^\lambda)$. In fact, one can capture  more accurate behavior for $Z^\lambda$. Recalling \eqref{341} and \eqref{343}, we have
\begin{equation*}
\begin{aligned}
&\frac{d}{dt}\|Z^\lambda\|_{\varepsilon,L^p}+(\kappa-C\lambda) \|Z^\lambda\|_{\varepsilon,L^p}\leq C\lambda^3 \|\Psi^\lambda\|_{L^p}+(\kappa-C\lambda)\varepsilon.
\end{aligned}
\end{equation*}
As $\lambda\leq \lambda_0<\frac{\kappa}{2C}$,  Gr\"onwall's inequality ensures that
\begin{equation}\label{Z:Lp}
\begin{aligned}
\|Z^\lambda\|_{L^p}&\leq e^{-\frac{\kappa}{2}t}\|Z^\lambda_0\|_{L^p}+C\lambda^3 \int_0^te^{-\frac{\kappa}{2}(t-\tau)}\|\Psi^\lambda\|_{L^p}\,d\tau\\
&\leq e^{-\frac{\kappa}{2}t}\|Z_0^\lambda\|_{L^p}+C\lambda^3 \sup_{\tau\in[0,t]}\Big(e^{R_1\lambda^2\tau}\|\Psi^\lambda\|_{L^p}\Big)\,\int_0^te^{-\frac{\kappa}{2}(t-\tau)} e^{-R_1\lambda^2\tau}\,d\tau\\
&\lesssim e^{-\frac{\kappa}{2}t}\|Z_0^\lambda\|_{L^p}+\lambda^3 e^{-R_1 \lambda^2t}\Big(\|\Psi^\lambda_0\|_{L^p}+\lambda \|Z_0^\lambda\|_{L^p}\Big)
\end{aligned}
\end{equation}
where we have used \eqref{totalPsiZ} and the fact that, if $\lambda\leq \sqrt{\frac{C'}{2c'}}$,
\begin{equation}\label{maomao1}
\begin{aligned}
\int_{0}^{t}e^{-C'(t-\tau)}e^{-c'\lambda^2 \tau}\, d\tau= e^{-C't}\int_{0}^{t} e^{(C'-c'\lambda^2)\tau}\,d\tau\leq \frac{1}{2C'} e^{-c'\lambda^2 t},
\,\quad c', C'>0.
\end{aligned}
\end{equation}

With these estimates of $\Psi^\lambda$ and $Z^\lambda$ at hand, we can get the desired estimates for $V_1^\lambda$ and $V_2^\lambda$. Indeed, owing to \eqref{V2}, \eqref{V1}, \eqref{totalPsiZ}, \eqref{Z:Lp} and \ref{lemma21}, it holds that
\begin{equation}\label{V1es:0}
\begin{aligned}
\|V_1^\lambda\|_{L^p}&\lesssim \|\Psi^\lambda\|_{L^p}+\lambda \|Z^\lambda\|_{L^p}\lesssim e^{-R_1 \lambda^2t}\Big(\|\Psi^\lambda_0\|_{L^p}+\lambda \|Z_0^\lambda\|_{L^p}\Big)+e^{-\frac{\kappa}{2}t}\lambda \|Z_0^\lambda\|_{L^p}
\end{aligned}
\end{equation}
and
\begin{equation}\label{V2es:0}
\begin{aligned}
\|V_2^\lambda\|_{L^p}&\lesssim \lambda \|\Psi^\lambda\|_{L^p}+\|Z^\lambda\|_{L^p}\lesssim e^{-R_1 \lambda^2t}\lambda\Big(\|\Psi^\lambda_0\|_{L^p}+\lambda \|Z_0^\lambda\|_{L^p}\Big)+e^{-\frac{\kappa}{2}t}\|Z_0^\lambda\|_{L^p}.
\end{aligned}
\end{equation}
It follows from \eqref{Psi0lambda} and \eqref{Z0} that
\begin{align}\label{V2es:001}
\|\Psi^\lambda_0\|_{L^p}\lesssim \|V_{1,0}^\lambda\|_{L^p}+\lambda \|V_{2,0}^\lambda\|_{L^p}\quad\text{and}\quad \|Z_0^\lambda\|_{L^p}\lesssim (1+\lambda)\|V_{2,0}^\lambda\|_{L^p}+\lambda \|V_{1,0}^\lambda\|_{L^p}.
\end{align}
Hence, \eqref{lowlocaluppertotal} and \eqref{lowlocalupper} are followed directly.

\begin{itemize}

\item \emph{Step 3: Improved $L^p$ bounds}

\end{itemize}

Furthermore, we have the asymptotic behaviors \eqref{asy:error} and \eqref{asy:error2}. Recall that $V_{1,*}^{\lambda}$ and $V_{2,*}^{\lambda}$ are given by \eqref{profilelambda}. By \eqref{formalcouple}, the error  $\widetilde{\Psi}^\lambda\coloneqq   \Psi^\lambda-V_{1,*}^{\lambda}$ satisfies
\begin{align}
\partial_t \widetilde{\Psi}^\lambda-\mathcal{A} \widetilde{\Psi}^\lambda=\bigg(\sum_{i,l=1}^{d} A^i_{1,2}D^{-1}A^l_{2,2} \partial^2_{x_{i} x_{l}}+\sum_{i=1}^{d}\mathcal{A}A_{1,2}^iD^{-1}\partial_{x_i} \bigg)V_2^\lambda,\quad \widetilde{\Psi}^\lambda|_{t=0}=0.\label{widePsi}
\end{align}
Then, arguing similarly to \eqref{fffffffs}-\eqref{340} yields
\begin{equation}
\begin{aligned}
\frac{d}{dt}\|\widetilde{\Psi}^\lambda\|_{\varepsilon,L^p}+R_1\lambda^2 \|\widetilde{\Psi}^\lambda\|_{\varepsilon,L^p} \lesssim \lambda^2  \|V_2^\lambda\|_{L^p}+\varepsilon.
\end{aligned}
\end{equation}
This, together with the estimate for $V_2^\lambda$ in $\eqref{lowlocalupper}$, leads to
\begin{equation}\label{fgggdd}
\begin{aligned}
\|\widetilde{\Psi}^\lambda\|_{L^p}&\lesssim \lambda^2 \int_0^t e^{-R_1 \lambda^2 (t-\tau)} \|V_2^\lambda\|_{L^p}\,d\tau\\
&\lesssim  \lambda^3 \int_0^t e^{-R_1  \lambda^2 (t-\tau)}  e^{-R_1  \lambda^2 \tau} \,d\tau\, \Big(\|V_{1,0}^\lambda\|_{L^p}+\lambda \|V_{2,0}^\lambda\|_{L^p}\Big)\\
&\quad+ \lambda^2 \int_0^t e^{-R_1  \lambda^2 (t-\tau)}  e^{-\kappa_1 \tau} \,d\tau \,\Big(\lambda^2\|V_{1,0}^\lambda\|_{L^p}+\lambda \|V_{2,0}^\lambda\|_{L^p}\Big)\\
&\lesssim \lambda  e^{-R_1  \lambda^2 t } \Big(\|V_{1,0}^\lambda\|_{L^p}+\lambda \|V_{2,0}^\lambda\|_{L^p}\Big),
\end{aligned}
\end{equation}
where we have used \eqref{maomao1} and
\begin{equation}\label{maomao}
\begin{aligned}
\int_{0}^{t}e^{-C'\lambda^2(t-\tau)}e^{-c'\lambda^2 \tau}\, d\tau=e^{-C'\lambda^2t}  \int_{0}^{t} e^{(C'-c')\lambda^2 \tau}d\tau\lesssim  e^{-c'\lambda^2t}\lambda^{-2},
\,\quad 0<c'< C'.
\end{aligned}
\end{equation}
According to \eqref{lowlocalupper}, \eqref{Psilambda} and \eqref{fgggdd}, we have
\begin{align*}
\|V_1^\lambda-V_{1,*}^\lambda\|_{L^p}&\lesssim \|\widetilde{\Psi}^\lambda\|_{L^p}+\lambda\|V_2^\lambda\|_{L^p}\lesssim\lambda  e^{-R_1  \lambda^2 t } \Big(\|V_{1,0}^\lambda\|_{L^p}+\lambda \|V_{2,0}^\lambda\|_{L^p}\Big)+   e^{-\kappa_1 t } \Big(\lambda^2\|V_{1,0}^\lambda\|_{L^p}+ \lambda \|V_{2,0}^\lambda\|_{L^p}\Big).
\end{align*}
This is \eqref{asy:error} exactly.

On the other hand, it follows from \eqref{Z}
 that
\begin{equation}\nonumber
\begin{aligned}
V_2^{\lambda}-V_{2,*}^{\lambda}&=-\sum_{i=1}^{d} D^{-1} A_{2,1}^i \partial_{x_i} (V_1^{\lambda}-V_{1,*}^{\lambda})-\sum_{i=1}^d D^{-1}A_{2,2}^i \partial_{x_i} V_2^{\lambda}+Z^{\lambda},
\end{aligned}
\end{equation}
which implies, by \eqref{lowlocalupper} and \eqref{asy:error} and \eqref{Z:Lp}, that
\begin{equation}\nonumber
\begin{aligned}
\| V_2^\lambda- V_{2,*}^{\lambda} \|_{L^p} \lesssim  \lambda^2  e^{-R_1  \lambda^2 t } \Big(\|V_{1,0}^\lambda\|_{L^p}+\lambda \|V_{2,0}^\lambda\|_{L^p}\Big)+ e^{-\kappa_1 t}\Big(\|V_{1,0}^\lambda\|_{L^p}+\lambda\|V_{2,0}^\lambda\|_{L^p}\Big).
\end{aligned}
\end{equation}
The inequality
\eqref{asy:error2} follows. Hence, the proof of Proposition \ref{LemmaspectrallocalHp} is complete.
\end{proof}

In the following proposition, we study the large-time behavior of solutions to \eqref{hplinear} pertaining to data in Besov spaces. Let $\mathbf{P}$ and $\{\mathbf{I} - \mathbf{P}\}$ denote the projections such that \eqref{3} holds, and $\dot{\mathbb{B}}^{\sigma,s}_{p,2}$ stand for the hybrid space endowed with the norm
$$
\|\cdot\|_{\dot{\mathbb{B}}^{\sigma,s}_{p,2}}=\|\cdot^{\ell}\|_{\dot{B}^{\sigma}_{p,1}}+\|\cdot\|_{\dot{B}^{s}_{2,1}}^{h}.
$$
See Appendix \ref{appendixA} for more details of Besov spaces or hybrid Besov spaces.

\begin{prop}\label{proplinear}
Let $1< p<\infty$, $\sigma_1, s\in \mathbb{R}$,  $\sigma>\sigma_1$ and $\sigma'>\sigma_1+1$. Assume that the initial data $V_{0}=\mathbf{P}V_0+\{\mathbf{I} - \mathbf{P}\}V_0$ fulfill $(\{\mathbf{I} - \mathbf{P}\}V_0)^{\ell}\in \dot{B}^{\sigma_1+1}_{p,\infty}$ and $V^h_0\in \dot{B}^{s}_{2,1}$ such that $\|V_0\|_{\dot{B}^{s}_{2,1}}^{h}<\infty$. Then the  solution $V=\mathbf{P}V+\{\mathbf{I} - \mathbf{P}\}V$ to \eqref{hplinear}
has the following decay properties.
\begin{itemize}
    \item  (Upper bounds){\rm:}  For any $t>0$, it holds that
    \begin{equation}\label{upperlinear}
\left\{
\begin{aligned}
&\|\mathbf{P}V(t)\|_{\dot{\mathbb{B}}^{\sigma,s}_{p,2}}\leq C \langle t\rangle^{-\frac{1}{2}(\sigma-\sigma_{1})}, \\
&\|\{\mathbf{I} - \mathbf{P}\}V(t)\|_{\dot{\mathbb{B}}^{\sigma',s}_{p,2}}\leq C \langle t\rangle^{-\frac{1}{2}(\sigma'-\sigma_{1})-\frac{1}{2}},
\end{aligned}
\right.
\end{equation}
if and only if $(\mathbf{P}V_{0})^{\ell}\in\dot{B}^{\sigma_{1}}_{p,\infty}$.

    \item  (Asymptotics){\rm:} If additionally $(\mathbf{P}V_{0})^{\ell}\in\dot{B}^{\sigma_{1}}_{p,\infty}$, then there exists a time $\tilde{t}_0>0$ such that
\begin{equation}
\left\{
\begin{aligned}
&\|\mathbf{P}(V-V^{*})(t)\|_{\dot{\mathbb{B}}^{\sigma,s}_{p,2}}\leq C \langle t\rangle^{-\frac{1}{2}(\sigma-\sigma_{1})-\frac{1}{2}},\\
&\|\{\mathbf{I} - \mathbf{P}\}(V-V^{*})(t)\|_{\dot{\mathbb{B}}^{\sigma',s}_{p,2}}\leq C \langle t\rangle^{-\frac{1}{2}(\sigma'-\sigma_{1})-1},\label{asy:linear}
\end{aligned}
\right.
\end{equation}
for any  $t\geq \tilde{t}_0$.

    \item  (Two-sided bounds){\rm:} There exists a time $\tilde{t}_1>0$ such that $V$ satisfies \eqref{upperlinear} and\begin{equation}\label{twosidelinear}
\left\{
\begin{aligned}
&c\langle t\rangle^{-\frac{1}{2}(\sigma-\sigma_{1})}\leq \|\mathbf{P}V(t)\|_{\dot{\mathbb{B}}^{\sigma,s}_{p,2}}\leq C \langle t\rangle^{-\frac{1}{2}(\sigma-\sigma_{1})}, \\
&c\langle t\rangle^{-\frac{1}{2}(\sigma'-\sigma_{1})-\frac{1}{2}}\leq \|\{\mathbf{I} - \mathbf{P}\} V(t)\|_{\dot{\mathbb{B}}^{\sigma',s}_{p,2}}\leq C\langle t\rangle^{-\frac{1}{2}(\sigma'-\sigma_{1})-\frac{1}{2}}
\end{aligned}
\right.
\end{equation}
for $t\geq \tilde{t}_1$, if and only if $\Psi_0^{\ell}\in\dot{\mathcal{B}}^{\sigma_{1}}_{p,\infty}$. 
\end{itemize}
Here $c,C>0$ are some constants independent of $t$, $\Psi_0$ and $V^*$ are given by \eqref{W0} and \eqref{Vp}, respectively. 
\end{prop}

\noindent
\textbf{{Proof.}} Let the threshold $J_0$ (see Appendix \ref{appendixA}) satisfy $2^{J_0}\leq \lambda_{0}$, where the small constant $\lambda_{0}$ is given by Proposition \ref{LemmaspectrallocalHp}. 
From \eqref{eigh:SK} and Parseval's equality, it is shown that there exists a constant $R_*>0$ such that
the solution $V$ has the following exponential decay in high frequencies
\begin{align}
&\|V\|_{\dot{B}^{s}_{2,1}}^{h}\lesssim e^{-R_{*}t}\|V_{0}\|_{\dot{B}^{s}_{2,1}}^{h}\label{hhhh}
\end{align}
for $s\in \mathbb{R}$. See also Proposition 4.1 in \cite{XK2} for similar details.

In order to get the upper bounds in \eqref{upperlinear}, we need to focus only on the sharp algebra decay in low frequencies under the additional condition $(\mathbf{P}V_0)^{\ell}\in\dot{B}^{\sigma_1}_{p,\infty}$.
It follows from $\eqref{lowlocalupper}$ that
\begin{equation}\label{ssfg}
\left\{
\begin{aligned}
\|\dot{\Delta}_{j}(\mathbf{P}V)^{\ell}(t)\|_{L^p}&\lesssim e^{-R_{1}2^{2j}t}\Big(\|\dot{\Delta}_{j} (\mathbf{P}V_0)^{\ell}\|_{L^p}+2^{j}\|\dot{\Delta}_{j} (\{\mathbf{I} - \mathbf{P}\}V_0)^{\ell}\|_{L^p}\Big)\\
\|\dot{\Delta}_{j}(\{\mathbf{I} - \mathbf{P}\}V)^{\ell}(t)\|_{L^p}&\lesssim e^{-R_{1}2^{2j}t}2^{j}\Big(\|\dot{\Delta}_{j} (\mathbf{P}V_0)^{\ell}\|_{L^p}+2^{j}\|\dot{\Delta}_{j} (\{\mathbf{I} - \mathbf{P}\}V_0)^{\ell}\|_{L^p}\Big)\\
&\quad+e^{-\kappa_{1}t}\|\dot{\Delta}_{j} (\{\mathbf{I} - \mathbf{P}\}V_0)^{\ell}\|_{L^p}.
\end{aligned}
\right.
\end{equation}
Through the standard procedure and \eqref{opoo},
 we arrive at
\begin{equation}\label{ffff}
\begin{aligned}
\|(\mathbf{P}V)^{\ell}(t)\|_{\dot{B}^{\sigma}_{p,1}}&\lesssim t^{-\frac{1}{2}(\sigma-\sigma_{1})}\Big(\|(\mathbf{P}V_0)^{\ell}\|_{\dot{B}^{\sigma_{1}}_{p,\infty}}+\|(\{\mathbf{I} - \mathbf{P}\}V_0)^{\ell}\|_{\dot{B}^{\sigma_{1}+1}_{p,\infty}}\Big),\quad \sigma>\sigma_1,
\end{aligned}
\end{equation}
and
\begin{equation}\label{ffff0}
\begin{aligned}
\|(\{\mathbf{I} - \mathbf{P}\}V)^{\ell}(t)\|_{\dot{B}^{\sigma'}_{p,1}}&\lesssim  t^{-\frac{1}{2}(\sigma'-\sigma_{1}+1)}\Big(\|(\mathbf{P}V_{0})^{\ell}\|_{\dot{B}^{\sigma_{1}}_{p,\infty}}+\|(\{\mathbf{I} - \mathbf{P}\}V_{0})^{\ell}\|_{\dot{B}^{\sigma_{1}+1}_{p,\infty}}\Big)\\
&\quad+e^{-\kappa_1t}\|(\{\mathbf{I} - \mathbf{P}\}V_0)^{\ell}\|_{\dot{B}^{\sigma'}_{p,1}}\\
&\lesssim t^{-\frac{1}{2}(\sigma'-\sigma_{1}+1)}\Big(\|(\mathbf{P}V_{0})^{\ell}\|_{\dot{B}^{\sigma_{1}}_{p,\infty}}+\|(\{\mathbf{I} - \mathbf{P}\}V_{0})^{\ell}\|_{\dot{B}^{\sigma_{1}+1}_{p,\infty}}\Big)\\
&\quad+e^{-\kappa_1t}\|(\{\mathbf{I} - \mathbf{P}\}V_0)^{\ell}\|_{\dot{B}^{\sigma_1+1}_{p,\infty}},\quad \sigma'>\sigma_1+1.
\end{aligned}
\end{equation}
In addition, for any $\sigma>\sigma_1$ and  $\sigma'>\sigma_1+1$, respectively, it holds that
\begin{equation}
\left\{
\begin{aligned}
\|(\mathbf{P}V)^{\ell}(t)\|_{\dot{B}^{\sigma}_{p,1}}&\lesssim \|(\mathbf{P}V)^{\ell}(t)\|_{\dot{B}^{\sigma_{1}}_{p,\infty}}\lesssim \|(\mathbf{P}V_0)^{\ell}\|_{\dot{B}^{\sigma_{1}}_{p,\infty}}+ \|(\{\mathbf{I} - \mathbf{P}\}V_{0})^{\ell}\|_{\dot{B}^{\sigma_{1}+1}_{p,\infty}},\\
\|(\{\mathbf{I} - \mathbf{P}\}V)^{\ell}(t)\|_{\dot{B}^{\sigma'}_{p,1}}&\lesssim \|(\{\mathbf{I} - \mathbf{P}\}V)^{\ell}(t)\|_{\dot{B}^{\sigma_{1}+1}_{p,\infty}}\lesssim\|(\mathbf{P}V_0)^{\ell}\|_{\dot{B}^{\sigma_{1}}_{p,\infty}}+ \|(\{\mathbf{I} - \mathbf{P}\}V_{0})^{\ell}\|_{\dot{B}^{\sigma_{1}+1}_{p,\infty}}.\label{llll}
\end{aligned}
\right.
\end{equation}
Note that  the fact that $e^{-\kappa_1t}\leq C_{k}(1+t)^{-k}$ for any $t>0$ and $k\geq 0$,
gathering \eqref{hhhh} and \eqref{ffff}-\eqref{llll}, we get the upper bounds \eqref{upperlinear}. 

Conversely, if we assume that $V$ satisfies \eqref{upperlinear}, then
\begin{equation}\label{upperlinearP}
\begin{aligned}
\|(\mathbf{P}V)^{\ell}(t)\|_{\dot{B}^{\sigma}_{p,1}}\lesssim \langle t\rangle^{-\frac{1}{2}(\sigma-\sigma_1)}.
\end{aligned}
\end{equation}
On the other hand, it follows from $\eqref{asy:error}$ that
\begin{equation}\label{asy:error11}
\begin{aligned}
&\quad \|\dot{\Delta}_{j}(\mathbf{P}(V-V^{*}))^{\ell}(t)\|_{L^p}\\
& \lesssim 2^{j}e^{-R_1 2^{2j}t}\Big(\|\dot{\Delta}_j(\mathbf{P}V_0)^{\ell}\|_{L^p}+2^{j}\|\dot{\Delta}_j(\{\mathbf{I} - \mathbf{P}\}V_0)^{\ell}\|_{L^p}\Big)\\
&\quad+e^{-\kappa_1 t}\Big( 2^{2j}\|\dot{\Delta}_j(\mathbf{P}V_0)^{\ell}\|_{L^p}+2^{j}\|\dot{\Delta}_j(\{\mathbf{I} - \mathbf{P}\}V_0)^{\ell}\|_{L^p}\Big)\\
&\lesssim 2^{j}e^{-R_1 2^{2j}t}\|\dot{\Delta}_j(\mathbf{P}V_0)^{\ell}\|_{L^p}+\Big(2^{2j}e^{-R_1 2^{2j}t}+ 2^{j} e^{-\kappa_1 t} \Big)\|\dot{\Delta}_j(\{\mathbf{I} - \mathbf{P}\}V_0)^{\ell}\|_{L^p}.
\end{aligned}
\end{equation}
So, one may employ Proposition \ref{propgeneral2} for the profile $V^{*}$ given in \eqref{Vp}. It follows from  the [SK] condition and \eqref{eq:SK-elliptic} that $-\mathcal{A}$ is the strongly elliptic operator. Consequently, one can construct an orthogonal matrix $P(\xi)$ such that \eqref{symbol} holds true for $\alpha=1$. By Lemma \ref{lemma22}, we deduce that there exist two constants $r_0, R_0>0$ such that
\begin{equation}\label{GGG11}
\begin{aligned}
e^{-r_{0}\lambda^2t} \|\dot{\Delta}_j\Psi_0^{\ell}\|_{L^p}\lesssim \|e^{t\mathcal{A}}\dot{\Delta}_j\Psi_0^{\ell}\|_{L^{p}}\lesssim e^{-R_{0}\lambda^2t} \|\dot{\Delta}_j\Psi_0^{\ell}\|_{L^p}.
\end{aligned}
\end{equation}
To proceed, we need to absorb the higher-order term involving $(\mathbf{P}V_0)^{\ell}$ on the right-hand side of \eqref{asy:error11}. To that end, from \eqref{GGG11}, \eqref{asy:error11} and $\|\dot{\Delta}_{j}(\mathbf{P}V_0)^{\ell}\|_{L^p}\lesssim \|\dot{\Delta}_{j}(\Psi_0)^{\ell}\|_{L^p}+2^j \|\dot{\Delta}_j (\{\mathbf{I} - \mathbf{P}\}V_0)^\ell\|_{L^p}$, one infers
\begin{equation}\nonumber
\begin{aligned}
\|\dot{\Delta}_{j}(\mathbf{P}V)^{\ell}(t)\|_{L^p}&\gtrsim \|e^{t\mathcal{A}}\dot{\Delta}_j\Psi^{\ell}_0\|_{L^{p}}-\|\dot{\Delta}_{j}(\mathbf{P}(V-V^{*}))^{\ell}(t)\|_{L^p}\\
&\gtrsim e^{-r_{0}2^{2j}t} \|\dot{\Delta}_{j}\Psi_0^{\ell}\|_{L^p}-2^{j} e^{-R_1 2^{2j}t} \|\dot{\Delta}_{j}(\mathbf{P}V_0)^{\ell}\|_{L^p}\\
&\quad-\Big(2^{2j}e^{-R_1 2^{2j}t}+ 2^{j} e^{-\kappa_1 t} \Big)\|\dot{\Delta}_j(\{\mathbf{I} - \mathbf{P}\}V_0)^{\ell}\|_{L^p}\\
&\gtrsim \Big(e^{-r_{0}2^{2j}t}-2^{j} e^{-R_1 2^{2j}t} \Big) \|\dot{\Delta}_{j}(\mathbf{P}V_0)^{\ell}\|_{L^p}\\
&\quad-\Big(2^{j}e^{-\min\{r_0,R_1\} 2^{2j}t}+ 2^{j} e^{-\kappa_1 t} \Big)\|\dot{\Delta}_j(\{\mathbf{I} - \mathbf{P}\}V_0)^{\ell}\|_{L^p},
\end{aligned}
\end{equation}
Thus, there exists a suitably small constant $r_*>0$ such that
\begin{equation*}
\begin{aligned}
\|\Lambda^{\sigma}e^{r_*t \Delta}\dot{\Delta}_{j}(\mathbf{P}V_0)^{\ell}\|_{L^p}&\lesssim
e^{-r_{1}2^{2j}t}2^{j\sigma}\|\dot{\Delta}_{j}(\mathbf{P}V_0)^{\ell}\|_{L^p}\\
&\lesssim 2^{j\sigma}\|\dot{\Delta}_{j}(\mathbf{P}V)^{\ell}(t)\|_{L^p}+2^{j(\sigma+1)} e^{-R_1 2^{2j}t}  \|\dot{\Delta}_{j}(\mathbf{P}V_0)^{\ell}\|_{L^p}\\
&\quad+\Big(e^{-\min\{r_0,R_1\} 2^{2j}t}+ e^{-\kappa_1 t} \Big) 2^{j(\sigma+1)}\|\dot{\Delta}_j(\{\mathbf{I} - \mathbf{P}\}V_0)^{\ell}\|_{L^p},
\end{aligned}
\end{equation*}
where $r_*>0$ is a suitably small constant. Thus, for all $\sigma>\sigma_1$, $t>0$ and $j\in\mathbb{Z}$, we deduce from \eqref{opoo} and the decay of $(\mathbf{P}V)^{\ell}$ in \eqref{upperlinearP}  that
\begin{align}
\|\Lambda^{\sigma}e^{r_*t \Delta}\dot{\Delta}_{j}(\mathbf{P}V_0)^{\ell}\|_{L^p}& \lesssim \|(\mathbf{P}V)^{\ell}(t)\|_{\dot{B}^{\sigma}_{p,1}}+  \|(\mathbf{P}V_0)^{\ell}\|_{\dot{B}^{\sigma_1}_{p,\infty}} \sum_{j\in \mathbb{Z}} 2^{(\sigma-\sigma_1+1)j} e^{-R_1 2^{2j}t} \nonumber\\
&\quad+  \|(\{\mathbf{I} - \mathbf{P}\}V_0)^{\ell}\|_{\dot{B}^{\sigma_1+1}_{p,\infty}} \Big( \sum_{j\in \mathbb{Z}} 2^{(\sigma-\sigma_1)j} e^{-\min\{r_0,R_1\} 2^{2j}t} +e^{-\kappa_1 t} \sum_{j\leq J_0} 2^{(\sigma-\sigma_1)j} \Big)\nonumber\\
& \lesssim t^{-\frac{1}{2}(\sigma-\sigma_{1})}+t^{-\frac{1}{2}(\sigma-\sigma_{1}+1)}+e^{-\kappa_1t}\lesssim (1+t)^{-\frac{1}{2}(\sigma-\sigma_{1})},\label{nnnnnnn}
\end{align}
for $t\geq 1$. On the other hand, one has the same bound on $[0,1]$ owing to the uniform estimate $$\|\Lambda^{\sigma}e^{r_*t \Delta}\dot{\Delta}_{j}(\mathbf{P}V_0)^{\ell}\|_{L^p}\lesssim \|(\mathbf{P}V_0)^{\ell}\|_{\dot{B}^{\sigma}_{p,1}}\lesssim \|(\mathbf{P}V_0)^{\ell}\|_{\dot{B}^{\sigma_1}_{p,\infty}}.$$  Then, similar to the proof of Proposition \ref{propgeneral2}, we have
\begin{align}
\Delta_{j}(\mathbf{P}V_0)^{\ell}&=\frac{1}{\Gamma(\frac{1}{2}(\sigma-\sigma_{1})+1)} \int_{0}^{\infty}t^{\frac{1}{2}(\sigma-\sigma_{1})} \mathcal{F}^{-1}\bigg( (2r_*|\xi|^2)^{1-\frac{\sigma_1}{2}} \mathcal{F}\Big(\Delta_{j}\Lambda^{\sigma} e^{2r_*t \Delta}(\mathbf{P}V_0)^{\ell}\Big) \,dt \bigg).\nonumber
\end{align}
In light of \eqref{nnnnnnn} and the classical Fourier multiplier theorem (\cite{bahouri1}[Lemma 2.2]), it follows that, for $j\in\mathbb{Z}$ and some constant $R^*>0$,
\begin{equation}\nonumber
\begin{aligned}
\|\Delta_{j}(\mathbf{P}V_0)^{\ell}\|_{L^p}&\lesssim \int_{0}^{\infty}t^{\frac{1}{2}(\sigma-\sigma_{1})} 2^{(2-\sigma_1)j} \|\Lambda^{\sigma} e^{2r_*t \Delta}(\mathbf{P}V_0)^{\ell}\|_{L^p}dt\\
&\lesssim \int_{0}^{\infty}t^{\frac{1}{2}(\sigma-\sigma_{1})} 2^{(2-\sigma_1)j} e^{-R^* 2^{2j}t} \|\Lambda^{\sigma} e^{r_*t \Delta}(\mathbf{P}V_0)^{\ell}\|_{L^p}dt\lesssim 2^{(2-\sigma_1)j} \int_{0}^{\infty}  e^{-R^* 2^{2j}t} dt\lesssim 2^{-\sigma_1 j}.
\end{aligned}
\end{equation}
This yields $(\mathbf{P}V_0)^{\ell}\in \dot{B}^{\sigma_{1}}_{p,\infty}$. 

Next, we turn to proving the asymptotic behavior \eqref{asy:linear}, that is, the faster decay rates of $\mathbf{P}(V-V^{*})$ and $\{\mathbf{I} - \mathbf{P}\}(V-V^{*})$. Note that $\|\dot{\Delta}_j\Psi_0\|_{L^2}\lesssim 2^{j}\|\dot{\Delta}_j V_0\|_{L^2}$ for all $j\geq J_0-1$. By using Lemma \ref{lemma22} and \eqref{opoo}, we arrive at
\begin{align*}
\|\mathbf{P}V^*(t)\|_{\dot{B}^s_{2,1}}^h&\lesssim \sum_{j\geq J_0-1 } 2^{s j} e^{-R_0 2^{2j} t} \|\dot{\Delta}_j\Psi_0\|_{L^2} \\
&\lesssim e^{-\frac{1}{2}R_0 2^{2(J_0-1)} t} \sum_{j\in\mathbb{Z}} 2^j    e^{-\frac{1}{2}R_0 2^{2j} t} \, \sup_{j\geq J_0}2^{(s-1)j} \|\dot{\Delta}_j \Psi_0\|_{L^2}\\
&\lesssim e^{-\frac{1}{2}R_0 2^{2(J_0-1)} t} t^{-\frac{1}{2}} \|V_0\|_{\dot{B}^{s}_{2,1}}^h.
\end{align*}
which, together with \eqref{hhhh}, \eqref{upperlinearP} and \eqref{asy:error11}, leads to the first decay estimate \eqref{asy:linear} for $t\geq \tilde{t}_0$ with $\tilde{t}_0>0$:
$$
\|\mathbf{P}(V-V^{*})(t)\|_{\dot{\mathbb{B}}^{\sigma,s}_{p,2}}\lesssim \|(\mathbf{P}(V-V^{*}))^{\ell}(t)\|_{\dot{B}^{\sigma}_{p,1}}+\|(V,\mathbf{P} V^{*})(t)\|_{\dot{B}^{s}_{2,1}}^h\leq C \langle t\rangle^{-\frac{1}{2}(\sigma-\sigma_{1})-\frac{1}{2}}.
$$

Furthermore, it follows from \eqref{asy:error2} that
\begin{equation}\nonumber
\begin{aligned}
\|\dot{\Delta}_{j}(\{\mathbf{I} - \mathbf{P}\}(V-V^*))
^{\ell}\|_{L^p}&\lesssim e^{-R_12^{2j}t} 2^{2j}\Big(\|\dot{\Delta}_{j}(\mathbf{P}V_{0})^{\ell}\|_{L^p}+2^{j}\|\dot{\Delta}_{j}(\{\mathbf{I} - \mathbf{P}\}V_{0})^{\ell}\|_{L^p}\Big)\\
&\quad+ e^{-\kappa_1 t} \Big(\|\dot{\Delta}_{j}\mathbf{P}V_{0})^{\ell}\|_{L^p}+2^{j}\|\dot{\Delta}_{j}(\{\mathbf{I} - \mathbf{P}\}V_{0})^{\ell}\|_{L^p}\Big).
\end{aligned}
\end{equation}
Recalling \eqref{opoo}, for any $\sigma'>\sigma_1+1$, we get
\begin{equation}\label{sfggvv}
\begin{aligned}
\|(\{\mathbf{I} - \mathbf{P}\}(V-V^*))
^{\ell}(t)\|_{\dot{B}^{\sigma'}_{p,1}}&\lesssim t^{-\frac{1}{2}(\sigma'-\sigma_1+2)}\Big(\|(\mathbf{P}V_{0})^{\ell}\|_{\dot{B}^{\sigma_1}_{p,\infty}}+\|(\{\mathbf{I} - \mathbf{P}\}V_{0})^{\ell}\|_{\dot{B}^{\sigma_1+1}_{p,\infty}}\Big)\\
&\quad +e^{-\kappa_1 t} \Big(\|(\mathbf{P}V_{0})^{\ell}\|_{\dot{B}^{\sigma}_{p,1}}+\|(\{\mathbf{I} - \mathbf{P}\}V_{0})^{\ell}\|_{\dot{B}^{\sigma}_{p,1}}\Big)\\
&\lesssim t^{-\frac{1}{2}(\sigma'-\sigma_1+2)} \Big(\|(\mathbf{P}V_{0})^{\ell}\|_{\dot{B}^{\sigma_1}_{p,\infty}}+\|(\{\mathbf{I} - \mathbf{P}\}V_{0})^{\ell}\|_{\dot{B}^{\sigma_1+1}_{p,\infty}}\Big).
\end{aligned}
\end{equation}
On the other hand, owing to Lemma \ref{lemma22}, \eqref{opoo} and \eqref{GGG11}, the high-frequency part of $\{\mathbf{I} - \mathbf{P}\}V^*$ satisfies
\begin{equation}
 \begin{aligned}\label{highV*}
 \|\{\mathbf{I} - \mathbf{P}\}V^*(t)\|_{\dot{B}^{s}_{2,1}}^{h}\lesssim  \|e^{t\mathcal{A}}\Psi_0\|_{\dot{B}^{s+1}_{2,1}}^h  &\lesssim \sum_{j\geq J_0-1 } 2^{(s+1) j} e^{-R_0 2^{2j} t} \|\dot{\Delta}_j\Psi_0\|_{L^2}\\
 &\lesssim   e^{-\frac{1}{2}R_0 2^{2(J_0-1)} t} t^{-1} \|V_0\|_{\dot{B}^{s}_{2,1}}^h.
 \end{aligned}
 \end{equation}
Therefore, the second decay estimate \eqref{asy:linear}  for $t\geq \tilde{t}_0>0$ is proved by \eqref{hhhh}, \eqref{sfggvv} and \eqref{highV*}.


Let us focus on the
``if'' part of two-sided bounds of \eqref{twosidelinear}.  
The upper bounds in \eqref{twosidelinear} follow directly from the embedding $\dot{\mathcal{B}}^{\sigma_{1}}_{p,\infty} \hookrightarrow \dot{B}^{\sigma_{1}}_{p,\infty}$.
We only need to show the lower bound properties in \eqref{twosidelinear}.
Note that $e^{\mathcal{A}t}\Psi_0$ is the solution to the linear parabolic equation \eqref{diffusion}. It follows from Proposition \ref{propgeneral2}  that
\begin{equation}\label{V1ell}
\begin{aligned}
\|(\mathbf{P}V^*)^{\ell}(t)\|_{\dot{B}^{\sigma}_{p,1}}=\|e^{t\mathcal{A}}\Psi_0^{\ell}\|_{\dot{B}^{\sigma}_{p,1}}\gtrsim \langle t\rangle^{-\frac{1}{2}(\sigma-\sigma_1)},\quad\quad \sigma>\sigma_1.
\end{aligned}
\end{equation}
Furthermore, it is shown, by $\eqref{asy:linear}_1$ and \eqref{V1ell}, that
\begin{equation}\nonumber
\begin{aligned}
 \|(\mathbf{P}V)^{\ell}(t)\|_{\dot{B}^{\sigma}_{p,1}} &\geq \|(\mathbf{P}V^*)^{\ell}(t)\|_{\dot{B}^{\sigma}_{p,1}}- \|(\mathbf{P}(V-V^*))^{\ell}(t)\|_{\dot{B}^{\sigma}_{p,1}}\\
&\geq \frac{1}{C_0}\langle t\rangle^{-\frac{1}{2}(\sigma-\sigma_1)}-C_0\langle t\rangle^{-\frac{1}{2}(\sigma-\sigma_1+1)}\geq \frac{1}{2C_0} \langle t\rangle^{-\frac{1}{2}(\sigma-\sigma_1)}
\end{aligned}
\end{equation}
for some uniform constant $C_0>1$  and suitably large time $t$.

On the other hand, we claim that there  exists a constant $c_{2,1}>0$ such that $
\Big|\sum\limits_{i=1}^{d}A_{2,1}^{i}\omega_{i}\Big|\geq c_{2,1}>0$ for $\omega\in S^{d-1}$. Indeed, if  $\sum\limits_{i=1}^{d}A_{2,1}^{i}\omega_{i}=0$, then in view of $A_{1,1}^{i}=0$ and the symmetry of $A^{i}$, we have
\begin{equation}\nonumber
\begin{aligned}
\sum_{i=1}^{d}A^{i}\omega_{i}=\begin{pmatrix} 0 & 0 \\ 0 & \sum_{i=1}^{d}A_{2,2}^{i}\omega_{i}
    \end{pmatrix}.
\end{aligned}
\end{equation}
This contradicts the [SK] condition in the sense of Definition \ref{defnSK}. Thus, we deduce that
$$
\Big|\sum_{i=1}^{d} A_{2,1}^{i} \xi_i \mathcal{F}\Big(e^{t\mathcal{A}}\Psi_0^{\ell}\Big) \Big|\geq c_{2,1}|\xi| \Big|\mathcal{F}\Big( e^{t\mathcal{A}}\Psi_0^{\ell}\Big)\Big|,
$$
which, together with \eqref{Vp} and \eqref{V1ell}, leads to
\begin{equation}\label{V2ell}
\begin{aligned}
\|(\{\mathbf{I} - \mathbf{P}\}V^*)^{\ell}(t)\|_{\dot{B}^{\sigma'}_{p,1}}&\gtrsim \|(\mathbf{P}V^*)^{\ell}(t)\|_{\dot{B}^{\sigma'+1}_{p,1}}\gtrsim \langle t\rangle^{-\frac{1}{2}(\sigma'-\sigma_1+1)},\quad \sigma'>\sigma_1+1.
\end{aligned}
\end{equation}
In view of $\eqref{asy:linear}_2$ and \eqref{V2ell}, one arrives at
\begin{equation}\nonumber
\begin{aligned}
\|\{\mathbf{I} - \mathbf{P}\}V(t)\|_{\dot{\mathbb{B}}^{\sigma,s}_{p,2}}&\geq \|(\{\mathbf{I} - \mathbf{P}\}V)^{\ell}(t)\|_{\dot{B}^{\sigma}_{p,1}}\\
&\geq \|(\{\mathbf{I} - \mathbf{P}\}V^*)^{\ell}(t)\|_{\dot{B}^{\sigma}_{p,1}}-\|(\{\mathbf{I} - \mathbf{P}\}(V-V^*))^{\ell}(t)\|_{\dot{B}^{\sigma}_{p,1}}\\
&\geq \frac{1}{C_0} \langle t\rangle^{-\frac{1}{2}(\sigma-\sigma_1+1)}-C_0\langle t\rangle^{-\frac{1}{2}(\sigma-\sigma_1+2)}\geq \frac{1}{2C_0}\langle t\rangle^{-\frac{1}{2}(\sigma-\sigma_1+1)},
\end{aligned}
\end{equation}
for some uniform constant $C_0$ and all $t\geq \tilde{t}_1$ with a suitably large time $\tilde{t}_1>0$. Thus, the lower bound part of \eqref{twosidelinear} is established for $t > t_0$ with some time $t_0 > 0$.

Conversely, we assume \eqref{upperlinear} for $t>0$ and \eqref{twosidelinear} for $t>\tilde{t}_1$. To show the ``only if'' part, let us first establish the following lemma
for the improved decay estimate of $\mathbf{P}(V-V^*)$.
\begin{lemma}
If \eqref{upperlinear} holds for $t>0$, then
\begin{equation}\label{con:asy}
\begin{aligned}
\|(\mathbf{P}(V-V^*))^{\ell}(t)\|_{\dot{B}^{\sigma}_{p,1}}=\|(V_1-V_1^*)^{\ell}(t)\|_{\dot{B}^{\sigma}_{p,1}}\lesssim \langle t\rangle^{-\frac{1}{2}(\sigma-\sigma_1+1)},\quad \sigma>\sigma_1.
\end{aligned}
\end{equation}
\end{lemma}
\begin{proof}
Keep in mind in Lemma \ref{lemma32} that the effective quantity $\Psi\coloneqq  V_1-\sum\limits_{i=1}^{d}A^{i}_{1,2}D^{-1} \partial_{x_{i}} V_2$ fulfills
\begin{equation}\label{errorPsi000}
\left\{
\begin{aligned}
&\partial_t(\Psi-V_1^*)-\mathcal{A}(\Psi-V_1^*)=\bigg(\sum_{i,l=1}^{d} A^i_{1,2}D^{-1}A^l_{2,2} \partial^2_{x_{i} x_{l}}+\sum_{i=1}^{d}\mathcal{A}A_{1,2}^i D^{-1}\partial_{x_i} \bigg) V_2,\\
&(\Psi-V_1^*)|_{t=0}=0.
\end{aligned}
\right.
\end{equation}
Here $V_1^*=e^{t\mathcal{A}}\Psi_0$. Applying Duhamel's principle to \eqref{errorPsi000} and Lemma \ref{lemma22}, we have
\begin{equation*}
\begin{aligned}
\|\dot{\Delta}_j(\Psi-V_1^*)^\ell(t)\|_{L^p}&\lesssim \int_0^t e^{-R_0 2^{2j} t} 2^{2j} \|\dot{\Delta}_j V_2^\ell\|_{L^p}d\tau,
\end{aligned}
\end{equation*}
which implies that
\begin{equation*}
\begin{aligned}
\|(\Psi-V_1^*)^\ell(t)\|_{\dot{B}^{\sigma}_{p,1}}\lesssim \int_0^t \|V_2^\ell\|_{\dot{B}^{\sigma+2}_{p,1}}d\tau\lesssim \int_0^t \|V_2^\ell\|_{\dot{B}^{\sigma+1}_{p,\infty}}d\tau,\quad \sigma>\sigma_1.
\end{aligned}
\end{equation*}
In view of \eqref{opoo}, one also has
\begin{equation*}
\begin{aligned}
\|(\Psi-V_1^*)^\ell(t)\|_{\dot{B}^{\sigma}_{p,1}}\lesssim \int_0^t (t-\tau)^{-\frac{1}{2}(\sigma-\sigma_1+1)}\|V_2^\ell\|_{\dot{B}^{\sigma+1}_{p,\infty}}d\tau,\quad \sigma>\sigma_1.
\end{aligned}
\end{equation*}
Consequently, employing the decay of $V_2$ in \eqref{twosidelinear}, we derive
\begin{equation*}
\begin{aligned}
\|(\Psi-V_1^*)^\ell(t)\|_{\dot{B}^{\sigma}_{p,1}}&\lesssim \int_0^t \langle t-\tau\rangle^{-\frac{1}{2}(\sigma-\sigma_1+1)}\|V_2^\ell\|_{\dot{B}^{\sigma+1}_{p,\infty}}d\tau\\
&\lesssim \int_0^t \langle t-\tau\rangle^{-\frac{1}{2}(\sigma-\sigma_1+1)} \langle \tau\rangle^{-\frac{1}{2}(\sigma-\sigma_1+2)}d\tau\lesssim \langle t\rangle^{-\frac{1}{2}(\sigma-\sigma_1+1)},
\end{aligned}
\end{equation*}
 Using $V_1=\Psi+\sum\limits_{i=1}^{d}A^{i}_{1,2}D^{-1} \partial_{x_{i}} V_2$ and employing the decay of $V_2$ once again yields \eqref{con:asy}.
\end{proof}

To proceed, it follows from \eqref{twosidelinear} and $\eqref{con:asy}$ that
\begin{equation}\nonumber
\begin{aligned}
\|e^{t\mathcal{A}}\Psi_0^{\ell}\|_{\dot{B}^{\sigma}_{p,1}}= \|(\mathbf{P}V^*)^{\ell}(t)\|_{\dot{B}^{\sigma}_{p,1}}&\lesssim \|(\mathbf{P}V)^{\ell}(t)\|_{\dot{B}^{\sigma}_{p,1}}+ \|(\mathbf{P}(V-V^*))^{\ell}(t)\|_{\dot{B}^{\sigma}_{p,1}}\lesssim \langle t\rangle^{-\frac{1}{2}(\sigma-\sigma_1)}
\end{aligned}
\end{equation}
for all $t>0$ and
\begin{equation}\nonumber
\begin{aligned}
\|e^{t\mathcal{A}}\Psi_0^{\ell}\|_{\dot{B}^{\sigma}_{p,1}}&\geq \|(\mathbf{P}V)^{\ell}(t)\|_{\dot{B}^{\sigma}_{p,1}}- \|(\mathbf{P}(V-V^*))^{\ell}(t)\|_{\dot{B}^{\sigma}_{p,1}}\\
&\geq \frac{1}{C_0} \langle t\rangle^{-\frac{1}{2}(\sigma-\sigma_1)}-C_0\langle t\rangle^{-\frac{1}{2}(\sigma-\sigma_1+1)}\geq \frac{1}{2C_0} \langle t\rangle^{-\frac{1}{2}(\sigma-\sigma_1)}
\end{aligned}
\end{equation}
for all $t>\tilde{t}_2$, where $C_0>1$ is a constant, and $\tilde{t}_2>\tilde{t}_1>0$ is a suitably large time. 
Hence, following the argument \eqref{mm11}-\eqref{267} in the proof of Proposition \ref{propgeneral2}, we can infer $\Psi_0^{\ell}\in \dot{\mathcal{B}}^{\sigma_1}_{p,\infty}$. The proof of Proposition \ref{proplinear} is eventually finished.\hfill $\Box$

\section{Global existence}\label{global existence}
In this section, we prove the global existence of strong solutions in the $L^p$ framework, where initial data are a small
perturbation of constant
equilibrium. Here, we pinpoint the \textit{critical} regularity for the global solutions in hybrid
Besov spaces with different exponents in low and high frequencies. 
Precisely, we assume that
\begin{align}
\text{the system \eqref{GEQSYM} has an entropy in the sense of Definition \ref{defnentropy}},\quad\label{H1}
\end{align}
and
\begin{align}
\text{the linearized system of \eqref{m1} satisfies the [SK] condition in the sense  of Definition \ref{defnSK}.} \quad\label{H2}
\end{align}
In addition,  we suppose the block in \eqref{block} fulfills
\begin{equation}\label{blockL}
\begin{aligned}
A^i_{1,1}=A^i_{1,1}(\bar{U})=0.
    \end{aligned}
\end{equation}

Since \eqref{GEQSYM} and the reformulated system \eqref{m1} are equivalent in the regime of classical solutions,  we focus on \eqref{m1} in the following. The special structure on $R^i$ ($i=1,\dots,d$) and the property $\mathbf{P}R^0=0$
have been playing a key role in the global existence and decay estimates of solutions.

From the point of view of spectral analysis, a good candidate for the solution space is the set of functions $V\in \mathcal{C}^{1}(\mathbb{R}_{+}\times \mathbb{R}^{d}; \mathbb{R}^{n})$ defined by
\begin{multline*}
\mathbf{E}\coloneqq  \Big\{V\in \mathcal{C}^{1}(\mathbb{R}_{+}\times \mathbb{R}^{d}; \mathbb{R}^{n}): \mathbf{P}V\in \mathcal{C}(\mathbb{R}_{+};\dot{\mathbb{B}}^{\frac{d}{p}-1,\frac{d}{2}+1}_{p,2})\cap L^1(\mathbb{R}_{+};\dot{\mathbb{B}}^{\frac{d}{p}+1,\frac{d}{2}+1}_{p,2}), \\
\{\mathbf{I}-\mathbf{P}\}V \in \mathcal{C}(\mathbb{R}_{+};\dot{\mathbb{B}}^{\frac{d}{p},\frac{d}{2}+1}_{p,2})\cap L^1(\mathbb{R}_{+};\dot{\mathbb{B}}^{\frac{d}{p},\frac{d}{2}+1}_{p,2}) \Big\}.
\end{multline*}
See Appendix \ref{appendixA} for more details on the hybrid Besov space $\dot{\mathbb{B}}^{\sigma,s}_{p,2}$.

The global existence
of classical solutions to the Cauchy problem \eqref{m1}-\eqref{m1d} in the critical $L^{p}$ framework is stated as follows.
\begin{theorem}\label{theorem0}
Let $d\geq1$  and
\begin{equation}\label{p}
\left\{\begin{array}{l}
2\leq p\leq4,\quad\quad \quad\,  \text{if}\quad  1\leq d\leq4,\\
2\leq p\leq\frac{2d}{d-2},\quad \quad \text{if}\quad d\geq 5.
\end{array}\right.
\end{equation}
Assume that \eqref{H1}, \eqref{H2} and \eqref{blockL} are in force. There exists a constant $\varepsilon_0>0$ such that if
\begin{equation}\label{a1}
\begin{aligned}
E_{p,0}\coloneqq  \|\mathbf{P}V_{0}\|_{\dot{\mathbb{B}}^{\frac{d}{p}-1,\frac{d}{2}+1}_{p,2}}+\|\{\mathbf{I} - \mathbf{P}\}V_{0}\|_{\dot{\mathbb{B}}^{\frac{d}{p},\frac{d}{2}+1}_{p,2}}\leq \varepsilon_0,
\end{aligned}
\end{equation}
then the Cauchy problem \eqref{m1}-\eqref{m1d} admits a unique global solution $V$ in the space $\mathbf{E}$.
Moreover, it holds that
\begin{equation}\label{r2}
\begin{aligned}
&\widetilde{E}_{p}(t)+D_{p}(t)\leq C E_{p,0}\quad \mbox{for all} \quad t>0,
\end{aligned}
\end{equation}
where $C>0$ is a generic positive constant. The energy and dissipation functionals are defined by
\begin{equation}\label{Et}
\begin{aligned}
\widetilde{E}_{p}(t)\coloneqq  & \|\mathbf{P}V\|_{\widetilde{L}^{\infty}_{t}(\dot{\mathbb{B}}^{\frac{d}{p}-1,\frac{d}{2}+1}_{p,2})}+\|\{\mathbf{I}-\mathbf{P}\}V\|_{\widetilde{L}^{\infty}_{t}(\dot{\mathbb{B}}^{\frac{d}{p},\frac{d}{2}+1}_{p,2})}
\end{aligned}
\end{equation}
and
\begin{equation}\label{Dt}
\begin{aligned}
D_{p}(t)\coloneqq &\|\mathbf{P}V\|_{L^1_{t}(\dot{\mathbb{B}}^{\frac{d}{p}+1,\frac{d}{2}+1}_{p,2})}+\|\{\mathbf{I}-\mathbf{P}\}V\|_{L^1_{t}(\dot{\mathbb{B}}^{\frac{d}{p},\frac{d}{2}+1}_{p,2})}.
\end{aligned}
\end{equation}
\end{theorem}


\begin{remark}\normalfont
Crin-Barat and Danchin \cite{c3} assumed that
$A^i_{1,1}(\bar U)=0$ in \eqref{blockL} and
established global solutions in $L^p$-type
hybrid Besov spaces when $V_0^{\ell}\in \dot B^{d/p}_{p,1}$.
By contrast, Theorem~\ref{theorem0} considers different regularity requirements
for the conservative component $\mathbf{P}V_0$ and the dissipative component
$\{\mathbf{I}-\mathbf{P}\}V_0$ based on the sharp frequency-localized decay properties (see Proposition~\ref{LemmaspectrallocalHp}). In
particular, the dissipation \eqref{Dt} yields
$V\in L^{1}(\mathbb{R}_{+};\dot{B}^{d/p+1}_{p,1})$ and hence
$\nabla V\in L^{1}(\mathbb{R}_{+};L^{\infty})$, which is consistent with the classical
blow-up criterion for hyperbolic systems in \cite{bahouri1} and \cite[Theorem~1.4]{XK1}. Moreover,
the functional setting allows to remove the additional structural assumptions on
$F^i(U)$ and $G(U)$ that are
required in \cite{c3}.

\end{remark}

The proof of Theorem \ref{theorem0} relies on the following \emph{a priori} estimates of the solution $V$ to \eqref{m1} and the standard continuation argument.
\begin{prop}\label{propapriori}
Let $V$ be the smooth solution to the Cauchy problem \eqref{m1} defined on $[0,T)\times\mathbb{R}^d$. For any $t\in[0,T)$, if $V$ satisfies
\begin{equation}\label{abound}
\begin{aligned}
\|V\|_{L^{\infty}_{t}(L^{\infty})}\leq 1,
\end{aligned}
\end{equation}
then there exists a uniform constant $M^*>0$ such that
\begin{align}
E_{p}(t)+D_{p}(t)\leq M^* e^{M^* D_p(t)}(1+E_p(t))^2\bigg((1+E_{p,0})E_{p,0}+(1+E_p(t))^2 E_p(t)D_{p}(t) \bigg), \label{prioribound}
\end{align}
where
\begin{align}
    E_{p}(t)\coloneqq \|\mathbf{P}V\|_{L^{\infty}_{t}(\dot{\mathbb{B}}^{\frac{d}{p}-1,\frac{d}{2}+1}_{p,2})}+\|\{\mathbf{I}-\mathbf{P}\}V\|_{L^{\infty}_{t}(\dot{\mathbb{B}}^{\frac{d}{p},\frac{d}{2}+1}_{p,2})}.\label{Etr}
\end{align}
\end{prop}

\begin{proof}
The \emph{a priori} estimates of $V$ can be split into three steps.
\begin{itemize}
\item {\emph{Step 1: Low-frequency analysis}}
\end{itemize}

Compared with the direct energy estimates developed in \cite{c3,danchinnote2}, the low-frequency analysis mainly resorts to Duhamel's principle and Proposition \ref{LemmaspectrallocalHp}.
Let $\{\mathcal{G}(t)\}_{t>0}$ be the semigroup generated by the linear operator on the left-hand side of \eqref{m1}. We always choose $J_0$ to be small enough such that $2^{J_0}\leq \lambda_0$, as required in Proposition \ref{LemmaspectrallocalHp}. Employing the low-frequency cutoff operator $\dot{S}_{J_0}$ to \eqref{m1} and applying the localized operator $\dot{\Delta}_{j}$ to the resulting system yields
\begin{equation}\label{localized}
    \begin{aligned}
&\dot{\Delta}_{j}V^{\ell}_{t}+\sum_{i=1}^{d}A^{i}\partial_{x_{i}}\dot{\Delta}_{j}V^{\ell}+L\dot{\Delta}_{j}V^{\ell}=\dot{\Delta}_{j}(R^{0})^{\ell}+\sum_{i=1}^{d}\partial_{x_{i}}\dot{\Delta}_{j}(R^{i})^{\ell},\quad \dot{\Delta}_j V^\ell|_{t=0}=\dot{\Delta}_j V_0^\ell.
\end{aligned}
\end{equation}
Furthermore, one may derive the frequency-localized Duhamel formula:
\begin{equation}\label{Duhanel1}
    \begin{aligned}
    \dot{\Delta}_{j}V^{\ell}=\mathcal{G}(t)\dot{\Delta}_jV_0^\ell+\int_0^t \mathcal{G}(t-\tau)\Big(\dot{\Delta}_{j}(R^{0})^{\ell}+\sum_{i=1}^{d}\partial_{x_{i}}\dot{\Delta}_{j}(R^{i})^{\ell}\Big)\,d\tau.
\end{aligned}
\end{equation}
It then follows from Proposition~\ref{LemmaspectrallocalHp} that, for any
$j\leq J_0$,
\begin{equation}\label{f-l}
\begin{aligned}
\|\mathbf{P}\mathcal{G}(t)\dot{\Delta}_{j}V_0^\ell\|_{L^p}&\lesssim e^{-R_{1}2^{2j}t} \Big(\|\mathbf{P} \dot{\Delta}_{j} V_0^\ell \|_{L^p}+2^{j} \|\{\mathbf{I} - \mathbf{P}\} \dot{\Delta}_{j} V_0^\ell\|_{L^p} \Big),\\
\|\mathbf{P}\mathcal{G}(t)\dot{\Delta}_{j}(R^{0})^{\ell}\|_{L^p}
&\lesssim 2^{j}  e^{-R_{1}2^{2j}t} \|\dot{\Delta}_{j}(R^{0})^{\ell}\|_{L^p},\\
\|\mathbf{P}\mathcal{G}(t)\dot{\Delta}_{j}\partial_{x_{i}}(R^{i})^{\ell}\|_{L^p}&\lesssim 2^{j} e^{-R_{1}2^{2j}t} \|\dot{\Delta}_{j}(R^{i})^{\ell}\|_{L^p},\quad i=1,\dots,d.
\end{aligned}
\end{equation}
Indeed, the first inequality in \eqref{f-l} follows directly from
\eqref{lowlocalupper}, provided that $J_0$ is chosen sufficiently small.
The second estimate, with the factor $2^j$, relies on the key identities
$\mathbf{P}R^{0}=0$ and $\{\mathbf{I}-\mathbf{P}\}R^{0}=R^{0}$.
The third inequality is a consequence of \eqref{lowlocaluppertotal}, since
each term $\partial_{x_i}R^{i}$ ($i=1,\dots,d$) involves one spatial
derivative. Consequently, we have
\begin{equation}\label{PVj}
\begin{aligned}
\|\dot{\Delta}_{j}(\mathbf{P}V)^{\ell}\|_{L^p}&\lesssim e^{-R_1 2^{2j}t}\Big(\|\dot{\Delta}_{j}(\mathbf{P}V_0)^{\ell}\|_{L^p}+2^{j}\|\dot{\Delta}_{j}(\{\mathbf{I}-\mathbf{P}\}V_0)^{\ell}\|_{L^p}\Big)+\int_{0}^{t} e^{-R_{1}2^{2j}(t-\tau)} 2^{j} \sum_{i=0}^{d}\|\dot{\Delta}_{j}(R^{i})^{\ell}\|_{L^{p}}\,d\tau
\end{aligned}
\end{equation}
for $j\leq J_0$. In view of Young's inequality for \eqref{PVj} in time, we arrive at
\begin{equation}\label{localPV}
\begin{aligned}
&\quad\|\dot{\Delta}_{j}(\mathbf{P}V)^{\ell}\|_{L^{\infty}_{t}(L^p)}+2^{2j}\|\dot{\Delta}_{j}(\mathbf{P}V)^{\ell}\|_{L^1_t(L^p)}\\
& \lesssim \|\dot{\Delta}_{j}(\mathbf{P}V_0)^{\ell}\|_{L^p}+2^{j}\|\dot{\Delta}_{j}(\{\mathbf{I}-\mathbf{P}\}V_0)^{\ell}\|_{L^p} +2^{j}\sum_{i=0}^{d}\| (R^i)^{\ell}\|_{L^1_t(L^p)}.
\end{aligned}
\end{equation}
Similarly, one can obtain the frequency-localization estimate for $\{\mathbf{I}-\mathbf{P}\}V$. With the aid of  Proposition \ref{LemmaspectrallocalHp},
we obtain
\begin{multline*}
\begin{aligned}
\|\{\mathbf{I} - \mathbf{P}\}\mathcal{G}(t)\dot{\Delta}_{j}V_0^\ell\|_{L^p}\lesssim e^{-R_{1}2^{2j}t}2^j \Big(\|\mathbf{P} \dot{\Delta}_{j} V_0^\ell \|_{L^p}+2^{j} \|\{\mathbf{I} - \mathbf{P}\} \dot{\Delta}_{j} V_0^\ell\|_{L^p} \Big)+e^{-\kappa_1 t}\|\{\mathbf{I} - \mathbf{P}\} \dot{\Delta}_{j} V_0^\ell\|_{L^p},
\end{aligned}
\end{multline*}
and similarly,
\begin{equation*}
\begin{aligned}
&\|\{\mathbf{I} - \mathbf{P}\}\mathcal{G}(t)\dot{\Delta}_{j}(R^{0})^{\ell}\|_{L^p}\lesssim \Big(2^{2j} e^{-R_{1}2^{2j}t} +e^{-\kappa_1 t} \Big)\|\dot{\Delta}_{j}(R^{0})^{\ell}\|_{L^p},\\
&\|\{\mathbf{I} - \mathbf{P}\}\mathcal{G}(t)\dot{\Delta}_{j}\partial_{x_{i}}(R^{i})^{\ell}\|_{L^p}\lesssim \Big(2^{2j} e^{-R_{1}2^{2j}t} +e^{-\kappa_1 t} 2^j \Big) \|\dot{\Delta}_{j}(R^{i})^{\ell}\|_{L^p},\quad i=1,\dots,d
\end{aligned}
\end{equation*}
for any $j\leq J_0$. Furthermore, it follows from \eqref{Duhanel1} that
\begin{equation}\nonumber
\begin{aligned}
\|\dot{\Delta}_{j}(\{\mathbf{I}-\mathbf{P}\}V)^{\ell}\|_{L^{p}}&\lesssim 2^{j} e^{-R_1 2^{2j}t} \Big(\|\dot{\Delta}_{j}(\mathbf{P}V_0)^{\ell}\|_{L^p}+2^{j}\|\dot{\Delta}_{j}(\{\mathbf{I}-\mathbf{P}\}V_0)^{\ell}\|_{L^p} \Big)+e^{-\kappa_1 t} \|\dot{\Delta}_{j}(\{\mathbf{I}-\mathbf{P}\}V_0)^{\ell}\|_{L^p}\\
&\quad+\int_{0}^{t}\Big( e^{-R_{1}2^{2j}(t-\tau)} 2^{2j}+e^{-\kappa_{1}(t-\tau)}  \Big) \sum_{i=0}^{d}\|\dot{\Delta}_{j}(R^{i})^{\ell}\|_{L^{p}}\, d\tau
\end{aligned}
\end{equation}
for $j\leq J_0$,  which leads to
\begin{equation}\label{localIPV}
\begin{aligned}
& \quad 2^j \|\dot{\Delta}_{j}(\{\mathbf{I} - \mathbf{P}\}V)^{\ell}\|_{L^{\infty}_{t}(L^p)}+ 2^j \|\dot{\Delta}_{j}(\{\mathbf{I} - \mathbf{P}\}V)^{\ell}\|_{L^{1}_{t}(L^p)}\\
&\lesssim \|\dot{\Delta}_{j}(\mathbf{P}V_0)^{\ell}\|_{L^p}+ 2^j \|\dot{\Delta}_{j}(\{\mathbf{I}-\mathbf{P}\}V_0)^{\ell}\|_{L^p} + 2^j \sum_{i=0}^{d}\| (R^i)^{\ell}\|_{L^1_t(L^p)}.
\end{aligned}
\end{equation}
Actually, those estimates \eqref{localPV}-\eqref{localIPV} hold for any $j\in \mathbb{Z}$ due to the fact that $\dot{\Delta}_{j}\dot{S}_{J_0}=0$ when $j\geq J_0+1$.
Adding up \eqref{localPV} and \eqref{localIPV}, multiplying the resulting  inequality by  $2^{j(\frac{d}{p}-1)}$ and then taking the summation over $j\in\mathbb{Z}$, we deduce that
\begin{equation}\nonumber
\begin{aligned}
&\quad \|(\mathbf{P}V)^{\ell}\|_{\widetilde{L}^{\infty}_{t}(\dot{B}^{\frac{d}{p}-1}_{p,1})\cap L^1_{t}(\dot{B}^{\frac{d}{p}+1}_{p,1})}+\|(\{\mathbf{I} - \mathbf{P}\}V)^{\ell}\|_{\widetilde{L}^{\infty}_{t}(\dot{B}^{\frac{d}{p}}_{p,1})\cap L^1_{t}(\dot{B}^{\frac{d}{p}}_{p,1})}\\
& \lesssim \|(\mathbf{P}V_0)^{\ell}\|_{\dot{B}^{\frac{d}{p}-1}_{p,1}}+\|(\{\mathbf{I} - \mathbf{P}\}V_0)^{\ell}\|_{\dot{B}^{\frac{d}{p}}_{p,1}}+\sum_{i=0}^{d}\|(R^i)^{\ell}\|_{L^1_{t}(\dot{B}^{\frac{d}{p}}_{p,1})}.
\end{aligned}
\end{equation}
Note that $\mathbf{r}(\cdot)=(r^1,r^2,\dots,r^n)(\cdot)$ fulfills $r^i(\bar{W})=D_{\mathcal{W}}r^i(\bar{W})=0$ ($i=0,2,\dots,n$). Hence, using Lemmas \ref{prop3.2}-\ref{lemma64} and  \eqref{abound} ensures that
\begin{equation}\nonumber
\begin{aligned}
\sum_{i=0}^{d}\|R^i\|_{L^1_{t}(\dot{B}^{\frac{d}{p}}_{p,1})}^{\ell}&\lesssim \|V\|_{L^2_{t}(\dot{B}^{\frac{d}{p}}_{p,1})}^2.
\end{aligned}
\end{equation}
Furthermore, we see that, by employing \eqref{inter}, the decomposition $V=(\mathbf{P}V)^\ell+(\{\mathbf{I-\mathbf{P}}\}
V)^\ell+V^h$ and Bernstein's inequality  \ref{lemma21},
\begin{align}\label{VL2}
\|V\|_{L^2_{t}(\dot{B}^{\frac{d}{p}}_{p,1})}&\lesssim \|(\mathbf{P}V)^\ell\|_{L^2_{t}(\dot{B}^{\frac{d}{p}}_{p,1})}+\|(\{\mathbf{I} - \mathbf{P}\}V)^\ell\|_{L^2_{t}(\dot{B}^{\frac{d}{p}}_{p,1})}+\|V\|_{L^{2}_t(\dot{B}^{\frac{d}{2}+1}_{2,1})}^h\lesssim \sqrt{E_{p}(t)D_{p}(t)}. 
\end{align}
Consequently, we conclude that
\begin{equation}\label{localglobal}
\begin{aligned}
&\quad\|(\mathbf{P}V)^{\ell}\|_{\widetilde{L}^{\infty}_{t}(\dot{B}^{\frac{d}{p}-1}_{p,1})\cap L^1_{t}(\dot{B}^{\frac{d}{p}+1}_{p,1})}+\|(\{\mathbf{I} - \mathbf{P}\}V)^{\ell}\|_{\widetilde{L}^{\infty}_{t}(\dot{B}^{\frac{d}{p}}_{p,1})\cap L^1_{t}(\dot{B}^{\frac{d}{p}}_{p,1})}\\
&\lesssim\|(\mathbf{P}V_0)^{\ell}\|_{\dot{B}^{\frac{d}{p}-1}_{p,1}}+\|(\{\mathbf{I} - \mathbf{P}\}V_0)^{\ell}\|_{\dot{B}^{\frac{d}{p}}_{p,1}}+E_{p}(t)D_{p}(t).
\end{aligned}
\end{equation}

\begin{itemize}
\item {\emph{Step 2: High-frequency analysis}}
\end{itemize}

As mentioned before,  Proposition \ref{LemmaspectrallocalHp} only gives the sharp frequency-localization estimates in low frequencies, which are
not sufficient for the high-frequency estimates. In the
$L^2$ framework, we shall go back to the system \eqref{entropeq}, relying on its symmetric structure.

Set $\mathcal{W}\coloneqq  W-\bar{W}$. Using \eqref{tildeH} and \eqref{rtildeU}, we write
\begin{equation}\label{normale}
    \begin{aligned}
    &A^0(\mathcal{W}+\bar{W})\partial_{t}\mathcal{W}+\sum_{i=1}^{d}A^{i}(\mathcal{W}+\bar{W})\partial_{x_{i}}\mathcal{W}+L(\bar{W})\mathcal{W} =\mathbf{r}(\mathcal{W}+\bar{W}).
\end{aligned}
\end{equation}
The initial data are prescribed by
\begin{align}
\mathcal{W}(0,x)=\mathcal{W}_0(x)\coloneqq  W(V_0)(x)-\bar{W}.\label{normald}
\end{align}
We emphasize that $\mathcal{W}\rightarrow V$ is a diffeomorphism in a ball around the zero equilibrium, so one can recover the desired estimates arising from $\mathcal{W}$.


To proceed, we shall take advantage of commutator estimates.  Applying $\dot{\Delta}_{j}$ \eqref{normale} and using \eqref{tildeH}, we get
\begin{equation}\label{normalj}
    \begin{aligned}
    &\quad A^0(\mathcal{W}+\bar{W})\partial_{t}\dot{\Delta}_{j}\mathcal{W}+\sum_{i=1}^{d}A^{i}(\mathcal{W}+\bar{W})\partial_{x_{i}}\dot{\Delta}_{j}\mathcal{W}+L(\bar{W})\dot{\Delta}_{j}\mathcal{W}\\
    & =-\mathcal{R}_{1,j}-\mathcal{R}_{2,j}+\dot{\Delta}_{j}\mathbf{r}(\mathcal{W}+\bar{W}),
\end{aligned}
\end{equation}
where  $\mathcal{R}_{i,j}$ ($i=1,2$) are the commutator terms
\begin{align*}
 &\mathcal{R}_{1,j}\coloneqq  \sum_{i=1}^{d}A^0(\mathcal{W}+\bar{W})[\dot{\Delta}_j,A^0(\mathcal{W}+\bar{W})^{-1}A^{i}(\mathcal{W}+\bar{W})]\partial_{x_{i}}\mathcal{W},\\
 &\mathcal{R}_{2,j}\coloneqq  A^0(\mathcal{W}+\bar{W})[\dot{\Delta}_{j},A^0(\mathcal{W}+\bar{W})^{-1}]L(\bar{W})\mathcal{W}.
 \end{align*}
Recall that $L(\bar{W})$ is nonnegative definite, and its null space coincides with $\mathcal{M}$. Taking the $L^2$ inner product of \eqref{normalj} with  $\dot{\Delta}_{j}\mathcal{W}$, we obtain
\begin{equation}\label{L11}
    \begin{aligned}
    &\quad\frac{1}{2}\frac{d}{dt}\|\dot{\Delta}_{j}\mathcal{W}\|_{L^2}^2+\tilde{\kappa} \|\{\mathbf{I} - \mathbf{P}\}\dot{\Delta}_{j}\mathcal{W}\|_{L^2}^2\\
    &\leq \frac{1}{2}\sum_{i=1}^{d}\|\partial_{x_i}A^{i}(\mathcal{W}+\bar{W})\|_{L^{\infty}}\|\dot{\Delta}_{j}\mathcal{W}\|_{L^2}^2+\Big(\mathcal{R}_{1,j}+\mathcal{R}_{2,j}+\|\dot{\Delta}_{j}\mathbf{r}(\mathcal{W}+\bar{W})\|_{L^2}\Big)\|\dot{\Delta}_{j}\mathcal{W}\|_{L^2}
  \end{aligned}
\end{equation}
for some constant $\tilde{\kappa}>0$.

To create the dissipation of $\mathbf{P}\mathcal{W}$,
the [SK] condition given by Definition \ref{defnSK} is now so classical. The stability assumption was originally introduced in \cite{SK} and guarantees
the necessary coupling between conserved/non-conserved quantities in order to have dissipation in each state variable, which leads to the global in time smooth solutions for small initial data, see for example  \cite{HN, KY1, Yong041} in the Sobolev framework and \cite{c2,c3, XK2} in the critical Besov framework.
It is convenient to rewrite \eqref{entropeq} in the following form
\begin{equation}\label{UUe}
    \begin{aligned}
&A^0(\bar{W})\partial_{t}\mathcal{W}+\sum_{i=1}^{d}A^{i}(\bar{W})\partial_{x_{i}}\mathcal{W}+L(\bar{W})\mathcal{W}=\sum_{i=1}^3h_i(\mathcal{W}),
\end{aligned}
\end{equation}
where $A^0=A^0(\bar{W})$, $A^{i}=A^{i}(\bar{W})$ and $L=L(\bar{W})$ are constant matrices. The nonlinear terms $h_{i}\,(i=1,2,3)$ are formulated in order by
\begin{equation}\nonumber
    \begin{aligned}
h_1(\mathcal{W})&\coloneqq  A^0(\bar{W})A^0(\mathcal{W}+\bar{W})^{-1}\mathbf{r}(\mathcal{W}+\bar{W}),\\
    h_2(\mathcal{W})&\coloneqq  -\sum_{i=1}^{d}A^0(\bar{W}) \Big( A^0(\mathcal{W}+\bar{W})^{-1}A^{i}(\mathcal{W}+\bar{W})-A^0(\bar{W})^{-1}A^{i}(\bar{W}) \Big)\partial_{x_i}\mathcal{W},\\
    h_3(\mathcal{W})&\coloneqq  A^0(\bar{W})\Big( A^0(\mathcal{W}+\bar{W})^{-1}-A^0(\bar{W})^{-1}\Big)L(\bar{W})\mathcal{W}.
\end{aligned}
\end{equation}
Applying the homogeneous operator $\dot{\Delta}_{j}$
to \eqref{UUe} gives
\begin{equation}\label{localizedS}
    \begin{aligned}
&A^0(\bar{W})\dot{\Delta}_{j}\partial_{t}\mathcal{W}+\sum_{i=1}^{d}A^{i}(\bar{W})\dot{\Delta}_{j}\partial_{x_{i}}\mathcal{W}+L(\bar{W})\dot{\Delta}_{j}\mathcal{W}=\sum_{i=1}^3\dot{\Delta}_{j}h_i(\mathcal{W}).
\end{aligned}
\end{equation}
As in \cite{SK}, the [SK] condition implies that there exists a $n\times n$ matrix $K(\omega)$ depending smoothly on $\omega\in \mathbb{S}^{d-1}$ satisfying $K(-\omega)=-K(\omega)$ and that $K(\omega)$ is skew-symmetric. In addition, $[K(\omega)A(\omega)]^{{\rm{sym}}}+L(\bar{W})$ is positive definite for $\omega\in \mathbb{S}^{d-1}$, where $A(\omega)\coloneqq  \sum_{i=1}^{d}A^{i}(\bar{W})\omega_{i}$ with $\omega\in \mathbb{S}^{d-1}$ such that $\omega |\xi|=\xi$. Perform the Fourier transform with respect to $x$ on both sides of \eqref{localizedS} and multiply the resulting equation by $-\mathrm{i}|\xi|K(\omega)$.
Furthermore, performing the inner product  with $\widehat{\dot{\Delta}_{j}\mathcal{W}}$ and taking the real part of each term in the resulting equality yields
\begin{equation}\label{L12}
    \begin{aligned}
    \frac{1}{2}&\frac{d}{dt} |\xi|  {\rm Im}\Big\langle   K(\omega) A^0(\bar{W})\widehat{\dot{\Delta}_{j}\mathcal{W}}, \widehat{\dot{\Delta}_{j}\mathcal{W}}\Big\rangle\\
    &\quad\quad~+|\xi|^2\Big\langle  [ K(\omega)A(\omega)]^{{\rm{sym}}}\widehat{\dot{\Delta}_{j}\mathcal{W}}, \widehat{\dot{\Delta}_{j}\mathcal{W}}\Big\rangle-|\xi| {\rm Im} \langle  K(\omega) L(\bar{W}) \widehat{\dot{\Delta}_{j}\mathcal{W}},\widehat{\dot{\Delta}_{j}\mathcal{W}} \Big\rangle\\
    &~={\rm Im}\Big\langle   \sum_{i=1}^3 \widehat{\dot{\Delta}_{j} h_i(\mathcal{W})}, |\xi|K(\omega) \widehat{\dot{\Delta}_{j}\mathcal{W}}  \Big\rangle,
  \end{aligned}
\end{equation}
where we used that equality
$|\xi|{\rm Im}\left\langle  K(\omega) A^{0}\frac{d}{dt} \widehat{\dot{\Delta}_{j}\mathcal{W}}, \widehat{\dot{\Delta}_{j}\mathcal{W}}\right\rangle =\frac{1}{2}\frac{d}{dt} |\xi|  {\rm Im}\Big\langle K(\omega) A^0(\bar{W})\widehat{\dot{\Delta}_{j}\mathcal{W}}, \widehat{\dot{\Delta}_{j}\mathcal{W}}\Big\rangle$ owing to the skew-symmetry of $K(\omega) A^{0}$. Together \eqref{L12} and \eqref{L11}, we get the Lyapunov inequality for any $j\geq J_0-1$:
    \begin{equation}\label{L}
    \begin{aligned}
    \frac{d}{dt}\mathcal{L}_{j}(t)+\mathcal{D}_{j}(t)&\lesssim \sum_{i=1}^{d}\|\partial_{x_i}A^{i}(\mathcal{W}+\bar{W})\|_{L^{\infty}}\|\dot{\Delta}_{j}\mathcal{W}\|_{L^2}^2\\
    &\quad+(\mathcal{R}_{1,j}+\mathcal{R}_{2,j}+\|\dot{\Delta}_{j}\mathbf{r}(\mathcal{W}+\bar{W})\|_{L^2})\|\dot{\Delta}_{j}\mathcal{W}\|_{L^2}+\eta 2^{-j}\sum_{i=1}^{3} \| \dot{\Delta}_{j}h_i(\mathcal{W})\|_{L^2},
    \end{aligned}
    \end{equation}
   where
    \begin{equation}\nonumber
    \begin{aligned}
    \mathcal{L}_{j}(t)&\coloneqq  \frac{1}{2}\|\dot{\Delta}_{j}\mathcal{W}\|_{L^2}^2+\frac{\eta}{2}\int_{\mathbb{R}^d}  |\xi|{\rm Im} \Big\langle K(\omega) A^0(\bar{W}) \widehat{\dot{\Delta}_{j}\mathcal{W}}, \widehat{\dot{\Delta}_{j}\mathcal{W}}\Big\rangle \,d\xi
  \end{aligned}
\end{equation}
  and
    \begin{equation}\nonumber
    \begin{aligned}
     \mathcal{D}_{j}(t)&\coloneqq  \tilde{\kappa} \|\{\mathbf{I} - \mathbf{P}\}\dot{\Delta}_{j}\mathcal{W}\|_{L^2}^2\\
     &\quad+\eta \int_{\mathbb{R}^d}\Big( |\xi|^2\Big\langle [ K(\omega)A(\omega)]^{{\rm{sym}}}\widehat{\dot{\Delta}_{j}\mathcal{W}}, \widehat{\dot{\Delta}_{j}\mathcal{W}}\Big\rangle-|\xi| {\rm Im} \Big\langle  K(\omega) L(\bar{W})  \widehat{\dot{\Delta}_{j}\mathcal{W}},\widehat{\dot{\Delta}_{j}\mathcal{W}} \Big\rangle\Big) d\xi.
  \end{aligned}
\end{equation}
Choosing $\eta>0$ small enough and taking advantage of the positivity of $[K(\omega)A(\omega)]^{{\rm{sym}}}+L$, we have
        \begin{equation}\nonumber
    \begin{aligned}
    \mathcal{D}_{j}(t)\gtrsim \mathcal{L}_{j}(t)\sim \|\dot{\Delta}_{j}\mathcal{W}\|_{L^2}^2,\quad\quad j\geq J_0-1.
  \end{aligned}
\end{equation}
Furthermore, it follows from \eqref{L} that
\begin{equation}\label{sqrtL}
    \begin{aligned}
    & \frac{d}{dt}\sqrt{\mathcal{L}_{j}(t)+\zeta}+c\sqrt{\mathcal{L}_{j}(t)+\zeta}\\
    &\lesssim \sqrt{\zeta}+\|\nabla \mathcal{W}\|_{L^{\infty}}\sqrt{\mathcal{L}_{j}(t)+\zeta}+\mathcal{R}_{1,j}+\mathcal{R}_{2,j}+\|\dot{\Delta}_{j}\mathbf{r}(\mathcal{W}+\bar{W})\|_{L^2}+2^{-j}\sum_{i=1}^{3}  \| \dot{\Delta}_{j}h_{i}(\mathcal{W})\|_{L^2}
    \end{aligned}
    \end{equation}
for any $\zeta>0$. Integrating  \eqref{sqrtL} over time and having $\zeta\rightarrow 0$.  With the aid of G\"onwall's inequality, we obtain
\begin{equation}\label{Z12345}
   \begin{aligned}
\|\dot{\Delta}_{j}\mathcal{W}\|_{L^{\infty}_{t}(L^2)\cap L^1_{t}(L^2)}&\lesssim e^{C\|\nabla\mathcal{W}\|_{L^1_t(L^{\infty})}}\bigg(\|\dot{\Delta}_{j}\mathcal{W}_0\|_{L^2}+\sum_{i=1}^{5}\int_{0}^{t} Z_{i,j} d\tau \bigg),
    \end{aligned}
    \end{equation}
where $Z_{i,j}$ ($i=1,2,3,4,5$) are given by
 \begin{equation*}
    \begin{aligned}
    Z_{1,j}&\coloneqq  \mathcal{R}_{1,j}, \quad\quad  Z_{2,j}\coloneqq  \mathcal{R}_{2,j},\\
    Z_{3,j}&\coloneqq   \| \dot{\Delta}_{j}\mathbf{r}(\mathcal{W}+\bar{W})\|_{L^2}+2^{-j}\| \dot{\Delta}_{j}h_1(\mathcal{W})\|_{L^2},\\
    Z_{4,j}&\coloneqq   2^{-j}\| \dot{\Delta}_{j}h_2(\mathcal{W})\|_{L^2},\quad Z_{5,j}\coloneqq   2^{-j}\| \dot{\Delta}_{j}h_3(\mathcal{W})\|_{L^2}.
    \end{aligned}
    \end{equation*}
By multiplying both sides of \eqref{Z12345} by $2^{j(\frac{d}{2}+1)}$ and
summing up over $j\geq J_0-1$, we deduce that
\begin{equation}\label{highexistence}
   \begin{aligned}
&\|\mathcal{W}\|_{\widetilde{L}^{\infty}_{t}(\dot{B}^{\frac{d}{2}+1}_{2,1})\cap L^1_{t}(\dot{B}^{\frac{d}{2}+1}_{2,1})}^{h}\lesssim e^{C\|\nabla\mathcal{W}\|_{L^1_t(L^{\infty})}}\bigg(\|\mathcal{W}_0\|_{\dot{B}^{\frac{d}{2}+1}_{2,1}}^{h}+\sum_{i=1}^{5}\int_{0}^{t}\sum_{j\geq J_0-1} 2^{j(\frac{d}{2}+1)} Z_{i,j} d\tau\bigg).
    \end{aligned}
    \end{equation}
 In what follows, we turn to bound those nonlinear terms on the right-hand side of \eqref{highexistence}. It follows from  the commutator estimates in Lemma \ref{lemmacommutator} and the embedding $\dot{\mathbb{B}}^{s,s-\frac{d}{p}+\frac{d}{2}}_{p,2}\hookrightarrow \dot{B}^{s}_{p,1}$ for $2\leq p\leq \min\{4,\frac{2p}{p-2}\}$ that
    \begin{equation}\label{firstnon}
    \begin{aligned}
    \int_{0}^{t} \sum_{j\geq J_0-1}2^{(\frac{d}{2}+1)j} Z_{1,j} d\tau &\lesssim \int_{0}^{t}\sum_{j\geq J_0-1}2^{(\frac{d}{2}+1)j}\sum_{i=1}^d \|[\dot{\Delta}_j, A^0(\mathcal{W}+\bar{W})^{-1}A^i(\mathcal{W}+\bar{W})]\partial_{x_i}\mathcal{W}\|_{L^2}\, d\tau\\
    &\lesssim  \Big\|\nabla\Big(A^0(\mathcal{W}+\bar{W})^{-1}A^i(\mathcal{W}+\bar{W})\Big)\Big\|_{L^{\infty}_{t}(\dot{\mathbb{B}}^{\frac{d}{p},\frac{d}{2}}_{p,2})} \|\mathcal{W}\|_{L^{1}_t(\dot{\mathbb{B}}^{\frac{d}{p}+1,\frac{d}{2}+1}_{p,2})}.
    \end{aligned}
    \end{equation}
To bound the composite function in \eqref{firstnon}, we employ Corollary \ref{corocom}
and \eqref{abound} to get
    \begin{equation*}
    \begin{aligned}
&\Big\|\nabla\Big(A^0(\mathcal{W}+\bar{W})^{-1}A^i(\mathcal{W}+\bar{W})\Big)\Big\|_{L^{\infty}_{t}(\dot{\mathbb{B}}^{\frac{d}{p},\frac{d}{2}}_{p,2})}
\lesssim (1+\|\mathcal{W}\|_{L^{\infty}_{t}(\dot{\mathbb{B}}^{\frac{d}{p},\frac{d}{2}+1}_{p,2})})\|\mathcal{W}\|_{L^{\infty}_{t}(\dot{\mathbb{B}}^{\frac{d}{p}+1,\frac{d}{2}+1}_{p,2})},
    \end{aligned}
    \end{equation*}
from which one can obtain
\begin{equation}\label{firstnon2}
\int_{0}^{t}\sum_{j\geq J_0-1}2^{(\frac{d}{2}+1)j} Z_{1,j} d\tau\lesssim (1+\|\mathcal{W}\|_{L^{\infty}_{t}(\dot{\mathbb{B}}^{\frac{d}{p},\frac{d}{2}+1}_{p,2})})\|\mathcal{W}\|_{L^{\infty}_{t}(\dot{\mathbb{B}}^{\frac{d}{p}+1,\frac{d}{2}+1}_{p,2})}\|\mathcal{W}\|_{L^{1}_t(\dot{\mathbb{B}}^{\frac{d}{p}+1,\frac{d}{2}+1}_{p,2})}.
 \end{equation}
Once again using Lemma
\ref{lemmacommutator} and
Corollary \ref{corocom} for $2\leq p\leq \min\{4,\frac{2p}{p-2}\}$, the second nonlinear
term  $Z_{2,j}$ can be  bounded by
\begin{equation}\label{secondnon}
   \begin{aligned}
    &\quad\int_{0}^{t} \sum_{j\geq J_0-1}2^{(\frac{d}{2}+1)j} Z_{2,j} d\tau\\
    &\lesssim \int_{0}^{t}\sum_{j\geq J_0-1}2^{(\frac{d}{2}+1)j}\sum_{i=1}^d \|[\dot{\Delta}_j, A^0(\mathcal{W}+\bar{W})^{-1}A^i(\mathcal{W}+\bar{W})]L(\bar{W})\mathcal{W}\|_{L^2}\, d\tau\\
    &\lesssim (1+\|\mathcal{W}\|_{L^{\infty}_{t}(\dot{\mathbb{B}}^{\frac{d}{p},\frac{d}{2}+1}_{p,2})})\|\mathcal{W}\|_{L^{\infty}_{t}(\dot{\mathbb{B}}^{\frac{d}{p},\frac{d}{2}+1}_{p,2})}\|\mathcal{W}\|_{L^{1}_t(\dot{\mathbb{B}}^{\frac{d}{p}+1,\frac{d}{2}+1}_{p,2})}.
    \end{aligned}
    \end{equation}
Note that $\mathbf{r}(\cdot)=(r^1,r^2,\dots,r^n)(\cdot)$ fulfills $r^i(\bar{W})=D_{\mathcal{W}}r^i(\bar{W})=0$ ($i=0,2,\dots,n$). Then,
by applying  Lemma \eqref{corocom} to
the composite functions $\mathbf{r}(\mathcal{W}+\bar{W})$ and $h_1(\mathcal{W})$, we arrive at
    \begin{equation}\label{thirdnon}
    \begin{aligned}
       &\quad\int_{0}^{t} \sum_{j\geq J_0-1}2^{(\frac{d}{2}+1)j} Z_{3,j} \,d\tau\lesssim  \|\mathcal{W}\|_{L^{\infty}_{t}(\dot{\mathbb{B}}^{\frac{d}{p},\frac{d}{2}+1}_{p,2})}  \|\mathcal{W}\|_{L^{1}_{t}(\dot{\mathbb{B}}^{\frac{d}{p}+1,\frac{d}{2}+1}_{p,2})}.
    \end{aligned}
    \end{equation}
The product laws and composite function estimates
in hybrid Besov
spaces (see Corollaries \ref{coroApp1} and \ref{corocom}) adapted to $Z_{4,j}$ and $Z_{5,j}$
 indicate that
    \begin{equation}\label{lasttwonon}
    \begin{aligned}
      &\quad\int_{0}^{t} \sum_{j\geq J_0-1}2^{(\frac{d}{2}+1)j}(Z_{4,j}+Z_{5,j}) \,d\tau\\
 &\quad+\Big\|A^0(\bar{W})\Big( A^0(\mathcal{W}+\bar{W})^{-1}-A^0(\bar{W})^{-1}\Big)L(\bar{W})\mathcal{W}\Big\|_{L^1_t(\dot{B}^{\frac{d}{2}+1}_{2,1})}^h\\
      &\lesssim \sum_{i=1}^d\Big\|A^0(\mathcal{W}+\bar{W})^{-1}A^{i}(\mathcal{W}+\bar{W})-A^0(\bar{W})^{-1}A^{i}(\bar{W})\Big\|_{L^{\infty}_{t}(\dot{\mathbb{B}}^{\frac{d}{p},\frac{d}{2}}_{p,2})}\| \mathcal{W}\|_{L^{1}_{t}(\dot{\mathbb{B}}^{\frac{d}{p}+1,\frac{d}{2}+1}_{p,2})}\\
      &\quad+\|A^0(\bar{W})\Big( A^0(\mathcal{W}+\bar{W})^{-1}-A^0(\bar{W})^{-1}\Big)\|_{L^{\infty}_{t}(\dot{\mathbb{B}}^{\frac{d}{p},\frac{d}{2}+1}_{p,2})}\| \mathcal{W}\|_{L^{1}_{t}(\dot{\mathbb{B}}^{\frac{d}{p}+1,\frac{d}{2}+1}_{p,2})}\\
      &\lesssim (1+\|\mathcal{W}\|_{L^{\infty}_{t}(\dot{\mathbb{B}}^{\frac{d}{p},\frac{d}{2}+1}_{p,2})}) \|\mathcal{W}\|_{L^{\infty}_{t}(\dot{\mathbb{B}}^{\frac{d}{p},\frac{d}{2}+1}_{p,2})} \| \mathcal{W}\|_{L^{1}_{t}(\dot{\mathbb{B}}^{\frac{d}{p}+1,\frac{d}{2}+1}_{p,2})}.
        \end{aligned}
    \end{equation}
Finally, adding up \eqref{firstnon}-\eqref{lasttwonon} and using $\dot{\mathbb{B}}^{\frac{d}{p},\frac{d}{2}+1}_{p,2}\hookrightarrow  \dot{\mathbb{B}}^{\frac{d}{p}+1,\frac{d}{2}+1}_{p,2}$ yields for all $t>0$,
\begin{equation}\label{Uhigh00}
   \begin{aligned}
&\quad\|\mathcal{W}\|_{\widetilde{L}^{\infty}_{t}(\dot{B}^{\frac{d}{2}+1}_{2,1})\cap L^1_{t}(\dot{B}^{\frac{d}{2}+1}_{2,1})}^{h}\\
&\lesssim e^{C\|\nabla\mathcal{W}\|_{L^1_t(L^{\infty})}}\bigg(\|\mathcal{W}_0\|_{\dot{B}^{\frac{d}{2}+1}_{2,1}}^{h}+ (1+\|\mathcal{W}\|_{L^{\infty}_{t}(\dot{\mathbb{B}}^{\frac{d}{p},\frac{d}{2}+1}_{p,2})}) \|\mathcal{W}\|_{L^{\infty}_{t}(\dot{\mathbb{B}}^{\frac{d}{p},\frac{d}{2}+1}_{p,2})} \| \mathcal{W}\|_{L^{1}_{t}(\dot{\mathbb{B}}^{\frac{d}{p}+1,\frac{d}{2}+1}_{p,2})}\bigg)
    \end{aligned}
    \end{equation}
This implies our desired \emph{a priori} estimates for $V$. Indeed, observe that $\mathcal{W}$ is the smooth function on $V$, so it follows from the $L^1$-time integrability control of Lipschitz bound that
\begin{align}\label{Lipschitz}
\|\nabla W\|_{L^1_t(L^{\infty})}\lesssim \|\nabla V\|_{L^1_t(L^{\infty})}\lesssim \|\nabla V\|_{L^1_t(\dot{B}^{\frac{d}{p}}_{p,1})}\lesssim \|V\|_{L^1_t(\dot{\mathbb{B}}^{\frac{d}{p}+1,\frac{d}{2}+1}_{p,2})}\lesssim D_p(t).
\end{align}
On the other hand, owing to $\mathcal{W}(0)=V(0)=0$, by applying Corollary \ref{corocom} to the composite functions $\mathcal{W}_0=\mathcal{W}(V_0)$ and $\mathcal{W}=\mathcal{W}(V)$, we deduce that
\begin{align}
  \|\mathcal{W}_0\|_{\dot{B}^{\frac{d}{2}+1}_{2,1}}^{h}&\lesssim (1+\|V_0\|_{\dot{\mathbb{B}}^{\frac{d}{p},\frac{d}{2}+1}_{p,2}}) \|V_0\|_{\dot{\mathbb{B}}^{\frac{d}{p},\frac{d}{2}+1}_{p,2}}\lesssim (1+E_{p,0})E_{p,0},\label{420}\\
  \|\mathcal{W}\|_{L^{\infty}_{t}(\dot{\mathbb{B}}^{\frac{d}{p},\frac{d}{2}+1}_{p,2})}&\lesssim (1+\|V\|_{L^{\infty}_{t}(\dot{\mathbb{B}}^{\frac{d}{p},\frac{d}{2}+1}_{p,2})})\|V\|_{L^{\infty}_{t}(\dot{\mathbb{B}}^{\frac{d}{p},\frac{d}{2}+1}_{p,2})},\label{421}\\
\| \mathcal{W}\|_{L^{1}_{t}(\dot{\mathbb{B}}^{\frac{d}{p}+1,\frac{d}{2}+1}_{p,2})}&\lesssim (1+\|V\|_{L^{\infty}_{t}(\dot{\mathbb{B}}^{\frac{d}{p},\frac{d}{2}+1}_{p,2})})\|V\|_{L^{1}_{t}(\dot{\mathbb{B}}^{\frac{d}{p}+1,\frac{d}{2}+1}_{p,2})}.\label{422}
\end{align}
 Regarding $V=V(\mathcal{W})$, one also has
\begin{equation}\label{423}
\begin{aligned}
\|V\|_{L^{\infty}_{t}(\dot{B}^{\frac{d}{2}+1}_{2,1})\cap L^1_t(\dot{B}^{\frac{d}{2}+1}_{2,1})}^h&\lesssim (1+\|\mathcal{W}\|_{L^{\infty}_{t}(\dot{\mathbb{B}}^{\frac{d}{p},\frac{d}{2}+1}_{p,2})})\|\mathcal{W}\|_{L^{\infty}_{t}(\dot{\mathbb{B}}^{\frac{d}{p},\frac{d}{2}+1}_{p,2})\cap L^1_{t}(\dot{\mathbb{B}}^{\frac{d}{p}+1,\frac{d}{2}+1}_{p,2})}\\
&\lesssim(1+\|V\|_{L^{\infty}_{t}(\dot{\mathbb{B}}^{\frac{d}{p},\frac{d}{2}+1}_{p,2})})^2 \|\mathcal{W}\|_{L^{\infty}_{t}(\dot{\mathbb{B}}^{\frac{d}{p},\frac{d}{2}+1}_{p,2})\cap L^1_{t}(\dot{\mathbb{B}}^{\frac{d}{p}+1,\frac{d}{2}+1}_{p,2})}.
\end{aligned}
\end{equation}
In fact, all the above norms in \eqref{421}-\eqref{423} can be replaced by Chemin-Lerner norms. Inserting those estimates \eqref{Lipschitz}-\eqref{423} into \eqref{Uhigh00}, we conclude that
\begin{equation}\label{Uhigh}
   \begin{aligned}
&\quad\|V\|_{L^{\infty}_{t}(\dot{B}^{\frac{d}{2}+1}_{2,1})\cap L^1_{t}(\dot{B}^{\frac{d}{2}+1}_{2,1})}^{h}\lesssim e^{CD_p(t)}(1+E_p(t))^2\bigg((1+E_{p,0})E_{p,0}+(1+E_p(t))^2 E_p(t)D_{p}(t) \bigg),
    \end{aligned}
    \end{equation}
where we used the bounds
$\|V\|_{L^{\infty}_t(\dot{\mathbb{B}}^{\frac{d}{p},\frac{d}{2}+1}_{p,2})}\lesssim E_{p}(t)$ and $\|V\|_{L^{1}_t(\dot{\mathbb{B}}^{\frac{d}{p}+1,\frac{d}{2}+1}_{p,2})}\lesssim D_{p}(t)$.

Combining \eqref{localglobal} and \eqref{Uhigh}, we conclude \eqref{prioribound} and finish the proof of Proposition \ref{propapriori}.
\end{proof}

\vspace{2mm}

\noindent
\emph{Proof of Theorem \ref{theorem0}.} For $n\geq1$, one can construct a sequence $\{V_0^n\}_{n\geq1}$ of smooth initial data satisfying
$$
\|\mathbf{P}V_0^n\|_{\dot{\mathbb{B}}^{\frac{d}{p}-1,\frac{d}{2}+1}_{p,2}}+\|\{\mathbf{I} - \mathbf{P}\}V_0^n\|_{\dot{\mathbb{B}}^{\frac{d}{p},\frac{d}{2}+1}_{p,2}}\lesssim E_{p,0}.
$$
According to the classical local well-posedness theory for hyperbolic symmetric systems (see {\emph{e.g.}}, \cite{bahouri1}), the system \eqref{m1} with the initial data $V_0^n$ for $n\geq1$ admits a unique classical solution $V^n$ in $\mathcal{C}([0,T^n);H^s(\mathbb{R}^d))$ ($s>d/2+1)$ on $[0,T^n)\times\mathbb{R}^d$ with the maximal time $T^n>0$.

Set $X^n(t)\coloneqq E_{p}(V^n) +D_p(V^n)$ with $E_{p}(V)$ and $D_p(V)$ defined in \eqref{Etr} and \eqref{Dt}, respectively. Define the time
\begin{align*}
T^n_*\coloneqq  \sup \Big\{t\in (0, T^n)\,:\, X^n(t)\leq 2M^*E_{p,0}\Big\}.
\end{align*}
Here the constant $M^*$ is given in Proposition \ref{propapriori}. It is clear that $0<T_n^*\leq T^n$.  It follows from Proposition \ref{propapriori} and $\|V^n\|_{L^{\infty}_t(L^{\infty})}\lesssim X^n(t)$ that, by choosing $E_{p,0}\leq \varepsilon_0$ with a sufficiently small $\varepsilon_0>0$, one has
\begin{align*}
X^n(t)\leq \frac{3}{2}M^* E_{p,0}\quad \text{for}\quad 0\leq t<T^n_*.
\end{align*}
Therefore, a classical continuity argument leads to  $T^n_*=T^n$. Furthermore, the gain of the Lipschitz bound $\|V^n\|_{L^1_t(L^{\infty})}\lesssim X^n(t)\lesssim E_{p,0}$ for $0\leq t<T^n_*$, combined with the classical blow-up criterion (see \cite{bahouri1}[Page 188]), ensures that $T^n_*=T^n=\infty$, which implies that $V^n$ is indeed a global classical solution.

Finally, with the aid of the above uniform bounds, a standard process based on Ascoli's Theorem and Cantor’s diagonal extraction indicates that there exists a limit $V$ such that, up to a subsequence, as $n\rightarrow \infty$, $V^n$ converges to $V$ weakly-* in $L^{\infty}(\mathbb{R}_{+};\dot{B}^{\frac{d}{p}}_{p,1}\cap \dot{B}^{\frac{d}{p}+1}_{p,1})$ and strongly in $L^{\infty}(0,T;\dot{B}^{\frac{d}{p}}_{p,1}(K))$ with any compact subset $K$ of $\mathbb{R}^d$ and time $T>0$, which implies that the limit $V$ satisfies the original system \eqref{m1} subject to the initial data $V_0$. Due to Fatou's property, $V$ satisfies the uniform bound $E_p(t)+D_p(t)\lesssim E_{p,0}$ for all $t\geq0$. Following similar computations in the proof of Proposition \ref{propapriori}, we can obtain $\widetilde{E}_p(t)\lesssim \widetilde{E}_{p,0}$.
Then, the continuity properties $\mathbf{P}V\in \mathcal{C}(\mathbb{R}_+;\dot{\mathbb{B}}^{\frac{d}{p}-1,\frac{d}{2}+1}_{p,2})$ and $\{\mathbf{I} - \mathbf{P}\}V\in \mathcal{C}(\mathbb{R}_+;\dot{\mathbb{B}}^{\frac{d}{p},\frac{d}{2}+1}_{p,2})$ can be proved by a frequency cutoff argument (cf. \cite{bahouri1}[Page 196]). 
The uniqueness in $L^{\infty}(\mathbb{R}_+;\dot{\mathbb{B}}^{\frac{d}{p},\frac{d}{2}}_{p,2})$ can be established by the standard argument  (for example, see \cite{c3}). We feel free to omit the details for brevity.   \hfill $\Box$

\section{Sharp characterization of
large-time asymptotics}\label{section:decay}
In this section, the main purpose is to study the sharp characterization of large-time asymptotics  for the Cauchy problem \eqref{m1}-\eqref{m1d} based on Theorem \ref{theorem0}. In the previous pioneering works, in particular, by the Japanese School in the
'70s and early '80s, for instance, Kawashima, Matsumura and Nishida (see \cite{KMN}), the $L^1$-assumption on initial data was additionally assumed, which leads to the large-time behavior for compressible Navier-Stokes equations and the Boltzmann equation. Later, that condition has been greatly developed by Schonbek \cite{S-ARMA,S-CPDE,S-JAMS} and Wiegner \cite{wiegner1} for incompressible Navier-Stokes equations.

In\cite{UKS}, Umeda, Kawashima and Shizuta  first employed the pointwise energy estimates in Fourier
spaces to show the time-decay property for a general class of symmetric hyperbolic-parabolic composition systems (including the symmetric hyperbolic system), if initial data belong to $L^2\cap L^{q}(1\leq q<2)$. They found that the overall decay rate in these norms was
the heat kernel rate. The main ingredients of analysis lie in the estimates of low-frequency and high-frequency integrals in Fourier spaces. Subsequently, Kawashima
\cite{Kawashimadoctoral} in his doctoral thesis investigated the corresponding nonlinear symmetric systems and obtained the upper bounds of decay estimates of classical small-amplitude solutions in $H^{l}(\mathbb{R}^{n})\cap L^p(\mathbb{R}^{n}) (l>2+d/2, 1\leq p<2)$.
For one-dimensional systems of dissipative balance laws endowed with the convex entropy,  Ruggeri and Serre \cite{RS1} showed that the constant equilibrium state $\bar{U}$ is time asymptotically $L^2$-stable under the zero-mass initial disturbance,  by constructing an appropriate Lyapunov functional. Furthermore, it was proved by Bianchini, Hanouzet and Natalini \cite{BHN} that classical solutions approached the constant equilibrium state in the $L^{p}$-norm at the rate $O(t^{-\frac{d}{2}(1-\frac{1}{p})})$, as $t\rightarrow\infty$, for $p\in[\min\{n,2\},\infty]$. The second author and Kawashima \cite{XK2,XK3} proposed the additional integrability in terms of negative Besov spaces, namely the sets $\dot{B}^{s}_{2,\infty}$. Owing to the critical embedding
$$
\|z\|_{\dot{B}^{\frac{d}{2}-\frac{d}{q}}_{2,\infty}}\lesssim \|z\|_{L^q}, \quad 1\leq q \leq 2,
$$
it makes assumptions in spaces $\dot{B}^{s}_{2,\infty}$
with a negative $s$ that is weaker than the pioneering work \cite{KMN}(this corresponds to the endpoint value $s=-d/2$) or (see \cite{UKS}) in $L^q$ for some $q\in [1,2)$ (taking $s=d/2-d/q$).
The readers are referred to the recent works \cite{c2,c3}, where
Crin-Barat and Danchin
employed the energy method of Lyapunov type
on partially dissipative and deduced the upper bounds of decay estimates. Their approach still requires the $\dot{B}^{\sigma_1}_{2,\infty}(\sigma_1\in (-d/2,d/2-1])$ condition on the low-frequency
part of initial data but not necessarily small.

To the best of our knowledge, the lower bounds of decay rates for the partially dissipative hyperbolic systems \eqref{m1}-\eqref{m1d} are left open since from the seminal work \cite{UKS},
which has been a challenging problem since
Green's function to \eqref{hp} is not explicit in the multi-dimensions. In this paper, we aim at solving that issue and provide a necessary and sufficient regularity condition for time-decay estimates of solutions to \eqref{m1}-\eqref{m1d}
in the new critical functional setting. More precisely, based on Theorem \ref{theorem0}, we
propose the following low-frequency assumption,
according to different regularity indices of
the conservative component
$\mathbf{P}V_0$ and the dissipative component
$\{\mathbf{I}-\mathbf{P}\}V_0$:
\begin{align*}
(\mathbf{P}V_0)^{\ell}\in \dot{B}^{\sigma_1}_{p,\infty}\quad\text{and}\quad  (\{\mathbf{I} - \mathbf{P}\}V_0)^{\ell}\in \dot{B}^{\sigma_1+1}_{p,\infty}\quad \text{for}\quad -\frac{d}{p}\leq \sigma_1<\frac{d}{p}-1.
\end{align*}
The $L^p$-type low-frequency regularity condition was first introduced to investigate the large-time behavior of global solutions so far, revealing a new parabolic profile that enables us to achieve the upper bounds for decay rates. Furthermore, we prove that the condition $(\mathbf{P}V_0)^{\ell}\in \dot{B}^{\sigma_1}_{p,\infty}$ gives a sharp (not only sufficient but also necessary) condition on the upper bounds of decay rates.

\begin{theorem}\label{thm1}{\rm(}Upper bound{\rm)}.
Let the assumptions in Theorem \ref{theorem0}, $d\geq 2$ and $p\neq 4$ when $d=2$ be in force. 
Let $V$ be the global solution to the Cauchy problem \eqref{m1}-\eqref{m1d}. Additionally, assume that $(\{\mathbf{I}-\mathbf{P}\}V_{0})^{\ell}\in\dot{B}^{\sigma_1+1}_{p,\infty}$ with $-\frac{d}{p}\leq \sigma_{1}<\frac{d}{p}-1$. Then, there exists a universal constant $C>0$ such that the solution $V=\mathbf{P}V+\{\mathbf{I} - \mathbf{P}\}V$ has the decay estimates
\begin{equation}\label{decayupper}
\left\{
\begin{aligned}
&\|\mathbf{P}V(t)\|_{\dot{\mathbb{B}}^{\sigma,\frac{d}{2}+1}_{p,2}}\leq C (1+t)^{-\frac{1}{2}(\sigma-\sigma_{1})},\quad\quad\quad\quad\quad\quad ~\sigma_1<\sigma\leq \frac{d}{p}+1, \\
&\|\{\mathbf{I}-\mathbf{P}\}V(t)\|_{\dot{\mathbb{B}}^{\sigma',\frac{d}{2}+1}_{p,2}}\leq C (1+t)^{-\frac{1}{2}(\sigma'-\sigma_{1})-\frac{1}{2}},\quad\quad \sigma_1+1<\sigma'\leq \frac{d}{p}
\end{aligned}
\right.
\end{equation}
for all $t>0$ if and only if $(\mathbf{P}V_{0})^{\ell}\in\dot{B}^{\sigma_{1}}_{p,\infty}$.

\end{theorem}

As mentioned in Subsection \ref{section:linear}, the parabolic system \eqref{diffusion} supplemented with the initial data $\Psi_0$ given by \eqref{W0} is a
good approximation of the conservative part of the solution to partially dissipative hyperbolic systems. As an intermediate step of
 lower bounds of decay rates of solutions to \eqref{m1}-\eqref{m1d}, we study the asymptotic stability of the conservative component $\mathbf{P}V $ and dissipative component $\{\mathbf{I} - \mathbf{P}\}V$,
which exhibits faster decay rates of $V-V^*$ compared to those of $V$ in \eqref{decayupper}.
\begin{theorem}\label{thm2}
{\rm(}Asymptotic stability{\rm)}. Let the assumptions in Theorem \ref{theorem0}, $d\geq 2$ and $p\neq 4$ when $d=2$ be in force. 
 Let $V$ be the global solution to the Cauchy problem \eqref{m1}–\eqref{m1d} with initial data $V_0$. Suppose, in addition, that the low-frequency part satisfies $(\{\mathbf{I} - \mathbf{P}\}V_{0})^{\ell}\in\dot{B}^{\sigma_{1}+1}_{p,\infty}$ and $(\mathbf{P}V_{0})^{\ell}\in\dot{B}^{\sigma_{1}}_{p,\infty}$ with $-\frac{d}{p}\leq \sigma_{1}<\frac{d}{p}-1$. Then there exists a time $t_0$ such that for all $t>t_0$, we have
\begin{equation}\label{asy:stab}
\begin{aligned}
&\|\mathbf{P}(V-V^{*})(t)\|_{\dot{\mathbb{B}}^{\sigma,\frac{d}{2}+1}_{p,2}}\leq
\begin{cases}
(1+t)^{-\frac{1}{2}(\sigma-\sigma_1+\alpha^*)},
& \mbox{\quad\quad\quad~~$\sigma_{1}<\sigma\leq \frac{d}{p}-1$},\\
(1+t)^{-\frac{1}{2}(\sigma-\sigma_1+\frac{1}{2}\alpha^*)},
& \mbox{\quad\quad\quad ~~$\frac{d}{p}-1<\sigma\leq \frac{d}{p}+1-\alpha^*$}
\end{cases}
\end{aligned}
\end{equation}
and
\begin{equation}\label{asy:stab1}
\begin{aligned}
&\|\{\mathbf{I} - \mathbf{P}\}(V-V^{*})(t)\|_{\dot{\mathbb{B}}^{\sigma',\frac{d}{2}+1}_{p,2}}\leq
\begin{cases}
\langle t\rangle^{-\frac{1}{2}(\sigma'-\sigma_1+1+\alpha^*)},
& \mbox{~ $\sigma_{1}+1<\sigma'\leq \frac{d}{p}-2$},\\
\langle t\rangle^{-\frac{1}{2}(\sigma'-\sigma_1+1+\frac{1}{2}\alpha^*)},
& \mbox{~ $\frac{d}{p}-2<\sigma'\leq \frac{d}{p}-\alpha^*$}.
\end{cases}
\end{aligned}
\end{equation}
Here, the profile $V^{*}$ is defined in \eqref{Vp}, and  the constant $\alpha^{*}\in(0,1]$ is given by
\begin{equation}\label{sigma2}
\begin{aligned}
&\alpha^{*}=
\begin{cases}
1,
& \mbox{\quad if ~~$-\frac{d}{p}\leq \sigma_{1}<\frac{d}{p}-2$},\\
\frac{d}{p}-1-\sigma_{1}-\varepsilon,
& \mbox{\quad if  ~~$\frac{d}{p}-2\leq \sigma_{1}<\frac{d}{p}-1$}
\end{cases}
\end{aligned}
\end{equation}
for any $0<\varepsilon<\frac{d}{p}-1-\sigma_1$.
\end{theorem}

By imposing the sharp assumption $\Psi_0^{\ell}\in\dot{\mathcal{B}}^{\sigma_{1}}_{p,\infty}$ (see \eqref{Bsubset}) rather than
$(\mathbf{P}V_0)^{\ell}\in\dot{\mathcal{B}}^{\sigma_{1}}_{p,\infty}$, we can establish the upper and lower bounds for decay rates.



\begin{theorem}\label{thm3} {\rm(}Upper and lower bounds{\rm)}. Let the assumptions in Theorem \ref{theorem0}, $d\geq 2$ and $p\neq 4$ when $d=2$ be in force. Set
$$
\Psi_0\coloneqq  V_{1,0}-\sum_{i=1}^{d}A^{i}_{1,2}D^{-1} \partial_{x_{i}} V_{2,0}.
$$
Let $V$ be the global solution to the Cauchy problem \eqref{m1}-\eqref{m1d}. Additionally, assume $(\{\mathbf{I}-\mathbf{P}\}V_{0})^{\ell}\in\dot{B}^{\sigma_{1}+1}_{p,\infty}$ with $-\frac{d}{p}\leq \sigma_{1}<\frac{d}{p}-1$. Then there exist two universal constants $c, C>0$ and a time $t_1>0$ such that the solution $V=\mathbf{P}V+\{\mathbf{I} - \mathbf{P}\}V$ satisfies 
\begin{equation}\label{decaytwoside}
\left\{
\begin{aligned}
c(1+t)^{-\frac{1}{2}(\sigma-\sigma_{1})}
&\le \|\mathbf{P}V(t)\|_{\dot{\mathbb{B}}^{\sigma,\frac{d}{2}+1}_{p,2}}
\le C (1+t)^{-\frac{1}{2}(\sigma-\sigma_{1})},
\quad\quad\,\,\,
\phantom{\sigma_1+1{}}\sigma_1<\sigma\le \frac{d}{p}+1, \\[0.4em]
c(1+t)^{-\frac{1}{2}(\sigma'-\sigma_{1})-\frac{1}{2}}
&\le \|\{\mathbf{I}-\mathbf{P}\}V(t)\|_{\dot{\mathbb{B}}^{\sigma',\frac{d}{2}+1}_{p,2}}
\le C(1+t)^{-\frac{1}{2}(\sigma'-\sigma_{1})-\frac{1}{2}},
\quad
\sigma_1+1<\sigma'\le \frac{d}{p}
\end{aligned}
\right.
\end{equation}
for all $t>t_1$ if and only if $\Psi_0^{\ell}\in\dot{\mathcal{B}}^{\sigma_{1}}_{p,\infty}$.  

\end{theorem}

Theorems \ref{thm1}-\ref{thm3} provide a sharp decay characterization for
partially dissipative systems of balance laws in multidimensions, which is of independent interest. The proof is mainly motivated by Wiegner's argument \cite{wiegner1} regarding the energy
decay of Leray's weak solutions to the incompressible Navier-Stokes equations and inverse Wiegner’s argument \cite{skalak1} (bounding the discrepancy between the nonlinear solution and the linear solution). As a side effect, the smallness of low frequencies of initial data can be removed in the Fourier semigroup framework in contrast to \cite{XK2,XK3}.

\begin{remark}\normalfont
We comment on a few points of immediate relevance:
\begin{itemize}
\item  The decay properties of hyperbolic--parabolic composite systems
(including hyperbolic balance systems) were first investigated in
\cite{UKS} more than forty years ago. Since then, numerous works have been
devoted to the large-time behavior of solutions (see, for instance,
\cite{Kawashimadoctoral,BHN,XK2,XK3,c2,c3}). However, the optimality of the
corresponding time-decay rates is left open. To the best of our
knowledge, the paper provides the first proof of lower bounds
for decay estimates not only in the $L^2$ framework but also in the general
$L^p$ one, which  allows to rely less on
traditional spectral/Green's function analysis. A key novelty lies in the introduction of the
new effective quantity $\Psi(t,x)$, which captures the intrinsic interaction
between the conservative and dissipative components. Moreover,
Theorem~\ref{thm2} identifies the asymptotic profiles of both $\mathbf{P}V$
and $\{\mathbf{I}-\mathbf{P}\}V$, yielding faster decay rates for the
remainders. In particular, decay bounds are obtained for the
highest-order derivatives ($\sigma=d/p+1$), and hence for the critical
Lipschitz norm $\|\nabla V\|_{L^\infty}$, which is not available in
the previous literature.



\item The present work also provides the first study of the {\emph{inverse problem}} for the
large-time asymptotics of partially dissipative hyperbolic systems.
More precisely, we show that the additional
$\dot{B}^{\sigma_{1}}_{p,\infty}$ assumption is sharp in comparison with
previous results: it is not only sufficient but also {\emph{necessary}} to
obtain two-sided decay estimates for solutions. In particular, we establish a theory of decay characters for partially dissipative systems of balance laws: Under
$(\{\mathbf{I}-\mathbf{P}\}V_{0})^{\ell}\in\dot{B}^{\sigma_{1}+1}_{2,\infty}$
with $-\frac{d}{2}\leq \sigma_{1}<\frac{d}{2}-1$, the solution $V$
satisfies
\begin{equation}\nonumber
\left\{
\begin{aligned}
c(1+t)^{-\frac{1}{2}(\sigma-\sigma_1)}
&\leq \|\Lambda^{\sigma}\mathbf{P}V(t)\|_{\dot{B}^{0}_{2,1}}
\leq C (1+t)^{-\frac{1}{2}(\sigma-\sigma_1)},
\qquad \sigma_1<\sigma\leq \frac{d}{2}+1, \\
c(1+t)^{-\frac{1}{2}(\sigma'-\sigma_1+1)}
&\leq \|\Lambda^{\sigma'}\{\mathbf{I}-\mathbf{P}\}V(t)\|_{\dot{B}^{0}_{2,1}}
\leq C(1+t)^{-\frac{1}{2}(\sigma'-\sigma_1+1)},
\qquad \sigma_1+1<\sigma'\leq \frac{d}{2},
\end{aligned}
\right.
\end{equation}
if and only if  $\Psi_0$ admits the decay character
$r^*=r^*(\Psi_0)=-\sigma_1-\frac{d}{2}$ (see
\cite{bjorland1,brandolese1,niche1}), that is,
\begin{align*}
0
&< \liminf_{r\to0^+}
r^{-2(r^*+\frac{d}{2}+\sigma)}
\int_{|\xi|\le r} |\xi|^{2\sigma}|\widehat{\Psi_0}(\xi)|^2\,d\xi \le \limsup_{r\to0^+}
r^{-2(r^*+\frac{d}{2}+\sigma)}
\int_{|\xi|\le r} |\xi|^{2\sigma}|\widehat{\Psi_0}(\xi)|^2\,d\xi
<\infty.
\end{align*}

\end{itemize}
\end{remark}

In what follows, we provide illustrations on the proofs of Theorems \ref{thm1}-\ref{thm3}.
The decay characterization for the linear problem \eqref{hplinear} has been investigated in Proposition \ref{proplinear}. Furthermore, to establish
the optimal time-decay bounds of the solution to the nonlinear problem \eqref{m1}, the well-known
Wiegner's argument for viscous incompressible flows \cite{wiegner1}
 can be
adapted to suit partially dissipative hyperbolic systems.


 First, we prove that if $V_L$ has the upper bounds for decay estimates (Proposition \ref{proplinear}), then $\delta V=V-V_L$ obeys faster rates (see Proposition \ref{properror1} below). 
 This result is essentially derived by
 the frequency-localization Duhamel's principle:
 \begin{equation}\label{LocalLDuhamel}
\begin{aligned}
\dot{\Delta}_j\delta V=\int_{0}^{t} \mathcal{G}(t-\tau) \dot{\Delta}_j R^{0} d\tau+\sum_{i=1}^{d}\int_{0}^{t} \mathcal{G}(t-\tau)\partial_{x_{i}}\dot{\Delta}_j R^{i} d\tau.
\end{aligned}
\end{equation}
Observe that there is the
spatial derivative on $R^i$ ($i=1,2,\dots,d$) in \eqref{LocalLDuhamel}.
However, the challenging difficulty lies in the presence of the low-order nonlinear term $R^0$, which cannot provide
faster time-decay rates directly. To address this issue,
we make the best use of the property $\mathbf{P}R^0=0$ 
and quantitative pointwise estimates for $\mathbf{P}\delta V$ and  $\{\mathbf{I} - \mathbf{P}\}\delta V$ separately. Consequently, we have 
 the factor ``$2^j$''-gain (see \eqref{555} and \eqref{deltaJdeltaIPV} for details)
and leads to the improvement in the decay rates of $\delta V$ up to $t^{-1/2}$. On the other hand, to analyze the nonlinear term without using the smallness of $\|{\mathbf{P}}V_0\|_{\dot{B}^{\sigma_1}_{p,\infty}}$ and $\|\{\mathbf{I} - \mathbf{P}\}V_0\|_{\dot{B}^{\sigma_{1}+1}_{p,\infty}}$,
we need an accurate decomposition of the nonlinear terms $R^i$ and make full use of faster decay rates of the component
 $\{\mathbf{I} - \mathbf{P}\}V$. Precisely,
 \begin{align*}
| R^i|\sim |V|^2&\lesssim |V_L|^2+\delta \mathcal{I}_1+\delta\mathcal{I}_2
 \end{align*}
with
 \begin{align*}
 \delta \mathcal{I}_1&=|\delta V|(|\{\mathbf{I} - \mathbf{P}\} V_L|+|\{\mathbf{I} - \mathbf{P}\} V|)+|\{\mathbf{I} - \mathbf{P}\}\delta V| (|\mathbf{P} V_L|+|\mathbf{P} V|), \quad
 \delta \mathcal{I}_2 =|\mathbf{P}\delta V|( |\mathbf{P}V|+|\mathbf{P}V_L|),
 \end{align*}
which enables us to establish faster decay rates between the nonlinear solution and the linear solution for low frequencies, by quantitatively utilizing  the different decay properties of $\mathbf{P}V_L$, $\{\mathbf{I} - \mathbf{P}\}V_L$, $\mathbf{P}\delta V$ and $\{\mathbf{I} - \mathbf{P}\}\delta V$.
In the second step, we establish the time-weighted estimates of high frequencies for the error $\delta V$. That step strongly relies on the elementary $L^2$ approach and the [SK] condition. The symmetric form \eqref{entropeq} in terms of entropy enables us to avoid the technical obstacle that there is a loss of one derivative.
It is worth noting that non-standard
product laws, composite estimates and commutator estimates have been well developed
by means of Bony's para-product decomposition in order to handle nonlinear terms in hybrid norms of $L^p$-$L^2$ type (see Lemmas \ref{NonClassicalProLaw1} and \ref{lemmacommutator}, Corollaries \ref{coroApp1} and \ref{corocom} ).
The last step is
dedicated to establishing the regularity evolution of
$\mathbf{P} V$ and $\{\mathbf{I} - \mathbf{P}\}V$, which allows to close the time-weighted estimate
for $\delta V$.

Finally, we prove the necessary part of the low-frequency assumption in terms of $\dot{B}^{\sigma_{1}}_{p,\infty}$
on the upper and lower bounds for decay rates. For that purpose, we develop the inverse Wiegner's argument from incompressible flows \cite{skalak1} to the partially dissipative hyperbolic systems in the $L^p$ framework (Proposition \ref{properrorin} ). Indeed, one can prove that
$\delta V$ satisfies faster rates if the global-in-time solution $V$ constructed in Theorem \ref{theorem0} satisfies \eqref{decayupper}. Furthermore, it can be shown that the linear
solution $(\mathbf{P} V_{L}, \{\mathbf{I} - \mathbf{P}\}V_{L})$
to \eqref{hplinear} has the same decay rates as $(\mathbf{P} V, \{\mathbf{I} - \mathbf{P}\}V)$. Consequently, the linear theorem (Proposition \ref{proplinear}) implies the desired low-frequency regularity assumptions.





\subsection{Sufficient conditions for decay}\label{sectionsufficient}
In this section, we shall develop the classical Wiegner's argument from incompressible Navier-Stokes equations \cite{wiegner1} to
partially dissipative hyperbolic
systems and prove the “if" part in Theorems \ref{thm1} or \ref{thm3} and Theorem \ref{thm2}.

\subsubsection{Error estimates}

A key ingredient in Wiegner's argument is to establish faster decay estimates
of $
\delta V\coloneqq  V-V_{L}
$, where $V_{L}$
denotes the solution to the linear problem \eqref{hp}. This indicates that the nonlinear solution $V$ is asymptotically equivalent to its linear counterpart $V_L$ in large times. 
 Specifically, the difference $\delta V$ solves
\begin{equation}\label{deltav}
    \begin{aligned}
    &\partial_{t}\delta V+\sum_{i=1}^{d}A^{i}\partial_{x_{i}}\delta V+L\delta V=R^{0}+\sum_{i=1}^{d}\partial_{x_{i}}R^{i},\quad\quad \delta V|_{t=0}=0,
\end{aligned}
\end{equation}
where the nonlinear terms $R^{0}$ and $R^{i}$ are given by \eqref{R0Ri}.

For the reader’s convenience, we first introduce the difference functionals.  Define the following norm for initial data
\begin{align}
\mathcal{D}_{p,\sigma_1}\coloneqq  \|(\mathbf{P}V_0)^{\ell}\|_{\dot{B}^{\sigma_1}_{p,\infty}}+\|(\{\mathbf{I} - \mathbf{P}\}V_0)^{\ell}\|_{\dot{B}^{\sigma_1+1}_{p,\infty}}+\|V_0\|_{\dot{B}^{\frac{d}{2}+1}_{2,1}}^{h}\label{mathcalD0}
\end{align}
and the difference functional
\begin{equation}\label{mathcalDt}
\begin{aligned}
\mathcal{D}_{p}(t)&\coloneqq  \sup_{\sigma_{1}<\sigma\leq  \frac{d}{p}+1-\alpha^*} \|\langle t\rangle^{\frac{1}{2}(\sigma-\sigma_{1}+\sigma_{1}^*)}\mathbf{P}\delta V\|_{L^{\infty}_{t}(\dot{B}^{\sigma}_{p,1})}^{\ell}\\
&\quad+\sup_{\sigma_{1}+1<\sigma'\leq \frac{d}{p}-\alpha^*} \|\langle t\rangle^{\frac{1}{2}(\sigma'-\sigma_{1}+1+\sigma_{2}^*)}\{\mathbf{I}-\mathbf{P}\}\delta V\|_{L^{\infty}_{t}(\dot{B}^{\sigma'}_{p,1})}^{\ell}+\|\langle \tau\rangle^{\beta^*} \delta V\|_{\widetilde{L}^{\infty}_{t}(\dot{B}^{\frac{d}{2}+1}_{2,1})}^{h},
\end{aligned}
\end{equation}
where
\begin{align}\label{5.4}
   \alpha^*\coloneqq  \begin{cases}
1,
& \mbox{\quad if  ~~$-\frac{d}{p}\leq \sigma_1<\frac{d}{p}-2$},\\
\frac{d}{p}-1-\sigma_1-\varepsilon_1,
& \mbox{\quad if  ~~ $\frac{d}{p}-2\leq \sigma_1<\frac{d}{p}-1$}
\end{cases}
\end{align}
for any $0<\varepsilon_1<\frac{d}{p}-1-\sigma_1$. The pair $(\sigma_{1}^*,\sigma_{2}^*)\in (0,1]$ is given by
\begin{equation}\label{sigma222}
\begin{aligned}
&\sigma_{1}^*\coloneqq
\begin{cases}
\alpha^*,
& \mbox{\quad if  ~~$\sigma_{1}<\sigma\leq \frac{d}{p}-1$},\\
\frac{1}{2}\alpha^*,
& \mbox{\quad if  ~~$\frac{d}{p}-1<\sigma\leq \frac{d}{p}+1-\alpha^*$},
\end{cases}
&\sigma_{2}^*\coloneqq
\begin{cases}
\alpha^*,
& \mbox{\quad if ~~$\sigma_{1}+1<\sigma'\leq \frac{d}{p}-2$},\\
\frac{1}{2}\alpha^*,
& \mbox{\quad if  ~~$\frac{d}{p}-2<\sigma'\leq \frac{d}{p}-\alpha^*$},
\end{cases}
\end{aligned}
\end{equation}
and $\beta^*>0$ is the maximal decay exponent $\beta^*\coloneqq  \frac{1}{2}(\frac{d}{p}+1-\sigma_{1}-\frac{1}{2}\alpha^*-\varepsilon_1)$.

As a consequence, the time-weighted energy estimate for $\delta V$ is included in the following proposition, which is the analogue of Wiegner's theorem (see \cite{wiegner1}).

\begin{prop}\label{properror1}
Let the assumptions of Theorem \ref{theorem0} be in force. 
Assume additionally that the initial datum $V_{0}$ satisfies $(\mathbf{P}V_{0})^{\ell}\in\dot{B}^{\sigma_{1}}_{p,\infty}$ and $(\{\mathbf{I}-\mathbf{P}\}V_{0})^{\ell}\in\dot{B}^{\sigma_{1}+1}_{p,\infty}$ with $-\frac{d}{p}\leq \sigma_{1}<\frac{d}{p}-1$. Then  the difference $
\delta V\coloneqq  V-V_{L}
$ fulfills the following time-weighted inequality
\begin{equation}
\begin{aligned}
&\mathcal{D}_{p}(t)\leq C(1+\mathcal{D}_{p,\sigma_1})\mathcal{D}_{p,\sigma_1} \label{errores}
\end{aligned}
\end{equation}
for any $t>0$, where $C$ is a positive constant independent of the time $t$.
\end{prop}

Proposition \ref{properror1} provides the decay rates of $\mathbf{P}\delta V$ and $\{\mathbf{I} - \mathbf{P}\}\delta V$, which are faster than those of $\mathbf{P}V_{L}$ and $\{\mathbf{I} - \mathbf{P}\}V_{L}$, respectively. 

\begin{coro}\label{properror}
For any $t>0$, the difference $\delta V=V-V_{L}$ satisfies that
\begin{equation}\label{nonlinearfaster}
\left\{
\begin{aligned}
\|\mathbf{P}\delta V\|_{\dot{\mathbb{B}}^{\sigma,\frac{d}{2}+1}_{p,2}}&\lesssim (1+\mathcal{D}_{p,\sigma_1})\mathcal{D}_{p,\sigma_1}\langle t\rangle^{-\frac{1}{2}(\sigma-\sigma_1+\sigma_1^*)},\quad\quad\quad \sigma_1<\sigma\leq \frac{d}{p}+1-\alpha^*,\\
\|\{\mathbf{I} - \mathbf{P}\}\delta V\|_{\dot{\mathbb{B}}^{\sigma',\frac{d}{2}+1}_{p,2}}&\lesssim (1+\mathcal{D}_{p,\sigma_1})\mathcal{D}_{p,\sigma_1}\langle t\rangle^{-\frac{1}{2}(\sigma'-\sigma_1+1+\sigma_2^*)},\quad\quad \sigma_1+1<\sigma'\leq \frac{d}{p}-\alpha^*.
\end{aligned}
\right.
\end{equation}
\end{coro}

Indeed, the proof of Proposition \ref{properror1} relies on both the low-frequency and high-frequency error estimates {\rm(}see Lemmas \ref{lemmalow}-\ref{lemmahigh}{\rm)} as well as the regularity evolution of $\delta V$ {\rm(}see Proposition \ref{propevolution}{\rm)}.
Let us keep in mind that the solution $V$ in Theorem \ref{theorem0} satisfies
\begin{equation}\label{Vuniform}
\begin{aligned}
&\|V\|_{\widetilde{L}^{\infty}_t(\dot{\mathbb{B}}^{\frac{d}{p},\frac{d}{2}+1}_{p,2})\cap L^1_t(\dot{\mathbb{B}}^{\frac{d}{p}+1,\frac{d}{2}+1}_{p,2})}\lesssim E_{p,0},\quad t\in\mathbb{R}_{+}.
\end{aligned}
\end{equation}
By applying the same argument as in \eqref{r2}, one can arrive at
\begin{equation}\label{Etlinear}
\begin{aligned}
\|\mathbf{P}V_{L}\|_{\widetilde{L}^{\infty}_{t}(\dot{\mathbb{B}}^{\frac{d}{p}-1,\frac{d}{2}+1}_{p,2})\cap L^1_{t}(\dot{\mathbb{B}}^{\frac{d}{p}+1,\frac{d}{2}+1}_{p,2})}+\|\{\mathbf{I}-\mathbf{P}\}V_{L}\|_{\widetilde{L}^{\infty}_{t}(\dot{\mathbb{B}}^{\frac{d}{p},\frac{d}{2}+1}_{p,2})\cap L^1_{t}(\dot{\mathbb{B}}^{\frac{d}{p},\frac{d}{2}+1}_{p,2})}\lesssim E_{p,0},
\end{aligned}
\end{equation}
which in turn implies
\begin{equation}\label{VLuniform}
\begin{aligned}
&\|V_L\|_{\widetilde{L}^{\infty}_t(\dot{\mathbb{B}}^{\frac{d}{p},\frac{d}{2}+1}_{p,2})\cap L^1_t(\dot{\mathbb{B}}^{\frac{d}{p}+1,\frac{d}{2}+1}_{p,2})}\lesssim E_{p,0},\quad t\in\mathbb{R}_{+}.
\end{aligned}
\end{equation}
As established by Proposition \ref{proplinear},
the linear solution $V_{L}$ to \eqref{hplinear} has the decay estimate \eqref{upperlinear}.
As a matter of fact,
it follows from the proof of Proposition \ref{proplinear} (see \eqref{hhhh} and \eqref{llll}) that
\begin{equation}\label{preciselinear1}
\left\{
\begin{aligned}
  & \| \mathbf{P}V_{L}\|_{\dot{B}^{\sigma}_{p,1}}\lesssim \|\mathbf{P} V_{L}\|_{\dot{\mathbb{B}}^{\sigma,\frac{d}{2}+1}_{p,2}}\lesssim \langle t\rangle^{-\frac{1}{2}(\sigma-\sigma_2)}\mathcal{D}_{p,\sigma_2},\quad\quad\quad\quad\quad \quad\quad\quad \sigma_2<\sigma\leq \frac{d}{p}+1,\\
   & \|\{\mathbf{I} - \mathbf{P}\} V_{L}\|_{\dot{B}^{\sigma'}_{p,1}}\lesssim\|\{\mathbf{I} - \mathbf{P}\} V_{L}\|_{\dot{\mathbb{B}}^{\sigma',\frac{d}{2}+1}_{p,2}}\lesssim \langle t\rangle^{-\frac{1}{2}(\sigma'-\sigma_2+1)}\mathcal{D}_{p,\sigma_2},\quad \sigma_2+1<\sigma'\leq \frac{d}{p}+1,\\
   &\| V_{L}\|_{\dot{B}^{\sigma''}_{p,1}}\lesssim \|V_{L}\|_{\dot{\mathbb{B}}^{\sigma'',\frac{d}{2}+1}_{p,2}}\lesssim \langle t\rangle^{-\frac{1}{2}(\sigma''-\sigma_2)}\mathcal{D}_{p,\sigma_2},\quad \quad~~\quad\quad\quad\quad\quad\quad\quad \sigma_2+1<\sigma''\leq \frac{d}{p}+1
\end{aligned}
\right.
\end{equation}
for $ \sigma_1\leq \sigma_2<\frac{d}{p}+1$, where the norm $\mathcal{D}_{p,\sigma}$ associated with the initial data is defined by
\begin{equation}\label{varepsilonLsigma'}
\begin{aligned}
\mathcal{D}_{p,\sigma}&\coloneqq  \|(\mathbf{P}V_0)^{\ell}\|_{\dot{B}^{\sigma}_{p,\infty}}+\|(\{\mathbf{I} - \mathbf{P}\}V_0)^{\ell}\|_{\dot{B}^{\sigma+1}_{p,\infty}}+\|V_0\|_{\dot{B}^{\frac{d}{2}+1}_{2,1}}^{h}.
\end{aligned}
\end{equation}
Note that
$\mathcal{D}_{p,\sigma_1}$ is not necessarily small; however,  there is some smallness of $\mathcal{D}_{p,\sigma_2}$ with $\sigma_1<\sigma_2<\frac{d}{p}-1$ due to the interpolation between $\mathcal{D}_{p,\sigma_1}$ and $E_{p,0}$, which  plays a role in handling those nonlinear terms in \eqref{deltav}.

In addition, we shall repeatedly use the fact that for $0\leq \gamma_{1}\leq \gamma_{2}$ and $\gamma_2>1$,
\begin{equation}\label{ineq}
\begin{aligned}
&\int_{0}^{t} \langle t-\tau\rangle^{-\gamma_{1}} \langle\tau\rangle^{-\gamma_{2}} \,d\tau\lesssim
\langle t\rangle^{-\gamma_{1}}.
\end{aligned}
\end{equation}
Also, owing to
the definition \eqref{mathcalDt} of $\mathcal{D}_{p}(t)$, we have the decay estimates
\begin{equation}
\left\{
\begin{aligned}\label{Dtdecay}
&\|\mathbf{P}\delta V(t)\|_{\dot{B}^{\sigma}_{p,1}}\lesssim \langle t\rangle^{-\frac{1}{2}(\sigma-\sigma_1+\sigma_1^*)} \mathcal{D}_{p}(t),\quad\quad\quad \quad\quad\sigma_1<\sigma\leq \frac{d}{p}+1-\alpha^*,\\
&\|\{\mathbf{I} - \mathbf{P}\}\delta V(t)\|_{\dot{B}^{\sigma'}_{p,1}}\lesssim \langle t\rangle^{-\frac{1}{2}(\sigma'-\sigma_1+1+\sigma_2^*)} \mathcal{D}_{p}(t),\quad \sigma_1+1<\sigma'\leq \frac{d}{p}-\alpha^*,\\
& \|\delta V(t)\|_{\dot{B}^{\sigma''}_{p,1}}\lesssim \langle t\rangle^{-\frac{1}{2}(\sigma''-\sigma_1+\sigma_1^*)} \mathcal{D}_{p}(t),\quad\quad\quad \quad\quad\quad~ \sigma_1+1<\sigma''\leq \frac{d}{p}+1-\alpha^*
\end{aligned}
\right.
\end{equation}
for all $t>0$.


\subsubsection{Bounds for the low frequencies}

This part is devoted to the time-decay estimates of $\mathbf{P}\delta V$ and $\{\mathbf{I} - \mathbf{P}\}\delta V$ in low frequencies.

\begin{lemma}\label{lemmalow}
Let $\sigma_1<\sigma\leq  \frac{d}{p}+1-\alpha^*$. It holds that
\begin{equation}\label{PV1}
\begin{aligned}
\|\mathbf{P}\delta V\|_{\dot{B}^{\sigma}_{p,1}}^{\ell} &\lesssim \langle t\rangle^{-\frac{1}{2}(\sigma-\sigma_1+\sigma_1^*)}\Big(\mathcal{D}_{p,\sigma_1+\frac{1}{2}\alpha^*}+\|(\mathbf{P}V,\mathbf{P}\delta V)\|_{L^{\infty}_{t}(\dot{B}^{\sigma_1+\frac{1}{2}\alpha^*}_{p,\infty})}+\| \delta V\|_{L^{\infty}_{t}(\dot{B}^{\sigma_1+1}_{p,\infty})}\Big)\mathcal{D}_{p}(t)\\
&\quad+\langle t\rangle^{-\frac{1}{2}(\sigma-\sigma_1+\sigma_1^*)}\Big( \|\delta V\|_{L^{\infty}_{t}(\dot{B}^{\sigma_1+1}_{p,\infty})}\mathcal{D}_{p,\sigma_1}+ \mathcal{D}_{p,\sigma_1}^2 \Big),
\end{aligned}
\end{equation}
where $\alpha^*$ and $\sigma_1^*$  are defined by  \eqref{5.4} and \eqref{sigma222}, respectively.
\end{lemma}

\begin{proof}
Applying Duhamel's principle to \eqref{deltav} yields that
\begin{equation}\label{duhamel}
\begin{aligned}
\delta V=\int_{0}^{t} \mathcal{G}(t-\tau)R^{0} d\tau+\sum_{i=1}^{d}\int_{0}^{t} \mathcal{G}(t-\tau)\partial_{x_{i}}R^{i} d\tau.
\end{aligned}
\end{equation}
Following essentially the same argument as \eqref{Duhanel1}-\eqref{PVj}, we arrive at
\begin{equation}\label{555}
\begin{aligned}
\|\dot{\Delta}_{j}\mathbf{P}\delta V\|_{L^{p}}&\lesssim
\int_{0}^{t} e^{-R_{1}2^{2j}(t-\tau)} 2^{j} \sum_{i=0}^{d}\|\dot{\Delta}_{j}R^{i}\|_{L^{p}}\,d\tau,\quad j\leq J_0.
\end{aligned}
\end{equation}
Since the nonlinear terms $R^{i}$ ($i=0,1,2,\dots,d$) defined in \eqref{R0Ri} are quadratic, there exists a $n\times n$ functional matrix $S^{i}(V)$ depending on $V$ smoothly such that $S^{i}(0)=0$
and $R^i=S^{i}(V) V$.
To handle the integral on the right-hand side of \eqref{555}, we decompose them in terms of the linear part and the difference part of solutions:
\begin{equation}\label{RiS}
\begin{aligned}
R^i=S^{i}(V) V=\mathcal{I}^i_L+\delta \mathcal{I}^i_1+\delta \mathcal{I}^i_2
\end{aligned}
\end{equation}
with
\begin{equation}
\begin{aligned}
(\text{double linear part})\quad \mathcal{I}^i_L&\coloneqq   S^i(V_L)V_L ,\\
(\text{error part with $\{\mathbf{I} - \mathbf{P}\}$})\quad \delta \mathcal{I}_1^i&\coloneqq  (S^i(V)-S^i(\mathbf{P} V) )\delta V+(S^i(V)-S^i(V_L))\{\mathbf{I} - \mathbf{P}\} V_L \\
&\quad+ S^i(\mathbf{P}V) \{\mathbf{I} - \mathbf{P}\} \delta V,\\
(\text{error part with $\mathbf{P}$})\quad \delta \mathcal{I}_2^i&\coloneqq  (S^i(V)-S^i(V_L)) \mathbf{P} V_L +S^i(\mathbf{P}V)\mathbf{P}\delta V.
\end{aligned}
\end{equation}
Multiplying \eqref{555} by $2^{\sigma j}$, taking the summation over $j\leq J_0$ and applying \eqref{opoo}, we get
\begin{equation}\label{fgggggg}
\begin{aligned}
\|\mathbf{P}\delta V(t)\|_{\dot{B}^{\sigma}_{p,1}}^{\ell}&\lesssim\int_{0}^{t} \langle t-\tau\rangle^{-\frac{1}{2}(\sigma-\sigma_{1}+\sigma_1^*)} \sum_{i=0}^d\|(\mathcal{I}^i_L,\delta\mathcal{I}^i_{1},\delta\mathcal{I}_{2}^i)\|_{\dot{B}^{\sigma_{1}+1-\sigma_1^*}_{p,\infty}}^{\ell}\, d\tau
\end{aligned}
\end{equation}
for $\sigma>\sigma_1$ and $0<\sigma_1^*\leq 1$.
For bounding the integral with respect to time in \eqref{fgggggg},  we  proceed differently depending on whether
$\sigma_1<\sigma\leq \frac{d}{p}-1$ and
$\frac{d}{p}-1<\sigma\leq \frac{d}{p}+1-\alpha^*$.

\begin{itemize}
\item \emph{Case 1: $\sigma_1<\sigma\leq \frac{d}{p}-1$ and $-\frac{d}{p}\leq \sigma_1<\frac{d}{p}-2$.}
\end{itemize}
In that case, we set  $\sigma_1^*=1$, which ensures the faster $\langle t\rangle^{-\frac{1}{2}}$-rate in \eqref{fgggggg}. 
The nonlinear terms $\mathcal{I}_{L}^i$, $\delta\mathcal{I}_1^i$ and  $\delta\mathcal{I}_2^i$ on the right-hand side of \eqref{fgggggg} are analyzed separately as follows. It follows from  \eqref{preciselinear1} with $\sigma_2=\sigma_1$, \eqref{product3} and \eqref{F1} that
\begin{equation}\label{mmfgg}
\begin{aligned}
\|\mathcal{I}_{L}^i\|_{\dot{B}^{\sigma_{1}}_{p,\infty}}&\lesssim \|S(V_L)\|_{\dot{B}^{\frac{d}{p}-1-\varepsilon}_{p,1}}\|V_L\|_{\dot{B}^{\sigma_1+1+\varepsilon}_{p,\infty}}\lesssim \|V_L\|_{\dot{B}^{\frac{d}{p}-1-\varepsilon}_{p,1}} \|V_L\|_{\dot{B}^{\sigma_1+1+\varepsilon}_{p,1}}\lesssim \langle t\rangle^{-\frac{1}{2}(\frac{d}{p}-\sigma_1)} \mathcal{D}_{p,\sigma_1}^2,
\end{aligned}
\end{equation}
where we considered $0<\varepsilon\ll\frac{d}{p}-2-\sigma_1$ such that $\sigma_1+1<\frac{d}{p}-1-\varepsilon$.
Then, we turn to bound each term in $\delta \mathcal{I}_{1}^i$.

Thanks to the estimate \eqref{F11} for the composite function $S^i(V)-S^i(V_L)$ and the faster rate of $\{\mathbf{I} - \mathbf{P}\}V_L$ in \eqref{preciselinear1} with $\sigma_2=\sigma_1$, we obtain
\begin{equation*}
\begin{aligned}
\|(S^i(V)-S^i(V_L))\{\mathbf{I} - \mathbf{P}\}V_L\|_{\dot{B}^{\sigma_{1}}_{p,\infty}}&\lesssim \|S^i(V)-S^i(V_L)\|_{\dot{B}^{\sigma_1+1}_{p,\infty}}\|\{\mathbf{I} - \mathbf{P}\} V_L\|_{\dot{B}^{\frac{d}{p}-1}_{p,1}}\\
&\lesssim \|\delta V\|_{\dot{B}^{\sigma_1+1}_{p,\infty}} \|\{\mathbf{I} - \mathbf{P}\} V_L\|_{\dot{B}^{\frac{d}{p}-1}_{p,1}}\\
&\lesssim  \|\delta V\|_{\dot{B}^{\sigma_1+1}_{p,\infty}} \mathcal{D}_{p,\sigma_1}\langle t\rangle^{-\frac{1}{2}(\frac{d}{p}-\sigma_1)},
\end{aligned}
\end{equation*}
where  \eqref{product3} was used. Similarly, it follows from \eqref{Dtdecay}, \eqref{product3} and \eqref{F11} that
\begin{equation*}
\begin{aligned}
\|S^i(\mathbf{P}V)\{\mathbf{I} - \mathbf{P}\}\delta V\|_{\dot{B}^{\sigma_{1}}_{p,\infty}}&\lesssim \|\mathbf{P}V\|_{\dot{B}^{\sigma_1+1}_{p,\infty}} \|\{\mathbf{I} - \mathbf{P}\}\delta V\|_{\dot{B}^{\frac{d}{p}-1}_{p,1}}\lesssim \|\mathbf{P}V\|_{\dot{B}^{\sigma_1+1}_{p,\infty}} \mathcal{D}_{p}(t)\langle t\rangle^{-\frac{1}{2}(\frac{d}{p}-\sigma_1+1)} .
\end{aligned}
\end{equation*}
The estimate of $(S^i(V)-S^i(\mathbf{P}V))\delta V$  requires the decomposition that
 $V-\mathbf{P}V=\{\mathbf{I} - \mathbf{P}\}V=\{\mathbf{I} - \mathbf{P}\}(V_L+\delta V)$. We get
\begin{equation*}
\begin{aligned}
\|(S^i(V)-S^i(\mathbf{P}V))\delta V\|_{\dot{B}^{\sigma_{1}}_{p,\infty}}
&\lesssim \Big(\|\{\mathbf{I} - \mathbf{P}\}V_L\|_{\dot{B}^{\frac{d}{p}-1}_{p,1}} +\|\{\mathbf{I} - \mathbf{P}\}\delta V\|_{\dot{B}^{\frac{d}{p}-1}_{p,1}} \Big)\|\delta V\|_{\dot{B}^{\sigma_1+1}_{p,\infty}}\\
&\lesssim \|\delta V\|_{\dot{B}^{\sigma_1+1}_{p,\infty}}(\mathcal{D}_{p}(t)+\mathcal{D}_{p,\sigma_1})\langle t\rangle^{-\frac{1}{2}(\frac{d}{p}-\sigma_1)}.
\end{aligned}
\end{equation*}
Hence, we conclude that
\begin{equation}\label{I1i}
\begin{aligned}
\|\delta \mathcal{I}_1^i\|_{\dot{B}^{\sigma_{1}}_{p,\infty}}&\lesssim  \Big(\|\mathbf{P}V\|_{\dot{B}^{\sigma_1+1}_{p,\infty}}+\|\delta V\|_{\dot{B}^{\sigma_1+1}_{p,\infty}})\mathcal{D}_{p}(t)+\|\delta V\|_{\dot{B}^{\sigma_1+1}_{p,\infty}} \mathcal{D}_{p,\sigma_1}\Big)\langle t\rangle^{-\frac{1}{2}(\frac{d}{p}-\sigma_1)}.
\end{aligned}
\end{equation}
Now, arguing as for proving \eqref{I1i}, we easily get the estimate for $\delta \mathcal{I}_2^i$,
\begin{equation}\label{I2i}
\begin{aligned}
\|\delta \mathcal{I}_2^i\|_{\dot{B}^{\sigma_{1}}_{p,\infty}}&\lesssim  \|\delta V\|_{\dot{B}^{\frac{d}{p}-1}_{p,1}} \|(\mathbf{P}V_L,\mathbf{P}V)\|_{\dot{B}^{\sigma_1+1}_{p,\infty}}\lesssim  \|(\mathbf{P}V_L,\mathbf{P}V)\|_{\dot{B}^{\sigma_1+1}_{p,\infty}}\mathcal{D}_{p}(t) \langle t\rangle^{-\frac{1}{2}(\frac{d}{p}-\sigma_1)}.
\end{aligned}
\end{equation}
Furthermore, one can employ \eqref{preciselinear1} with $\sigma_2=\sigma_1+\frac{1}{2}$ to get
\begin{equation}
\begin{aligned}
\|\mathbf{P}V_L\|_{\dot{B}^{\sigma_1+1}_{p,\infty}}&\lesssim \langle t\rangle^{-\frac{1}{4}}\mathcal{D}_{p,\sigma_1+\frac{1}{2}}.
\end{aligned}
\end{equation}
Notice that $\frac{1}{2}(\frac{d}{p}-\sigma_1)>1$ and $\frac{1}{2}(\sigma-\sigma_1+1)\leq \frac{1}{2}(\frac{d}{p}-\sigma_1)$ for $\sigma_{1}<\sigma\leq \frac{d}{p}-1$. Plugging the above estimates \eqref{mmfgg}-\eqref{I2i} for $\mathcal{I}_{L}$, $\delta \mathcal{I}_1^i$ and $\delta \mathcal{I}_2^i$ into \eqref{fgggggg},
we end up with, thanks to \eqref{ineq},
\begin{equation}\label{fg1}
\begin{aligned}
\|\mathbf{P}\delta V\|_{\dot{B}^{\sigma}_{p,1}}^{\ell}
\lesssim&\int_{0}^{t} \langle t-\tau\rangle^{-\frac{1}{2}(\sigma-\sigma_{1}+1)} \|(\mathcal{I}_{L}, \delta \mathcal{I}_1^i, \delta \mathcal{I}_2^i)\|_{\dot{B}^{\sigma_1}_{p,\infty}}\,d\tau\\
\lesssim & \langle t\rangle^{-\frac{1}{2}(\sigma-\sigma_1+1)}\bigg( \Big(\mathcal{D}_{p,\sigma_1+\frac{1}{2}}+\|\mathbf{P}V\|_{\dot{B}^{\sigma_1+1}_{p,\infty}}+\|\delta V\|_{\dot{B}^{\sigma_1+1}_{p,\infty}}\Big)\mathcal{D}_{p}(t)+ \|\delta V\|_{\dot{B}^{\sigma_1+1}_{p,\infty}} \mathcal{D}_{p,\sigma_1} +\mathcal{D}_{p,\sigma_1}^2\bigg)
\end{aligned}
\end{equation}
for  $\sigma_{1}<\sigma\leq \frac{d}{p}-1$ and $-\frac{d}{p}\leq \sigma_1<\frac{d}{p}-2$.

\begin{itemize}
\item \emph{Case 2: $\sigma_1<\sigma\leq \frac{d}{p}-1$ and $\frac{d}{p}-2\leq \sigma_1<\frac{d}{p}-1$.}
\end{itemize}
Since $\sigma_1$ is close to $\frac{d}{p}-1$, we may not expect the best $t^{-\frac{1}{2}}$-rate for the error $\delta V$. 
In fact, we take $\sigma_1^*=\alpha^*=\frac{d}{p}-1-\sigma_1-\varepsilon_1$ and transfer the $(1-\alpha^*)$-order regularity to deal with those nonlinear terms in \eqref{fgggggg} involving $\mathcal{I}_{L}^i$ and $\delta\mathcal{I}_1^i$, which allows us to ensure the sufficient decay on nonlinear terms and then obtain the $t^{-\frac{1}{2}(\frac{d}{p}-1-\sigma_1-\varepsilon_1)}$ -rate for $\delta V$.

Given that  $\frac{d}{p}-\alpha^*-\frac{\varepsilon_1}{2}=\sigma_1+1+\frac{\varepsilon_1}{2}>1+\sigma_1$, we perform the product law \eqref{product3}, \eqref{F1} and \eqref{preciselinear1} to obtain
\begin{equation*}
\begin{aligned}
\|\mathcal{I}_{L}^i\|_{\dot{B}^{\sigma_{1}+1-\alpha^*}_{p,\infty}}&\lesssim
\|S^i(V_L)\|_{\dot{B}^{\frac{d}{p}-\alpha^*-\frac{\varepsilon_1}{2}}_{p,1}} \| V_L\|_{\dot{B}^{\sigma_{1}+1+\frac{\varepsilon_1}{2}}_{p,\infty}}\\
&\lesssim \| V_L\|_{\dot{B}^{\frac{d}{p}-\alpha^*-\frac{\varepsilon_1}{2}}_{p,1}}  \| V_L\|_{\dot{B}^{\sigma_{1}+1+\frac{\varepsilon_1}{2}}_{p,1}}\lesssim \langle t\rangle^{-\frac{1}{2}\left(\frac{d}{p}-\sigma_1+1-\alpha^*\right)}\mathcal{D}_{p,\sigma_1}^2.
\end{aligned}
\end{equation*}
Arguing similarly as for proving \eqref{I1i}, we get
\begin{equation*}
\begin{aligned}
\|\delta \mathcal{I}^i_1\|_{\dot{B}^{\sigma_{1}+1-\alpha^*}_{p,\infty}}&\lesssim \|(S^i(V)-S^i(V_L))\{\mathbf{I} - \mathbf{P}\}V_L\|_{\dot{B}^{\sigma_{1}+1-\alpha^*}_{p,\infty}}+\|S^i(\mathbf{P}V)\{\mathbf{I} - \mathbf{P}\}\delta V\|_{\dot{B}^{\sigma_{1}+1-\alpha^*}_{p,\infty}}\\
&\quad+\|(S^i(V)-S^i(\mathbf{P}V))\delta V\|_{\dot{B}^{\sigma_{1}+1-\alpha^*}_{p,\infty}}\\
&\lesssim \|\delta V\|_{\dot{B}^{\sigma_1+1}_{p,\infty}}\|\{\mathbf{I} - \mathbf{P}\} V_L\|_{\dot{B}^{\frac{d}{p}-\alpha^*}_{p,1}}+\|\mathbf{P}V\|_{\dot{B}^{\sigma_1+1}_{p,\infty}} \|\{\mathbf{I} - \mathbf{P}\}\delta V\|_{\dot{B}^{\frac{d}{p}-\alpha^*}_{p,1}}\\
&\quad+\Big(\|\{\mathbf{I} - \mathbf{P}\}V_L\|_{\dot{B}^{\frac{d}{p}-\alpha^*}_{p,1}} +\|\{\mathbf{I} - \mathbf{P}\}\delta V\|_{\dot{B}^{\frac{d}{p}-\alpha^*}_{p,1}}\Big)\|\delta V\|_{\dot{B}^{\sigma_1+1}_{p,\infty}}\\
&\lesssim \|\delta V\|_{\dot{B}^{\sigma_1+1}_{p,\infty}}(\mathcal{D}_{p}(t)+\mathcal{D}_{p,\sigma_1})\bigg(\langle t\rangle^{-\frac{1}{2}\left(\frac{d}{p}-\sigma_1+1-\alpha^*\right)}+\langle t\rangle^{-\frac{1}{2}\left(\frac{d}{p}-\sigma_1+1-\frac{1}{2}\alpha^*\right)}\bigg).
\end{aligned}
\end{equation*}
To handle the term involving $\delta \mathcal{I}_{2}^i$, we may use the decomposition $V=\mathbf{P}V+\{\mathbf{I}-\mathbf{P}\}V$ and derive the following three terms:
\begin{equation*}
\begin{aligned}
\|\delta \mathcal{I}_2^i\|_{\dot{B}^{\sigma_{1}+1-\alpha^*}_{p,\infty}}&\lesssim \|(S^i(V)-S^i(V_L))\mathbf{P}V_L\|_{\dot{B}^{\sigma_{1}+1-\alpha^*}_{p,\infty}}+\|S^i(\mathbf{P}V_L)\mathbf{P}\delta V\|_{\dot{B}^{\sigma_{1}+1-\alpha^*}_{p,\infty}}\\
&\quad+\|(S^i(\mathbf{P}V)-S^i(\mathbf{P}V_L))\mathbf{P}\delta V\|_{\dot{B}^{\sigma_{1}+1-\alpha^*}_{p,\infty}}.
\end{aligned}
\end{equation*}
It follows from the standard product laws that
\begin{align*}
\|(S^i(V)-S^i(V_L))\mathbf{P}V_L\|_{\dot{B}^{\sigma_{1}+1-\alpha^*}_{p,\infty}}&\lesssim \|\delta V\|_{\dot{B}^{\frac{d}{p}-\frac{1}{2}\alpha^*}_{p,1}} \|\mathbf{P}V_L\|_{\dot{B}^{\sigma_1+1-\frac{1}{2}\alpha^*}_{p,\infty}},\\
\|S^i(\mathbf{P}V_L)\mathbf{P}\delta V\|_{\dot{B}^{\sigma_{1}+1-\alpha^*}_{p,\infty}}&\lesssim \|\mathbf{P}\delta V\|_{\dot{B}^{\frac{d}{p}-\frac{1}{2}\alpha^*}_{p,1}} \|\mathbf{P}V_L\|_{\dot{B}^{\sigma_1+1-\frac{1}{2}\alpha^*}_{p,\infty}},\\
\|(S^i(\mathbf{P}V)-S^i(\mathbf{P}V_L))\mathbf{P}\delta V\|_{\dot{B}^{\sigma_{1}+1-\alpha^*}_{p,\infty}}&\lesssim \|\mathbf{P}\delta V\|_{\dot{B}^{\frac{1}{2}\left(\sigma_1+1+\frac{d}{p}-\alpha^*\right)}_{p,1}}^2.
\end{align*}
Note that  $\|\delta V\|_{\dot{B}^{\frac{d}{p}-\frac{1}{2}\alpha^*}_{p,1}}$ decays at the rate $\langle t\rangle^{-\frac{1}{2}(\frac{d}{p}-\sigma_1)}$  (see $\eqref{Dtdecay}$), which may not be sufficient in this case. To ensure enough decay, we exploit \eqref{preciselinear1} with $\sigma_2=\sigma_1+\frac{1}{2}\alpha^*$ to obtain
$$
\|\mathbf{P}V_L\|_{\dot{B}^{\sigma_1+1-\frac{1}{2}\alpha^*}_{p,1}}\lesssim \langle t\rangle^{-\frac{1}{2}(1-\alpha^*)}\mathcal{D}_{p,\sigma_1+\frac{1}{2}\alpha^*}.
$$
In addition, the interpolation inequality \eqref{inter}, together with \eqref{Dtdecay}, implies
$$
\|\mathbf{P}\delta V\|_{\dot{B}^{\frac{1}{2}(\sigma_1+1+\frac{d}{p}-\alpha^*)}_{p,1}}^2\leq \|\mathbf{P}\delta V\|_{\dot{B}^{\frac{d}{p}+1-\alpha^*}_{p,\infty}} \|\mathbf{P}\delta V\|_{\dot{B}^{\sigma_1+1}_{p,\infty}}\lesssim \langle t\rangle^{-\frac{1}{2}(\frac{d}{p}-\sigma_1+1-\frac{1}{2}\alpha^*)} \|\delta V\|_{\dot{B}^{\sigma_1+1}_{p,\infty}}.
$$
Therefore, we have
\begin{equation*}
\begin{aligned}
\|\delta \mathcal{I}^i_2\|_{\dot{B}^{\sigma_{1}+1-\alpha^*}_{p,\infty}}&\lesssim\Big(  \mathcal{D}_{p,\sigma_1+\frac{1}{2}\alpha^*}+\|\delta V\|_{\dot{B}^{\sigma_1+1}_{p,\infty}} \Big) \mathcal{D}_{p}(t) \bigg(\langle t\rangle^{-\frac{1}{2}\left(\frac{d}{p}-\sigma_1+1-\alpha^*\right)}+\langle t\rangle^{-\frac{1}{2}\left(\frac{d}{p}-\sigma_1+1-\frac{1}{2}\alpha^*\right)}\bigg).
\end{aligned}
\end{equation*}
Owing to \eqref{ineq}, it is clear that
\begin{equation}\nonumber
\begin{aligned}
&\int_{0}^{t} \langle t-\tau\rangle^{-\frac{1}{2}(\sigma-\sigma_{1}+\alpha^*)} \langle \tau \rangle^{-\frac{1}{2}\left(\frac{d}{p}-\sigma_{1}+1-\alpha^*\right)}\,d\tau\lesssim  \langle t\rangle^{-\frac{1}{2}(\sigma-\sigma_1+\alpha^*)},
\end{aligned}
\end{equation}
since $\frac{1}{2}(\sigma-\sigma_1+\alpha^*)\leq \frac{1}{2}(\frac{d}{p}-\sigma_1)$ for $\sigma\leq \frac{d}{p}-1$ and $\frac{1}{2}\left(\frac{d}{p}-\sigma_1+1-\alpha^*\right)>1$. Therefore, collecting the above estimates for $\mathcal{I}^i_L$, $\delta \mathcal{I}^i_1$ and $\delta \mathcal{I}^i_2$, it follows  from \eqref{fgggggg} that
\begin{equation}\label{fg2}
\begin{aligned}
\|\mathbf{P}\delta V\|_{\dot{B}^{\sigma}_{p,1}}^{\ell}&\lesssim \langle t\rangle^{-\frac{1}{2}(\sigma-\sigma_1+\alpha^*)}\bigg( \Big(\mathcal{D}_{p,\sigma_1+\frac{1}{2}\alpha^*}+\|\mathbf{P}V\|_{\dot{B}^{\sigma_1+\frac{1}{2}\alpha^*}_{p,\infty}}+\|\delta V\|_{\dot{B}^{\sigma_1+1}_{p,\infty}}\Big)\mathcal{D}_{p}(t)\\
&\quad+\mathcal{D}_{p,\sigma_1}\|\delta V\|_{\dot{B}^{\sigma_1+1}_{p,\infty}}+\mathcal{D}_{p,\sigma_1}^2\bigg).
\end{aligned}
\end{equation}

\begin{itemize}
\item \emph{Case 3: $\frac{d}{p}-1<\sigma\leq \frac{d}{p}+1-\alpha^*$.}
\end{itemize}
Let us end the error estimate for $\mathbf{P}\delta V$
for the last case where we take $\sigma_1^*=\frac{1}{2}\alpha^*=\frac{1}{2}\min\Big\{1,\frac{d}{p}-1-\sigma_1-\varepsilon_1  \Big\}$ with $0<\varepsilon_1<\frac{d}{p}-1-\sigma_1$. First, it follows  from \eqref{product2} and \eqref{preciselinear1} with  $\frac{d}{p}-\alpha^*>\sigma_1+1$ that
\begin{equation*}
\begin{aligned}
\|\mathcal{I}_{L}^i\|_{\dot{B}^{\sigma_{1}+1-\frac{1}{2}\alpha^*}_{p,\infty}}&\lesssim  \| V_L\|_{\dot{B}^{\sigma_{1}+1+\frac{1}{2}\alpha^*}_{p,\infty}}  \| V_L\|_{\dot{B}^{\frac{d}{p}-\alpha^*}_{p,1}}\lesssim \langle t\rangle^{-\frac{1}{2}\left(\frac{d}{p}-\sigma_1+1-\frac{1}{2}\alpha^*\right)}\mathcal{D}_{p,\sigma_1}^2.
\end{aligned}
\end{equation*}
In light of \eqref{preciselinear1} and the definition of $\mathcal{D}_p(t)$, we arrive at
\begin{equation*}
\begin{aligned}
\|\delta \mathcal{I}^i_1\|_{\dot{B}^{\sigma_{1}+1-\frac{1}{2}\alpha^*}_{p,\infty}}&\lesssim \|(S^i(V)-S^i(V_L))\{\mathbf{I} - \mathbf{P}\}V_L\|_{\dot{B}^{\sigma_{1}+1-\frac{1}{2}\alpha^*}_{p,\infty}}+\|S^i(\mathbf{P}V)\{\mathbf{I} - \mathbf{P}\}\delta V\|_{\dot{B}^{\sigma_{1}+1-\frac{1}{2}\alpha^*}_{p,\infty}}\\
&\quad+\|(S^i(V)-S^i(\mathbf{P}V))\delta V\|_{\dot{B}^{\sigma_{1}+1-\frac{1}{2}\alpha^*}_{p,\infty}}\\
&\lesssim \|\delta V\|_{\dot{B}^{\sigma_1+1}_{p,\infty}}\|\{\mathbf{I} - \mathbf{P}\} V_L\|_{\dot{B}^{\frac{d}{p}-\frac{1}{2}\alpha^*}_{p,1}}+\|\mathbf{P}V\|_{\dot{B}^{\sigma_1+1}_{p,\infty}} \|\{\mathbf{I} - \mathbf{P}\}\delta V\|_{\dot{B}^{\frac{d}{p}-\frac{1}{2}\alpha^*}_{p,1}}\\
&\quad+\Big(\|\{\mathbf{I} - \mathbf{P}\}V_L\|_{\dot{B}^{\frac{d}{p}-\frac{1}{2}\alpha^*}_{p,1}} +\|\{\mathbf{I} - \mathbf{P}\}\delta V\|_{\dot{B}^{\frac{d}{p}-\frac{1}{2}\alpha^*}_{p,1}}\Big)\|\delta V\|_{\dot{B}^{\sigma_1+1}_{p,\infty}}\\
&\lesssim \|\delta V\|_{\dot{B}^{\sigma_1+1}_{p,\infty}}(\mathcal{D}_{p}(t)+\mathcal{D}_{p,\sigma_1})\langle t\rangle^{-\frac{1}{2}\left(\frac{d}{p}-\sigma_1+1-\alpha^*\right)},
\end{aligned}
\end{equation*}
where we used the fact that
$$
\|\{\mathbf{I} - \mathbf{P}\}\delta V\|_{\dot{B}^{\frac{d}{p}-\frac{1}{2}\alpha^*}_{p,1}}\lesssim \|\{\mathbf{I} - \mathbf{P}\}\delta V\|_{\dot{B}^{\frac{d}{p}-\alpha^*}_{p,1}}^{\ell}+\|\delta V\|_{\dot{B}^{\frac{d}{2}+1}_{2,1}}^h\lesssim \mathcal{D}_{p}(t) \langle t\rangle^{-\frac{1}{2}\left(\frac{d}{p}-\sigma_1+1-\frac{1}{2}\alpha^*\right)} .
$$
Bounding $\delta \mathcal{I}_2^i$
requires a more elaborate analysis. Precisely,
\begin{equation*}
\begin{aligned}
\|\delta \mathcal{I}_2^i\|_{\dot{B}^{\sigma_{1}+1-\frac{1}{2}\alpha^*}_{p,\infty}}&\lesssim \|(S^i(V)-S^i(V_L))\mathbf{P}V_L\|_{\dot{B}^{\sigma_{1}+1-\frac{1}{2}\alpha^*}_{p,\infty}}+\|S^i(\mathbf{P}V_L)\mathbf{P}\delta V\|_{\dot{B}^{{\sigma_{1}+1}-\frac{1}{2}\alpha^*}_{p,\infty}}\\
&\quad+\|(S^i(\mathbf{P}V)-S^i(\mathbf{P}V_L))\mathbf{P}\delta V\|_{\dot{B}^{{\sigma_{1}+1}-\frac{1}{2}\alpha^*}_{p,\infty}}.
\end{aligned}
\end{equation*}
In the case that
$-\frac{d}{p}\leq\sigma_1<\frac{d}{p}-2$, we take $\alpha^*=1$. It follows from \eqref{product3}, \eqref{F1} and \eqref{Dtdecay} that
\begin{equation*}
\begin{aligned}
&\quad\|(S^i(V)-S^i(V_L))\mathbf{P}V_L\|_{\dot{B}^{{\sigma_{1}+1}-\frac{1}{2}\alpha^*}_{p,\infty}}+\|S^i(\mathbf{P}V_L)\mathbf{P}\delta V\|_{\dot{B}^{{\sigma_{1}+1}-\frac{1}{2}\alpha^*}_{p,\infty}}\\
& \lesssim \|\delta V\|_{\dot{B}^{\frac{d}{p}}_{p,1}}\|\mathbf{P}V_L\|_{\dot{B}^{\sigma_1+1-\frac{1}{2}\alpha^*}_{p,\infty}}\lesssim       \langle t\rangle^{-\frac{1}{2}\left(\frac{d}{p}-\sigma_1+\frac{1}{2}\alpha^*\right)}\|\mathbf{P}V_L\|_{\dot{B}^{\sigma_1+1-\frac{1}{2}\alpha^*}_{p,1}}\mathcal{D}_{p}(t),
\end{aligned}
\end{equation*}
which gives the desired decay rate.
However, in the case that $\frac{d}{p}-2\leq \sigma_1<\frac{d}{p}-1$, we take $\alpha^*=\frac{d}{p}-1-\sigma_1-\varepsilon_1<1$, which implies that   $\frac{1}{2}\alpha^*<1-\frac{1}{2}\alpha^*$. Hence, we need to get some additional decay from $\|\mathbf{P}V_L\|_{\dot{B}^{\sigma_1+1-\frac{1}{2}\alpha^*}_{p,1}}$.
Indeed, it follows from \eqref{preciselinear1} with $\sigma_2=\sigma_1+\frac{1}{2}\alpha^*$ that
$$
\|\mathbf{P}V_L\|_{\dot{B}^{\sigma_1+1-\frac{1}{2}\alpha^*}_{p,1}}\lesssim \langle t\rangle^{-\frac{1}{2}(1-\alpha^*)}\mathcal{D}_{p,\sigma_1+\frac{1}{2}\alpha^*}.
$$
To deal with $\|(S^i(\mathbf{P}V)-S^i(\mathbf{P}V_L))\mathbf{P}\delta V\|_{\dot{B}^{\frac{d}{p}-\frac{1}{2}\alpha^*}_{p,1}}$,
using the product estimate \eqref{product2} and the interpolation inequality \eqref{inter}, we have
\begin{equation*}
\begin{aligned}
\|(S^i(\mathbf{P}V)-S^i(\mathbf{P}V_L))\mathbf{P}\delta V\|_{\dot{B}^{\sigma_1+1-\frac{1}{2}\alpha^*}_{p,\infty}}&\lesssim \|\mathbf{P}\delta V\|_{\dot{B}^{\frac{1}{2}(\frac{d}{p}+\sigma_1+1-\frac{1}{2}\alpha^*)}_{p,\infty}}^2\\
&\lesssim \|\mathbf{P}\delta V\|_{\dot{B}^{\sigma_1+\frac{1}{2}\alpha^*}_{p,1}} \|\mathbf{P}\delta V\|_{\dot{B}^{\frac{d}{p}+1-\alpha^*}_{p,1}}\\
 &\lesssim \langle t\rangle^{-\frac{1}{2}\left(\frac{d}{p}-\sigma_1+1-\frac{1}{2}\alpha^*\right)}\|\mathbf{P} \delta V\|_{\dot{B}^{\sigma_1+\frac{1}{2}\alpha^*}_{p,\infty}} \mathcal{D}_p(t).
\end{aligned}
\end{equation*}
Putting together all the above estimates involving
 $\mathcal{I}^i_L$, $\delta \mathcal{I}_1^i$ and $\delta \mathcal{I}_2^i$,
we can obtain
\begin{equation}\label{fgggggg000}
\begin{aligned}
\|\mathbf{P}\delta V(t)\|_{\dot{B}^{\sigma}_{p,1}}^{\ell}&\lesssim\int_{0}^{t} \langle t-\tau\rangle^{-\frac{1}{2}\left(\sigma-\sigma_{1}+\frac{1}{2}\alpha^*\right)} \langle \tau\rangle^{-\frac{1}{2}\left(\frac{d}{p}-\sigma_1+1-\frac{1}{2}\alpha^*\right)} \, d\tau\,\bigg(\mathcal{D}_{p,\sigma_1}^2+\mathcal{D}_{p,\sigma_1}\|\delta V\|_{\dot{B}^{\sigma_1+1}_{p,\infty}} \\
&\quad +\Big(\mathcal{D}_{p,\sigma_1+\frac{1}{2}\alpha^*}+\|(\mathbf{P} V,\mathbf{P}\delta V)\|_{\dot{B}^{\sigma_1+\frac{1}{2}\alpha^*}_{p,\infty}}+\|\delta V\|_{\dot{B}^{\sigma_1+1}_{p,\infty}}\Big) \mathcal{D}_{p}(t) \bigg)
\end{aligned}
\end{equation}
for $\frac{d}{p}-1<\sigma\leq \frac{d}{p}+1-\alpha^*$.

Keep in mind that $\sigma-\sigma_{1}+\frac{1}{2}\alpha^*\leq \frac{d}{p}-\sigma_{1}+1-\frac{1}{2}\alpha^*$ and $\frac{1}{2}\left(\frac{d}{p}-\sigma_{1}+1-\frac{1}{2}\alpha^*\right)>1$ for $\frac{d}{p}-1<\sigma<\frac{d}{p}+1-\alpha^*$.
It is shown by \eqref{ineq} that
\begin{equation*}
\begin{aligned}
 &\int_{0}^{t}\langle t-\tau\rangle^{-\frac{1}{2}\left(\sigma-\sigma_{1}+\frac{1}{2}\alpha^*\right)}\langle\tau\rangle^{-\frac{1}{2}\left(\frac{d}{p}-\sigma_{1}+1-\frac{1}{2}\alpha^*\right)}d\tau\,\lesssim  \langle t\rangle^{-\frac{1}{2}\left(\sigma-\sigma_{1}+\frac{1}{2}\alpha^*\right)}.
\end{aligned}
\end{equation*}
Therefore, we deduce that
\begin{equation}\label{fg3}
\begin{aligned}
&\quad\|\mathbf{P}\delta V(t)\|_{\dot{B}^{\sigma}_{p,1}}^{\ell}\lesssim\langle t\rangle^{-\frac{1}{2}\left(\sigma-\sigma_{1}+\frac{1}{2}\alpha^*\right)}\bigg(\mathcal{D}_{p,\sigma_1}\|\delta V\|_{\dot{B}^{\sigma_1+1}_{p,\infty}} +\mathcal{D}_{p,\sigma_1}^2\\
&\hspace{30mm} +\Big(\mathcal{D}_{p,\sigma_1+\frac{1}{2}\alpha^*}+\|(\mathbf{P} V,\mathbf{P}\delta V)\|_{\dot{B}^{\sigma_1+\frac{1}{2}\alpha^*}_{p,\infty}}+\|\delta V\|_{\dot{B}^{\sigma_1+1}_{p,\infty}}\Big) \mathcal{D}_{p}(t) \bigg).
\end{aligned}
\end{equation}

Combining those estimates \eqref{fg1}-\eqref{fg3}, 
we  conclude that \eqref{PV1} holds. Hence, the proof of Lemma \ref{lemmalow} is complete.
\end{proof}

Similarly, we have the time-weighted decay estimate for $\{\mathbf{I} - \mathbf{P}\}\delta V$ in low frequencies.
\begin{lemma}
Let $\sigma_1+1<\sigma'\leq  \frac{d}{p}-\alpha^*$. It holds that
\begin{equation}\label{PV2}
\begin{aligned}
\|\{\mathbf{I} - \mathbf{P}\}\delta V\|_{\dot{B}^{\sigma'}_{p,1}}^{\ell}&\lesssim \langle t\rangle^{-\frac{1}{2}(\sigma'-\sigma_1+1+\sigma_2^*)}\Big(\mathcal{D}_{p,\sigma_1+\frac{1}{2}\alpha^*}+\|(\mathbf{P} V,\mathbf{P}\delta V)\|_{\dot{B}^{\sigma_1+\frac{1}{2}\alpha^*}_{p,\infty}}+\|\delta V\|_{L^{\infty}_{t}(\dot{B}^{\sigma_1+1}_{p,\infty})}\Big)\mathcal{D}_{p}(t)\\
&\quad+\langle t\rangle^{-\frac{1}{2}(\sigma'-\sigma_1+1+\sigma_2^*)}\Big(\mathcal{D}_{p,\sigma_1}^2+\mathcal{D}_{p,\sigma_1}\|\delta V\|_{L^{\infty}_{t}(\dot{B}^{\sigma_1+1}_{p,\infty})}\Big),
\end{aligned}
\end{equation}
where $\alpha^*$ and $\sigma_2^*$  are defined by  \eqref{5.4} and \eqref{sigma222}, respectively.
\end{lemma}

\begin{proof}
In view of Proposition \ref{LemmaspectrallocalHp} and $\mathbf{P}R^0=0$, $\dot{\Delta}_j\{\mathbf{I} - \mathbf{P}\}\delta V$ satisfies
\begin{equation}\label{deltaJdeltaIPV}
\begin{aligned}
\|\dot{\Delta}_{j}\{\mathbf{I} - \mathbf{P}\}\delta V\|_{L^{p}}&\lesssim
\int_{0}^{t} \Big(e^{-R_{1}2^{2j}(t-\tau)} 2^{2j}+e^{-\kappa_1 (t-\tau)}\Big)\sum_{i=0}^{d}\|\dot{\Delta}_{j}R^{i}\|_{L^{p}}\,d\tau.
\end{aligned}
\end{equation}
Thus, we have
\begin{equation}\label{fgggggg1}
\begin{aligned}
\|\{\mathbf{I}-\mathbf{P}\}\delta V(t)\|_{\dot{B}^{\sigma'}_{p,1}}^{\ell}&\lesssim
\int_{0}^{t} \langle t-\tau\rangle^{-\frac{1}{2}(\sigma'-\sigma_{1}+1+\sigma_2^*)}  \sum_{i=0}^d\|(\mathcal{I}^i_L,\delta\mathcal{I}^i_{1},\delta\mathcal{I}_{2}^i)\|_{\dot{B}^{\sigma_{1}+1-\sigma_2^*}_{p,\infty}}^{\ell}\, d\tau\\
&\quad+\int_{0}^{t} e^{-\kappa_1 (t-\tau)} \sum_{i=0}^d\|(\mathcal{I}^i_L,\delta\mathcal{I}^i_{1},\delta\mathcal{I}^i_{2})\|_{\dot{B}^{\sigma'}_{p,1}}^{\ell}\,d\tau.
\end{aligned}
\end{equation}
Following from the similar procedure leading to \eqref{PV1}, we can deduce that
\begin{equation}\label{I1I2I3d}
\begin{aligned}
\|(\mathcal{I}^i_L,\delta\mathcal{I}^i_1,\delta\mathcal{I}^i_2)\|_{\dot{B}^{\sigma_1+1-\sigma_2^*}_{p,\infty}}^{\ell}
&\lesssim \bigg( \mathcal{D}_{p,\sigma_1}^2+\mathcal{D}_{p,\sigma_1}\|\delta V\|_{\dot{B}^{\sigma_1+1}_{p,\infty}}\\
&\quad\quad+\Big(\mathcal{D}_{p,\sigma_1+\frac{1}{2}\alpha^*}+\|(\mathbf{P} V,\mathbf{P}\delta V)\|_{\dot{B}^{\sigma_1+\frac{1}{2}\alpha^*}_{p,\infty}}+\|\delta V\|_{\dot{B}^{\sigma_1+1}_{p,\infty}}\Big)\mathcal{D}_{p}(t)\bigg)\\
&\quad\times \begin{cases}
\langle t\rangle^{-\frac{1}{2}\left(\frac{d}{p}-\sigma_{1}\right)},
& \mbox{\quad if ~~$\sigma_1+1<\sigma'\leq \frac{d}{p}-2$},\\
\langle t\rangle^{-\frac{1}{2}\left(\frac{d}{p}-\sigma_{1}+1-\frac{1}{2}\alpha^*\right)},
& \mbox{\quad if ~~$\frac{d}{p}-2<\sigma'\leq \frac{d}{p}-\alpha^*$},
\end{cases}
\end{aligned}
\end{equation}
which enables us to obtain
\begin{equation}\label{Expod}
\begin{aligned}
&\quad\int_{0}^{t} \langle t-\tau\rangle^{-\frac{1}{2}(\sigma'-\sigma_{1}+1+\sigma_2^*)} \sum_{i=0}^d\|(\mathcal{I}^i_L,\delta\mathcal{I}^i_{1},\delta\mathcal{I}_{2}^i)\|_{\dot{B}^{\sigma_{1}+1-\sigma_2^*}_{p,\infty}}^{\ell}\, d\tau\\
&\lesssim  \langle t\rangle^{-\frac{1}{2}(\sigma'-\sigma_{1}+1+\sigma_2^*)}\bigg( \mathcal{D}_{p,\sigma_1}^2+\mathcal{D}_{p,\sigma_1}\|\delta V\|_{L^{\infty}_{t}(\dot{B}^{\sigma_1+1}_{p,\infty})}\\
&\quad+\Big(\mathcal{D}_{p,\sigma_1+\frac{1}{2}\alpha^*}+\|(\mathbf{P} V,\mathbf{P}\delta V)\|_{L^{\infty}_{t}(\dot{B}^{\sigma_1+\frac{1}{2}\alpha^*}_{p,\infty})}+\|\delta V\|_{L^{\infty}_{t}(\dot{B}^{\sigma_1+1}_{p,\infty})}\Big)\mathcal{D}_{p}(t)\bigg).
\end{aligned}
\end{equation}
On the other hand, observe that $\sigma'>\sigma_1+1>\sigma_1+1-\sigma_2^*$, we infer that the second integration on the right-hand side of \eqref{fgggggg1} can be bounded by
\begin{equation}\label{Expod2}
\begin{aligned}
&\quad\int_{0}^{t} e^{-\kappa_1 (t-\tau)}\|(\mathcal{I}^i_L,\delta\mathcal{I}^i_1,\delta\mathcal{I}^i_2)\|_{\dot{B}^{\sigma'}_{p,1}}^{\ell}\,d\tau\\
&\lesssim\int_{0}^{t} e^{-\kappa_1 (t-\tau)} \|(\mathcal{I}^i_L,\delta\mathcal{I}^i_1,\delta\mathcal{I}^i_2)\|_{\dot{B}^{\sigma_1+1-\sigma_2^*}_{p,\infty}}^{\ell}\,d\tau\\
&\lesssim  \bigg( \mathcal{D}_{p,\sigma_1}^2+\mathcal{D}_{p,\sigma_1}\|V\|_{L^{\infty}_{t}(\dot{B}^{\sigma_1+1}_{p,\infty})}+\Big(\mathcal{D}_{p,\sigma_1+\frac{1}{2}\alpha^*}+\|(\mathbf{P} V,\mathbf{P}\delta V)\|_{L^{\infty}_{t}(\dot{B}^{\sigma_1+\frac{1}{2}\alpha^*}_{p,\infty})}+\|\delta V\|_{L^{\infty}_{t}(\dot{B}^{\sigma_1+1}_{p,\infty})}\Big)\mathcal{D}_{p}(t)\bigg)\\
&\quad\quad \times\begin{cases}
\langle t\rangle^{-\frac{1}{2}\left(\frac{d}{p}-\sigma_{1}\right)},
& \mbox{\quad if ~~$\sigma_1+1<\sigma'\leq \frac{d}{p}-2$},\\
\langle t\rangle^{-\frac{1}{2}\left(\frac{d}{p}-\sigma_{1}+1-\frac{1}{2}\alpha^*\right)},
& \mbox{\quad if ~~$\frac{d}{p}-2<\sigma'\leq \frac{d}{p}-\alpha^*$}.
\end{cases}
\end{aligned}
\end{equation}
Thanks to  $\frac{1}{2}(\sigma'-\sigma_{1}+1+\sigma_2^*)\leq \frac{1}{2}(\frac{d}{p}-\sigma_1)$ for $\sigma_1+1<\sigma'\leq \frac{d}{p}-2$ and $\frac{1}{2}(\sigma'-\sigma_{1}+1+\sigma_2^*)\leq \frac{1}{2}\left(\frac{d}{p}-\sigma_1+1-\frac{1}{2}\alpha^*\right)$ for $\frac{d}{p}-2<\sigma'\leq \frac{d}{p}-\alpha^*$, we  arrive at \eqref{PV2} readily with the combination of \eqref{Expod}-\eqref{Expod2}.
\end{proof}

\subsubsection{Bounds for the high frequencies}
To achieve the high-frequency estimates of $\delta V$ in Proposition \ref{properror1}, it is natural to look at the difference system \eqref{deltav}. The problem is that the nonlinear term on the right-hand side of \eqref{deltav} will cause a loss of one derivative. In the critical regularity framework, however, one cannot afford any loss of regularity in  high frequencies.  To overcome the difficulty,  we use the perturbed norm form \eqref{normale} and track the decay exponent for high frequencies according to the definition of $\mathcal{D}_{p}(t)$.
Precisely, we have the following time-weighted estimate of $\delta V$ in high frequencies.
\begin{lemma}\label{lemmahigh}
Under those assumptions of Proposition \ref{properror1}, it holds that
   \begin{equation}\label{high}
    \begin{aligned}
    \|\langle \tau\rangle^{\beta^*}\delta V\|_{\widetilde{L}^{\infty}_{t}(\dot{B}^{\frac{d}{2}+1}_{2,1})}^{h}&\lesssim (1+\mathcal{D}_{p,\sigma_1})\mathcal{D}_{p,\sigma_1}+E_{p,0}  \mathcal{D}_{p}(t)
    \end{aligned}
    \end{equation}
with $\beta^*=\frac{1}{2}(\frac{d}{p}-\sigma_{1}+1-\frac{1}{2}\alpha^*-\varepsilon_1)$.
\end{lemma}
\begin{proof}
Following from the same energy argument
as in Section \ref{global existence}, we can
achieve the high frequency estimate \eqref{sqrtL}. Thus, by employing   Gr\"onwall's inequality to  \eqref{sqrtL} and then taking the limit  as $\zeta\rightarrow0$, we arrive at
    \begin{equation}\label{DjU}
    \begin{aligned}
&\|\dot{\Delta}_{j}\mathcal{W}\|_{L^2}\lesssim e^{C\|V\|_{L^1_t(L^{\infty})}}\bigg(e^{-ct}\|\dot{\Delta}_{j}\mathcal{W}_0\|_{L^2}+\sum_{i=1}^{5}\int_{0}^{t}e^{-c(t-\tau)}  Z_{i,j} d\tau \bigg),
    \end{aligned}
    \end{equation}
    where $j\geq J_0-1$ and $Z_{i,j}$ ($i=1,2,3,4,5$) are given by \eqref{Z12345}. It follows from \eqref{r2}, \eqref{Lipschitz} and \eqref{DjU} that
\begin{equation}\label{highe}
    \begin{aligned}
    \|\langle \tau\rangle^{\beta^*}\mathcal{W}\|_{\widetilde{L}^{\infty}_{t}(\dot{B}^{\frac{d}{2}+1}_{2,1})}^{h}&\lesssim \|\mathcal{W}_0\|_{\dot{B}^{\frac{d}{2}+1}_{2,1}}^{h}+\sum_{j\geq J_0-1} 2^{(\frac{d}{2}+1)j} \sum_{i=1}^5 \sup_{\tau\in[0,t]}\langle \tau\rangle^{\beta^*}\int_{0}^{\tau}e^{-c(\tau-s)}  Z_{i,j} \,ds.
    \end{aligned}
    \end{equation}
Those nonlinear terms on the right-hand side of \eqref{highe} can be estimated similarly to the proof of global existence. Precisely, thanks to Lemma \ref{lemmacommutator} and Corollary \ref{corocom} for $2\leq p\leq \min\{4,\frac{2p}{p-2}\}$, we obtain
    \begin{equation*}
    \begin{aligned}
    &\quad \sum_{j\geq J_0-1} 2^{(\frac{d}{2}+1)j} \sup_{\tau\in[0,t]}\langle \tau\rangle^{\beta^*}\int_{0}^{\tau}e^{-c(\tau-s)} ( Z_{1,j}+Z_{2,j}) ds\\
    & \lesssim \sup_{\tau\in[0,t]}\langle \tau\rangle^{\beta^*}\int_{0}^{\tau}e^{-c(\tau-s)}\langle s\rangle^{-\beta^*}\,ds\,\|\mathcal{W}\|_{\widetilde{L}^{\infty}_{t}(\dot{B}^{\frac{d}{p}+1}_{p,1})}  \|\langle \tau\rangle^{\beta^*}\mathcal{W}\|_{\widetilde{L}^{\infty}_{t}(\dot{\mathbb{B}}^{\frac{d}{p}+1,\frac{d}{2}+1}_{p,2})}\\
    &\lesssim  \|\mathcal{W}\|_{\widetilde{L}^{\infty}_{t}(\dot{\mathbb{B}}^{\frac{d}{p}+1,\frac{d}{2}+1}_{p,2})} \|\langle \tau\rangle^{\beta^*}\mathcal{W}\|_{\widetilde{L}^{\infty}_{t}(\dot{\mathbb{B}}^{\frac{d}{p}+1,\frac{d}{2}+1}_{p,2})} .
    \end{aligned}
    \end{equation*}
Moreover, using the composite estimate in Lemma \ref{corocom}, we have
    \begin{equation*}
    \begin{aligned}
     &\quad\sum_{j\geq J_0-1} 2^{(\frac{d}{2}+1)j} \sup_{\tau\in[0,t]}\langle \tau\rangle^{\beta^*}\int_{0}^{\tau}e^{-c(\tau-s)} Z_{3,j} ds\lesssim \|\mathcal{W}\|_{\widetilde{L}^{\infty}_{t}(\dot{\mathbb{B}}^{\frac{d}{p},\frac{d}{2}+1}_{p,2})}  \|\langle \tau\rangle^{\beta^*}\mathcal{W}\|_{\widetilde{L}^{\infty}_{t}(\dot{\mathbb{B}}^{\frac{d}{p}+1,\frac{d}{2}+1}_{p,2})}.
    \end{aligned}
    \end{equation*}
Taking advantage of Corollaries \ref{corocom} and \ref{coroApp1} for $2\leq p\leq \min\{4,\frac{2p}{p-2}\}$ to bound the nonlinear terms concerning $Z_{4,j}$ and $Z_{5,j}$, we get
    \begin{equation*}
    \begin{aligned}
      &\quad\sum_{j\geq J_0-1} 2^{(\frac{d}{2}+1)j} \sup_{\tau\in[0,t]}\langle \tau\rangle^{\beta^*}\int_{0}^{\tau}e^{-c(\tau-s)} (Z_{4,j}+Z_{5,j}) ds\\
 &\lesssim \sum_{i=1}^d\|A^0(\mathcal{W}+\bar{W})^{-1}A^{i}(\mathcal{W}+\bar{W})-A^0(\bar{W})^{-1}A^{i}(\bar{W})\|_{\widetilde{L}^{\infty}_{t}(\dot{\mathbb{B}}^{\frac{d}{p},\frac{d}{2}}_{p,2})}\|\langle \tau\rangle^{\beta^*} \mathcal{W}\|_{\widetilde{L}^{\infty}_{t}(\dot{\mathbb{B}}^{\frac{d}{p}+1,\frac{d}{2}+1}_{p,2})}\\
      &\quad+\big\|A^0(\bar{W})\big( A^0(\mathcal{W}+\bar{W})^{-1}-A^0(\bar{W})^{-1}\big)\big\|_{\widetilde{L}^{\infty}_{t}(\dot{\mathbb{B}}^{\frac{d}{p},\frac{d}{2}+1}_{p,2})}\| \langle \tau\rangle^{\beta^*}\mathcal{W}\|_{\widetilde{L}^{\infty}_{t}(\dot{\mathbb{B}}^{\frac{d}{p}+1,\frac{d}{2}+1}_{p,2})}\\
      &\lesssim (1+\|\mathcal{W}\|_{\widetilde{L}^{\infty}_{t}(\dot{\mathbb{B}}^{\frac{d}{p},\frac{d}{2}+1}_{p,2})}) \|\mathcal{W}\|_{\widetilde{L}^{\infty}_{t}(\dot{\mathbb{B}}^{\frac{d}{p},\frac{d}{2}+1}_{p,2})} \| \langle \tau\rangle^{\beta^*}\mathcal{W}\|_{\widetilde{L}^{\infty}_{t}(\dot{\mathbb{B}}^{\frac{d}{p}+1,\frac{d}{2}+1}_{p,2})}.
        \end{aligned}
    \end{equation*}
Hence, plugging the above estimates into $Z_{i,j}$ ($i=1,2,3,4,5$) in \eqref{highe}, remembering  \eqref{Vuniform} and arguing similarly to \eqref{420}-\eqref{422}, we end up with
    \begin{equation}\label{highe1}
    \begin{aligned}
 \|\langle \tau\rangle^{\beta^*}\mathcal{W}\|_{\widetilde{L}^{\infty}_{t}(\dot{B}^{\frac{d}{2}+1}_{2,1})}^{h} &\lesssim \|\mathcal{W}_0\|_{\dot{B}^{\frac{d}{2}+1}_{2,1}}^{h}+(1+\|\mathcal{W}\|_{\widetilde{L}^{\infty}_{t}(\dot{\mathbb{B}}^{\frac{d}{p},\frac{d}{2}+1}_{p,2})}) \|\mathcal{W}\|_{\widetilde{L}^{\infty}_{t}(\dot{\mathbb{B}}^{\frac{d}{p},\frac{d}{2}+1}_{p,2})}\|\langle \tau\rangle^{\beta^*}\mathcal{W}\|_{\widetilde{L}^{\infty}_{t}(\dot{\mathbb{B}}^{\frac{d}{p}+1,\frac{d}{2}+1}_{p,2})}\\
   &\lesssim E_{p,0}+E_{p,0}\|\langle \tau\rangle^{\beta^*}\mathcal{W}\|_{\widetilde{L}^{\infty}_{t}(\dot{\mathbb{B}}^{\frac{d}{p}+1,\frac{d}{2}+1}_{p,2})}.
    \end{aligned}
    \end{equation}
  In what follows, we establish several calculus inequalities to deduce the desired high-frequency decay estimate \eqref{high}.
In fact, we claim that
\begin{equation}\label{UV10}
\left\{
\begin{aligned}
&\|\langle \tau\rangle^{\beta^*}\mathcal{W}\|_{\widetilde{L}^{\infty}_{t}(\dot{\mathbb{B}}^{\frac{d}{p}+1,\frac{d}{2}+1}_{p,2})}\lesssim (1+\mathcal{D}_{p,\sigma_1})\mathcal{D}_{p,\sigma_1}+\mathcal{D}_{p}(t),\\
&\|\langle \tau\rangle^{\beta^*}\delta V\|_{\widetilde{L}^{\infty}_{t}(\dot{B}^{\frac{d}{2}+1}_{2,1})}^{h}\lesssim \|\langle \tau\rangle^{\beta^*}\mathcal{W}\|_{\widetilde{L}^{\infty}_{t}(\dot{B}^{\frac{d}{2}+1}_{2,1})}^{h}+\mathcal{D}_{p,\sigma_1}+E_{p,0}\mathcal{D}_{p}(t),
\end{aligned}
\right.
\end{equation}
which provides the link between the time-weighted norms for $\mathcal{W}$ and $\delta V$.

To handle the time weight $\langle t\rangle^{\beta^*}$ in \eqref{UV10}, we make the best use of decay rates of $\delta V$ in the norm $\dot{B}^{\frac{d}{p}+1-\alpha^*}_{p,1}$ and  $V_L$ in the norm $\dot{B}^{\frac{d}{p}+1}_{p,1}$. This requires us to decompose $\mathcal{W}(V)$ and perform elaborate estimates in low and high frequencies. We remember that $\mathcal{W}=\mathcal{W}(V)$ and $V=V(\mathcal{W})$ are smooth functions and satisfy $\mathcal{W}(0)=V(0)=0$. By Taylor's expansion for vector-valued functions, there exists a matrix $g_1(V)$ such that $g_{1}(0)=\nabla\mathcal{W}(0)$ is positive definite and $\mathcal{W}(V)=g_1(V)  V$. Precisely, we consider the decomposition
\begin{equation}\label{UVde}
\begin{aligned}
\mathcal{W}(V)&=g_1(0)  V+\big(g_1(V)-g_1(0)\big)   V\\
&=g_1(0)  V_{L}+g_1(0)   \delta V+\big(g_1(V)-g_1(0)\big)   \delta V+\big(g_1(V)-g_1(V_{L})\big)V_L+\big(g_1(V_{L})-g_1(0)\big)  V_{L}.
\end{aligned}
\end{equation}
Note that the norm $\|V_{L}(t)\|_{\dot{\mathbb{B}}^{\frac{d}{p}+1,\frac{d}{2}+1}_{p,2}}$ can be bounded by $\mathcal{D}_{p,\sigma_1}\langle t\rangle^{-\frac{1}{2}(\frac{d}{p}+1-\sigma_1)}$. Owing to \eqref{preciselinear1} and $\frac{1}{2}(\frac{d}{p}+1-\sigma_1)>\beta^*$, we have
\begin{equation}\label{VL}
\begin{aligned}
\|\langle \tau\rangle^{\beta^*} V_{L}\|_{\widetilde{L}^{\infty}_t(\dot{\mathbb{B}}^{\frac{d}{p}+1,\frac{d}{2}+1}_{p,2})}\lesssim \|\langle \tau\rangle^{\frac{1}{2}(\frac{d}{p}+1-\sigma_1-\varepsilon_1)} V_{L}\|_{L^{\infty}_t(\dot{\mathbb{B}}^{\frac{d}{p}+1-\varepsilon_1,\frac{d}{2}+1}_{p,2})}\lesssim \mathcal{D}_{p,\sigma_1},
\end{aligned}
\end{equation}
 where we used the fact that, for any $s\in\mathbb{R}, s'>0$ and $u^{\ell}\in \dot{B}^{s-s'}_{p,\infty}$,
\begin{align}
    \|u^{\ell}\|_{\widetilde{L}^{\infty}_{t}(\dot{B}^{s}_{p,1})}\lesssim \sum_{j\leq J_0} 2^{s'j} \|u^{\ell}\|_{L^{\infty}_{t}(\dot{B}^{s-s'}_{p,\infty})}\lesssim \|u^{\ell}\|_{L^{\infty}_{t}(\dot{B}^{s-s'}_{p,\infty})}.\label{lowll}
\end{align}
Observe that the low-frequency norm $\|\langle \tau\rangle^{\beta^*} \delta V^\ell\|_{\dot{B}^{\frac{d}{p}+1-\alpha^*}_{p,1}}$ decays at the rate $\langle t\rangle^{-\frac{1}{2}(\frac{d}{p}+1-\sigma_1-\frac{1}{2}\alpha^*)}$, which is faster than $\langle t\rangle^{-\beta^*}$. It follows from \eqref{lowll} and the definition of $\mathcal{D}_{p}(t)$  that
\begin{equation}\label{deltaVw}
\begin{aligned}
\|\langle \tau\rangle^{\beta^*} \delta V\|_{\widetilde{L}^{\infty}_t(\dot{\mathbb{B}}^{\frac{d}{p}+1,\frac{d}{2}+1}_{p,2})}\lesssim  \|\langle \tau\rangle^{\beta^*} \delta V^\ell\|_{L^{\infty}_t(\dot{B}^{\frac{d}{p}+1-\alpha^*}_{p,1})}+\|\langle \tau\rangle^{\beta^*} \delta V\|_{\widetilde{L}^{\infty}_t(\dot{B}^{\frac{d}{2}+1}_{2,1})}^h\lesssim \mathcal{D}_{p}(t).
\end{aligned}
\end{equation}
Furthermore, we employ \eqref{r2}, \eqref{VLuniform},  \eqref{deltaVw}, \eqref{realmoser2}, \eqref{comhybrid}
and derive that
\begin{equation}\label{554}
\begin{aligned}
&\quad\|\langle \tau\rangle^{\beta^*} \big(g_1(V)-g_1(0)\big)  \delta V\|_{\widetilde{L}^{\infty}_t(\dot{\mathbb{B}}^{\frac{d}{p}+1,\frac{d}{2}+1}_{p,2})}\\
&\lesssim \| \big(g_1(V)-g_1(0) \big)\|_{\widetilde{L}^{\infty}_t(\dot{\mathbb{B}}^{\frac{d}{p},\frac{d}{2}+1}_{p,2})}\|\langle \tau\rangle^{\beta^*}\delta V\|_{\widetilde{L}^{\infty}_{t}(\dot{\mathbb{B}}^{\frac{d}{p}+1,\frac{d}{2}+1}_{p,2})}\\
&\quad+\|\langle \tau\rangle^{\beta^*} \big(g_1(V)-g_1(0) \big)\|_{\widetilde{L}^{\infty}_t(\dot{\mathbb{B}}^{\frac{d}{p}+1,\frac{d}{2}+1}_{p,2})}\|\delta V\|_{\widetilde{L}^{\infty}_{t}(\dot{\mathbb{B}}^{\frac{d}{p},\frac{d}{2}+1}_{p,2})}\\
&\lesssim (1+\|V\|_{\widetilde{L}^{\infty}_t(\dot{\mathbb{B}}^{\frac{d}{p},\frac{d}{2}+1}_{p,2})})\bigg(\|V\|_{\widetilde{L}^{\infty}_t(\dot{\mathbb{B}}^{\frac{d}{p},\frac{d}{2}+1}_{p,2})} \|\langle \tau\rangle^{\beta^*}\delta V\|_{\widetilde{L}^{\infty}_{t}(\dot{\mathbb{B}}^{\frac{d}{p}+1,\frac{d}{2}+1}_{p,2})}\\
&\quad+\|\langle \tau\rangle^{\beta^*}(V_L,\delta V)\|_{\widetilde{L}^{\infty}_t(\dot{\mathbb{B}}^{\frac{d}{p}+1,\frac{d}{2}+1}_{p,2})} \|(V,V_L)\|_{\widetilde{L}^{\infty}_{t}(\dot{\mathbb{B}}^{\frac{d}{p},\frac{d}{2}+1}_{p,2})}\bigg)\lesssim E_{p,0}\mathcal{D}_{p}(t)
\end{aligned}
\end{equation}
and
\begin{equation}\label{55500}
\begin{aligned}
&\quad\|\langle \tau\rangle^{\beta^*} (g_1(V)-g_1(V_{L}))  V_{L}\|_{\widetilde{L}^{\infty}_t(\dot{\mathbb{B}}^{\frac{d}{p}+1,\frac{d}{2}+1}_{p,2})}\\
&\lesssim \|\langle \tau\rangle^{\beta^*}V_{L}\|_{\widetilde{L}^{\infty}_{t}(\dot{\mathbb{B}}^{\frac{d}{p}+1,\frac{d}{2}+1}_{p,2})}\|g_1(V)-g_1(0)\|_{\widetilde{L}^{\infty}_t(\dot{\mathbb{B}}^{\frac{d}{p},\frac{d}{2}+1}_{p,2})}\\
&\quad+\|V_{L}\|_{\widetilde{L}^{\infty}_{t}(\dot{\mathbb{B}}^{\frac{d}{p},\frac{d}{2}+1}_{p,2})}\|\langle \tau\rangle^{\beta^*}\big(g_1(V)-g_1(0) \big)\|_{\widetilde{L}^{\infty}_t(\dot{\mathbb{B}}^{\frac{d}{p}+1,\frac{d}{2}+1}_{p,2})}\lesssim E_{p,0}\mathcal{D}_{p}(t).\\
\end{aligned}
\end{equation}
Owing to  \eqref{VLuniform}, \eqref{preciselinear1} and $\frac{1}{2}\beta^*\leq \frac{1}{2}(\frac{d}{p}-\sigma_1-\varepsilon_1)$, the term $\big(g_1(V_L)-g_1(0) \big)  V_{L}$ can be similarly estimated by
\begin{equation}\label{errorVL}
\begin{aligned}
&\quad\|\langle \tau\rangle^{\beta^*} \big{(}g_1(V_{L})-g_1(0) \big{)} V_{L}\|_{\widetilde{L}^{\infty}_t(\dot{\mathbb{B}}^{\frac{d}{p}+1,\frac{d}{2}+1}_{p,2})}\\
&\lesssim \|\langle \tau\rangle^{\frac{1}{2}(\frac{d}{p}-\sigma_1-\varepsilon_1)}V_{L}\|_{\widetilde{L}^{\infty}_{t}(\dot{\mathbb{B}}^{\frac{d}{p}+1,\frac{d}{2}+1}_{p,2})}\|\langle \tau\rangle^{\frac{1}{2}(\frac{d}{p}-\sigma_1-\varepsilon_1)}\big(g_1(V_L)-g_1(0) \big)\|_{\widetilde{L}^{\infty}_t(\dot{\mathbb{B}}^{\frac{d}{p},\frac{d}{2}+1}_{p,2})}\\
&\quad+\|\langle \tau\rangle^{\frac{1}{2}(\frac{d}{p}-\sigma_1-\varepsilon_1)}V_{L}\|_{\widetilde{L}^{\infty}_{t}(\dot{\mathbb{B}}^{\frac{d}{p},\frac{d}{2}+1}_{p,2})}\|\langle \tau\rangle^{\frac{1}{2}(\frac{d}{p}-\sigma_1-\varepsilon_1)}\big(g_1(V_L)-g_1(0) \big)\|_{\widetilde{L}^{\infty}_t(\dot{\mathbb{B}}^{\frac{d}{p}+1,\frac{d}{2}+1}_{p,2})}\\
&\lesssim (1+\|V_{L}\|_{\widetilde{L}^{\infty}_t(\dot{\mathbb{B}}^{\frac{d}{p},\frac{d}{2}+1}_{p,2})}) \|\langle \tau\rangle^{\frac{1}{2}(\frac{d}{p}-\sigma_1-\varepsilon_1)}V_{L}\|_{\widetilde{L}^{\infty}_{t}(\dot{B}^{\frac{d}{p}}_{p,1})}\|\langle \tau\rangle^{\frac{1}{2}(\frac{d}{p}-\sigma_1-\varepsilon_1)}V_{L}\|_{\widetilde{L}^{\infty}_t(\dot{\mathbb{B}}^{\frac{d}{p},\frac{d}{2}+1}_{p,2})}\lesssim \mathcal{D}_{p,\sigma_1}^2.
\end{aligned}
\end{equation}
Therefore, adding up \eqref{VL}-\eqref{errorVL} yields the first inequality in \eqref{UV10}.
On the other hand, we perform the small modification on  \eqref{VL}-\eqref{errorVL} in high frequencies only and obtain
\begin{equation}
\begin{aligned}
\|\langle \tau\rangle^{\beta^*} \delta V\|_{\widetilde{L}^{\infty}_{t}(\dot{B}^{\frac{d}{2}+1}_{2,1})}^{h}
&\lesssim \|\langle \tau\rangle^{\beta^*}\mathcal{W}\|_{\widetilde{L}^{\infty}_{t}(\dot{B}^{\frac{d}{2}+1}_{2,1})}^{h}+\mathcal{D}_{p,\sigma_1}+\mathcal{D}_{p,\sigma_1}^2+E_{p,0}\mathcal{D}_{p}(t),
\end{aligned}
\end{equation}
which is exactly the second inequality in
\eqref{UV10}. Hence, the proof of Lemma \ref{lemmahigh} is eventually complete.
\end{proof}

\noindent
\textbf{Proof of Proposition \ref{properror1}.}
Having Lemmas \ref{lemmalow}-\ref{lemmahigh} in hand, we shall finish the proof of
Proposition \ref{properror1}. It follows from \eqref{PV1}, \eqref{PV2} and \eqref{high} that
    \begin{equation}\label{Derror}
    \begin{aligned}
\mathcal{D}_{p}(t)&\lesssim \Big(\mathcal{D}_{p,\sigma_1+\frac{1}{2}\alpha^*}+\|(\mathbf{P} V,\mathbf{P}\delta V)\|_{L^{\infty}_{t}(\dot{B}^{\sigma_1+\frac{1}{2}\alpha^*}_{p,\infty})}+\| \delta V\|_{L^{\infty}_{t}(\dot{B}^{\sigma_1+1}_{p,\infty})}\Big)\mathcal{D}_{p}(t)\\
&\quad +(1+\| \delta V\|_{L^{\infty}_{t}(\dot{B}^{\sigma_1+1}_{p,\infty})})\mathcal{D}_{p,\sigma_1}+\mathcal{D}_{p,\sigma_1}^2,
    \end{aligned}
    \end{equation}
where we used the smallness of $E_{p,0}$. The norm notation $\mathcal{D}_{p,\sigma}$ and $\alpha^*\in (0,1]$ are given by \eqref{varepsilonLsigma'} and \eqref{5.4}, respectively. To get the time-weighted energy estimate \eqref{errores},
it suffices to show that   $\mathcal{D}_{p,\sigma_1+\frac{1}{2}\alpha^*}+\|\mathbf{P}V\|_{L^{\infty}_{t}(\dot{B}^{\sigma_1+\frac{1}{2}\alpha^*}_{p,\infty})}+\| \delta V\|_{L^{\infty}_{t}(\dot{B}^{\sigma_1+1}_{p,\infty})}$ is suitably small.
In fact, combining with the interpolation inequality \eqref{inter}, $\sigma_1<\sigma_1+\frac{1}{2}\alpha^*<\frac{d}{p}$ and the smallness of $E_{p,0}$, one can see that
$$
    \mathcal{D}_{p,\sigma_1+\frac{1}{2}\alpha^*} \lesssim \mathcal{D}_{p,\sigma_1}^{\theta_1}E_{p,0}^{1-\theta_1}\ll1,
   $$
   where $\theta_1\in(0,1)$ is given by $\sigma_1+\frac{1}{2}\alpha^*=\sigma_1 \theta_1+\frac{d}{p}(1-\theta_1)$.
It follows from \eqref{Vuniform} and \eqref{Etlinear}  that
\begin{align*}
\|\mathbf{P}V\|_{L^{\infty}_{t}(\dot{B}^{\sigma_1+\frac{1}{2}\alpha^*}_{p,\infty})}^h&\lesssim \|V\|_{L^{\infty}_{t}(\dot{B}^{\frac{d}{2}+1}_{2,1})}^h\lesssim E_{p,0}\ll1,\\
\|\mathbf{P}\delta V\|_{\dot{B}^{\sigma_1+\frac{1}{2}\alpha^*}_{p,\infty}}^h+\| \delta V\|_{L^{\infty}_{t}(\dot{B}^{\sigma_1+1}_{p,\infty})}^h&\lesssim \|(V,V_L)\|_{L^{\infty}_{t}(\dot{B}^{\frac{d}{2}+1}_{2,1})}^h\lesssim E_{p,0}\ll1.
\end{align*}

Now, we  {\textbf{claim}} that
    \begin{equation}\label{lowevolution}
    \begin{aligned}
&\|(\mathbf{P}V,\mathbf{P}\delta V )^\ell\|_{L^{\infty}_{t}(\dot{B}^{\sigma_1}_{p,\infty})}+\| (\{\mathbf{I} - \mathbf{P}\}\delta V)^\ell\|_{L^{\infty}_{t}(\dot{B}^{\sigma_1+1-\alpha^*}_{p,\infty})}\lesssim \mathcal{D}_{p,\sigma_1}
    \end{aligned}
    \end{equation}
for all $t>0$. The proof is left to the next section for clarity. If so, let $\theta_2\in(0,1)$ be given by $\sigma_1+1=(\sigma_1+1-\alpha^*)\theta_2+\frac{d}{p}(1-\theta_2)$. Then, employing the interpolation inequality \eqref{inter} once again, we have
       \begin{equation}\nonumber
    \begin{aligned}
    \|(\mathbf{P} V,\mathbf{P}\delta V)^\ell\|_{L^{\infty}_{t}(\dot{B}^{\sigma_1+\frac{1}{2}\alpha^*}_{p,\infty})}&\lesssim \|(\mathbf{P} V,\mathbf{P}\delta V)^\ell\|_{L^{\infty}_{t}(\dot{B}^{\sigma_1}_{p,\infty})}^{\theta_1} \|(\mathbf{P}V,\mathbf{P}\delta V)^\ell\|_{L^{\infty}_{t}(\dot{B}^{\frac{d}{p}-1}_{p,1})}^{1-\theta_1}\lesssim  \mathcal{D}_{p,\sigma_1}^{\theta_1}E_{p,0}^{1-\theta_1}\ll1,\\
   \|(\{\mathbf{I} - \mathbf{P}\}\delta V)^\ell\|_{L^{\infty}_{t}(\dot{B}^{\sigma_1+1}_{p,\infty})}&\lesssim \|(\{\mathbf{I} - \mathbf{P}\}\delta V)^\ell\|_{L^{\infty}_{t}(\dot{B}^{\sigma_1+1-\alpha^*}_{p,\infty})}^{\theta_2}\|\{\mathbf{I} - \mathbf{P}\}(V_L^\ell,V^\ell)\|_{L^{\infty}_{t}(\dot{B}^{\frac{d}{p}}_{p,1})}^{1-\theta_2}
\lesssim \mathcal{D}_{p,\sigma_1}^{\theta_{2}} E_{p,0}^{1-\theta_{2}}\ll1,
    \end{aligned}
    \end{equation}
owing to \eqref{Vuniform}, \eqref{VLuniform} and \eqref{lowevolution}.
Consequently,  one can conclude that \eqref{errores} for all time since $E_{p,0}$
is small enough. The proof of Proposition \ref{properror1} is complete.\hfill $\Box$

\subsubsection{The regularity evolution of low-frequency norm}

The remainder is to prove {\textbf{Claim}} \eqref{lowevolution},
which indicates the regularity evolution of low frequencies of solutions to \eqref{m1}-\eqref{m1d}. Clearly, this fact plays a key role in closing \eqref{errores} in the last section.

\begin{prop}\label{propevolution}
Let $V$ be the global solution to \eqref{m1}-\eqref{m1d} given in Theorem \ref{theorem0}. Under the additional assumptions $(\mathbf{P}V_0)^{\ell}\in\dot{B}^{\sigma_{1}}_{p,\infty}$ and $(\{\mathbf{I} - \mathbf{P}\}V_0)^{\ell}\in\dot{B}^{\sigma_{1}+1}_{p,\infty}$ with  $-\frac{d}{p} \leq \sigma_{1}<\frac{d}{p}-1$, it holds that
\begin{equation}
\begin{aligned}
\|(\mathbf{P}V)^{\ell}\|_{L^{\infty}_{t}(\dot{B}^{\sigma_{1}}_{p,\infty})}+\|(\{\mathbf{I} - \mathbf{P}\}V)^{\ell}\|_{L^{\infty}_{t}(\dot{B}^{\sigma_{1}+1}_{p,\infty})}\lesssim \mathcal{D}_{p,\sigma_1}\label{evolution1}
\end{aligned}
\end{equation}
and
\begin{align}
\| (\mathbf{P}\delta V)^\ell\|_{L^{\infty}_{t}(\dot{B}^{\sigma_1}_{p,\infty})}+\| (\{\mathbf{I} - \mathbf{P}\}\delta V)^\ell\|_{L^{\infty}_{t}(\dot{B}^{\sigma_1+1-\alpha^*}_{p,\infty})}\lesssim \mathcal{D}_{p,\sigma_1}\label{evolution2}
\end{align}
for all $t>0$ where $\alpha^*\in(0,1]$ is given by \eqref{5.4}.
\end{prop}


\begin{proof}
The strategy is inspired by the work \cite{xin1} of Xin and the second author. Set $V^{\ell}=\dot{S}_{J_0-1}V$ with the threshold $J_0$ satisfying $2^{J_0}\leq \lambda_{0}$ for a small constant $\lambda_{0}$ given by Proposition \ref{LemmaspectrallocalHp}.  Performing the same procedure leading to \eqref{localPV} and \eqref{localIPV}, we
obtain
\begin{equation}\label{5311}
\begin{aligned}
&\quad\|(\mathbf{P}V)^{\ell}\|_{\widetilde{L}^{\infty}_{t}(\dot{B}^{\sigma_{1}}_{p,\infty})\cap \widetilde{L}^{1}_{t}(\dot{B}^{\sigma_{1}+2}_{p,\infty})}+\|(\{\mathbf{I} - \mathbf{P}\}V)^{\ell}\|_{\widetilde{L}^{\infty}_{t}(\dot{B}^{\sigma_{1}+1}_{p,\infty})\cap \widetilde{L}^{1}_{t}(\dot{B}^{\sigma_{1}+1}_{p,\infty})}\\
& \lesssim \|(\mathbf{P}V_0)^{\ell}\|_{\dot{B}^{\sigma_1}_{p,\infty}}+\|(\{\mathbf{I} - \mathbf{P}\}V_0)^{\ell}\|_{\dot{B}^{\sigma_1+1}_{p,\infty}}+\|R^0\|_{\widetilde{L}^1_{t}(\dot{B}^{\sigma_1+1}_{p,\infty})}^{\ell}+\sum_{i=1}^{d}\|R^i\|_{\widetilde{L}^1_{t}(\dot{B}^{\sigma_1+1}_{p,\infty})}^{\ell}.
\end{aligned}
\end{equation}
Remembering that $R^i=S^i(V)V$ ($i=0,1,\dots,d$) with  smooth functions $S^i(0)=0$ and combining with \eqref{product3}, \eqref{F11} and  $-\frac{d}{p}\leq \sigma_1+1<\frac{d}{p}$ yields
\begin{equation}\label{Ripinfty}
\begin{aligned}
\sum_{i=0}^{d}\|R^i\|_{\widetilde{L}^1_{t}(\dot{B}^{\sigma_1+1}_{p,\infty})}^{\ell}
&\lesssim
\|V\|_{\widetilde{L}^2_{t}(\dot{B}^{\frac{d}{p}}_{p,1})}\|V\|_{\widetilde{L}^2_{t}(\dot{B}^{\sigma_1+1}_{p,\infty})}\\
&\lesssim \|V\|_{\widetilde{L}^2_{t}(\dot{B}^{\frac{d}{p}}_{p,1})}\Big(\|(\mathbf{P}V)^{\ell}\|_{\widetilde{L}^2_{t}(\dot{B}^{\sigma_1+1}_{p,\infty})}+\|(\{\mathbf{I} - \mathbf{P}\}V)^{\ell}\|_{\widetilde{L}^2_{t}(\dot{B}^{\sigma_1+1}_{p,\infty})}+\|V\|_{\widetilde{L}^2_{t}(\dot{B}^{\frac{d}{2}+1}_{2,1})}^{h}\Big).
\end{aligned}
\end{equation}
The interpolation \eqref{inter} ensures that
\begin{equation}\label{L2small}
\begin{aligned}
&\quad\|(\mathbf{P}V)^{\ell}\|_{\widetilde{L}^2_{t}(\dot{B}^{\sigma_1+1}_{p,\infty})}+\|(\{\mathbf{I} - \mathbf{P}\}V)^{\ell}\|_{\widetilde{L}^2_{t}(\dot{B}^{\sigma_1+1}_{p,\infty})}\\
&\lesssim \|(\mathbf{P}V)^{\ell}\|_{\widetilde{L}^{\infty}_{t}(\dot{B}^{\sigma_{1}}_{p,\infty})\cap \widetilde{L}^{1}_{t}(\dot{B}^{\sigma_{1}+2}_{p,\infty})}+\|(\{\mathbf{I} - \mathbf{P}\}V)^{\ell}\|_{\widetilde{L}^{\infty}_{t}(\dot{B}^{\sigma_{1}+1}_{p,\infty})\cap \widetilde{L}^{1}_{t}(\dot{B}^{\sigma_{1}+1}_{p,\infty})}.
\end{aligned}
\end{equation}
Recall that $V$ satisfies the uniform estimate \eqref{Vuniform}. It follows from H\"older's inequality that
$$
\|V\|_{\widetilde{L}^2_{t}(\dot{B}^{\frac{d}{2}+1}_{2,1})}^{h}\lesssim \bigg( \|V\|_{\widetilde{L}^{\infty}_{t}(\dot{B}^{\frac{d}{2}+1}_{2,1})}^{h}\bigg)^{\frac{1}{2}} \bigg( \|V\|_{L^{1}_{t}(\dot{B}^{\frac{d}{2}+1}_{2,1})}^{h}\bigg)^{\frac{1}{2}} \lesssim E_{p,0}.
$$
As in \eqref{VL2}, we also get
\begin{equation}\label{L2small0}
\begin{aligned}
\|V\|_{\widetilde{L}^2_{t}(\dot{B}^{\frac{d}{p}}_{p,1})}&\lesssim E_{p,0}.
\end{aligned}
\end{equation}
Inserting \eqref{L2small} and \eqref{L2small0} into \eqref{5311}, one can conclude \eqref{evolution1} immediately by the smallness of $E_{p,0}$ and the fact $E_{p,0}\lesssim \mathcal{D}_{p,\sigma_1}$.

Next, let us show \eqref{evolution2} for the regularity evolution of $\delta V$. The regularity estimate of $\mathbf{P}\delta V$ can be derived by similar computations on \eqref{555} as before. Taking the supremum of \eqref{deltaJdeltaIPV} on $[0,t]$ leads to
\begin{equation}\nonumber
\begin{aligned}
\|\dot{\Delta}_{j}(\{\mathbf{I} - \mathbf{P}\}\delta V)^{\ell}\|_{L^{\infty}_t(L^p)}&\lesssim \int_{0}^{t} \Big(e^{-R_{1}2^{2j}(t-\tau)} 2^{2j}+e^{-\kappa_1 (t-\tau)}\Big) \Big(\|\dot{\Delta}_{j} (R^0)^{\ell}\|_{L^{p}}+\sum_{i=1}^{d}\|\dot{\Delta}_{j}(R^{i})^{\ell}\|_{L^{p}}\Big)\,d\tau\\
&\lesssim 2^{2j}\sum_{i=0}^{d}\|\dot{\Delta}_{j}(R^{i})^{\ell}\|_{L^1_t(L^{p})}+\sum_{i=0}^{d}\|\dot{\Delta}_{j}(R^{i})^{\ell}\|_{L^{\infty}_t(L^{p})},
\end{aligned}
\end{equation}
from which we infer that
\begin{equation}\label{ddgggh}
\begin{aligned}
&\|(\{\mathbf{I} - \mathbf{P}\}\delta V)^{\ell}\|_{\widetilde{L}^{\infty}_{t}(\dot{B}^{\sigma_{1}+1-\alpha^*}_{p,\infty})}\lesssim \sum_{i=0}^{d}\|R^i\|_{\widetilde{L}^1_{t}(\dot{B}^{\sigma_1+3-\alpha^*}_{p,\infty})}^{\ell}+\sum_{i=0}^{d}\|R^i\|_{\widetilde{L}^{\infty}_{t}(\dot{B}^{\sigma_1+1-\alpha^*}_{p,\infty})}^{\ell} .
\end{aligned}
\end{equation}
The first term on the right-hand side of \eqref{ddgggh} can be dealt with by \eqref{Ripinfty} in low frequencies. As $-\frac{d}{p}\leq\sigma_1<\frac{d}{p}-2$, i.e. $\alpha^*=1$, it follows from \eqref{product3} and \eqref{F11} that
\begin{equation}\nonumber
\begin{aligned}
\sum_{i=0}^{d}\|R^i\|_{L^{\infty}_{t}(\dot{B}^{\sigma_1}_{p,\infty})}^{\ell}&\lesssim \|V\|_{L^{\infty}_t(\dot{B}^{\frac{d}{p}-1}_{p,1})}\|V\|_{L^{\infty}_t(\dot{B}^{\sigma_1+1}_{p,\infty})}\\
&\lesssim \Big(\|V^\ell\|_{L^{\infty}_t(\dot{B}^{\frac{d}{p}-1}_{p,1})}+ \|V\|_{L^{\infty}_t(\dot{B}^{\frac{d}{2}+1}_{2,1})}^h\Big) \Big(\|V^\ell\|_{L^{\infty}_t(\dot{B}^{\sigma_1+1}_{p,\infty})}+\|V\|_{L^{\infty}_t(\dot{B}^{\frac{d}{2}+1}_{2,1})}^h\Big)\\
&\lesssim \Big(\|V^\ell\|_{L^{\infty}_t(\dot{B}^{\sigma_1+1}_{p,\infty})} +\|V^h\|_{L^{\infty}_t(\dot{B}^{\frac{d}{2}+1}_{2,1})}\Big)^2\lesssim \mathcal{D}_{p,\sigma_1}^2.
\end{aligned}
\end{equation}
As $\frac{d}{p}-2\leq \sigma_1<\frac{d}{p}-1$, i.e. $\alpha^*=\frac{d}{p}-1-\sigma_1-\varepsilon_1$ with $0<\varepsilon_1<\frac{d}{p}-1-\sigma_1$, we have $\frac{d}{p}-\alpha^*=\sigma_1+1+\varepsilon_1>\sigma_1+1$. Thus, it holds that
\begin{equation}\nonumber
\begin{aligned}
\sum_{i=0}^{d}\|R^i\|_{L^{\infty}_{t}(\dot{B}^{\sigma_1+1-\alpha^*}_{p,\infty})}^{\ell}&\lesssim \|V\|_{L^{\infty}_t(\dot{B}^{\frac{d}{p}-\alpha^*}_{p,1})}\|V\|_{L^{\infty}_t(\dot{B}^{\sigma_1+1}_{p,\infty})}\\
&\lesssim \Big(\|V^\ell\|_{L^{\infty}_t(\dot{B}^{\sigma_1+1}_{p,\infty})}^2 +\|V\|_{L^{\infty}_t(\dot{B}^{\frac{d}{2}+1}_{2,1})}^h\Big)^2\lesssim \mathcal{D}_{p,\sigma_1}^2.
\end{aligned}
\end{equation}
Thus, we get \eqref{evolution2} and finish the proof of Proposition \ref{propevolution}.
\end{proof}

\subsubsection{Lower and upper bounds of decay}

The subsection is devoted to the proof
of “if" parts in Theorems \ref{thm1}, \ref{thm2} and \ref{thm3}. Besides
 \eqref{a1}, we additionally assume $(\mathbf{P}V_{0})^{\ell}\in\dot{B}^{\sigma_{1}}_{p,\infty}$ and $(\{\mathbf{I}-\mathbf{P}\}V_{0})^{\ell}\in\dot{B}^{\sigma_{1}+1}_{p,\infty}$ with $-\frac{d}{p}\leq \sigma_{1}<\frac{d}{p}-1$.

\vspace{2mm}

\noindent
\textbf{Proof of “if" part in Theorem \ref{thm1}}.
Recall that the solution $V_{L}$ to the linear problem \eqref{hplinear} has decay estimates \eqref{upperlinear} 
and the difference $\delta V=V-V_{L}$ decays at  faster rates (see \eqref{nonlinearfaster}). 
It follows that
\begin{equation}\label{dxq1}
\begin{aligned}
\|\mathbf{P}V(t)\|_{\dot{\mathbb{B}}^{\sigma,\frac{d}{2}+1}_{p,2}}&\lesssim \|\mathbf{P}V_{L}(t)\|_{\dot{\mathbb{B}}^{\sigma,\frac{d}{2}+1}_{p,2}}+\|(\mathbf{P}\delta V)^{\ell}(t)\|_{\dot{B}^{\sigma}_{p,1}}+\|\delta V(t)\|_{\dot{B}^{\frac{d}{2}+1}_{2,1}}^{h}\\
&\lesssim \langle t\rangle^{-\frac{1}{2}(\sigma-\sigma_1)}+\langle t\rangle^{-\frac{1}{2}(\sigma-\sigma_1+\sigma_1^*)}+\langle t\rangle^{-\beta^*}\\
&\lesssim \langle t\rangle^{-\frac{1}{2}(\sigma-\sigma_1)},\quad \sigma_1<\sigma\leq \frac{d}{p}+1-\alpha^*,
\end{aligned}
\end{equation}
and
\begin{equation}\label{dxq2}
\begin{aligned}
\|\{\mathbf{I} - \mathbf{P}\}V(t)\|_{\dot{\mathbb{B}}^{\sigma',\frac{d}{2}+1}_{p,2}}&\lesssim \|\{\mathbf{I} - \mathbf{P}\}V_{L}(t)\|_{\dot{\mathbb{B}}^{\sigma',\frac{d}{2}+1}_{p,2}}+\|(\{\mathbf{I} - \mathbf{P}\}\delta V)^{\ell}(t)\|_{\dot{B}^{\sigma'}_{p,1}}+ \|\delta V(t)\|_{\dot{B}^{\frac{d}{2}+1}_{2,1}}^{h}\\
&\lesssim \langle t\rangle^{-\frac{1}{2}(\sigma'-\sigma_1+1)}+\langle t\rangle^{-\frac{1}{2}(\sigma'-\sigma_1+\sigma_2^*+1)}+\langle t\rangle^{-\beta^*}\\
&\lesssim \langle t\rangle^{-\frac{1}{2}(\sigma'-\sigma_1+1)},\quad \sigma_1+1<\sigma'\leq \frac{d}{p}-\alpha^*.
\end{aligned}
\end{equation}
Furthermore, one can
improve the above decay estimate  \eqref{dxq1} to remain true for $\frac{d}{p}+1-\alpha^*< \sigma\leq \frac{d}{p}+1$ and \eqref{dxq2} for
 $\frac{d}{p}-\alpha^*\leq \sigma'\leq \frac{d}{p}$. 
We recall
\eqref{PVj}:  \begin{equation}\nonumber
\begin{aligned}
\|\dot{\Delta}_{j}(\mathbf{P} V)^\ell\|_{L^p}&\lesssim e^{-R_{1}2^{2j}t}\Big(\|\dot{\Delta}_{j}(\mathbf{P} V)^\ell\|_{L^p}+2^j\|\dot{\Delta}_{j}(\{\mathbf{I} - \mathbf{P}\} V_0)^\ell\|_{L^p}\Big)\\
&\quad+\int_{0}^{t} e^{-R_{1}2^{2j}(t-\tau)} 2^{j} \sum_{i=0}^{d}\|\dot{\Delta}_{j}R^{i}\|_{L^{p}}\,d\tau,
\end{aligned}
\end{equation}
which leads to
\begin{equation}\nonumber
\begin{aligned}
\|(\mathbf{P} V)^\ell(t)\|_{\dot{B}^{\sigma}_{p,1}}&\lesssim \langle t\rangle^{-\frac{1}{2}(\sigma-\sigma_1)}\Big(\|\mathbf{P}V_0^\ell\|_{\dot{B}^{\sigma_1}_{p,\infty}}+\|\{\mathbf{I} - \mathbf{P}\}V_0^\ell\|_{\dot{B}^{\sigma_1+1}_{p,\infty}}\Big)\\
&\quad+\int_{0}^{t}\langle t-\tau\rangle^{-\frac{1}{2}(\sigma-\sigma_1)} \sum_{i=0}^{d}\|R^i\|_{\dot{B}^{\sigma_1+1}_{p,\infty}}^{\ell}\,d\tau,\quad \sigma_1<\sigma\leq \frac{d}{p}+1.
\end{aligned}
\end{equation}
Note that $R^i=S^i(V)V$ with $S^i(0)=0$ and $V$ decays in $\dot{B}^{\sigma}_{p,1}$ for $\sigma\in (\sigma_1+1,\frac{d}{p}+1-\alpha^*)$ at the rate $\langle t\rangle^{-\frac{1}{2}(\sigma-\sigma_1)}$ due to \eqref{dxq1}-\eqref{dxq2}. Together with \eqref{product3} and  \eqref{F1} for $S^i(V)$, we deduce that
\begin{equation}\nonumber
\begin{aligned}
\sum_{i=0}^{d}\|R^i\|_{\dot{B}^{\sigma_1+1}_{p,\infty}}^{\ell}&\lesssim  \| S^i(V)V \|_{\dot{B}^{\sigma_1+1}_{p,\infty}}\lesssim  \|V\|_{\dot{B}^{\frac{d}{p}-\varepsilon_1}_{p,1}}\|V\|_{\dot{B}^{\sigma_1+1+\varepsilon_1}_{p,\infty}}\lesssim \langle t\rangle^{-\frac{1}{2}(\frac{d}{p}+1-\sigma_1)}
\end{aligned}
\end{equation}
for $0<\varepsilon_1<\frac{d}{p}-1-\sigma_1$.

Consequently, for $\sigma_1<\sigma\leq \frac{d}{p}+1$, owing to $\frac{1}{2}(\sigma-\sigma_1)\leq \frac{1}{2}(\frac{d}{p}+1-\sigma_1)$ and $\frac{1}{2}(\frac{d}{p}+1-\sigma_1)>1$, it follows from  \eqref{ineq} that
\begin{equation}\label{dxq3}
\begin{aligned}
\|(\mathbf{P}V)^{\ell}(t)\|_{\dot{B}^{\sigma}_{p,1}}&\lesssim \langle t\rangle^{-\frac{1}{2}(\sigma-\sigma_1)}+\int_{0}^{t} \langle t-\tau\rangle^{-\frac{1}{2}(\sigma-\sigma_1)} \langle \tau\rangle^{-\frac{1}{2}(\frac{d}{p}+1-\sigma_1)}\,d\tau\lesssim \langle t\rangle^{-\frac{1}{2}(\sigma-\sigma_1)}.
\end{aligned}
\end{equation}
For $\sigma_1+1<\sigma'\leq \frac{d}{p}$, a similar argument enables us to get
\begin{equation}\label{dxq4}
\begin{aligned}
\|(\{\mathbf{I} - \mathbf{P}\}V)^{\ell}(t)\|_{\dot{B}^{\sigma'}_{p,1}}
\lesssim \langle t\rangle^{-\frac{1}{2}(\sigma'-\sigma_1+1)}.
\end{aligned}
\end{equation}

Concerning the decay of the high-frequency norm, following similar computations leading to \eqref{high}, we can obtain
    \begin{equation*}
    \begin{aligned}
\|\langle \tau\rangle^{2\beta_{**}}\mathcal{W}\|_{\widetilde{L}^{\infty}_{t}(\dot{B}^{\frac{d}{2}+1}_{2,1})}^{h}    &\lesssim \|\mathcal{W}_0\|_{\dot{B}^{\frac{d}{2}+1}_{2,1}}^{h}+\Big(1+\|\mathcal{W}\|_{\widetilde{L}^{\infty}_{t}(\dot{\mathbb{B}}^{\frac{d}{p},\frac{d}{2}+1}_{p,2})}\Big) \|\langle \tau\rangle^{\beta_{**}}\mathcal{W}\|_{\widetilde{L}^{\infty}_{t}(\dot{\mathbb{B}}^{\frac{d}{p},\frac{d}{2}+1}_{p,2})}\|\langle \tau\rangle^{\beta_{**}}\mathcal{W}\|_{\widetilde{L}^{\infty}_{t}(\dot{\mathbb{B}}^{\frac{d}{p}+1,\frac{d}{2}+1}_{p,2})}
    \end{aligned}
    \end{equation*}
with $\beta_{**}=\frac{1}{2}(\frac{d}{p}-\sigma_1)-\varepsilon_1$. From \eqref{r2}, \eqref{421} and \eqref{errores}, it holds that
    \begin{equation}\nonumber
    \begin{aligned}
    \|\mathcal{W}\|_{\widetilde{L}^{\infty}_{t}(\dot{\mathbb{B}}^{\frac{d}{p},\frac{d}{2}+1}_{p,2})}\lesssim 1,\quad \|\langle \tau\rangle^{\beta_{**}}\mathcal{W}\|_{\widetilde{L}^{\infty}_{t}(\dot{\mathbb{B}}^{\frac{d}{p}+1,\frac{d}{2}+1}_{p,2})}\lesssim \|\langle \tau\rangle^{\beta_{**}}\mathcal{W}\|_{\widetilde{L}^{\infty}_{t}(\dot{\mathbb{B}}^{\frac{d}{p},\frac{d}{2}+1}_{p,2})}&\lesssim 1.
    \end{aligned}
    \end{equation}
We thus conclude $\|\langle \tau\rangle^{2\beta_{**}}\mathcal{W}\|_{\widetilde{L}^{\infty}_{t}(\dot{B}^{\frac{d}{2}+1}_{2,1})}^{h} \lesssim 1$. Following the same line in \eqref{UV10}, one discovers that $\|\langle \tau\rangle^{2\beta_{**}}V\|_{\widetilde{L}^{\infty}_{t}(\dot{B}^{\frac{d}{2}+1}_{2,1})}^{h}$ is uniformly bounded. Thus, we arrive at
    \begin{equation}\label{dxq5}
    \begin{aligned}
    &\|V(t)\|_{\dot{B}^{\frac{d}{2}+1}_{2,1}}^h\lesssim \langle t\rangle^{-2\beta_{**}}.
   \end{aligned}
\end{equation}
Since $2\beta_{**}=\frac{d}{p}-\sigma_1-2\varepsilon_1>\frac{1}{2}(\frac{d}{p}+1-\sigma_1)$ due to $\sigma_1<\frac{d}{p}-1$ when $0<\varepsilon_1\ll1$, combining this with \eqref{dxq3}-\eqref{dxq5}, we eventually obtain \eqref{decayupper}. \hfill $\Box$

\vspace{2mm}

\noindent
\textbf{Proof of Theorem \ref{thm2}}. Let $t>t_0$ with a fixed time $t_0>0$. Recalling the profile $V^*$ as in \eqref{Vp}, we take advantage of \eqref{asy:linear} and \eqref{nonlinearfaster} to get
\begin{equation*}
\begin{aligned}
\|\mathbf{P}(V-V^*)(t)\|_{\dot{\mathbb{B}}^{\sigma,\frac{d}{2}+1}_{p,2}}&\lesssim \|\mathbf{P}(V_L-V^*)(t)\|_{\dot{\mathbb{B}}^{\sigma,\frac{d}{2}+1}_{p,2}}+\|\mathbf{P}\delta V(t)\|_{\dot{\mathbb{B}}^{\sigma,\frac{d}{2}+1}_{p,2}}
\lesssim \langle t\rangle^{-\frac{1}{2}(\sigma-\sigma_1+\sigma_1^*)}
\end{aligned}
\end{equation*}
for $\sigma_1<\sigma\leq \frac{d}{p}+1-\alpha^*$ and
\begin{equation*}
\begin{aligned}
\|\{\mathbf{I} - \mathbf{P}\}(V-V^*)(t)\|_{\dot{\mathbb{B}}^{\sigma',\frac{d}{2}+1}_{p,2}}&\lesssim \|\{\mathbf{I} - \mathbf{P}\}(V_L-V^*)(t)\|_{\dot{\mathbb{B}}^{\sigma',\frac{d}{2}+1}_{p,2}}+\|\{\mathbf{I} - \mathbf{P}\}\delta V(t)\|_{\dot{\mathbb{B}}^{\sigma',\frac{d}{2}+1}_{p,2}}
\lesssim \langle t\rangle^{-\frac{1}{2}(\sigma'-\sigma_1+1+\sigma_1^*)}
\end{aligned}
\end{equation*}
for $\sigma_1+1<\sigma'\leq \frac{d}{p}-\alpha^*$, which are \eqref{asy:stab} and \eqref{asy:stab1} exactly.
This completes the proof of Theorem \ref{thm2}.\hfill $\Box$

\vspace{2mm}

\noindent
\textbf{Proof of “if" part in Theorem \ref{thm3}}.
With the assumption that $\Psi_0\in\dot{\mathcal{B}}^{\sigma_{1}}_{p,\infty}$ and $(\{\mathbf{I}-\mathbf{P}\}V_{0})^{\ell}\in\dot{B}^{\sigma_{1}+1}_{p,\infty}$, we establish the two-sided decay estimates \eqref{decaytwoside}.
Indeed, we are able to infer that $(\mathbf{P}V_{0})^{\ell}\in \dot{B}^{\sigma_{1}}_{p,\infty}$, since $\Psi_0\in \dot{\mathcal{B}}^{\sigma_{1}}_{p,\infty} $(which is the subset of $\dot{B}^{\sigma_{1}}_{p,\infty}$) and $(\{\mathbf{I}-\mathbf{P}\}V_{0})^{\ell}\in\dot{B}^{\sigma_{1}+1}_{p,\infty}$. Consequently,  the upper bound of decay estimates in \eqref{decaytwoside} follows. To show the lower bound, it follows from  Propositions \ref{propgeneral2} and \ref{properror1} that
\begin{equation}
    \begin{aligned}
    \|(\mathbf{P}V)^{\ell}(t)\|_{\dot{B}^{\sigma}_{p,1}}&\geq \|(\mathbf{P}V_L)^{\ell}(t)\|_{\dot{B}^{\sigma}_{p,1}}-\|\mathbf{P}\delta V(t)\|_{\dot{B}^{\sigma}_{p,1}}^{\ell}\\
    &\geq \frac{1}{C_{0}}\langle t\rangle^{-\frac{1}{2}(\sigma-\sigma_{1})}-C_{0}\langle t\rangle^{-\frac{1}{2}(\sigma-\sigma_{1}+\sigma_1^{*})}\geq \frac{1}{2C_{0}}\langle t\rangle^{-\frac{1}{2}(\sigma-\sigma_{1})}\label{sam2}
    \end{aligned}
\end{equation}
for $\sigma_{1}<\sigma\leq \frac{d}{p}+1-\alpha^*$ and suitably large time  $t> t_{1}$, where $C_{0}>1$ is chosen to be a greater constant if necessary. For the case $\frac{d}{p}+1-\alpha^*\leq \sigma\leq \frac{d}{p}+1$, it follows from the interpolation inequality \eqref{inter} that
\begin{equation}\nonumber
    \begin{aligned}
    \|(\mathbf{P}V)^{\ell}(t)\|_{\dot{B}^{\frac{d}{p}}_{p,1}}\lesssim \|(\mathbf{P}V)^{\ell}(t)\|_{\dot{B}^{\sigma_1+1}_{p,1}}^{\theta}\|(\mathbf{P}V)^{\ell}(t)\|_{\dot{B}^{\sigma}_{p,1}}^{1-\theta}\quad\text{with}\quad \frac{d}{p}=(\sigma_1+1)\theta+\sigma(1-\theta),
    \end{aligned}
\end{equation}
which, together with \eqref{decayupper} and \eqref{sam2}, gives rise to
\begin{equation}\nonumber
    \begin{aligned}
    \|(\mathbf{P}V)^{\ell}(t)\|_{\dot{B}^{\sigma}_{p,1}}&\gtrsim  \|(\mathbf{P}V)^{\ell}(t)\|_{\dot{B}^{\frac{d}{p}}_{p,1}}^{\frac{1}{1-\theta}} \|(\mathbf{P}V)^{\ell}(t)\|_{\dot{B}^{\sigma_1+1}_{p,1}}^{-\frac{\theta}{1-\theta}}\gtrsim  \Big( \langle t\rangle^{-\frac{1}{2}(\frac{d}{p}-\sigma_1)}\Big)^{\frac{1}{1-\theta}} \Big( \langle t\rangle^{\frac{1}{2}}\Big)^{-\frac{\theta}{1-\theta}}\gtrsim \langle t\rangle^{-\frac{1}{2}(\sigma-\sigma_1)}
    \end{aligned}
\end{equation}
  for all $t>t_1$ and $\frac{d}{p}+1-\alpha^*<\sigma\leq \frac{d}{p}+1$. This gives the lower bound of decay estimates for $\mathbf{P}V$.

Due to Propositions \ref{proplinear} and \ref{properror}, one has the lower bounds for $\{\mathbf{I} - \mathbf{P}\}V$. In fact, it follows from \eqref{nonlinearfaster} that
\begin{equation}\label{I-Pfinal}
\begin{aligned}
\|\{\mathbf{I} - \mathbf{P}\}\delta V^{\ell}(t)\|_{\dot{B}^{\sigma'}_{p,1}}&\lesssim \langle t\rangle^{-\frac{1}{2}(\sigma'-\sigma_1+1+\sigma_2^*)},\quad \sigma_{1}+1<\sigma
'\leq \frac{d}{p}-\alpha^*.
\end{aligned}
\end{equation}
Note that the decay rates of $(\{\mathbf{I} - \mathbf{P}\}\delta V)^\ell$ in $\dot{B}^{\sigma'}_{p,1}$ are always faster than $\langle t\rangle^{-\frac{1}{2}(\sigma'-\sigma_1+1)}$ for $\sigma_1+1<\sigma'\leq \frac{d}{p}-\alpha^*$, where the interval $(\sigma_1+1,\frac{d}{p}-\alpha^*)$ is well-defined due to $0<\alpha^*<\frac{d}{p}-1-\sigma_1$. Thus, the combination of \eqref{twosidelinear} and \eqref{I-Pfinal} ensures that
  \begin{equation}
  \begin{aligned}
  \|(\{\mathbf{I} - \mathbf{P}\} V)^\ell(t)\|_{\dot{B}^{\sigma'}_{p,1}}&\geq \frac{1}{C_0} \|(\{\mathbf{I} - \mathbf{P}\}V_L)^\ell(t)\|_{\dot{B}^{\sigma'}_{p,1}}-C_0\|\{\mathbf{I} - \mathbf{P}\}\delta V(t)\|_{\dot{B}^{\sigma'}_{p,1}}^{\ell}\geq \frac{1}{2C_0} \langle t\rangle^{-\frac{1}{2}(\sigma'-\sigma_1+1)}
  \end{aligned}
  \end{equation}
for all $\sigma_1+1<\sigma'<\frac{d}{p}-\alpha^*$ and  suitably large time $t>t_1$.  In the case $\frac{d}{p}-\alpha^*\leq \sigma'\leq \frac{d}{p}$, one can carry out a similar interpolation argument. This completes the proof of the two-sided time-decay bounds \eqref{decaytwoside}. \hfill $\Box$

\subsection{Necessary conditions for decay}\label{sectionnecessary}
The section is devoted to the proof of the ``only if'' part of Theorems \ref{thm1} and \ref{thm2}. That is, if we assume $(\{\mathbf{I} - \mathbf{P}\}V_0)^{\ell}\in \dot{B}^{\sigma_1+1}_{p,\infty}$ with $-\frac{d}{p}\leq \sigma_1<\frac{d}{p}-1$ and \eqref{decayupper} (resp., two-sided bounds \eqref{decaytwoside}), then $(\mathbf{P}V_0)^{\ell}\in  \dot{B}^{\sigma_1}_{p,\infty}$ (resp., $\Psi_0^{\ell}\in\dot{\mathcal{B}}^{\sigma_{1}}_{p,\infty}$) holds true. To achieve this, the crucial ingredient  is to develop the inverse Wiegner’s argument (see \cite{skalak1}) from incompressible Navier-Stokes equations to partially dissipative hyperbolic systems.

\begin{prop}\label{properrorin}
Let $-\frac{d}{p}\leq \sigma_{1}<\frac{d}{p}-1$. If 
the solution $V$ to the Cauchy problem \eqref{m1}-\eqref{m1d} 
satisfies \eqref{decayupper}. Then, for all $t>0$, $\delta V\coloneqq   V-V_{L}$ has faster decay rates:
\begin{equation}\label{Nimprovednonlinear}
\begin{aligned}
\|\mathbf{P}\delta V(t)\|_{\dot{\mathbb{B}}^{\sigma,\frac{d}{2}+1}_{p,2}}&\leq C\langle t\rangle^{-\frac{1}{2}(\sigma-\sigma_{1}+\sigma_1^*)}
\end{aligned}
\end{equation}
for $\sigma_1<\sigma\leq \frac{d}{p}+1-\alpha^*$ and
\begin{equation}\label{Nimprovednonlinear1}
\begin{aligned}
&\|\{\mathbf{I} - \mathbf{P}\}\delta V(t)\|_{\dot{\mathbb{B}}^{\sigma',\frac{d}{2}+1}_{p,2}}\leq C\langle t\rangle^{-\frac{1}{2}(\sigma'-\sigma_{1}+1+\sigma_2^*)}
\end{aligned}
\end{equation}
for $\sigma_1<\sigma'\leq \frac{d}{p}-\alpha^*$. Here, $\sigma_1^*, \sigma_2^*, \alpha^*\in(0,1]$ are defined in \eqref{5.4}-\eqref{sigma222}.
\end{prop}

\begin{proof}
It suffices to estimate the low-frequency part of $\delta V$ and
the proof can be done by a similar procedure leading to Lemma \ref{lemmalow}.
It follows from
\eqref{decayupper}
that
\begin{align}
\|V(t)\|_{\dot{B}^{\sigma}_{p,1}}\lesssim \langle t\rangle^{-\frac{1}{2}(\sigma-\sigma_1)},\quad \sigma_1+1<\sigma\leq \frac{d}{p}.\label{decayupper1}
\end{align}
Taking advantage of \eqref{duhamel} and \eqref{RiS}, for $\sigma>\sigma_{1}$ we have
\begin{equation}\label{PwideV1}
\begin{aligned}
\|\mathbf{P}\delta V(t)\|_{\dot{B}^{\sigma}_{p,1}}^{\ell}\lesssim\int_{0}^{t} \langle t-\tau\rangle^{-\frac{1}{2}(\sigma-\sigma_{1}+\sigma_1^*)} \sum_{i=0}^{d}\| S^i(V)V \|_{\dot{B}^{\sigma_{1}+1-\sigma_1^*}_{p,\infty}}^{\ell} \,d\tau.
\end{aligned}
\end{equation}

To handle the integral on the right-hand side of \eqref{PwideV1}, we divide it into several cases.

\vspace{1mm}

\begin{itemize}
\item \emph{Case 1: $\sigma_1<\sigma\leq \frac{d}{p}-1$ and $-\frac{d}{p}\leq \sigma_1<\frac{d}{p}-2$.}
\end{itemize}


In that case, we choose $\sigma_1^*=\alpha^*=1$.  Taking advantage of  \eqref{decayupper1} and \eqref{product3} and remembering $\sigma_1+1<\frac{d}{p}-1$, we let $0<\varepsilon\ll\frac{d}{p}-2-\sigma_1$ and have
\begin{equation}\nonumber
\begin{aligned}
 \| S^i(V)V\|_{\dot{B}^{\sigma_{1}}_{p,\infty}}&\lesssim \|V\|_{\dot{B}^{\frac{d}{p}-1-\varepsilon}_{p,1}}\|V\|_{\dot{B}^{\sigma_1+1+\varepsilon}_{p,1}}\lesssim \langle t\rangle^{-\frac{1}{2}(\frac{d}{p}-\sigma_1)}.
\end{aligned}
\end{equation}
Putting the above estimate into \eqref{PwideV1} leads to
\begin{equation}\nonumber
\begin{aligned}
\|\mathbf{P}\delta V(t)\|_{\dot{B}^{\sigma}_{p,1}}^{\ell}&\lesssim\int_{0}^{t} \langle t-\tau\rangle^{-\frac{1}{2}(\sigma-\sigma_{1}+1)} \langle \tau \rangle^{-\frac{1}{2}\left(\frac{d}{p}-\sigma_{1}\right)}\,d\tau\lesssim
\langle t\rangle^{-\frac{1}{2}(\sigma-\sigma_{1}+1)}.
\end{aligned}
\end{equation}

\begin{itemize}
\item \emph{Case 2: $\sigma_1<\sigma\leq \frac{d}{p}-1$ and $\frac{d}{p}-2\leq \sigma_1<\frac{d}{p}-1$.}
\end{itemize}
We take $\sigma_1^*=\alpha^*=\frac{d}{p}-1-\sigma_1-\varepsilon_1$ for  $0<\varepsilon_1<\frac{d}{p}-\sigma_1-1$. As $\frac{d}{p}-\alpha^*-\frac{\varepsilon_1}{2}=\sigma_1+1+\frac{\varepsilon_1}{2}>1+\sigma_1$, 
it follows from \eqref{decayupper1} and \eqref{product2} that
\begin{equation}\nonumber
\begin{aligned}
\|S^i(V)V\|_{\dot{B}^{\sigma_{1}+1-\sigma_1^*}_{p,\infty}}&\lesssim \|V\|_{\dot{B}^{\frac{d}{p}-\alpha^*-\frac{\varepsilon_1}{2}}_{p,1}}\|V\|_{\dot{B}^{\sigma_1+1+\frac{\varepsilon_1}{2}}_{p,\infty}} \lesssim \langle t\rangle^{-\frac{1}{2}\left(\frac{d}{p}-\sigma_1+1-\alpha^*\right)}.
\end{aligned}
\end{equation}
Hence, we infer that for $\sigma_1<\sigma\leq \frac{d}{p}-1$,
\begin{equation}\nonumber
\begin{aligned}
\|\mathbf{P}\delta V(t)\|_{\dot{B}^{\sigma}_{p,1}}^{\ell}&\lesssim\int_{0}^{t} \langle t-\tau\rangle^{-\frac{1}{2}(\sigma-\sigma_{1}+\alpha^*)} \langle \tau \rangle^{-\frac{1}{2}(\frac{d}{p}-\sigma_{1}+1-\alpha^*)}\,d\tau\lesssim
\langle t\rangle^{-\frac{1}{2}(\sigma-\sigma_{1}+\alpha^*)},
\end{aligned}
\end{equation}
where we noticed  the fact that $\sigma-\sigma_1+\alpha^*< \frac{d}{p}-\sigma_1+1-\alpha^*$ and $\frac{1}{2}\left(\frac{d}{p}-\sigma_1+1-\alpha^*\right)=1+\frac{\varepsilon_1}{2}>1$.

\begin{itemize}
\item \emph{Case 3: $\frac{d}{p}-1<\sigma\leq \frac{d}{p}+1-\alpha^*$.}
\end{itemize}
We set $\sigma_1^*=\frac{1}{2}\alpha^*=\frac{1}{2}\min\Big\{1,\Big(\frac{d}{p}-1-\sigma_{1}-\varepsilon_1\Big)\Big\}$ for $0<\varepsilon_1<\frac{d}{p}-1-\sigma_{1}$. Similarly, owing to $\frac{d}{p}-\sigma_1^*-\frac{1}{2}\alpha^*>\sigma_1+1$, we have
\begin{equation}\nonumber
\begin{aligned}
\|S^i(V)V\|_{\dot{B}^{\sigma_{1}+1-\sigma_1^*}_{p,\infty}}&\lesssim
\|V\|_{\dot{B}^{\frac{d}{p}-\sigma_1^*-\frac{1}{2}\alpha^*}_{p,1}}\|V\|_{\dot{B}^{\sigma_1+1+\frac{1}{2}\alpha^*}_{p,1}}\lesssim \langle t\rangle^{-\frac{1}{2}(\frac{d}{p}-\sigma_1+1-\sigma_1^*)}.
\end{aligned}
\end{equation}
The choice of $\alpha^*$ indicates that
$\sigma-\sigma_{1}+\frac{1}{2}\alpha^*\leq \frac{d}{p}-\sigma_{1}+1-\frac{1}{2}\alpha^*$ with $\frac{d}{p}-1<\sigma\leq \frac{d}{p}+1-\alpha^*$ and $\frac{1}{2}\left(\frac{d}{p}-\sigma_{1}+1-\frac{1}{2}\alpha^*\right)>1$. Thus, according to \eqref{ineq} and the definition of $\sigma_1^*$, it follows that
\begin{equation}\nonumber
\begin{aligned}
\|\mathbf{P}\delta V(t)\|_{\dot{B}^{\sigma}_{p,1}}^{\ell}&\lesssim\int_{0}^{t} \langle t-\tau\rangle^{-\frac{1}{2}\left(\sigma-\sigma_{1}+\frac{1}{2}\alpha^*\right)} \langle \tau \rangle^{-\frac{1}{2}\left(\frac{d}{p}-\sigma_{1}+1-\frac{1}{2}\alpha^*\right)}\,d\tau\lesssim \langle t\rangle^{-\frac{1}{2}(\sigma-\sigma_1+\frac{1}{2}\alpha^*)}
\end{aligned}
\end{equation}
for $\frac{d}{p}-1<\sigma\leq \frac{d}{p}+1-\alpha^*$.

Thus,  \eqref{Nimprovednonlinear} follows. Similarly, the estimate \eqref{Nimprovednonlinear1} for the dissipative part is valid provided that $\sigma_1^*$ is replaced by $\sigma_2^*$. Indeed,
in light of \eqref{duhamel} and \eqref{RiS}, $\{\mathbf{I} - \mathbf{P}\}\delta V$ satisfies
\begin{equation}\label{PwideV2}
\begin{aligned}
\|\{\mathbf{I}-\mathbf{P}\}\delta V(t)\|_{\dot{B}^{\sigma'}_{p,1}}^{\ell}&\lesssim
\int_{0}^{t} \langle t-\tau\rangle^{-\frac{1}{2}(\sigma'-\sigma_{1}+1+\sigma_2^*)}\| S^i(V)V\|^{\ell}_{\dot{B}^{\sigma_{1}+1-\sigma_2^*}_{p,\infty}} \,d\tau\\
&\quad+\int_0^t e^{-\kappa_1 (t-\tau)}\| S^i(V)V\|^{\ell}_{\dot{B}^{\sigma'}_{p,1}} \,d\tau,\quad \sigma'>\sigma_1+1.
\end{aligned}
\end{equation}
The first integral in \eqref{PwideV2}
stands for the diffusion effect of solutions. Performing similar computations as in Cases $1-3$ gives
\begin{equation}\nonumber
\begin{aligned}
&\| S^i(V)V\|_{\dot{B}^{\sigma_{1}+1-\sigma_2^*}_{p,\infty}} \,\lesssim
\begin{cases}
\langle t\rangle^{-\frac{1}{2}\left(\frac{d}{p}-\sigma_{1}\right)},
& \mbox{\quad if ~~$\sigma_1+1<\sigma'\leq \frac{d}{p}-2$},\\
\langle t\rangle^{-\frac{1}{2}\left(\frac{d}{p}-\sigma_{1}+1-\frac{1}{2}\alpha^*\right)},
& \mbox{\quad if ~~$\frac{d}{p}-2<\sigma'\leq \frac{d}{p}-\alpha^*$}.
\end{cases}
\end{aligned}
\end{equation}
This yields
\begin{equation}\nonumber
\begin{aligned}
&\int_{0}^{t} \langle t-\tau\rangle^{-\frac{1}{2}(\sigma'-\sigma_{1}+1+\sigma_2^*)}\sum_{i=1}^d\| S^i(V)V\|_{\dot{B}^{\sigma_{1}+1-\sigma_2^*}_{p,\infty}} \,d\tau\lesssim \langle t-\tau\rangle^{-\frac{1}{2}(\sigma'-\sigma_{1}+1+\sigma_2^*)}.
\end{aligned}
\end{equation}
The second integral
in \eqref{PwideV2}
characterizes the damping effect. Using the frequency cutoff property, we arrive at
\begin{equation}\nonumber
\begin{aligned}
\int_0^t e^{-\kappa_1 (t-\tau)}\sum_{i=1}^d\| S^i(V)V\|^{\ell}_{\dot{B}^{\sigma'}_{p,1}}\,d\tau&\lesssim \int_0^t e^{-\kappa_1 (t-\tau)}\sum_{i=1}^d\| S^i(V)V\|_{\dot{B}^{\sigma_1+1-\sigma_2^*}_{p,1}}\,d\tau\\
&\lesssim
\begin{cases}
\langle t\rangle^{-\frac{1}{2}\left(\frac{d}{p}-\sigma_{1}\right)},
& \mbox{\quad if ~~$\sigma_1+1<\sigma'\leq \frac{d}{p}-2$},\\
\langle t\rangle^{-\frac{1}{2}\left(\frac{d}{p}-\sigma_{1}+1-\frac{1}{2}\alpha^*\right)},
& \mbox{\quad if ~~$\frac{d}{p}-2<\sigma'\leq \frac{d}{p}-\alpha^*$},
\end{cases}
\end{aligned}
\end{equation}
which is faster than the rates $\langle t\rangle^{-\frac{1}{2}(\sigma'-\sigma_1+1+\frac{1}{2}\alpha^*)}$ for $\sigma_1+1<\sigma'\leq \frac{d}{p}-\alpha^*$. Therefore,  the proof of Proposition \ref{properrorin} is finished.
\end{proof}
Finally, one can show that the linear solution $V_{L}$ to
\eqref{hplinear} has the same decay rates as the global-in-time
solution $V$ to the Cauchy problem \eqref{m1}-\eqref{m1d} constructed in Theorem \ref{theorem0}, which allows us to complete the “only if" part of Theorems \ref{thm1} and \ref{thm3}.

\vspace{2ex}
\noindent
\emph{Proof of “only if" part in Theorem \ref{thm1}}.
According to Proposition \ref{properrorin}, we find that $\delta V$  satisfies the faster decay \eqref{Nimprovednonlinear}-\eqref{Nimprovednonlinear1}. It follows from \eqref{decayupper}
and \eqref{Nimprovednonlinear} that
\begin{equation}\label{sdgbvb}
\begin{aligned}
\|\mathbf{P}V_{L}(t)\|_{\dot{\mathbb{B}}^{\sigma,\frac{d}{2}+1}_{p,2}}&\lesssim \|\mathbf{P}V(t)\|_{\dot{\mathbb{B}}^{\sigma,\frac{d}{2}+1}_{p,2}}+\|\mathbf{P}\delta V(t)\|_{\dot{\mathbb{B}}^{\sigma,\frac{d}{2}+1}_{p,2}}\lesssim  \langle t\rangle^{-\frac{1}{2}(\sigma-\sigma_{1})}
\end{aligned}
\end{equation}
 for all $t>0$ and some fixed $\sigma_{1}<\sigma\leq \frac{d}{p}+1-\alpha^*$. Hence, the upper bound of decay estimates of $\mathbf{P}V_{L}$ implies that $(\mathbf{P}V_0)^{\ell}\in \dot{B}^{\sigma_1}_{p,\infty}$ for $-\frac{d}{p}\leq \sigma_{1}<\frac{d}{p}-1$ with the aid of
Proposition \ref{proplinear}. This finishes the proof of Theorem \ref{thm1}. \hfill $\Box$

\noindent
\emph{Proof of “only if" part in Theorem \ref{thm3}}.
Suppose that \eqref{decaytwoside} holds for $t>t_1$ with some time $t_1>0$. Hence, we have \eqref{sdgbvb} for $t>t_1$. By using the uniform estimate \eqref{r2} in Theorem \ref{theorem0}, we deduce that the upper bound in \eqref{decaytwoside} holds for $0<t\leq t_1$ when $\sigma=\frac{d}{p}-1$. On the other hand, it follows from \eqref{decaytwoside} and \eqref{Nimprovednonlinear} that
\begin{equation}\nonumber
\begin{aligned}
\|\mathbf{P}V_{L}(t)\|_{\dot{\mathbb{B}}^{\sigma,\frac{d}{2}+1}_{p,2}}&\geq   \|\mathbf{P}V(t)\|_{\dot{\mathbb{B}}^{\sigma,\frac{d}{2}+1}_{p,2}}-\|\mathbf{P}\delta V(t)\|_{\dot{\mathbb{B}}^{\sigma,\frac{d}{2}+1}_{p,2}}\geq \frac{1}{C_0} \langle t\rangle^{-\frac{1}{2}(\sigma-\sigma_{1})}-C_0\langle t\rangle^{-\frac{1}{2}(\sigma-\sigma_{1}+\sigma_1^*)}\geq \frac{1}{2C_0}  \langle t\rangle^{-\frac{1}{2}(\sigma-\sigma_{1})}
\end{aligned}
\end{equation}
for suitably large time $t$ and some fixed $\sigma_{1}<\sigma\leq \frac{d}{p}+1-\alpha^*$, where  $C_0>1$ is a constant.  Therefore, applying Proposition \ref{proplinear} once again, we have $\Psi_0^{\ell}\in \dot{\mathcal{B}}^{\sigma_{1}}_{p,\infty}$.  This concludes the proof of Theorem \ref{thm3}.\hfill $\Box$

\appendix

\section{Analysis tools in Besov spaces}\label{appendixA}

For the convenience of  readers, we briefly recall the Littlewood–Paley theory and homogeneous Besov spaces.
We refer to \cite{bahouri1} for the complete explanation.

Let $\chi$ be a smooth radial function supported in $B(0,\frac{4}{3})$
and equal to $1$ on $B(0,\frac{3}{4})$, and set
$\varphi(\xi)=\chi(\xi/2)-\chi(\xi)$ such that $\sum_{j\in\mathbb{Z}}\varphi(2^{-j}\xi)=1$ ($\xi\neq 0$) and  $\operatorname{Supp}\varphi\subset
\bigl\{\tfrac{3}{4}\le|\xi|\le\tfrac{8}{3}\bigr\}.$ The homogeneous dyadic blocks are defined by $\dot{\Delta}_j u
:=\mathcal{F}^{-1}\!\bigl(\varphi(2^{-j}\cdot)\widehat u\bigr)$ and $\dot{S}_j u:=\sum_{k\le j-1}\dot{\Delta}_k u.$ Let $\mathcal{S}'_{h}$ denote the set of tempered distributions $z$ on
$\mathbb{R}^{d}$ such that $\dot{S}_{j}z\to0$ in $\mathcal{S}'$ as $j\to\infty$.
For any $u\in\mathcal{S}'_h$, it is classical to introduce the Littlewood--Paley decomposition $u=\sum_{j\in\mathbb{Z}}\dot{\Delta}_j u$ in $\mathcal{S}'_h$. Then, the homogeneous Besov norm is defined by
\[
\|u\|_{\dot B^s_{p,r}}
:=\bigl\|\{2^{js}\|\dot{\Delta}_j u\|_{L^p}\}_{j\in\mathbb{Z}}\bigr\|_{\ell^r}.
\]
The Besov space $\dot B^s_{p,r}$ consists of all $u\in\mathcal{S}'_h$
 for which this norm is finite. Here and below, all functional spaces are considered on
$\mathbb{R}^{d}$ and we omit the space dependence for simplicity.

\medskip
There are well-known embedding properties in Besov spaces. For example, for any
$s\in\mathbb{R}$, $1\le p_{1}\le p_{2}\le \infty$ and $1\le r_{1}\le r_{2}\le \infty$,
\begin{equation*}
\dot{B}^{s}_{p_{1},r_{1}}
\hookrightarrow
\dot{B}^{s-d\left(\frac{1}{p_{1}}-\frac{1}{p_{2}}\right)}_{p_{2},r_{2}}.
\end{equation*}
Moreover, for any $1\le p\le q\le\infty$,
\begin{equation*}
\dot{B}^{0}_{p,1}\hookrightarrow L^{p}\hookrightarrow \dot{B}^{0}_{p,\infty}
\hookrightarrow \dot{B}^{\sigma}_{q,\infty},
\qquad
\sigma=-d\Bigl(\frac{1}{p}-\frac{1}{q}\Bigr)<0.
\end{equation*}
For any $\sigma\in \mathbb{R}$, the Fourier multiplier $\Lambda^{\sigma}$
with symbol $|\xi|^{\sigma}$ is an isomorphism from $\dot{B}^{s}_{p,r}$ to
$\dot{B}^{s-\sigma}_{p,r}$.

\medskip
To study the dissipative structure in different frequency regimes, we split
Besov norms into low- and high-frequency contributions. For any $s\in\mathbb{R}$
and $1\le p,r\le \infty$, we set
\begin{equation*}
\|u\|_{\dot{B}^{s}_{p,r}}^{\ell}
:= \Big\|\{2^{js}\|\dot{\Delta}_{j}u\|_{L^{p}}\}_{j\leq J_0}\Big\|_{\ell^r},
\qquad
\|u\|_{\dot{B}^{s}_{p,r}}^{h}
:= \Big\|\{2^{js}\|\dot{\Delta}_{j}u\|_{L^{p}}\}_{j\geq J_0-1}\Big\|_{\ell^r},
\end{equation*}
where the integer $J_0$ is the \emph{threshold} between low and high frequencies,
to be chosen in Appendix~\ref{appendixA}.
We also denote the low- and high-frequency parts of $u\in\mathcal{S}'_h$ by
\[
u^{\ell}:=\sum_{j\leq J_0-1}\dot{\Delta}_{j}u,
\qquad
u^{h}:=u-u^{\ell}=\sum_{j\geq J_0 }\dot{\Delta}_{j}u.
\]
Then for any $s'>0$, $
\|u^{\ell}\|_{\dot{B}^{s}_{p,r}}
\lesssim \|u\|_{\dot{B}^{s}_{p,r}}^{\ell}
\lesssim \|u\|_{\dot{B}^{s-s'}_{p,\infty}}^{\ell}$ and
$\|u^{h}\|_{\dot{B}^{s}_{p,1}}
\lesssim \|u\|_{\dot{B}^{s}_{p,r}}^{h}
\lesssim \|u\|_{\dot{B}^{s+s'}_{p,r}}^{h}$.

\medskip
We now introduce a hybrid Besov space with different integrability indices in
the low and high frequency ranges. For $p\ge2$ and $s_1,s_2\in\mathbb{R}$, define
\[
\|u\|_{\dot{\mathbb{B}}^{s_1,s_2}_{p,2}}
:= \|u^{\ell}\|_{\dot{B}^{s_1}_{p,1}}
   + \|u\|_{\dot{B}^{s_2}_{2,1}}^{h},
\]
and let $\dot{\mathbb{B}}^{s_1,s_2}_{p,2}$ be the set of $u\in\mathcal{S}'_h$
such that this norm is finite.

\medskip
We recall the classical complex interpolation inequality: for all
$1\le p,r\le\infty$, $s_1<s_2$, $\theta\in(0,1)$ and
$s=\theta s_1+(1-\theta)s_2$, one has
\begin{equation}\label{inter}
\|u\|_{\dot{B}^{s}_{p,r}}
\lesssim
\|u\|_{\dot{B}^{ s_{1}}_{p,1}}^{\theta}\,
\|u\|_{\dot{B}^{s_{2}}_{p,r}}^{1-\theta}.
\end{equation}

\medskip
The study of non-stationary PDEs requires function spaces of the type
$L^{\varrho}(0,T;X)$ for some Banach space $X$. Since we frequently localize the
equations by Littlewood--Paley decomposition, we use the Chemin--Lerner spaces
$\widetilde{L}^{\varrho}(0,T;X)$ (see \cite{chemin1}).
For $T>0$, $s\in\mathbb{R}$ and $1\le \varrho,p,r\le\infty$, we define
$\widetilde{L}^{\varrho}(0,T;\dot{B}^{s}_{p,r})$ as the subset of
$L^1_{\rm{loc}}(0,T;\mathcal{S}_h')$ under the norm
\[
\|u\|_{\widetilde{L}^{\varrho}_{T}(\dot{B}^{s}_{p,r})}
:= \Big\|\{2^{js}\|\dot{\Delta}_{j}u\|_{L^{\varrho}_{T}(L^{p})}\}_{j\in\mathbb{Z}}\Big\|_{\ell^{r}}.
\]
Furthermore, for $T>0$, $s_1,s_2\in\mathbb{R}$ and $p\ge 2$, we define
$\widetilde{L}^{\varrho}(0,T;\dot{\mathbb{B}}^{s_1,s_2}_{p,2})$ as the completion
of $C([0,T];\mathcal{S})$ under the norm
\[
\|u\|_{\widetilde{L}^{\varrho}_{T}(\dot{\mathbb{B}}^{s_1,s_2}_{p,2})}
:= \|u^{\ell}\|_{\widetilde{L}^{\varrho}_{T}(\dot{B}^{s_1}_{p,1})}
 + \|u\|_{\widetilde{L}^{\varrho}_{T}(\dot{B}^{s_2}_{2,1})}^{h},
\]
where $\|u\|_{\widetilde{L}^{\varrho}_{T}(\dot{B}^{s_2}_{2,1})}^{h}
:= \sum_{j\geq J_0-1} 2^{js_2}\|\dot{\Delta}_j u\|_{L^{\varrho}_{T}(L^2)}.$ Finally, the Chemin--Lerner norm $\widetilde{L}^{\infty}_{T}(\dot{\mathbb{B}}^{s_1,s_2}_{p,2})$
is stronger than the usual Lebesgue--Besov norm
$L^{\infty}_{T}(\dot{\mathbb{B}}^{s_1,s_2}_{p,2})$.

We recall the following classical Bernstein inequality:
\begin{equation}\label{lemma21}
\begin{aligned}
&\operatorname{Supp} \widehat{u} \subset \left\{\xi\in\mathbb{R}^{d}\mid |\xi|\leq \lambda R \right\}
\Longrightarrow
\sum_{|\alpha|=k}\| \partial^\alpha u\|_{L^q}
\lesssim \lambda^{k+d\left(\frac{1}{p}-\frac{1}{q}\right)}\|u\|_{L^p}.
\end{aligned}
\end{equation}
This holds for any tempered distribution $u$, any $k\in\mathbb{N}$, and any
$1\leq p\leq q\leq \infty$.



Below, we recall some useful lemmas that are frequently used in our analysis. First, we need the nonlinear Bernstein's type lemma, which can be shown using a similar argument as in \cite[Lemma 8]{7}.

\begin{lemma}\label{lemmaA2}
Let $-A$ be a strongly elliptic operator of order $2$. If the Fourier transform of $u$ is supported in $\{\xi\in\mathbb{R}^{d}~|~ \lambda r\leq |\xi|\leq \lambda R\}$, then for $1<p<\infty$, there exists a constant $c_p^*>0$ depending only on $d$, $r$, $p$ and $R$ such that
\begin{align*}
-\int_{\mathbb{R}^d}Au |u|^{p-2} u\,dx\geq c_p^* \lambda^2\int_{\mathbb{R}^d}|u|^{p}\,dx.
\end{align*}
\end{lemma}

Furthermore, there are classical product laws that play a fundamental role in handling bilinear terms.
\begin{lemma}\label{prop3.2}
Let $1\leq p,r\leq\infty$. Then
\begin{itemize}
\item For $s>0$, we have
\begin{align}
\|uv\|_{\dot{B}^{s}_{p,r}}&\lesssim \|u\|_{L^{\infty}} \|v\|_{\dot{B}^{s}_{p,r}}+\|u\|_{\dot{B}^{s}_{p,r}}\|v\|_{L^{\infty}}.\label{product1}
\end{align}
\item For  $p\geq 2$ and $s_1,s_2\in \mathbb{R}$ satisfying $ -\frac{d}{p}<s_1, s_2\leq \frac dp$ and $s_{1}+s_{2}>0$, we have
\begin{align}
\|uv\|_{\dot{B}^{s_{1}+s_{2}-\frac dp}_{p,1}}&\lesssim \|u\|_{\dot{B}^{s_1}_{p,1}}\|v\|_{\dot{B}^{s_2}_{p,1}}.\label{product2}
\end{align}
\item For  $p\geq 2$ and $s_1,s_2\in \mathbb{R}$ satisfying $ -\frac{d}{p}<s_1\leq \frac{d}{p}$, $-\frac{d}{p}\leq s_2<\frac dp$ and $s_{1}+s_{2}\geq 0$, we have
\begin{align}
\|uv\|_{\dot{B}^{s_{1}+s_{2}-\frac dp}_{p,\infty}}&\lesssim \|u\|_{\dot{B}^{s_1}_{p,1}}\|v\|_{\dot{B}^{s_2}_{p,\infty}}.\label{product3}
\end{align}
\end{itemize}
\end{lemma}

To analyze those nonlinear terms that arise from the matrices $F^i(U)$ and $G(U)$,  the continuity for the composition of vector-valued functions plays a key role.
\begin{lemma}\label{lemma64}
Let $1\leq p,r\leq \infty$, $s>0$, and $m,n\geq1$.  Assume that $F$: $\mathbb{R}^n\rightarrow \mathbb{R}^m$ is a smooth function with $F(0)=0$. Then for any vector-valued function  $U: \mathbb{R}^d \to \mathbb{R}^n$ with $n\geq1$, if $U: \mathbb{R}^{d}\rightarrow \mathbb{R}^n$ belongs to $\dot{B}^{s}_{p,r}\cap L^{\infty}$, there exists a constant $C_{U}>0$ depending only on $\|U\|_{L^{\infty}}$, $F$, $s$, $p$, $d$, $m$ and $n$ such that
\begin{equation}
\begin{aligned}
\|F(U)\|_{\dot{B}^{s}_{p,r}}\leq C_{U}\|U\|_{\dot{B}^{s}_{p,r}}.\label{F1}
\end{aligned}
\end{equation}
If additionally, $\nabla_{U}F(0)=0$ holds, then we have
\begin{equation}
\begin{aligned}
\|F(U)\|_{\dot{B}^{s}_{p,r}}\leq C_{U}\|U\|_{\dot{B}^{\frac{d}{p}}_{p,1}}\|U\|_{\dot{B}^{s}_{p,r}}.\label{F11}
\end{aligned}
\end{equation}
Furthermore, if $p\geq 2$, $-\frac{d}{p}<s\leq \frac{d}{p}$ and assume the vector-valued functions $U_{1}, U_{2}\in\dot{B}^{s}_{p,r}\cap \dot{B}^{\frac{d}{p}}_{p,1}$, then we have
\begin{equation}
\begin{aligned}
\|F(U_{1})-F(U_{2})\|_{\dot{B}^{s}_{p,1}}\leq C_{U_{1},U_{2}}(1+\|(U_{1},U_{2})\|_{\dot{B}^{\frac{d}{p}}_{p,1}})\|U_{1}-U_{2}\|_{ \dot{B}^{s}_{p,1}},\label{F3}
\end{aligned}
\end{equation}
where the constant $C_{U_{1},U_{2}}>0$ depends only on $\|(U_{1},U_{2})\|_{L^{\infty}}$, $F$, $s$, $p$, $d$, $m$ and $n$.
\end{lemma}

Also, we will employ key estimates for vector-valued composition functions in hybrid spaces. The following lemma can be proved similarly to \cite{xu-zhang} designed for vector-valued functions, combined with Lemma \ref{lemma64}.

\begin{lemma}\label{corocom}
Let $2\leq p\leq \infty$, $0<s\leq \frac{d}{p}+1$, and $m,n\geq1$.  Assume that $F$: $\mathbb{R}^n\rightarrow \mathbb{R}^m$ is a smooth function with $F(0)=0$. Then for any vector-valued function $U\in \mathbb{R}^n$ with $n\geq1$, if $U: \mathbb{R}^{d}\rightarrow \mathbb{R}^n$ satisfies $U\in \dot{\mathbb{B}}^{s,\frac{d}{2}+1}_{p,2}\cap \dot{\mathbb{B}}^{\frac{d}{p},\frac{d}{2}+1}_{p,2}$, there exists a constant $C_{U,J_0}>0$ depending only on $\|U\|_{L^{\infty}}$, $F$, $s$, $p$, $d$, $m$, $n$ and $J_0$ such that
\begin{align}
\|F(U)\|_{\dot{\mathbb{B}}^{s,\frac{d}{2}+1}_{p,2}}\leq C_{U,J_0} (1+\|U\|_{\dot{\mathbb{B}}^{\frac{d}{p},\frac{d}{2}+1}_{p,2}}) \|U\|_{\dot{\mathbb{B}}^{s,\frac{d}{2}+1}_{p,2}}.\label{comhybrid}
\end{align}
Furthermore, if, in addition, $D_{U}F(0)=0$, it  holds
\begin{equation}\label{LemNewSmoothhigh21}
\begin{aligned}
\|F(U)\|_{\dot{\mathbb{B}}^{s,\frac{d}{2}+1}_{p,2}}
&\leq C_{U,J_0} \|U\|_{\dot{\mathbb{B}}^{\frac{d}{p},\frac{d}{2}+1}_{p,2}} \|U\|_{\dot{\mathbb{B}}^{s,\frac{d}{2}+1}_{p,2}}.
\end{aligned}
\end{equation}
\end{lemma}

Since our functional setting is somewhat different compared to \cite{c3}, we need non-standard product laws and commutator estimates to handle nonlinear terms
in hybrid spaces of $L^p$-$L^2$ type. 

\begin{lemma}\label{NonClassicalProLaw1}
Let $s>0$, $2\leq p\leq 4$ for $d=1$ and $2\leq p\leq \min\{4,\frac{2d}{d-2}\}$ for $d\geq 2$. Then, there exists a constant $C$ independent of $J_0$ such that
\begin{align}\label{newproduct}
\|uv\|_{\dot{B}_{2,1}^{s}}^{h}
&\leq C\Big(
\|v\|_{\dot{B}_{p,1}^{\frac{d}{p}}}\|u\|_{\dot{\mathbb{B}}^{s+\frac{d}{p}-\frac{d}{2},s}_{p,2}}
+\|u\|_{\dot{B}_{p,1}^{\frac{d}{p}}}\|v\|_{\dot{\mathbb{B}}^{s+\frac{d}{p}-\frac{d}{2},s}_{p,2}} \Big).
\end{align}
\end{lemma}

\begin{proof}
We recall Bony's paraproduct decomposition (see \cite{bahouri1}):
\begin{equation}\nonumber
\begin{aligned}
uv=\dot{T}_{u}v+\dot{R}[u,v]+\dot{T}_{v}u\quad \mbox{with}\quad
\dot{T}_{u}v\coloneqq  \sum_{j'\in \mathbb{Z}} \dot S_{j'-1}u \dot{\Delta}_{j'} v\quad\mbox{and}\quad
\dot{R}[u,v]\coloneqq  \sum_{|j'-j''|\leq 1} \dot\Delta_{j''} u \dot{\Delta}_{j'} v.
\end{aligned}
\end{equation}
We first deal with $\dot{T}_{u}v$ as follows:
\begin{equation}\nonumber
\begin{aligned}
\|\dot{T}_{u} v\|_{\dot{B}_{2,1}^{s}}^{h}
\leq
\sum_{\substack{j\geq J_0-1 \\ |j-j'|\leq 1}}2^{s j} \|\dot S_{j'-1} u \dot{\Delta}_{j} \dot{\Delta}_{j'} v\|_{L^2}
+\sum_{\substack{j\geq J_0-1 \\ |j-j'|\leq 4}} 2^{s j}\|[\dot{\Delta}_{j}, \dot S_{j'-1} u] \dot{\Delta}_{j'} v\|_{L^2}.
\end{aligned}
\end{equation}
It is straightforward to see that
\begin{equation}\nonumber
\begin{aligned}
\sum_{\genfrac{}{}{0pt}{}{j\geq J_0-1}{|j-j'|\leq 1}} 2^{sj}\|\dot S_{j'-1} u \dot{\Delta}_{j} \dot{\Delta}_{j'} v\|_{L^2}
\lesssim
\|u\|_{L^\infty}\sum_{j\geq J_0-1} 2^{sj} \|\dot{\Delta}_{j} v\|_{L^2}
\lesssim
\|u\|_{\dot{B}_{p,1}^{\frac{d}{p}}}\|v\|_{\dot{B}_{2,1}^{s}}^{h}.
\end{aligned}
\end{equation}
The commutator term is decomposed as follows:
\begin{equation}\nonumber
\begin{aligned}
\sum_{\genfrac{}{}{0pt}{}{j\geq J_0-1}{|j-j'|\leq 4}} 2^{s j}\|[\dot{\Delta}_{j}, \dot S_{j'-1} u] \dot{\Delta}_{j'} v\|_{L^2}\lesssim\bigg(\sum\limits_{\genfrac{}{}{0pt}{}{j'\geq J_0-1}{|j-j'|\leq 4}}
+ \sum\limits_{\genfrac{}{}{0pt}{}{J_0 -5 \leq j'\leq J_0-2}{|j-j'|\leq 4}} \bigg)2^{sj}
\|[\dot{\Delta}_{j}, \dot{S}_{j'-1} u] \Delta_{j'} v\|_{L^2}.
\end{aligned}
\end{equation}
It follows from the classical commutator estimate in \cite{bahouri1}[Proposition 2.97] Bernstein's lemma that
\begin{equation}\nonumber
\begin{aligned}
&\quad\sum_{\genfrac{}{}{0pt}{}{j'\geq J_0-1}{|j-j'|\leq 4}}  2^{sj}
\|[\dot{\Delta}_{j}, \dot{S}_{j'-1} u] \Delta_{j'} v\|_{L^2}\\
&\lesssim \sum_{\genfrac{}{}{0pt}{}{j'\geq J_0-1}{|j-j'|\leq 4}} 2^{(s-1)(j-j')}(2^{sj'}\|\Delta_{j'} v\|_{L^2})
(2^{-j'}\|\nabla \dot{S}_{j'-1} u\|_{L^{\infty}})\lesssim \|\nabla u\|_{\dot{B}_{\infty,1}^{-1}}
\|v\|_{\dot{B}_{2,1}^{s}}^{h}
\lesssim
\|u\|_{\dot{B}_{p,1}^{\frac{d}{p}}}
\|v\|_{\dot{B}_{2,1}^{s}}^{h}.
\end{aligned}
\end{equation}
On the other hand, since $\frac{2p}{p-2}\geq p$ and $\frac{d}{2}-\frac{d}{p}-1< 0$ due to $2\leq p\leq \min\{4,\frac{d}{d-2}\}$ for $d\geq 2$ and $2\leq p\leq 4$ for $d=1$, the commutator estimate, as well as $\dot{\Delta}_{j}u=\dot{\Delta}_{j}u^{\ell}$ for all $j\leq J_0-2$, also leads to
\begin{equation}\nonumber
\begin{aligned}
&\quad\sum\limits_{\genfrac{}{}{0pt}{}{J_0 -5 \leq j'\leq J_0-2}{|j-j'|\leq 4}} 2^{sj}
\|[\dot{\Delta}_{j}, \dot{S}_{j'-1} u] \Delta_{j'} v\|_{L^2}\\
&\lesssim
\sum\limits_{J_0 -5 \leq j'\leq J_0-2}\Big(2^{(s+\frac{d}{p}-\frac{d}{2}) j'}\|\Delta_{j'} v\|_{L^p}\Big)
\Big(2^{(\frac{d}{2}-\frac{d}{p}-1) j'}\|\nabla \dot{S}_{j'-1} u\|_{L^{\frac{2p}{p-2}}}\Big) \sum_{ |j-j'|\leq 4} 2^{(s-1)(j-j')}\\
&\lesssim
\|v^{\ell}\|_{\dot{B}_{p,1}^{s+\frac{d}{p}-\frac{d}{2}}} \|\nabla u^{\ell}\|_{\dot{B}_{\frac{2p}{p-2},1}^{\frac{d}{2}-\frac{d}{p}-1}}\lesssim
\|u^{\ell}\|_{\dot{B}_{p,1}^{\frac{d}{p}}}\|v^{\ell}\|_{\dot{B}_{p,1}^{s+\frac{d}{p}-\frac{d}{2}}}.
\end{aligned}
\end{equation}
Hence, it follows that
\begin{equation}\label{Tba}
\begin{aligned}
\|\dot{T}_{u} v\|_{\dot{B}_{2,1}^{s}}^{h} \lesssim \|u\|_{\dot{B}_{p,1}^{\frac{d}{p}}}(\|v\|_{\dot{B}_{2,1}^{s}}^{h}+ \|v^{\ell}\|_{\dot{B}_{p,1}^{s+\frac{d}{p}-\frac{d}{2}}}).
\end{aligned}
\end{equation}
The term $\|T_{v} u\|_{\dot{B}_{2,1}^{s}}^{h}$ can be treated in a similar line. Finally, as $s>0$, the classical remainder estimate (cf. \cite[Theorem 2.85]{bahouri1}) gives rise to
\begin{equation}\nonumber
\begin{aligned}
\|\dot{R}[u,v]\|_{\dot{B}_{2,1}^{s}}^{h}
&\leq
\|\dot{R}[u^h,v]\|_{\dot{B}_{2,1}^{s}}^{h}
+\|\dot{R}[u^\ell,v]\|_{\dot{B}_{2,1}^{s}}^{h}
\lesssim
\|u\|_{\dot{B}_{2,1}^{s}}^h \|v\|_{\dot{B}_{p,1}^{\frac{d}{p}}}
+
\|u^{\ell}\|_{\dot{B}_{p,1}^{s-\frac{d}{2}+\frac{d}{p}}}\|v\|_{\dot{B}_{p,1}^{\frac{d}{p}}} .
\end{aligned}
\end{equation}
This completes the proof of Lemma \ref{NonClassicalProLaw1}.
\end{proof}

As a consequence of \eqref{product2}, \eqref{newproduct} and Bernstein's lemma, we have the following product law in hybrid spaces.

\begin{coro}\label{coroApp1}
Let $s>0$, $2\leq p\leq 4$ for $d=1$ and $2\leq p\leq \min\{4,\frac{2d}{d-2}\}$ for $d\geq 2$. Then, there exists a constant $C$ independent of $J_0$ such that
\begin{align}
&\|uv\|_{\dot{\mathbb{B}}^{\frac{d}{p},\frac{d}{2}}_{p,2}}\leq C \|u\|_{\dot{\mathbb{B}}^{\frac{d}{p},\frac{d}{2}}_{p,2}} \|v\|_{\dot{\mathbb{B}}^{\frac{d}{p},\frac{d}{2}}_{p,2}},\label{realmoser}\\
&\|uv\|_{\dot{\mathbb{B}}^{\frac{d}{p},\frac{d}{2}+1}_{p,2}}\leq C \|u\|_{\dot{\mathbb{B}}^{\frac{d}{p},\frac{d}{2}+1}_{p,2}} \|v\|_{\dot{\mathbb{B}}^{\frac{d}{p},\frac{d}{2}+1}_{p,2}},\label{realmoser1}\\
&\|uv\|_{\dot{\mathbb{B}}^{\frac{d}{p}+1,\frac{d}{2}+1}_{p,2}}\leq C \|u\|_{\dot{\mathbb{B}}^{\frac{d}{p},\frac{d}{2}}_{p,2}} \|v\|_{\dot{\mathbb{B}}^{\frac{d}{p}+1,\frac{d}{2}+1}_{p,2}}+C \|v\|_{\dot{\mathbb{B}}^{\frac{d}{p},\frac{d}{2}}_{p,2}} \|u\|_{\dot{\mathbb{B}}^{\frac{d}{p}+1,\frac{d}{2}+1}_{p,2}}.\label{realmoser2}
\end{align}
\end{coro}

Finally, we prove new commutator estimates to handle the spatial derivatives of $A^i(W)$ in \eqref{entropeq} for high frequencies. 

\begin{lemma}\label{lemmacommutator}
Let $2\leq p\leq \min\{4,\frac{d}{d-2}\}$ and $s > 0$. There exists a constant $C$ independent of $J_0$ such that
\begin{equation}\label{mmmnew}
\begin{aligned}
&\quad \sum\limits_{j\geq J_0-1} 2^{sj}\|[w,\dot{\Delta}_{j}]z\|_{L^2} \leq    C\Big(\|\nabla w\|_{\dot{B}^{\frac{d}{p}}_{p,1}}\|z\|_{\dot{\mathbb{B}}^{s+\frac{d}{p}-\frac{d}{2}-1,s-1}_{p,2}}+\|z\|_{\dot{B}^{\frac{d}{p}}_{p,1}}\|\nabla w\|_{\dot{\mathbb{B}}^{s+\frac{d}{p}-\frac{d}{2}-1,s-1}_{p,2}}\Big).
\end{aligned}
\end{equation}
\end{lemma}

\begin{proof}
Recalling Bony's decomposition, we rewrite $[w,\dot{\Delta}_{j}] z$ as follows.
\begin{equation}\nonumber
\begin{aligned}
\relax[w,\dot{\Delta}_{j}] z&=\sum_{|j'-j|\leq 4}[\dot{S}_{j'-1}w,\dot{\Delta}_{j}]  \dot{\Delta}_{j'}z-\dot{\Delta}_{j} \dot{T}_{ z}w-\dot{\Delta}_{j} \dot{R}(w, z)+({\rm Id}-\dot{S}_{j-1}) w  \dot{\Delta}_{j} z+\sum_{|j-j'|\leq1} (\dot{S}_{j-1}w-\dot{S}_{j'-1}w)
\dot{\Delta}_{j}\dot{\Delta}_{j'}z.
\end{aligned}
\end{equation}
First, the classical commutator estimate
(see \cite{bahouri1}[Proposition 2.97]) leads to
\begin{equation}\nonumber
\begin{aligned}
&\sum_{j\geq J_0-1}2^{sj} \sum_{|j'-j|\leq 4}\|[\dot{S}_{j'-1}w,\dot{\Delta}_{j}]  \dot{\Delta}_{j'}z\|_{L^2} &\lesssim \sum_{j\geq J_0-1}2^{sj} \sum_{|j'-j|\leq 4}\|\nabla \dot{S}_{j'-1} w\|_{L^{\infty}} 2^{-j}\|\dot{\Delta}_{j'} z\|_{L^2}\lesssim \|\nabla w\|_{L^{\infty}} \|z\|_{\dot{B}^{s-1}_{2,1}}^{h}.
\end{aligned}
\end{equation}
Owing to \eqref{Tba} and $2\leq p\leq \min\{4,\frac{d}{d-2}\}$, $\dot{T}_{ z} w$ is analyzed by
\begin{equation}\nonumber
\begin{aligned}
\|\dot{T}_{ z} w\|_{\dot{B}^{s}_{2,1}}&\lesssim \|z\|_{\dot{B}_{p,1}^{\frac{d}{p}}}\Big(\|w\|_{\dot{B}_{2,1}^{s}}^{h}+ \|w^{\ell}\|_{\dot{B}_{p,1}^{s+\frac{d}{p}-\frac{d}{2}}}\Big)\lesssim \|z\|_{\dot{B}^{\frac{d}{p}+1}_{p,1}}\|w\|_{\dot{\mathbb{B}}^{s+\frac{d}{p}-\frac{d}{2},s}_{p,2}}.
\end{aligned}
\end{equation}
With respect to the remainder $\dot{R}(w,\partial_{x_k}z)$, we deduce from $s>0$ and $2\leq p\leq 4$ that
\begin{equation}\nonumber
\begin{aligned}
\|\dot{R}(w,z)\|_{\dot{B}^{s}_{2,1}}&\leq \|w\|_{\dot{B}^{\frac{d}{2}-\frac{d}{p}+1}_{\frac{2p}{p-2},\infty}} \|z\|_{\dot{B}^{s+\frac{d}{p}-\frac{d}{2}-1}_{p,1}}\lesssim \|w\|_{\dot{B}^{\frac{d}{p}+1}_{p,1}}\|z\|_{\dot{B}^{s+\frac{d}{p}-\frac{d}{2}-1}_{p,1}}\lesssim \|w\|_{\dot{B}^{\frac{d}{p}+1}_{p,1}}\|z\|_{\dot{\mathbb{B}}^{s+\frac{d}{p}-\frac{d}{2}-1,s-1}_{p,2}}.
\end{aligned}
\end{equation}
One concludes from Bernstein's and Young's inequalities that
\begin{equation}\nonumber
\begin{aligned}
&\sum_{j\geq J_0-1}2^{sj}  \|({\rm Id}-\dot{S}_{j-1}) w  {\Delta}_{j}z\|_{L^2}\\
&\quad\lesssim \sum_{j\geq J_0-1}   \sum_{j'\geq j} 2^{(\frac{d}{p}+1)j'} \|\dot{\Delta}_{j'}w\|_{L^p} 2^{(\frac{d}{p}+1)(j-j')} 2^{j(s-1)}  \|\dot{\Delta}_{j} z\|_{L^2}\lesssim \|w\|_{\dot{B}^{\frac{d}{p}+1}_{p,1}} \|z\|_{\dot{B}^{s-1}_{2,1}}^h.
\end{aligned}
\end{equation}
Finally, it also holds that
\begin{equation}\nonumber
\begin{aligned}
&\sum_{j\geq J_0-1} 2^{js}\sum_{|j-j'|\leq1} \|(\dot{S}_{j'-1}w-\dot{S}_{j-1}w)\dot{\Delta}_{j}\dot{\Delta}_{j'}z\|_{L^2}\\
&\quad\lesssim \sum_{j\geq J_0-1}  2^{j}(\|\dot{\Delta}_{j- 1}w\|_{L^{\infty}}+\|\dot{\Delta}_{j- 2}w\|_{L^{\infty}}) 2^{j(s-1)} \|\dot{\Delta}_{j}z\|_{L^2}\lesssim \|w\|_{\dot{B}^{1}_{\infty,\infty}}\|z\|_{\dot{B}^{s-1}_{2,1}}^h\lesssim \|w\|_{\dot{B}^{\frac{d}{p}+1}_{p,1}} \|z\|_{\dot{B}^{s-1}_{2,1}}^h.
\end{aligned}
\end{equation}
Collecting the above estimates, 
we arrive at \eqref{mmmnew}.
\end{proof}


\section{The compressible Euler system with damping: A motivation} \label{appendixB}
\setcounter{equation}{0}

As a typical example of partially dissipative hyperbolic systems, we study the isentropic compressible Euler equations with damping:
\begin{equation}\label{euler}
\left\{
\begin{aligned}
& \partial_{t}\rho+\dive (\rho u)=0,\\
& \partial_{t}(\rho u)+\dive (\rho u\otimes u)+\nabla P(\rho)=-\rho u,\quad x\in\mathbb{R}^{d},\quad t>0,
\end{aligned}
\right.
\end{equation}
where $\rho=\rho(t,x)$ denotes the fluid density, $u$ stands for the fluid velocity, and $P=P(\rho)$ is the pressure satisfying $P'(\rho)>0$. We consider the Cauchy problem of \eqref{euler} with initial data
\begin{align}
&(\rho, u)(0,x)=(\rho_0,u_0)(x),\quad x\in\mathbb{R}^{d} .\label{eulerd}
\end{align}

The system \eqref{euler} describes the compressible gas flow as it passes through a
porous medium and the medium induces a friction force, proportional to the linear
momentum in the opposite direction. To the best of our knowledge, there are lots of significant results, such as the well-posedness, blow-up, global-in-time existence and large-time behavior of classical solutions, singular limit problems and so on. See \cite{Nishida1,dafermos1,Hsiao,s1,wangyang1,tanwu1} and references therein.
As shown by \cite{XK1}, one can see that the system \eqref{euler} satisfies the structural conditions \eqref{H1}-\eqref{blockL}. Consequently, with  suitable modifications, Theorems \ref{theorem0}, \ref{thm1} \ref{thm2}, \ref{thm3} can be adapted to the Cauchy problem \eqref{euler}-\eqref{eulerd}.

This section may be regarded as an independent part, which also serves as the
main motivation of the present paper; nevertheless, we include it here for
completeness. When dealing with the nonlinear problem, the momentum structure
ensures a conservation-law formulation, which in turn ensures that the
large-time behavior is asymptotically equivalent to that of the linearized
system (see \eqref{m1} and Section~\ref{section:decay}). Accordingly, one may
rewrite \eqref{euler} in terms of $(\rho,m)$ with the momentum $m=\rho u$ and consider the linearization around a constant background density
$\bar\rho>0$:
\begin{equation}\label{Eulerlinear}
\left\{
\begin{aligned}
&\partial_{t}a+\dive m=0,\\
&\partial_{t} m+c_*^2\nabla a+m=0
\end{aligned}
\right.
\end{equation}
with $a\coloneqq  \rho-\bar{\rho}$, where $c_*=\sqrt{P'(\bar{\rho})}$ denotes the sound speed. Taking the Fourier transform
of \eqref{Eulerlinear}, we have
\begin{equation*}
\frac{d}{dt}\left(\begin{array}{l}
\widehat{a}\\
\widehat{m}
\end{array}
\right)
=A(\xi)
\left(\begin{array}{l}
\widehat{a}\\
\widehat{m}
\end{array}
\right),\quad\quad A(\xi)\coloneqq  \left(\begin{matrix}
0	&  -\mathrm{i}\xi^{\top}\\
- c_*^2 \mathrm{i}\xi	& -{\rm {I}_d}
\end{matrix}\right).
\end{equation*}
The eigenvalues of $A(\xi)$ are
\begin{equation}\nonumber
\begin{aligned}
\lambda_{1,2}(\xi)=-\frac{1}{2}\pm \frac{1}{2}r(\xi)\quad \text{with}\quad r(\xi)\coloneqq\sqrt{1-4c_*^2|\xi|^2},\quad\quad \lambda_3=\lambda_4=\cdot\cdot\cdot=\lambda_{d+1}=-1.
\end{aligned}
\end{equation}
The spectral analysis for \eqref{Eulerlinear} has been studied by the classical works \cite{s1,tanwu1,wangyang1}. For high frequencies  $|\xi|> \frac{1}{2c_*}$, $\lambda_{1}(\xi)$ and $\lambda_{2}(\xi)$ are conjugate complex, so the $L^2$ framework is sufficient in this regime and the solution $(a,m)$ exhibits an exponential decay. For low frequencies $|\xi|< \frac{1}{2c_*}$, the effects of damping and diffusion coexist, and the general $L^p$-estimates can be expected as eigenvalues are all real.

Moreover, we write $(a,m)^{\top}=G\ast (a_0,m_0)^{\top}$ with $(a_0, m_0)\coloneqq  (\rho_0-\bar{\rho}, \rho_0 u_0)$, where
$G$ denotes Green's function of \eqref{Eulerlinear}. It is easy to verify that $\widehat{G}(\xi,t)$ can be expressed as
\begin{equation}\nonumber
\begin{aligned}
\widehat{G}(\xi,t)=e^{\lambda_{1}(\xi)t}P_1+e^{\lambda_{2}(\xi)t}P_2+e^{\lambda_{3}(\xi)t}P_3,
\end{aligned}
\end{equation}
where the projection operators $P_{i}$ ($i=1,2,3$) are
\begin{equation}\nonumber
\begin{aligned}
P_1=\left(\begin{matrix}
\frac{1}{2}(1+\frac{1}{r(\xi)})	&  -\frac{\mathrm{i}\xi^{\top}}{r(\xi)}\\
-\frac{\mathrm{i}c_*^2\xi}{r(\xi)}	& \frac{1}{2}(1-\frac{1}{r(\xi)}) \frac{\xi \otimes\xi}{|\xi|^2}
\end{matrix}\right),\quad P_2=\left(\begin{matrix}
\frac{1}{2}(1-\frac{1}{r(\xi)})	&  \frac{\mathrm{i}\xi^{\top}}{r(\xi)}\\
\frac{\mathrm{i}c_*^2\xi}{r(\xi)}	& \frac{1}{2}(1+\frac{1}{r(\xi)}) \frac{\xi \otimes\xi}{|\xi|^2}
\end{matrix}\right),\quad P_3=\left(\begin{matrix}
0	&  0\\
0	& {\rm {I}_d}-\frac{\xi \otimes\xi}{|\xi|^2}
\end{matrix}\right).
\end{aligned}
\end{equation}
Using
$r(\xi)=1+\mathcal{O}(|\xi|^2)$, $\lambda_{1}(\xi)=-c_*^2|\xi|^2+\mathcal{O}(|\xi|^4)$, $\lambda_{2}(\xi)=-1+\mathcal{O}(|\xi|^2)$ and $1-\frac{1}{r(\xi)}=-2c_*^2|\xi|^2+\mathcal{O}(|\xi|^4)$, we have
\begin{equation}\label{a}
\begin{aligned}
\widehat{a}&= e^{\lambda_{1}(\xi)t}\bigg(\frac{1}{2}\Big(1+\frac{1}{r(\xi)} \Big)\widehat{a_0}-\frac{\mathrm{i}\xi\cdot \widehat{m_0}}{r(\xi)}\bigg)+e^{\lambda_{2}(\xi)t} \bigg(\frac{1}{2}\Big(1-\frac{1}{r(\xi)} \Big)\widehat{a_0}+\frac{\mathrm{i}\xi\cdot\widehat{m_0}}{r(\xi)} \bigg)\\
&\sim e^{-c_*^2|\xi|^2t} \widehat{\Psi_0}+\mathcal{O}(1)e^{-t}(|\xi|^2|\widehat{a_0}|+|\xi| |\widehat{m_0}|)
\end{aligned}
\end{equation}
and
\begin{equation}\label{m}
\begin{aligned}
\widehat{m}&= e^{\lambda_{1}(\xi)t}\bigg(-\frac{\mathrm{i}c^2_*\xi}{r(\xi)} \widehat{a_0}+\frac{1}{2}\Big(1-\frac{1}{r(\xi)} \Big)\frac{\xi(\xi\cdot \widehat{m_0})}{|\xi|^2}\bigg)\\
&\quad+e^{\lambda_{2}(\xi)t} \bigg(\frac{\mathrm{i}c^2_*\xi}{r(\xi)} \widehat{a_0}-\frac{1}{2}\Big(1+\frac{1}{r(\xi)} \Big)\frac{\xi(\xi\cdot \widehat{m_0})}{|\xi|^2}\bigg)+e^{-t}\Big(\widehat{m_0}-\frac{\xi(\xi\cdot \widehat{m_0})}{|\xi|^2}\Big)\\
&\sim -e^{-c_*^2|\xi|^2t}\mathrm{i}c^2_* \xi \widehat{\Psi_0}+\mathcal{O}(1)e^{-t}(|\xi||\widehat{a_0}|+|\widehat{m_0}|),
\end{aligned}
\end{equation}
where $\Psi_0$ is given by
\begin{align}
\Psi_0=a_0-\dive m_0.\label{Psi00}
\end{align}
It should be noted that the estimates \eqref{a}-\eqref{m} coincide with the localized $L^p$ estimates in Proposition \ref{LemmaspectrallocalHp}.  Consequently, in the $L^2$ framework, $a$ decays at the same rates as the heat equation, and $m$ decays half faster than $a$.

Inspired by the analysis of Green's function, $\Psi_0$ associated with the part $P_1(\widehat{a_0},\widehat{m_0})$ is a sharp quantity for decay estimates, and the low frequency regularity of $\Psi_0$ will play a key role in obtaining optimal decay estimates. However, the above asymptotic analysis may not be applied directly to the general $L^p$ framework with $p\neq 2$ since Parseval's equality is not available.

Furthermore, we
would like to use the linear Euler equations as a toy model and present the main idea for the sharp decay characterization of partially hyperbolic systems. Indeed, we develop the pure $L^p$ energy argument in the sense of spectral localization, which allows us not only to deduce decay estimates, but also to get asymptotic profiles for large times.
\begin{lemma}\label{lemc.1}
Let $1<p<\infty$. Assume that $(a^\lambda,m^\lambda)$ is a solution to
\eqref{Eulerlinear} with smooth initial data $(a_0^\lambda,m_0^\lambda)$, whose
Fourier transforms are supported in $\lambda\mathcal{C}$, where $\mathcal{C}$ is
an annulus and $0<\lambda\le\lambda_0$ for some sufficiently small $\lambda_0>0$.
Then $(a^\lambda,m^\lambda)$ satisfies the following decay estimates:
\begin{align}
 \| a^\lambda(t) \|_{L^p}
 &\lesssim e^{-\bar{c}\lambda^2 t}
 \Big(\| a^\lambda_0 \|_{L^p}+\lambda\| m^\lambda_0 \|_{L^p}\Big)
 +e^{-\frac{1}{2}t}
 \Big(\lambda^2\| a^\lambda_0 \|_{L^p}+\lambda\| m^\lambda_0 \|_{L^p}\Big),
 \label{V1asy}\\
 \| m^\lambda(t)\|_{L^p}
 &\lesssim e^{-\bar c\lambda^2 t}\lambda
 \Big(\| a^\lambda_0 \|_{L^p}+\lambda\| m^\lambda_0 \|_{L^p}\Big)
 +e^{-\frac{1}{2}t}
 \Big(\lambda\| a^\lambda_0 \|_{L^p}+\| m^\lambda_0 \|_{L^p}\Big).
 \label{masy}
\end{align}
Moreover, the following refined estimates hold:
\begin{align}
 \| a^\lambda(t)-e^{c_*^2 t\Delta}\Psi_0^\lambda \|_{L^p}
 &\lesssim e^{-\bar c\lambda^2 t}\lambda
 \Big(\| a^\lambda_0 \|_{L^p}+\lambda\| m^\lambda_0 \|_{L^p}\Big)
 +e^{-\frac{1}{2}t}
 \Big(\lambda^2\| a^\lambda_0 \|_{L^p}+\lambda\| m^\lambda_0 \|_{L^p}\Big),
 \label{V1asy1}\\
 \big\| m^\lambda(t)+c_*^2\nabla e^{c_*^2 t\Delta}\Psi_0^\lambda \big\|_{L^p}
 &\lesssim e^{-\bar c\lambda^2 t}\lambda^2
 \Big(\| a^\lambda_0 \|_{L^p}+\lambda\| m^\lambda_0 \|_{L^p}\Big)
 +e^{-\frac{1}{2}t}
 \Big(\lambda\| a^\lambda_0 \|_{L^p}+\| m^\lambda_0 \|_{L^p}\Big).
 \label{masy1}
\end{align}
Here $\bar c>0$ depends only on $c_*$ and $p$, and
$\Psi_0^\lambda:=a_0^\lambda-\dive  m_0^\lambda$.
\end{lemma}

\begin{proof}
Inspired by the spectral analysis, we introduce the effective quantity $\Psi^\lambda\coloneqq  a^\lambda-\dive m^\lambda
$, which solves
\begin{align}
\partial_{t}\Psi^\lambda-c_*^2\Delta \Psi^\lambda=c_*^2 \Delta \dive m^\lambda,\quad \Psi^\lambda|_{t=0}=\Psi^\lambda_0.
\label{y}
\end{align}
We also recall the damped mode associated with Darcy's law first introduced by \cite{c1}:
$$
z^\lambda \coloneqq c_*^2 \nabla a^\lambda+m^\lambda.
$$
Note that $a^\lambda=({\rm Id}+c_*^2 \Delta)^{-1}(\Psi^\lambda+\dive z^\lambda)$ and $m^\lambda=z^\lambda-c_*^2\nabla ({\rm Id}+c_*^2 \Delta)^{-1}(\Psi^\lambda+\dive z^\lambda)$. In terms of $(\Psi^\lambda, z^\lambda)$, one can perform a diagonalization of the system \eqref{Eulerlinear} as follows:
\begin{equation}\label{Psiz}
\left\{
\begin{aligned}
&\partial_{t}\Psi^\lambda-c_*^2 \Delta\Psi^\lambda=F(\Psi^\lambda,z^\lambda),\, \,\quad  \Psi^\lambda|_{t=0}=\Psi^\lambda_0,\\
&\partial_{t} z^\lambda+z^\lambda=G(\Psi^\lambda,z^\lambda),\quad\quad\quad \, \, \, \, z^\lambda|_{t=0}=z^\lambda_0\coloneqq c_*^2\nabla a^\lambda_0+m^\lambda_0,
\end{aligned}
\right.
\end{equation}
where the linear higher-order terms
$F(\Psi,z)$ and $G(\Psi,z)$ are given by
\begin{align*}
F(\Psi^\lambda,z^\lambda)&=c_*^2\dive\Delta \Big(z-c_*^2\nabla(1+c_*^2 \Delta)^{-1}(\Psi^\lambda+\dive z^\lambda)\Big),\\
G(\Psi^\lambda,z^\lambda)&=-c_*^2\nabla\dive\Big(z-c_*^2\nabla(1+c_*^2 \Delta)^{-1}(\Psi^\lambda+\dive z^\lambda)\Big).
\end{align*}
By employing Lemma \ref{lemmaA2} and the standard $L^p$-energy methods, we obtain
\begin{equation}\label{rrra}
\begin{aligned}
\frac{1}{p}\frac{d}{dt}\|\Psi^\lambda\|^{p}_{L^p}
+c_p^*\lambda^2\|\Psi^\lambda\|^{p}_{L^p}\leq \|F(\Psi^\lambda,z^\lambda)\|_{L^p}
\|\Psi^{\lambda}\|^{p-1}_{L^p},
\end{aligned}
\end{equation}
which implies that
\begin{equation}\label{rrra2}
\begin{aligned}
\frac{d}{dt}\|\Psi^\lambda\|_{\varepsilon,L^p}+c_p^*\lambda^2\|\Psi^\lambda\|_{\varepsilon,L^p}\leq \|F(\Psi^\lambda,z^\lambda)\|_{L^p}+c_p^*\lambda^2\varepsilon,
\end{aligned}
\end{equation}
where, using the notation  $\|\cdot\|_{\varepsilon,\, L^p}\coloneqq  (\|\cdot\|_{L^p}^p+\varepsilon^p)^{1/p}$. Similarly, one has
\begin{equation}\label{rrra3}
\begin{aligned}
\frac{d}{dt}\|\dive z^\lambda\|_{\varepsilon,L^p}+\|\dive z^\lambda\|_{\varepsilon,L^p}\leq
\|\dive G(\Psi^\lambda,z^\lambda)\|_{L^p}+
\varepsilon.
\end{aligned}
\end{equation}
Consequently, it follows from \eqref{rrra2} and \eqref{rrra3} that
\begin{equation}\label{combine}
\begin{aligned}
&\frac{d}{dt}\Big(\|\Psi^\lambda\|_{\varepsilon,L^p}+\|\dive z^\lambda\|_{\varepsilon,L^p}\Big)
+c_p^*\lambda^2\|\Psi^\lambda\|_{\varepsilon,L^p}+\|\dive z^\lambda\|_{\varepsilon,L^p}\\
&\leq \|F(\Psi^\lambda,z^\lambda)\|_{L^p}+\|\dive G(\Psi^\lambda,z^\lambda)\|_{L^p}+M_{\varepsilon},
\end{aligned}
\end{equation}
where $M_{\varepsilon}\coloneqq   c_p^*\lambda^2\varepsilon+\varepsilon$. Note that the operator $({\rm{Id}}+c_*^2\Delta)^{-1}$ is bounded in the subset of $L^p$ where any functions are truncated at low frequencies. As $\lambda\leq \lambda_0\ll1$, one can check that 
\begin{align}
&\|F(\Psi^\lambda,z^\lambda)\|_{L^p}\lesssim  \lambda^2 \|\dive z^\lambda\|_{L^p}+\lambda^4 \|\Psi^\lambda\|_{L^p},\label{non-f}\\
&\|\dive G(\Psi^\lambda,z^\lambda)\|_{L^p}\lesssim \lambda^2 \|\dive z^\lambda\|_{L^p}+\lambda^4 \|\Psi^\lambda\|_{L^p} .\label{non-g}
\end{align}
Substituting \eqref{non-f} and \eqref{non-g} into \eqref{combine} and using $\lambda\leq \lambda_0<<1$, we conclude that there exists a constant $c>0$ depending only on $c_p^*$ such that
\begin{equation}\label{combine2}
\begin{aligned}
\frac{d}{dt}\Big(\|\Psi^\lambda\|_{\varepsilon,L^p}+\|\dive z^\lambda\|_{\varepsilon,L^p}\Big)
+\bar{c}\lambda^2\Big(\|\Psi^\lambda\|_{\varepsilon,L^p}+\|\dive z^\lambda\|_{\varepsilon,L^p}\Big)\leq M_{\varepsilon}.
\end{aligned}
\end{equation}
Integrating in time on $[0,t]$, having $\varepsilon$ tends to $0$ and using Grönwall's inequality, we get
\begin{equation}\label{combine3}
\begin{aligned}
\|\Psi^\lambda\|_{L^p}+\|\dive z^\lambda\|_{L^p}\lesssim
e^{-\bar{c}\lambda^2 t}\Big(\|\Psi^\lambda_0\|_{L^p}+\lambda \|z^\lambda_0\|_{L^p}\Big) \quad \mbox{for} \quad t>0.
\end{aligned}
\end{equation}
Furthermore, we rewrite the equation of $z$ as
\begin{equation}
\begin{aligned}
\partial_tz^\lambda +\mathcal{D} z^\lambda=c_*^4\nabla\Delta (1+c_*^2\Delta)^{-1}\Psi^\lambda,\label{znew}
\end{aligned}
\end{equation}
with $\mathcal{D}\coloneqq  \rm {Id}+c_*^2\nabla \dive +c_*^4\nabla(1+c_*^2\Delta)^{-1}\dive $. Since $\lambda$ is small enough, one knows that $\|e^{-t\mathcal{D}}f^\lambda\|_{L^p}\lesssim e^{-\frac{t}{2}}\|f^\lambda\|_{L^p}$ for any distribution $f^\lambda$ supported in $\lambda \mathcal{C}$. Applying Duhamel's principle to \eqref{znew} implies that
(with the aid of \eqref{combine3})
\begin{equation}\label{combine4}
\begin{aligned}
\|z^\lambda\|_{L^p}&\leq
e^{-\frac{t}{2}}\|z^\lambda_0\|_{L^p}+\lambda^3\int^{t}_{0}e^{-\frac{1}{2}(t-\tau)}\|\Psi^\lambda\|_{L^p} \,d\tau\\& \lesssim
e^{-\frac{t}{2}}\|z^\lambda_0\|_{L^p}+
\lambda^3 e^{-\bar{c}\lambda^2 t}\Big(\|\Psi^\lambda_0\|_{L^p}+\lambda \|z^\lambda_0\|_{L^p}\Big),
\end{aligned}
\end{equation}
where we have used \eqref{maomao}.

We now recover the desired estimates of $(a,m)$ from $(\Psi^\lambda,z^\lambda)$. Indeed, recalling $a^\lambda=({\rm Id}+c_*^2 \Delta)^{-1}(\Psi^\lambda+\dive z^\lambda)$, it follows from  \eqref{combine3} and Bernstein's inequality that
\begin{equation*}
\begin{aligned}
 \|a^\lambda \|_{L^p}&\lesssim   \| \Psi^\lambda \|_{L^p}+  \| \dive z^\lambda \|_{L^p}\lesssim e^{-\bar{c}\lambda^2 t}\Big(\|\Psi^\lambda_0\|_{L^p}+\lambda \|z^\lambda_0\|_{L^p}\Big)+e^{-\frac{1}{2}t} \lambda \|z^\lambda_0\|_{L^p}.
\end{aligned}
\end{equation*}
Similarly, observing that $m^\lambda=z^\lambda-c_*^2\nabla ({\rm Id}+c_*^2 \Delta)^{-1}(\Psi^\lambda+\dive z^\lambda)$, one can get
\begin{equation*}
\begin{aligned}
 \| m^\lambda \|_{L^p}&\lesssim \| z^\lambda \|_{L^p}+\lambda \| \Psi^\lambda \|_{L^p}\lesssim  \lambda e^{-\bar{c}\lambda^2 t}\Big(\|\Psi^\lambda_0\|_{L^p}+\lambda \|z^\lambda_0\|_{L^p}\Big)+e^{-\frac{1}{2}t}\|z^\lambda_0\|_{L^p}.
\end{aligned}
\end{equation*}
According to the definitions of $\Psi_0$ and $z_0$, we obtain \eqref{V1asy}-\eqref{masy} for $a$ and $m$.

Furthermore, for \eqref{y}, by employing Duhamel's principle,  we deduce that
\begin{equation*}
\begin{aligned}
\|(\Psi^\lambda-e^{c_*^2t \Delta}\Psi^\lambda_0)(t)\|_{L^p}&\lesssim \lambda^3 \int_0^t e^{-C^*\lambda^2 (t-\tau)} \|m^\lambda\|_{L^p}\,d\tau\\
&\lesssim \lambda^3 \int_0^te^{-C^*\lambda^2 (t-\tau)} e^{-\bar{c}\lambda^2 \tau}\,d\tau \sup_{\tau\in[0,t]}\|e^{\bar{c}\lambda^2\tau}m^\lambda\|_{L^p}\lesssim  \lambda e^{-\bar{c}\lambda^2 t
}\sup_{\tau\in[0,t]}\|e^{\bar{c}\lambda^2\tau}m^\lambda\|_{L^p}
\end{aligned}
\end{equation*}
for some constant $C^*$ satisfying $0<\frac {c}{2}\leq C^*$, where we used \eqref{maomao}.

Together with $a^\lambda=\Psi^\lambda+\dive m^\lambda$ and  \eqref{masy}, we get  \eqref{V1asy1}. Finally, according to the definition of $z^\lambda$, we write  $m^\lambda+c_*^2\nabla e^{c_*^2t \Delta}\Psi^\lambda_0=-c_*^2\nabla(a^\lambda-e^{c_*^2 t \Delta}\Psi^\lambda_0)+z^\lambda$, so \eqref{masy1} follows from \eqref{V1asy1} and \eqref{combine4} directly.
\end{proof}

Those pointwise estimates in Lemma
\eqref{lemc.1} enable us to establish the sharp decay characterization for \eqref{Eulerlinear} in terms of the Besov regularity.
\begin{lemma}
Assume that $(a_0,m_0)$ satisfies $a_0^{\ell}\in \dot{B}^{\sigma_1}_{p,\infty}$,  $m_0^{\ell}\in \dot{B}^{\sigma_1+1}_{p,\infty}$ and $\|(a_0,m_0)\|_{\dot{B}^{s}_{2,1}}^h<\infty$ with $1<p<\infty$ and $\sigma_1, s\in \mathbb{R}$. It holds that
\begin{align}
&\|a(t)\|_{\dot{\mathbb{B}}^{\sigma,s}_{p,2}}\lesssim (1+t)^{-\frac{1}{2}(\sigma-\sigma_1)}, \,\,\quad\quad\quad~~ \sigma>\sigma_1,\label{C19}\\
&\|m(t)\|_{\dot{\mathbb{B}}^{\sigma',s}_{p,2}}\lesssim (1+t)^{-\frac{1}{2}(\sigma'-\sigma_1+1)}, \quad \sigma'>\sigma_1+1
\end{align}
for all $t>0$. Furthermore, if $\Psi_0^{\ell}\in \mathcal{\dot{B}}^{\sigma_1}_{p,\infty}$, then we have
\begin{align}
 (1+t)^{-\frac{1}{2}(\sigma-\sigma_1)}&\lesssim \|a(t)\|_{\dot{\mathbb{B}}^{\sigma,s}_{p,2}}\lesssim (1+t)^{-\frac{1}{2}(\sigma-\sigma_1)}, \,\,\quad\quad\quad~~ \sigma>\sigma_1,\label{C23}\\
 (1+t)^{-\frac{1}{2}(\sigma'-\sigma_1+1)}&\lesssim \|m(t)\|_{\dot{\mathbb{B}}^{\sigma',s}_{p,2}}\lesssim (1+t)^{-\frac{1}{2}(\sigma'-\sigma_1+1)},\quad \sigma'>\sigma_1+1\label{C24}
\end{align}
for $t\geq t_*$ with some time $t_*>0$.
\end{lemma}

\begin{proof}
It is well-known that the high-frequency norm $\|(a,m)\|_{\dot{B}^{s}_{2,1}}^h$ decays exponentially in time; see, for example \cite{XK2}. What remains is to bound the large time behavior of low frequencies of $(a,m)$ for $j\leq J_0$ (choosing sufficiently small if necessary). Following from the same procedure that leads to \eqref{V1asy} and \eqref{masy}, we have
\begin{align}
 \|\dot{\Delta}_{j} a(t) \|_{L^p}&\lesssim e^{-\bar{c} 2^{2j}t}  \Big(\|  \dot{\Delta}_{j}a_0 \|_{L^p}+ 2^j\| \dot{\Delta}_{j} m_0 \|_{L^p}\Big)+e^{-\frac{1}{2}t}\Big(2^{2j}\| \dot{\Delta}_{j} a_0 \|_{L^p}+ 2^j\| \dot{\Delta}_{j}m_0 \|_{L^p}\Big), \label{tildePsi000} \\
\| \dot{\Delta}_{j}m(t) \|_{L^p}& \lesssim e^{-\bar{c} 2^{2j} t} 2^j\Big(\| \dot{\Delta}_{j} a_0 \|_{L^p}+ 2^j\|\dot{\Delta}_{j} m_0 \|_{L^p}\Big)+e^{-\frac{1}{2}t}\Big(2^j\|\dot{\Delta}_{j} a_0 \|_{L^p}+\| \dot{\Delta}_{j}m_0 \|_{L^p}\Big).\label{masy000}
\end{align}
It follows from \eqref{opoo} and \eqref{tildePsi000} that
\begin{equation}\label{rhodecayasy}
\begin{aligned}
\|a^{\ell}(t)\|_{\dot{B}^{\sigma}_{p,1}}&\lesssim \langle t\rangle^{-\frac{1}{2}(\sigma-\sigma_1)}\Big(\|a_0^{\ell}\|_{\dot{B}^{\sigma_1}_{p,\infty}}+\|m_0^{\ell}\|_{\dot{B}^{\sigma_1+1}_{p,\infty}}\Big)
\end{aligned}
\end{equation}
for any $\sigma>\sigma_1$ and $t>0$.

In view of \eqref{masy000}, for $\sigma>\sigma_1+1$ and $t>0$, we also get
\begin{equation}\label{mdecayasy}
\begin{aligned}
\|m^{\ell}(t)\|_{\dot{B}^{\sigma'}_{p,1}}\lesssim \langle t\rangle^{-\frac{1}{2}(\sigma'-\sigma_1+1)}\Big(\|a_0^{\ell}\|_{\dot{B}^{\sigma_1}_{p,\infty}}+\|m_0^{\ell}\|_{\dot{B}^{\sigma_1+1}_{p,\infty}}\Big).
\end{aligned}
\end{equation}

In order to derive lower bounds, we additionally assume that the initial quantity $\Psi_0$ satisfies $\Psi_0^{\ell}\in\dot{\mathcal{B}}^{\sigma_1}_{p,\infty}$ with any $1\leq p\leq \infty$ (see \eqref{Bsubset}). It follows from Proposition \ref{propgeneral2} that
\begin{equation}\label{heatdecay}
\begin{aligned}
&\langle t\rangle^{-\frac{1}{2}(\sigma-\sigma_1)}\lesssim \|e^{c_*^2t \Delta}\Psi_0^{\ell}\|_{\dot{B}^{\sigma}_{p,1}}\lesssim  \langle t\rangle^{-\frac{1}{2}(\sigma-\sigma_1)},\quad \sigma>\sigma_1.
\end{aligned}
\end{equation}
On the other hand, having \eqref{V1asy1} and \eqref{masy1}, it is not difficult to deduce that
\begin{align}
\|(a^{\ell}-e^{c_*^2t \Delta}\Psi_0^{\ell})(t)\|_{\dot{B}^{\sigma}_{p,1}}&\lesssim (1+t)^{-\frac{1}{2}(\sigma-\sigma_1+1)},\quad \sigma>\sigma_1 \label{C21},\\
\Big\| m^{\ell}-\Big(-c_*^2\nabla e^{c_*^2t \Delta}\Psi_0^{\ell}\Big)(t)  \Big\|_{\dot{B}^{\sigma}_{p,1}}&\lesssim (1+t)^{-\frac{1}{2}(\sigma'-\sigma_1+2)},\quad \sigma'>\sigma_1+1.\label{C22}
\end{align}
Furthermore, it follows from \eqref{heatdecay}-\eqref{C22} that
\begin{equation}\nonumber
\begin{aligned}
\|a^{\ell}(t)\|_{\dot{B}^{\sigma}_{p,1}}&\geq \|e^{c_*^2t \Delta}\Psi_0^{\ell}\|_{\dot{B}^{\sigma}_{p,1}}-\|(a^{\ell}-e^{c_*^2t \Delta}\Psi_0^{\ell})(t)\|_{\dot{B}^{\sigma}_{p,1}}\gtrsim \langle t\rangle^{-\frac{1}{2}(\sigma-\sigma_1)}
\end{aligned}
\end{equation}
 and
\begin{equation}\nonumber
\begin{aligned}
\|m^{\ell}(t)\|_{\dot{B}^{\sigma'}_{p,1}}&\geq\|\nabla e^{c_*^2\Delta  t} \Psi_0^{\ell}\|_{\dot{B}^{\sigma'}_{p,1}}- \|(m^{\ell}- c_*^2\nabla e^{c_*^2\Delta  t} \Psi_0^{\ell})(t)\|_{\dot{B}^{\sigma'}_{p,1}}\gtrsim \langle t\rangle^{-\frac{1}{2}(\sigma'-\sigma_1+1)},
\end{aligned}
\end{equation}
when $t$ is suitably large. Hence, \eqref{C23} and \eqref{C24} follow directly.
\end{proof}


\vspace{3mm}

\textbf{Acknowledgments}
L.-Y. Shou is supported by the National Natural Science Foundation of China (12301275). J. Xu is partially supported by the National Natural Science Foundation of China (12271250). P. Zhang is partially  supported by National Key R$\&$D Program of China under grant 2021YFA1000800 and by the Natural Science Foundation of China under Grants  No. 12421001, No. 12494542 and No. 12288201.








\vspace{5ex}


(L.-Y. Shou)\par\nopagebreak
\noindent\textsc{School of Mathematical Sciences, Ministry of Education Key Laboratory of NSLSCS, and Key Laboratory of Jiangsu Provincial Universities of FDMTA, Nanjing Normal University, Nanjing 210023, China}

Email address: {\texttt{shoulingyun11@gmail.com}}

\vspace{3ex}

(J. Xu)\par\nopagebreak
\noindent\textsc{School of Mathematics, Nanjing University of Aeronautics and
Astronautics, Nanjing, 211106, China}

Email address: {\texttt{jiangxu\underline{~}79@nuaa.edu.cn}}

\vspace{3ex}

(P. Zhang)\par\nopagebreak
\noindent\textsc{State Key Laboratory of Mathematical Sciences, Academy of Mathematics $\&$ Systems Science, The Chinese Academy of
	Sciences, Beijing 100190, China, and School of Mathematical Sciences,
	University of Chinese Academy of Sciences, Beijing 100049, China}

Email address: {\texttt{zp@amss.ac.cn}}

\end{document}